\documentclass{amsart}
\usepackage{parskip}
\usepackage{lipsum}
\usepackage{tikz-cd}
\usepackage{hyperref}
\usepackage{graphicx,color}
\usepackage{stix}
\usepackage{amssymb,amsmath,bm,mathtools}
\usepackage{bm}
\usepackage{enumerate}
\usepackage{tkz-graph}
\usepackage{ragged2e}
\usepackage{enumitem}
\usepackage{adjustbox}
\usepackage{tipa}
\usepackage{comment}
\usepackage{tikz}
\usepackage[all]{xy}

\usetikzlibrary{arrows,decorations.pathmorphing}
\usetikzlibrary{backgrounds,positioning}
\usetikzlibrary{fit,petri,shapes.misc}
\usetikzlibrary{shapes.geometric}
\usetikzlibrary{decorations.markings}
\usetikzlibrary{calc}

\newlength{\defaultpgflinewidth}
\tikzset{auto}
\tikzset{empty/.style={circle,inner sep=0pt,minimum size=6mm}}
\tikzset{emptyvt/.style={circle,inner sep=0pt,minimum size=0mm}}

\tikzset{plain/.style={circle,draw,very thick,
inner sep=0pt,minimum size=6mm}}

\tikzset{xplain/.style={circle,draw,very thick,
inner sep=0pt,minimum size=8mm}}

\tikzset{smallplain/.style={circle,draw,very thick,
inner sep=0pt,minimum size=4mm}}

\tikzset{xsplain/.style={circle,draw, thick,
inner sep=0pt,minimum size=1.5mm}}

\tikzset{tinyplain/.style={circle,draw, thick,
inner sep=0pt,minimum size=1mm}}

\tikzset{smalldotted/.style={circle,draw,very thick, densely dotted,
inner sep=0pt,minimum size=3mm}}

\tikzset{dottedplain/.style={circle,draw,very thick, densely dotted,
inner sep=0pt,minimum size=6mm}}

\tikzset{normaldot/.style={circle,black,fill=black,inner sep=0pt,minimum size=2mm}}

\tikzset{rectplain/.style={rectangle,draw,very thick,minimum size=6mm}}
\tikzset{bigplain/.style={rectangle,draw,very thick,minimum size=1cm}}

\tikzset{triangular/.style={regular polygon, regular polygon sides=3, draw,very thick,
inner sep=0pt,minimum size=1.2cm}}

\tikzset{arrow/.style={->,thick}}
\tikzset{dashedarrow/.style={->,dashed,thick}}
\tikzset{dottedarrow/.style={->,dotted,thick}}
\tikzset{mapto/.style={|->,thick}}
\tikzset{->-/.style={decoration={markings, mark=at position #1 with {\arrow{>}}},postaction={decorate}}}

\tikzset{implies/.style={thick,double,double equal sign distance,-implies}} 

\tikzset{line/.style={thick}}
\tikzset{dottedline/.style={dotted,thick}}
\tikzset{dashedline/.style={dashed,thick}}

\tikzset{inputleg/.style={<-,thick}}
\tikzset{outputleg/.style={->,thick}}
\tikzset{dottedinput/.style={<-,dotted,thick}}

\usepackage{transparent}
\usepackage{graphicx,import}
\usepackage{caption}
\usepackage{subcaption}
\usepackage{enumitem}

\newcommand{\Z}{\mathbb{Z}}

\newcommand{\F}{\mathsf{F}}

\newcommand{\bC}{\mathbf{C}}

\newcommand{\bE}{\mathscr{E}}

\newcommand{\bM}{\mathscr{M}}

\newcommand{\calQ}{\mathcal{Q}}

\newcommand{\bS}{\mathbf{S}}

\newcommand{\calM}{\mathcal{M}}
\newcommand{\pantscpx}{\mathcal{C}}
\newcommand{\xra}[1]{\overset{#1}{\rightsquigarrow}}
\newcommand{\surf}{S}%% Notation for a surface. S is bad since we already have S-moves and he modular operad S. F is bad because of F-moves. Sigma is bad because of the symmetric group.

\newcommand{\sSet}{\mathscr{S}}

\newcommand{\isSet}{\mathtt{S}}

\newcommand{\Set}{\mathrm{Set}}
\newcommand{\Grpd}{\mathscr{G}}

\newcommand{\Gr}{\mathsf{Grp}}

\newcommand{\PaRB}{\mathsf{PaRB}}
\newcommand{\PaB}{\mathsf{PaB}}

\newcommand{\markS}{\mathcal{M}(S)}
\newcommand{\calP}{\mathcal{P}}

\newcommand{\twist}{\tau}

\newcommand{\galQ}{\mathbf{\mathrm{Gal}}(\overline{\mathbb{Q}}/\mathbb{Q})}

\newcommand{\Mod}{\mathrm{Mod}}
\newcommand{\dMod}{\mathrm{dMod}}

\newcommand{\Cyc}{\mathrm{Cyc}}
\newcommand{\Op}{\mathrm{Op}}

\newcommand{\op}{\mathrm{op}}

\newcommand{\Sc}{\mathrm{Sc}}

\newcommand{\widesthat}[1]{(#1)^{\wedge}}
\newcommand{\id}{\operatorname{id}}

\DeclareMathOperator*{\colim}{colim}
\newcommand{\Pro}{\operatorname{Pro}}
\newcommand{\cs}{\mathrm{B}}
\newcommand{\nerve}{\mathrm{N}}

\newcommand{\End}{\operatorname{End}}
\newcommand{\Hom}{\operatorname{Hom}}
\newcommand{\Ho}{\operatorname{Ho}}

\newcommand{\ob}{\operatorname{ob}}
\newcommand{\Aut}{\operatorname{Aut}}

\newcommand{\Map}{\operatorname{Map}}
\newcommand{\Alg}{\operatorname{Alg}}

\newcommand{\HoAut}{\operatorname{HoAut}}

\newcommand{\Diff}{\operatorname{Diff}}
\newcommand{\Homeo}{\operatorname{Homeo}}

\newcommand{\trun}[1]{\mathrm{tr}_{\leq{#1}}}

\usepackage{scalerel,stackengine}
\stackMath
\newcommand\reallywidehat[1]{%
\savestack{\tmpbox}{\stretchto{%
  \scaleto{%
    \scalerel*[\widthof{\ensuremath{#1}}]{\kern-.6pt\bigwedge\kern-.6pt}%
    {\rule[-\textheight/2]{1ex}{\textheight}}%WIDTH-LIMITED BIG WEDGE
  }{\textheight}% 
}{0.5ex}}%
\stackon[1pt]{#1}{\tmpbox}%
}
\newcommand{\GT}{\widehat{\mathsf{GT}}}
\newcommand{\NS}{\widehat{\mathsf{\Gamma}}}

\newcommand{\Br}{\mathsf{B}} %braid groups
\newcommand{\PB}{\mathsf{PB}} %braid groups

\newcommand{\RB}{\mathsf{RB}} %ribbon braid groups
\newcommand{\PRB}{\mathsf{PRB}} %pure ribbon braid groups

\newcommand{\Mmax}{\mathcal{M}^{max}}

\newcommand{\nb}{nb}

\newcommand{\corolla}{C}

\newcommand{\triv}{\raisebox{-2pt}{\includegraphics[height=14pt]{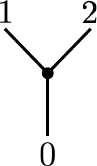}}}
\newcommand{\forkone}{\raisebox{-4pt}{\includegraphics[height=8pt]{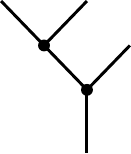}}}
\newcommand{\forktwo}{\raisebox{-4pt}{\includegraphics[height=8pt]{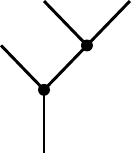}}}
\newcommand{\contraction}{\raisebox{-4pt}{\includegraphics[height=8pt]{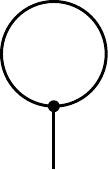}}}
\newcommand{\graphgtwo}{\raisebox{-4pt}{\includegraphics[height=8pt]{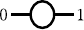}}}

\newtheorem{theorem}{Theorem}[section]
\newtheorem{thm}[theorem]{Theorem}%[section]
\newtheorem{prop}[theorem]{Proposition}%[section]
\newtheorem*{prop*}{Proposition}%[section]
\newtheorem{lemma}[theorem]{Lemma}%[section]
\newtheorem{cor}[theorem]{Corollary}%[section]
\newtheorem*{cor*}{Corollary}%[section]
\newtheorem{Idea}[theorem]{Idea}%[section]
\newtheorem*{thm*}{Theorem}

\theoremstyle{definition}
\newtheorem{definition}[theorem]{Definition}%[section]
\newtheorem{example}[theorem]{Example}%[section]
\newtheorem{remark}[theorem]{Remark}%[section]
\newtheorem{notation}[theorem]{Notation}

\numberwithin{equation}{section}

\let\oldtocsection=\tocsection
\let\oldtocsubsection=\tocsubsection
\renewcommand{\tocsection}[2]{\hspace{0em}\oldtocsection{#1}{#2}}
\renewcommand{\tocsubsection}[2]{\hspace{1em}\oldtocsubsection{#1}{#2}}
\DeclareRobustCommand{\SkipTocEntry}[5]{}
\renewcommand{\paragraph}[1]{\textbf{{#1}.}\hspace{5pt}}

\title[Galois actions on surfaces]{Galois actions on surfaces and a higher genus Grothendieck-Teichm\"{u}ller group}

\author[L. Basualdo Bonatto]{Luciana Basualdo Bonatto}
\address{Mathematical Institute, University of Oxford \\ Oxford, UK}
\email{luciana.basualdobonatto@maths.ox.ac.uk}

\author[M. Robertson]{Marcy Robertson}
\address{School of Mathematics and Statistics \\ The University of Melbourne \\ Melbourne, Victoria, Australia}
\email{marcy.robertson@unimelb.edu.au}

\date{\today}

\begin{document}

\begin{abstract}
We construct a modular operad $\bS$ of mapping class groupoids modelling the higher-genus Teichm\"uller tower, with compositions and contractions encoding surface gluing. We prove a presentation theorem for maps out of $\bS$ using genus-zero and genus-one generators and relations. Applying the work of Nakamura--Schneps, we obtain a faithful action of their subgroup $\NS\subseteq\GT$ on $\widehat{\bS}$, and hence a faithful Galois action. The genus-zero truncation identifies with the cyclic operad of parenthesized ribbon braids, whose object-fixing profinite automorphism group is $\GT$. The completed classifying spaces form a modular $\infty$-operad whose corolla values are geometric \'etale homotopy types of moduli stacks of curves with nonzero tangent vectors at their marked points. The $\NS$-action induces a faithful action by homotopy automorphisms on the associated tower of profinite homotopy $1$-types, which agrees with the \'etale model in genus zero and one.
\end{abstract}

\maketitle
\section{Introduction}

We construct a modular operad of surface groupoids and prove that it can be reconstructed from its genus-zero and genus-one parts.
This reconstruction gives a presentation for maps into arbitrary modular operads in groupoids and, after profinite completion, a corresponding determination theorem for homotopy endomorphisms. We apply these results to the actions of Nakamura--Schneps on profinite mapping class groups, obtaining faithful actions on a modular-operadic model of the Teichm\"uller tower and on its associated
tower of profinite homotopy $1$-types.

In his \emph{Esquisse d'un programme}, Grothendieck proposed studying $\operatorname{Gal}(\overline{\mathbb Q}/\mathbb Q)$ through its action on the moduli spaces of algebraic curves \cite{grothendieck_esquisse}. Let $\mathcal M_{g,n}$ denote the moduli stack over $\mathbb Q$ of smooth projective curves of genus $g$ with $n$ ordered distinct marked points. Already for $\mathcal M_{0,4}\cong\mathbb P^1\setminus\{0,1,\infty\}$, the outer Galois action on its geometric \'etale fundamental group is faithful, by Belyi's theorem \cite{belyi}. As the genus and number of marked points vary, the moduli stacks are related by forgetting and relabelling marked points and by gluing constructions associated with boundary strata of their compactifications. Grothendieck envisaged organising their geometric \'etale fundamental groupoids, together with suitable compatibility maps,
into a \emph{Teichm\"uller tower}. This tower offers a way to study the Galois action through its compatibility with the geometry of moduli spaces in all genera. Our results give an operadic form of his two-level principle, which predicts that the structure of the tower is controlled by low-genus generators and relations.

For $2g-2+n>0$, comparison over $\mathbb C$ identifies the geometric
\'etale fundamental group of $\mathcal M_{g,n}$ with the profinite
completion of the mapping class group of a genus-$g$ surface with
$n$ marked points \cite{matsumoto2000arithmetic,Oda}.
In the topological approach to the tower developed in
\cite{hls,Nakamura-Schneps}, one also allows surfaces with boundary,
so that gluing can be described directly in terms of surface topology.
Gluing boundary components and forgetting marked points then induce
homomorphisms between mapping class groups, translating compatibility
with the geometric operations in the tower into explicit conditions
on automorphisms of these groups.

To study these Galois actions, it is useful to embed $\galQ$ in a
group defined by explicit relations whose action satisfies the
same geometric compatibility conditions.
In genus zero, this is achieved by Drinfeld's profinite
Grothendieck--Teichm\"uller group $\GT$
\cite[Section~4]{Drin}.
Drinfeld constructs $\GT$ by studying its action on a tower of profinite braid groups, whose relation to genus-zero mapping class groups is recalled in Section~\ref{sec: operad of ribbon braids}.
The $\GT$-action on this tower respects forgetting strands and
replacing a strand by several parallel strands
(Definition~\ref{defn:GT}).
Under the embedding
\[
\galQ\hookrightarrow\GT,
\]
it recovers the Galois actions arising from the geometric
\'etale fundamental groups of genus-zero moduli spaces
\cite{ihara}.

The strand-replacement operations in this tower of braid groups are encoded by composition in
the operad of parenthesized braids $\PaB$.
In this language, $\GT$ is identified with the group of
object-fixing automorphisms of the profinite completion of $\PaB$
\cite{BN,FresseBook1}.
This description also has a homotopical interpretation.
The classifying spaces of $\PaB$ form an operad weakly equivalent
to the little $2$-disks operad, and Horel
\cite{Horel_profinite_groupoids} shows that $\GT$ is
the group of homotopy classes of automorphisms of its profinite
completion.

The framed little $2$-disks operad brings this description closer
to surface topology.
It admits a model whose spaces are classifying spaces of
genus-zero boundary mapping class groups and whose compositions
are induced by gluing boundary components.
Boavida de Brito, Horel, and Robertson
\cite{Boavida-Horel-Robertson} identify $\GT$ with the group
of homotopy classes of automorphisms of its profinite completion.
%In his Bourbaki lecture, Kontsevich discusses the cyclic operadic structure of the framed genus-zero Teichm\"uller tower and its extension to modular operads in higher genus \cite[Sections~2.4--2.5]{KontBourbaki}.

In higher genus, we must also allow two boundary components of the
same surface to be glued. This operation increases the genus and
is encoded by the contractions of a \emph{modular operad}
(Definition~\ref{def: modular operad}).
We construct a modular operad $\bS$ in groupoids whose components
satisfy
\[
\cs\bS(g,n+1)\simeq\cs\Gamma_{g,n+1}
\qquad (2g+n\geq1),
\]
where $\Gamma_{g,n+1}$ is the mapping class group of a genus-$g$
surface with $n+1$ boundary components fixed pointwise
(Definition~\ref{def: surface modular operad}).
Its objects are decorated graphs representing marked pants
decompositions, and its morphisms are mapping classes between
the associated surfaces.
Its compositions glue boundary components of separate surfaces,
and its contractions glue boundary components of the same surface.
The relation to the genus-zero models is made explicit by the
following identification.

\begin{prop}[Proposition~\ref{prop:PaRB-cyclic-is-S0}]
There is an isomorphism of cyclic operads
\[
\trun{0}\bS\cong\PaRB^{\mathrm{cyc}}.
\]
\end{prop}

Here $\trun{0}$ denotes genus-zero truncation, and
$\PaRB^{\mathrm{cyc}}$ is the operad of parenthesized ribbon
braids equipped with the cyclic structure described by Campos, Idrissi, and Willwacher \cite[Section~3.1]{campos2019configuration}. This structure treats all boundary components symmetrically,
without distinguishing one as an output. We use the description on generators and the verification of
compatibility with the defining relations and operadic composition given in \cite[Section~4.3]{robertson2025grothendieck}.

Our central structural result is a presentation for maps from $\bS$
into an arbitrary modular operad in groupoids.

\begin{thm}[Theorem~\ref{thm: presentation for maps out of bS}]
Let $\calQ$ be a modular operad in groupoids.
A map $\bS\to\calQ$ is equivalent to the choice of an
object $\mathbf m\in\ob(\calQ(0,3))$ and four morphisms
$\mathbf t,\mathbf b,\mathbf a,\mathbf s$, representing
a twist, a braiding, an associator, and a genus-one
switch, subject to the conditions specified in
Theorem~\ref{thm: presentation for maps out of bS}.
These conditions involve only genus zero and one,
with three additional relations in genus one.
\end{thm}

The presentation gives a reconstruction theorem for the modular
operad itself. Since all generators and relations occur in genus
at most one, every map from $\trun{1}\bS$ to $\trun{1}\calQ$
extends uniquely to a map $\bS\to\calQ$.
Equivalently, the counit of the genus-one truncation adjunction
is an isomorphism
\[
(\trun{1})_!(\trun{1}\bS)\cong\bS,
\]
where $(\trun{1})_!$ is left adjoint to genus-one truncation.
This is a universal property describing maps into every modular
operad in groupoids, and gives an operadic form of the
Lego--Teichm\"uller constructions of \cite{bk_marked_surfaces}.

The geometric input to the presentation is the two-dimensional
CW complex $\calM(S)$ of markings over pants decompositions of a
parametrised, labelled surface $S$
(Definition~\ref{def: complex of markings}).
Building on \cite{Hatcher_Thurston,hls,bk_marked_surfaces}, this
complex has marked pants decompositions as vertices, elementary
moves as edges, and relations among these moves as $2$-cells.
Twists and braidings change the markings on individual pairs of
pants, while associativity and genus-one moves change the
underlying pants decomposition.
Starting from the complex of maximal markings
of \cite{bk_marked_surfaces}, we retain only decompositions
into pairs of pants and adapt the $2$-cell relations to this
setting.
These include explicit gluing relations for Dehn twists
and a modified relation on a genus-one surface with two
boundary components.

\begin{prop}[Proposition~\ref{prop:our-complex-of-markings-is-simply-connected}]
For a parametrised, labelled surface $S$, the complex
$\calM(S)$ of markings over pants decompositions is
connected and simply connected.
\end{prop}

We prove this in Appendix~\ref{app: complex of markings} by
comparing the marking complex with the pants-decomposition
complex, controlling the additional marking data and the
relations introduced by gluing boundary components.
Lemma~\ref{lemma: the complexes of morphisms and markings} then
relates paths in $\calM(S)$ to morphisms in $\bS$.
Connectedness supplies the generators, while simple connectedness
ensures that the prescribed relations suffice.
Compatibility with gluing turns this description into a
presentation of the modular operad itself.

Passing to profinite completion gives a modular operad
$\widehat{\bS}$ in profinite groupoids
(Section~\ref{sec: profinite completion of modular operads in groupoids}).
The reconstruction theorem and the profinite-completion adjunction
give an isomorphism of monoids
\[
\End_0(\widehat{\bS})
\cong
\End_0(\trun{1}\widehat{\bS}),
\]
where the subscript $0$ denotes endomorphisms fixing the objects
of the underlying groupoids.
Thus every object-fixing endomorphism of the truncated modular
operad extends uniquely to the full profinite surface operad.

We apply this determination result to the arithmetic actions
constructed by Nakamura and Schneps.
They introduced a subgroup $\NS\subseteq\GT$ containing the image
of $\galQ$, and constructed actions of $\NS$ on profinite mapping
class groups using quilt decompositions of surfaces.
They describe changes of quilt decomposition by explicit inner
automorphisms and prove compatibility with inclusions of
subsurfaces carrying the induced decomposition
\cite[Theorems~1.3--1.4]{Nakamura-Schneps}.

Our markings determine the curves and Dehn twists appearing in
their formulae. The presentation theorem reduces the construction
of an action on $\widehat{\bS}$ to specifying the images of the
generating morphisms and checking the defining relations in
genus zero and one.
It then supplies a unique extension preserving compositions,
contractions, and symmetric-group actions.
The resulting $\NS$-action is faithful, and its genus-zero
restriction agrees with the $\GT$-action along $\NS\hookrightarrow\GT$
(Theorem~\ref{main theorem NS action}).
Restricting this action gives the following result.

\begin{theorem}[Consequence of Theorem~\ref{main theorem NS action}]
The absolute Galois group $\galQ$ acts faithfully on
$\widehat{\bS}$ through an explicit homomorphism
\[
\galQ\hookrightarrow\Aut_0(\widehat{\bS}),
\]
where $\Aut_0$ denotes automorphisms fixing the objects
of the underlying groupoids.
\end{theorem}

The modular dendroidal nerve $\nerve$ lets us regard a strict modular
operad as a modular $\infty$-operad, whose compositions and
contractions are compatible up to coherent homotopy.
We work with a genus-graded version of the graphical Segal
model introduced in \cite{hry1,hry_modular_nerve}, recalled in
Section~\ref{sec:modular infinity operads}.
We prove that the nerve is homotopically fully faithful
(Theorem~\ref{thm: nerve is homotopically fully faithful}),
so this passage preserves derived mapping spaces and allows
us to study homotopy endomorphisms in either model.

For $\widehat{\bS}$, we combine this comparison with the explicit
presentation to identify homotopy classes of automorphisms:
\[
\Aut_0(\widehat{\bS})
\cong
\pi_0\mathbb R\Aut(\nerve\widehat{\bS})
\]
(Proposition~\ref{prop: strict-to-homotopy-endomorphisms}).
Every homotopy-automorphism class therefore has a unique
object-fixing strict representative, and the arithmetic action
remains faithful after passage to homotopy classes.

The final part of the paper relates these results to the \'etale
homotopy types of moduli stacks of curves.
We use the formulation of \'etale homotopy types as shapes of
\'etale $\infty$-topoi \cite{Lurie,Carchedi}, compatible with the
classical constructions of \cite{Artin_Mazur,Friedlander}.
In the stable range, Oda's comparison identifies the profinite
\'etale homotopy types of moduli stacks of smooth pointed curves
with the profinite completions of the classifying spaces of their
mapping class groups \cite[Theorem~1]{Oda}.
For boundary mapping class groups, the corresponding moduli
problem includes a nonzero tangent vector at each marked point.
Writing $\mathcal M_{g,n+1}^{\tan}$ for this moduli stack over
$\overline{\mathbb Q}$, we establish an equivalence
\[
\widehat{\Pi_\infty^{\acute et}
(\mathcal M_{g,n+1}^{\tan})}
\simeq
\widehat{\cs\Gamma_{g,n+1}}
\qquad (2g+n\geq1)
\]
in Lemma~\ref{lemma:tangent-moduli-comparison}.

These profinite spaces are therefore the corolla values of
$\widehat{\cs\nerve\bS}$.
Since profinite completion does not preserve products in general,
verifying the modular Segal condition requires a separate product
comparison. We establish this in
Proposition~\ref{prop:haine-boundary-product} using finite CW
models for boundary mapping class groups and the criterion of
\cite{haine2024profinite}.
It follows that $\widehat{\cs\nerve\bS}$ is a modular
$\infty$-operad whose corolla values are profinite \'etale homotopy
types and whose structure maps encode gluing
(Lemma~\ref{lemma:teichmuller-tower-is-modular-infty-operad}).

Passing to profinite homotopy $1$-types gives a natural comparison
\[
\widehat{\cs\nerve\bS}
\longrightarrow
\widehat{\cs}\nerve\widehat{\bS}.
\]
On corollas, this replaces the profinite \'etale homotopy types
by their homotopy $1$-types, giving a modular $\infty$-operad
modelling the tower of profinite fundamental groupoids.
The strict-to-homotopy comparison and the homotopical full
faithfulness of the profinite classifying-space functor give
the following action on this tower.

\begin{thm*}[Proposition~\ref{NS action on topological operad}]
The action of $\NS$ on $\widehat{\bS}$ induces a faithful
action by homotopy automorphisms on
$\widehat{\cs}\nerve\widehat{\bS}$.
Equivalently, there is an injective homomorphism
\[
\NS\longrightarrow
\pi_0\mathbb R\Aut_{\Mod_\infty(\widehat{\isSet})}
\bigl(\widehat{\cs}\nerve\widehat{\bS}\bigr).
\]
Restriction along $\galQ\hookrightarrow\NS$ gives a
faithful Galois action on this modular $\infty$-operad.
\end{thm*}

Goodness of the boundary mapping class groups in genus zero and
one implies that the comparison is an equivalence after genus-one
truncation
(Proposition~\ref{prop: genus-one B N and completion commute}).
Combining this equivalence with the reconstruction theorem and
the comparison of strict and homotopy endomorphisms gives a
two-level theorem for the entire homotopy-endomorphism monoid.

\begin{thm*}[Corollary~\ref{prop: two-level profinite space model}]
Restriction and the comparison above induce an isomorphism
of monoids
\[
\pi_0\mathbb R\End_{\Mod_\infty(\widehat{\isSet})}
\bigl(\widehat{\cs}\nerve\widehat{\bS}\bigr)
\cong
\pi_0\mathbb R\End_{\Mod_{\infty,\leq1}(\widehat{\isSet})}
\bigl(\trun{1}\widehat{\cs\nerve\bS}\bigr).
\]
Consequently, $\NS$ and $\galQ$ act faithfully by homotopy
automorphisms on $\trun{1}\widehat{\cs\nerve\bS}$.
The genus-zero restriction of the $\NS$-action agrees
with the $\GT$-action along $\NS\hookrightarrow\GT$.
\end{thm*}

Thus restriction determines all homotopy classes of endomorphisms
of the full operadic tower of profinite homotopy $1$-types.
Every endomorphism of the genus-zero and genus-one truncated
modular $\infty$-operad of \'etale homotopy types extends uniquely
at the level of homotopy classes to this tower.
The determination of the Nakamura--Schneps and Galois actions
is a consequence of this general statement.

\subsection{Organization of the Paper}
Following the preliminaries, we construct and present the surface
operad $\bS$, prove its reconstruction from genus zero and one,
and construct the profinite arithmetic actions.
Sections~\ref{sec:modular infinity operads}
and~\ref{sec: homotopy action} then develop the homotopical
setting, the comparison of endomorphisms, and the relation
to \'etale homotopy types.

The homotopical results used in the paper are developed
in Appendices~\ref{sec: homotopy theory of mod op and truncated mod op}
and~\ref{sec: nerve is homotopically fully faithful}.
The first treats model structures on modular operads
and their genus truncations.
The second proves homotopical full faithfulness of the
modular dendroidal nerve, allowing derived mapping spaces
to be computed in either the strict or the graphical
Segal model.
These results underlie the comparison of automorphisms
and the homotopical two-level principle in
Section~\ref{sec: homotopy action}.
Appendix~\ref{app: complex of markings} supplies the
geometric input to the presentation theorem by proving
connectedness and simple connectedness of the marking
complexes.

\subsection{Acknowledgments}
The authors would like to thank Pedro Boavida de Brito, Philip
Hackney, Peter Haine, Geoffroy Horel, and Jan Steinebrunner for
many useful conversations and comments.
M.R. recognizes the support of Australian Research Council
Future Fellowship FT210100256.
L.B.B. is grateful to the Max Planck Institute for Mathematics
in Bonn for its hospitality and financial support.

\tableofcontents

\section{Preliminaries}
Throughout, we write $\bE=(\bE,\otimes,1_{\bE})$ for a cocomplete, cartesian symmetric monoidal category whose monoidal product $\otimes$
commutes with small colimits in each variable. In this paper, $\bE$ will be (completed versions of) the category of sets, groupoids, or simplicial sets. We work with categories enriched, tensored, and cotensored over
simplicial sets. Thus, for objects $X,Y\in\bE$, there is a simplicial mapping space\footnote{We follow a common abuse of terminology and often use the term \emph{space} when referring to simplicial sets.} $\Map_{\bE}(X,Y)$, whose set of $0$-simplices is $\Hom_{\bE}(X,Y)$. To reduce clutter, we often suppress the category $\bE$ and write $\Hom(X,Y)$ or $\Map(X,Y)$ when no confusion can arise.

\medskip 

In order to discuss homotopy classes of maps between (modular and cyclic) operads, we will make use of the language of Quillen model categories and $\infty$-categories. For Quillen model categories, we take \cite{hirsch} or \cite{balchin} as comprehensive references.  We will use the formalism of relative categories from \cite{barwick_kan} or \cite[Chapter 13]{balchin} as our model for $\infty$-categories.

\begin{definition}\label{def: relative category}
A \emph{relative category} is a pair $(\bE, \mathscr{W})$ where $\bE$ is a small category and $\mathscr{W}$ is a wide subcategory of $\bE$ whose morphisms are a class of designated \emph{weak equivalences} in $\bE$.   A functor of relative categories $F:(\bE,\mathscr{W})\rightarrow(\bE',\mathscr{W}')$ consists of a functor $F:\bE\rightarrow\bE'$ such that $F(\mathscr{W})\subset \mathscr{W}'$.
\end{definition}

\begin{example}
Any model category $\bE$ is a relative category in which $\mathscr{W}$ is the class of weak equivalences in the model structure on $\bE$.    
\end{example}

The \emph{homotopy category} of a relative category $\bE$, denoted $\Ho(\bE)$, is obtained by formally inverting the morphisms in $\mathscr{W}$.  To describe maps in the homotopy category $\Ho(\bE)$ we use the Dwyer-Kan enhancement of the homotopy category. This is a simplicially enriched category, $\mathcal{L}(\bE,\mathscr{W})$, which has the same objects as $\bE$ and, for any two objects $X,Y\in\ob(\bE)$, a space of maps $\mathbb{R}\Map(X,Y)$.\footnote{Here, the $\mathbb{R}$ is to indicate that we are taking the ``right derived'' version of the $\Map(-,-)$ functor. This is also commonly denoted $\Map^h(X,Y)$ in the literature, e.g. \cite{hirsch}.} By considering $\bE$ to be a discrete simplicial category, one can define a natural inclusion $\bE\rightarrow\mathcal{L}(\bE, \mathscr{W})$, with the property that, for every $X,Y\in\ob(\bE)$, \[\Hom_{\Ho(\bE)}(X,Y) =\pi_0\mathbb{R}\Map_{\bE}(X,Y).\] The fact that the connected components of $\mathbb{R}\Map_{\bE}(X,Y)$ correspond to morphisms in $\bE$ reflects that the simplicial category $\mathcal{L}(\bE,\mathscr{W})$ captures the higher homotopical data of morphisms in the homotopy category of $\bE$. In particular, for any weak equivalence $X\rightarrow X'$ and any $Y\in\ob(\bE)$ in $\bE$ the induced maps
\[\begin{tikzcd} \mathbb{R}\Map_{\bE}(X',Y)\arrow[r]& \mathbb{R}\Map_{\bE}(X,Y)\end{tikzcd} \quad \text{and} \quad \begin{tikzcd} \mathbb{R}\Map_{\bE}(Y,X)\arrow[r]& \mathbb{R}\Map_{\bE}(Y,X')\end{tikzcd}\] are weak equivalences of spaces.  

In general, it is difficult to compute the derived mapping spaces $\mathbb{R}\Map_{\bE}(X,Y)$. However, if $\bE$ admits the structure of a simplicial model category (\cite[Chapter 9]{hirsch}), we have the following useful fact (\cite[Corollary 4.7]{dk}):

\begin{thm}\label{thm: identifying mapping spaces in a simplicial model category}
Let $\bE$ be a simplicial model category. If $X$ is cofibrant and $Y$ is fibrant in $\bE$, then there exists a natural zigzag of weak equivalences \[\Map_{\bE}(X,Y)\simeq \mathbb{R}\Map_{\bE}(X,Y).\]
\end{thm}

Let $\mathscr{C}$ and $\mathscr{D}$ be two simplicial model categories. A simplicial Quillen adjunction between two simplicial model categories $
\begin{tikzcd} \mathscr{C}  \arrow[r, shift left =1, "\text{F}"]& \arrow[l, shift left =1, "\text{G}"]\mathscr{D} \end{tikzcd}$ gives rise to a homotopy adjunction \begin{equation}
\begin{tikzcd} \mathcal{L}(\mathscr{C},\mathscr{W}_{\mathscr{C}})  \arrow[r, shift left =1, "\mathbb{L}\text{F}"]& \arrow[l, shift left =1, "\mathbb{R}\text{G}"]\mathcal{L}(\mathscr{D},\mathscr{W}_{\mathscr{D}}). \end{tikzcd}
\end{equation} This means that there exists a weak equivalence of spaces \[\mathbb{R}\Map_{\mathscr{D}}(\mathbb{L}\text{F}(X),Y)\simeq \mathbb{R}\Map_{\mathscr{C}}(X,\mathbb{R}\text{G}(Y)),\] for all $X\in\mathscr{C}$ and $Y\in\mathscr{D}$. Here, the notation $\mathbb{L}$ and $\mathbb{R}$ denotes that we are using the left and right derived functors of $\text{F}$ and $\text{G}$, respectively. 

\subsection{Adjunctions between spaces and groupoids}\label{sec: classifying space functor}  Let $\sSet$ denote the category of simplicial sets with its usual Kan-Quillen simplicial model structure (e.g. \cite[Chapter 6.1]{balchin}) and let $\Grpd$ denote the category of groupoids equipped with the model structure in which weak equivalences are equivalences of categories, cofibrations are morphisms that are injective on objects, and fibrations are isofibrations (e.g. \cite{And}, \cite{Bous_model_cat_groupoid}).  Note that in this model structure, every object in $\Grpd$ is both fibrant and cofibrant.  

The functor which sends a space $X$ to its fundamental groupoid $\Pi_1(X)$, is the left adjoint in an adjunction \begin{equation}\label{classifying space adjunction}\begin{tikzcd} \Pi_{1}:\sSet  \arrow[r, shift left =1]& \arrow[l, shift left =1]\Grpd: \cs. \end{tikzcd}\end{equation} The right adjoint is called the \emph{classifying space} functor. The category of groupoids becomes a simplicial category via this adjunction. In particular, one defines the space of maps between two groupoids as \[\Map_{\Grpd}(\mathscr{C}, \mathscr{D}):=\Map_{\sSet}(\cs\mathscr{C}, \cs\mathscr{D}).\]  The classifying space functor preserves and reflects weak equivalences and fibrations which makes \eqref{classifying space adjunction} into a simplicial Quillen adjunction in which the right adjoint is \emph{homotopically fully faithful} in the sense that the induced map \[\begin{tikzcd}\mathbb{R}\Map_{\Grpd}(\mathscr{C},\mathscr{D})\arrow[r]& \mathbb{R}\Map_{\sSet}(\cs \mathscr{C},\cs \mathscr{D})\end{tikzcd}\] is a weak equivalence of spaces for all $\mathscr{C},\mathscr{D}\in\Grpd$. 

\subsection{Profinite completion of spaces and groupoids}
For any small category $\bE$, the category of pro-objects in $\bE$, denoted $\Pro(\bE)$, can be thought of as the category obtained by freely adding all co-filtered limits to $\bE$. More precisely, one can describe the \emph{category of pro-objects} in $\bE$, $\Pro(\bE)$, as the category whose objects are pairs $(I,X)$ where $I$ is a cofiltered category and $X=\{X_{i}\}_{i\in I}$ is a diagram $X:I\rightarrow \bE$. If $X:I\to\bE$ and $Y:J\to\bE$ are cofiltered diagrams, then morphisms in $\Pro(\bE)$ are defined by
\[
\Hom_{\Pro(\bE)}(\{X\}_{i}, \{Y\}_{j})
=\lim_{j\in J}\,\colim_{i\in I}\,\Hom_{\bE}(X_i,Y_j).
\]
Similarly, given a full subcategory $\mathscr{C}\subseteq \bE$, we write $\Pro_{\bE}(\mathscr{C})\subseteq \Pro(\bE)$ for the full subcategory spanned by those pro-objects $(I,X)$ such that $X_i\in \mathscr{C}$ for all $i\in I$; equivalently, $\Pro_{\bE}(\mathscr{C})\cong \Pro(\mathscr{C})$.

If the category $\bE$ is complete and $\mathscr{C}\subseteq\bE$ is a full subcategory closed under finite limits, we can define an adjunction \begin{equation}\label{profinite compeletion adjunction}
\begin{tikzcd} \widehat{(-)}:\bE\arrow[r, shift left =1] & \arrow[l, shift left =1] \Pro_{\bE}(\mathscr{C}): |-|. \end{tikzcd}
\end{equation} The left adjoint, $\widehat{(-)}$, is called the \emph{pro-$\mathscr{C}$-completion} functor. The right adjoint, $|-|$, sends a pro-object $\{X_i\}_{i\in I}$ to its limit in $\mathscr{E}$.  If, moreover, $\bE$ is a cofibrantly generated simplicial model category and $\mathscr{C}\subseteq \bE$ satisfies a specific list of conditions (e.g., the conditions of a \emph{fibration test category} in \cite[Definition 5.1]{BM20}) then $\Pro_{\bE}(\mathscr{C})$ admits a fibrantly generated model category structure and the adjunction \eqref{profinite compeletion adjunction} becomes a simplicial Quillen pair. 

\subsubsection{Profinite groupoids}
Let $\Gr$ denote the category of discrete groups and let $\Gr_{\mathrm{fin}}\subseteq \Gr$ be the full subcategory of finite groups. The pro-category
\[
\widehat{\Gr}:=\Pro_{\Gr}(\Gr_{\mathrm{fin}})
\] is the category of \emph{profinite groups}. There is an adjunction
\[
\begin{tikzcd}
\widehat{(-)}:\Gr \arrow[r, shift left=1] & \arrow[l, shift left=1] \widehat{\Gr}:|-|
\end{tikzcd}
\]
in which the right adjoint sends a pro-object $X=\{X_j\}_{j\in J}$ to its limit in $\Gr$.
The left adjoint is the classical profinite completion functor which sends a group $\mathsf G$ to the pro-object $\{\mathsf G/\mathsf N\}_{\mathsf N\triangleleft_f \mathsf G}$ indexed by the cofiltered poset of finite-index normal subgroups $\mathsf N\triangleleft_f \mathsf G$. Equivalently, one may identify $\widehat{\mathsf G}$ with the inverse limit
\[
\widehat{\mathsf G}:= \lim_{\mathsf N\triangleleft_f \mathsf G}\,\mathsf G/\mathsf N
\]
in $\Gr$.

\begin{example}\label{example: etale fundamental groups}
Let $X$ be a connected, locally Noetherian scheme and let
$\bar{x}\to X$ be a geometric point.  The \emph{$\acute etale$
fundamental group} $\pi_1^{\acute et}(X,\bar{x})$ may be defined as the
automorphism group of the fibre functor from finite $\acute etale$ covers of
$X$ to finite sets.  Equivalently, it is the inverse limit
\[
\pi_1^{\acute et}(X,\bar{x})
\cong
\lim_{(Y,y)\to (X,\bar{x})}\Aut_X(Y),
\]
where the limit ranges over connected finite $\acute etale$ Galois covers
$Y\to X$ equipped with a lift $y:\bar{x}\to Y$ of $\bar{x}$.  In
particular, $\pi_1^{\acute et}(X,\bar{x})$ is a profinite group. See \cite[\href{https://stacks.math.columbia.edu/tag/0BQ8}{0BQ8}]{stacks-project} or \cite[Expos\'e V]{grothendieck2003revetements} for more details. 
\end{example}

\begin{example}\label{example: mapping class groups}
Let $\calM_g^n$ denote the moduli stack of smooth genus $g$ curves with
$n$ marked points, defined over $\mathbb Q$.  For $2g-2+n>0$, the
geometric $\acute etale$ fundamental group of $\calM_g^n$ identifies,
after choosing a geometric base point, with the profinite completion of the
mapping class group of a genus $g$ surface with $n$ marked points fixed
pointwise:
\[
\pi_1^{\acute et}\bigl(\calM_g^n\otimes_{\mathbb Q}\overline{\mathbb Q}\bigr)
\cong \widehat{\Gamma_g^n}.
\]
See \cite{matsumoto2000arithmetic,boggi2009fundamental} for more details. 
%Later in this paper we will also work with mapping class groups of genus $g$ surfaces with $n$ boundary components, with each boundary component fixed pointwise. We denote this group by $\Gamma_{g,n}$ (see Definition~\ref{def: mapping class group}). Filling in each boundary component by a disc determines a natural surjection $\Gamma_{g,n}\twoheadrightarrow \Gamma_g^{n}$ whose kernel is generated by the Dehn twists about the boundary components. More precisely, there is a central short exact sequence \begin{equation}\label{eq:boundary_twist_extension} 1 \longrightarrow \mathbb{Z}^{n} \longrightarrow \Gamma_{g,n} \longrightarrow \Gamma_g^{n} \longrightarrow 1, \end{equation} where the map $\mathbb{Z}^{n}\to \Gamma_{g,n}$ sends the $i$th standard basis vector to the Dehn twist about the $i$th boundary component.
\end{example}

The profinite completion functor for groups admits a direct analogue for groupoids. We say that a
groupoid $\mathscr{E}$ is \emph{finite} if it has finitely many morphisms (hence finitely many objects),
and write $\Grpd_{\mathrm{fin}}\subseteq \Grpd$ for the full subcategory of finite groupoids. Since
$\Grpd_{\mathrm{fin}}$ is closed under finite limits, we define the category of \emph{profinite groupoids} by
\[
\widehat{\Grpd}:=\Pro_{\Grpd}(\Grpd_{\mathrm{fin}}).
\]
%Equivalently, the profinite completion of a groupoid $\mathscr{E}$ may be described as the cofiltered
%diagram of its finite quotients (obtained by quotienting by congruence relations so that the resulting
%groupoid has finitely many objects and morphisms). In particular, every finite groupoid is profinite,
%viewed as a constant pro-object.

The category of profinite groupoids $\widehat{\Grpd}$ admits a cocombinatorial model structure, and
the profinite completion functor fits into a Quillen adjunction
\[
\begin{tikzcd}
\widehat{(-)}:\Grpd \arrow[r, shift left=1] &
\arrow[l, shift left=1] \widehat{\Grpd}:|-|
\end{tikzcd}
\]
(see \cite[Theorem~4.12 and Proposition~4.22]{Horel_profinite_groupoids}), where $|-|$ denotes the
limit functor $\widehat{\Grpd}\to \Grpd$. 
The following proposition shows that, for groupoids with a finite set of objects, profinite completion
preserves binary products.

\begin{prop}\cite[Proposition~4.23]{Horel_profinite_groupoids}\label{prop: profinite completion preserves products of groupoids}
Let $\mathscr{E}$ and $\mathscr{D}$ be groupoids with a finite set of objects. Then the canonical map
\[
\widehat{\mathscr{E}\times \mathscr{D}}\longrightarrow
\widehat{\mathscr{E}}\times \widehat{\mathscr{D}}
\]
is an isomorphism in $\widehat{\Grpd}$.
\end{prop}

\subsubsection{Profinite spaces}\label{sec: profinite spaces}
A simplicial set is called \emph{lean} if it is degreewise finite
and $n$-coskeletal for some $n\geq 0$.
Let $\sSet_{\mathrm{lean}}\subseteq\sSet$ denote the full subcategory
of lean simplicial sets. This subcategory is closed under finite
limits and cotensors by finite simplicial sets.
We define the category of \emph{profinite spaces} by
\[
\widehat{\sSet}:=\Pro_{\sSet}(\sSet_{\mathrm{lean}}).
\]
It is equivalent to the category of simplicial profinite sets
\cite[Example~5.5]{BM20}.

We equip $\widehat{\sSet}$ with Quick's model structure
\cite[Theorem~2.12]{Quick}, in which cofibrations are monomorphisms
and weak equivalences are detected by $\pi_0$, $\pi_1$, and
cohomology with finite local coefficients.
This is also the model structure obtained from
\cite[Theorem~5.2]{BM20}; see
\cite[Example~5.5 and Corollary~6.6]{BM20}.
Profinite completion and the limit functor form a Quillen adjunction
\cite[Proposition~2.28]{Quick}:
\[
\begin{tikzcd}
\widehat{(-)}:\sSet \arrow[r, shift left=1] &
\arrow[l, shift left=1] \widehat{\sSet}:|-|.
\end{tikzcd}
\]

The classifying-space functor sends finite groupoids to lean
simplicial sets and induces a simplicial Quillen adjunction
\cite[Example~7.7]{BM20}:
\[
\begin{tikzcd}
\widehat{\Pi}_1:\widehat{\sSet} \arrow[r, shift left=1] &
\arrow[l, shift left=1] \widehat{\Grpd}:\widehat{\cs}.
\end{tikzcd}
\]
Moreover, these functors fit into a commutative square of simplicial
Quillen adjunctions \cite[Example~7.14]{BM20}:
\[
\begin{tikzcd}
\Grpd
  \arrow[d, shift right=1, "\widehat{(-)}"']
  \arrow[r, shift right=1, "\cs"']
&
\sSet
  \arrow[l, shift right=1, "\Pi_1"']
  \arrow[d, shift left=1, "\widehat{(-)}"]
\\
\widehat{\Grpd}
  \arrow[u, shift right=1, "|-|"']
  \arrow[r, shift right=1, "\widehat{\cs}"']
&
\widehat{\sSet}
  \arrow[l, shift right=1, "\widehat{\Pi}_1"']
  \arrow[u, shift left=1, "|-|"]
\end{tikzcd}
\]

\begin{example}\label{example: Oda}
Artin--Mazur \cite{Artin_Mazur} and Friedlander \cite{Friedlander}
construct the \'etale homotopy type of a scheme as a pro-simplicial
set. We write $\Pi^{\acute et}(X)$ for this pro-simplicial set,
and write $\widehat{\Pi}^{\acute et}(X)$ for its profinite
completion, viewed as a profinite space.
The same notation will be used for stacks, using \'etale hypercovers.
In Section~\ref{sec: tower of etale homotopy types}, we use the
corresponding $\infty$-categorical formulation, where the \'etale
homotopy type is viewed as a pro-object in the $\infty$-category
of spaces.

Let $\calM_g^n$ denote the moduli stack of smooth genus-$g$ curves
with $n$ distinct ordered marked points, defined over $\mathbb Q$.
Oda considers the \'etale homotopy type of the geometric stack
$\calM_g^n\otimes_{\mathbb Q}\overline{\mathbb Q}$.
If $\Gamma_g^n$ denotes the mapping class group of a genus-$g$
surface with $n$ marked points fixed pointwise, then, for
$2g-2+n>0$, \cite[Theorem~1]{Oda} gives a weak equivalence
of profinite spaces
\[
\widehat{\Pi}^{\acute et}
  \bigl(\calM_g^n\otimes_{\mathbb Q}\overline{\mathbb Q}\bigr)
\simeq
\widehat{\cs\Gamma_g^n}.
\]
Equivalently, the profinite \'etale homotopy type is the profinite
completion of the classifying space of the mapping class group
$\Gamma_g^n$.
\end{example}

\subsection{Graphs} We model the operations of operads, cyclic operads, modular operads and their $\infty$-versions with decorated graphs. 

\begin{notation}
Let $\Sigma_n=\Aut(\{1,\ldots,n\})$ and $\Sigma_n^+=\Aut(\{0,1,\ldots,n\})$; then $\Sigma_n$ identifies naturally with the subgroup of $\Sigma_n^+$ consisting of permutations that fix $0$.
\end{notation}

\begin{definition}\label{def: graph}
A \emph{graph} $G$ consists of a (possibly empty) finite set of \emph{vertices} $V_G$, a finite set of \emph{half-edges} $A_G$,\footnote{Half-edges are also called ``arcs'' in the literature, hence the use of the letter $A$ for the set of half-edges.} together with gluing data encoded by a diagram of finite sets
\[
\begin{tikzcd}
\arrow[loop left]{l}{\dagger} A_G & D_G \arrow[l,hook'] \arrow[r,"t"] & V_G .
\end{tikzcd}
\] Here $\dagger:A_G\to A_G$ is a fixed-point-free involution, and $D_G\subseteq A_G$ is the subset of half-edges adjacent to vertices, with $t:D_G\to V_G$ recording the incident vertex.
\end{definition}

An \emph{edge} of $G$ is an orbit of $\dagger$, written $e=[a,a^\dagger]$, and we write $E_G=\{[a,a^\dagger]\mid a\in A_G\}$ for the set of edges. For $v\in V_G$, the \emph{neighborhood} of $v$ is the set of incident half-edges $\nb(v):=t^{-1}(v)\subseteq D_G$; its elements are said to be \emph{adjacent} to $v$. An edge $e=[a,a^\dagger]$ is \emph{internal} if $a,a^\dagger\in D_G$; we write $iE_G\subseteq E_G$ for the set of internal edges. The remaining half-edges
\[
\partial(G):=A_G\setminus D_G
\] are called \emph{legs}, and $\partial(G)$ is called the \emph{boundary} of $G$.

\begin{example}
\begin{enumerate}
    \item The graph consisting of a single \emph{edge} and no vertices is denoted $\updownarrow$. Concretely, it has $V_{\updownarrow}=\varnothing$, $D_{\updownarrow}=\varnothing$, and $A_{\updownarrow}=\{a,a^{\dagger}\}$ with $\dagger(a)=a^{\dagger}$.

    \item The $(n+1)$st \emph{corolla}, denoted $C_{n+1}$, is the graph with a single vertex $v$ and no internal edges. It has $n+1$ half-edges adjacent to $v$, and each of these is paired under $\dagger$ with a boundary half-edge. These boundary half-edges are the $n+1$ legs of the corolla.
\end{enumerate}
Many additional examples of graphs can be found in \cite{hry1}, \cite{hry_modular_nerve}, and in the survey paper \cite{hackney2022segal}.
\end{example}

A graph $H$ is a \emph{subgraph} of a graph $G$ if there is a morphism of the defining diagrams
\[
\begin{tikzcd}
\arrow[loop left]{l}{\dagger}A_{H}\arrow[d] & D_{H}\arrow[l,hook']\arrow[d] \arrow[r, "t"] & V_{H}\arrow[d]\\
\arrow[loop left]{l}{\dagger}A_{G}          & D_{G}\arrow[l,hook']          \arrow[r, "t"] & V_{G}
\end{tikzcd}
\]
such that the square
\[
\begin{tikzcd}
D_H \arrow[r]\arrow[d,"t"'] & D_G \arrow[d,"t"]\\
V_H \arrow[r] & V_G
\end{tikzcd}
\]
is a pullback and the induced map $V_H\to V_G$ is injective (cf.\ \cite[Definition~1.13]{hry1}).

\begin{example}\label{example: corolla at a vertex}
If $G$ has at least one vertex and $v\in V_G$, there is a canonical subgraph $C_v\hookrightarrow G$, called the \emph{corolla at $v$} (cf.\ \cite[Definition~1.5]{hry1}). It has a single vertex $\{v\}$, and its adjacent half-edges are exactly the elements of $\nb(v)$. Each such half-edge is paired with a boundary half-edge in $C_v$, so the legs of $C_v$ are naturally in bijection with $\nb(v)$.
\end{example}

For graphs with at least one vertex, the gluing data of $G$ can be assembled into a coequalizer diagram (cf.\ \cite[Construction~1.18]{hry1}).

\begin{lemma}\label{lem:coequalizer-graph}
Let $\mathrm{Fin}$ denote the category of finite sets, and let $\mathscr{I}$ be the category with three objects and three generating morphisms of the shape
\[
\begin{tikzcd}
\arrow[loop left] \bullet & \bullet \arrow[l, hook'] \arrow[r] & \bullet .
\end{tikzcd}
\]
A connected graph $G$ with at least one vertex, viewed as an object of $\mathrm{Fin}^{\mathscr I}$, is the coequalizer
\[
\begin{tikzcd}
\coprod_{e\in iE_G}\updownarrow \arrow[r, shift left=1.2] \arrow[r, shift right=1.2] & \coprod_{v\in V_G} C_v \arrow[r] & G .
\end{tikzcd}
\]
The two arrows send each internal edge to the two corresponding boundary legs of the corollas at its incident vertices.
\end{lemma}

\subsubsection{Additional structure on graphs}
Our combinatorial model admits a geometric realization $|G|$ obtained by gluing a point for each vertex and a closed interval for each edge, identifying the endpoints of each interval with the incident vertices via the incidence map $t:D_G\to V_G$. We say that $G$ is \emph{connected} (respectively, \emph{simply connected}) if $|G|$ is connected (respectively, simply connected) (see also \cite[Definition~1.7]{hry1}).

\begin{definition}\label{def: tree}
A \emph{tree} is a simply connected graph. A tree $T$ is \emph{rooted} if it is equipped with a distinguished leg $r\in \partial(T)$, called the \emph{root}.
\end{definition}

\begin{remark}
Any rooted tree admits a canonical orientation by directing every edge toward the root. Since $T$ is simply connected, each vertex then has a unique \emph{outgoing} incident edge (the one on the path to the root), and all remaining incident edges are \emph{incoming} (cf.\ \cite[Definition~3.4]{hackney2022segal}).
\end{remark}

In order to use graphs to model operations in modular and cyclic operads we will equip them with additional labeling data.

\begin{definition}\label{def: labelled graph}
A graph $G$ is said to be \emph{labelled} if:
\begin{enumerate}
\item it comes with a total ordering of its vertices, i.e.\ a bijection
      $\lambda:\{1,\ldots,k\}\xrightarrow{\cong} V_G$;
\item for each vertex $v\in V_G$, it comes with a bijection 
      \[
      \ell_v:\{0,1,\ldots,|\nb(v)|-1\}\xrightarrow{\cong}\nb(v);
      \]
\item it comes with a bijection of the boundary (legs)
      \[
      \ell_G:\{0,1,\ldots,n\}\xrightarrow{\cong} \partial(G),
      \qquad\text{where } n=|\partial(G)|-1.
      \] 
\end{enumerate}
We write $(G,\lambda,\ell)$ for this data. 
\end{definition}

A \emph{genus grading} on a graph $G$ is a function $\epsilon:V_G\to \mathbb{N}$. The \emph{genus}
of a genus-graded graph $(G,\epsilon)$ is
\[
\beta_1(G)+\sum_{v\in V_G}\epsilon(v),
\]
where $\beta_1(G)$ is the first Betti number of the geometric realization of $G$.
If $G$ is connected and $G\neq \updownarrow$, then
\[
\beta_1(G)=|iE_G|-|V_G|+1.
\]
%A genus-graded graph $(G,\epsilon)$ is \emph{stable}\mrnote{I am not sure we use stable anywhere} if $G$ is connected and, for every vertex $v$,
%\[
%2\epsilon(v)+|\nb(v)|-2>0.
%\]
We adopt the convention that the edge $\updownarrow$ has genus zero.

\begin{notation}\label{not:decorated-graphs}
Graphs in this paper are often equipped with additional structure, such as a vertex-ordering $\lambda$,
labelings $\ell$ (Definition~\ref{def: labelled graph}), and/or a genus grading $\epsilon$.
When we wish to display this data explicitly we write
\[
G=(G,\lambda,\ell,\epsilon),
\]
omitting any entries that are not present. When no confusion can arise, we suppress some or all of
these decorations and simply write $G$ for the underlying graph.
\end{notation}

\begin{definition}\label{def: iso of graphs}
An \emph{isomorphism of graphs} $\varphi:G\to G'$ is a triple of bijections
\[
\varphi_A:A_G\xrightarrow{\cong} A_{G'},\qquad
\varphi_D:D_G\xrightarrow{\cong} D_{G'},\qquad
\varphi_V:V_G\xrightarrow{\cong} V_{G'}
\]
such that the following diagram commutes:
\[
\begin{tikzcd}
A_G \arrow[d,"\varphi_A"'] \arrow[loop left,"\dagger_G"] &
D_G \arrow[l,hook'] \arrow[d,"\varphi_D"] \arrow[r,"t_G"] &
V_G \arrow[d,"\varphi_V"] \\
A_{G'} \arrow[loop left,"\dagger_{G'}"] &
D_{G'} \arrow[l,hook'] \arrow[r,"t_{G'}"'] &
V_{G'} .
\end{tikzcd}
\]
If $G$ and $G'$ are labelled graphs, a \emph{strict isomorphism} $\varphi:G\to G'$ is an isomorphism that preserves the vertex ordering and the labelings.
\end{definition}

\subsection{Operads and cyclic operads}
A symmetric \emph{sequence} in $\bE$ is a collection of objects $\calP=\{\calP(n)\}_{n\ge 0}$ in $\bE$, together with a right action of the symmetric group $\Sigma_n$ on each $\calP(n)$. In particular, for each $n\ge 0$ and each $\sigma\in\Sigma_n$, there is a morphism
\[
\sigma^*:\calP(n)\to \calP(n).
\]

\begin{definition}\label{def:operad}
An \emph{operad} in $\bE$ consists of a symmetric sequence $\calP=\{\calP(n)\}_{n\ge 1}$ together with:
\begin{itemize}
\item a distinguished element $\id\in\calP(1)$;
\item partial compositions
\[
\begin{tikzcd}
\calP(n)\times \calP(m)\arrow[r,"\circ_i"]& \calP(n+m-1),
\end{tikzcd}
\qquad 1\le i\le n,
\]
\end{itemize}
satisfying the following axioms, for all $x\in\calP(n)$, $y\in\calP(m)$, $z\in\calP(\ell)$:

\begin{enumerate}
\item For $1\le i\le n$ and $1\le j\le m$,
\[
(x\circ_i y)\circ_{i+j-1} z = x\circ_i (y\circ_j z).
\]

\item For $1\le i<k\le n$,
\[
(x\circ_i y)\circ_{k+m-1} z = (x\circ_k z)\circ_i y.
\]

\item For $x\in\calP(n)$ and $1\le i\le n$, $\id\circ_1 x = x,$ and $x\circ_i \id = x.$ 

\item  For $\sigma\in\Sigma_n$ and $\tau\in\Sigma_m$,
\[
(\sigma^*x)\circ_i (\tau^*y) =
(\sigma\circ_i \tau)^*(x\circ_{\sigma^{-1}(i)} y),
\]
where $\sigma\circ_i\tau\in\Sigma_{n+m-1}$ denotes the standard block permutation determined by
inserting $\tau$ into the $i$th input of $\sigma$.
\end{enumerate}
\end{definition}

A map $f:\calP\to\calQ$ of operads is a $\Sigma_n$-equivariant map of symmetric sequences that preserves the unit and partial compositions. We write $\Op(\bE)$ for the category of operads in $\bE$.

\medskip 

A cyclic operad is an operad in which the symmetric group actions on inputs extend compatibly to actions on the output together with the inputs. We write
\[
\Sigma_n^+ := \Aut(\{0,1,\ldots,n\})
\]
and let $z_n^+\in\Sigma_n^+$ denote the cyclic permutation $z_n^+(i)=i+1\pmod{n+1}$.

\begin{definition}\label{def: cyclic structure}
Let $\calP=\{\calP(n)\}_{n\ge 0}$ be an operad in $\bE$. A \emph{cyclic structure} on $\calP$ consists, for each $n\ge 0$, of a right action of $\Sigma_n^+$ on $\calP(n)$, written
\[
\sigma^*:\calP(n)\to\calP(n),
\]
such that the following conditions hold:
\begin{enumerate}
\item the restriction along $\Sigma_n\subseteq\Sigma_n^+$ agrees with the given $\Sigma_n$-action on $\calP(n)$;

\item the operadic unit is fixed: $(z_1^+)^*(\mathrm{id})=\mathrm{id}$;

\item for $x\in\calP(n)$ and $y\in\calP(m)$, with $n\ge 1$ and $m\ge 0$, the following identities hold in the indicated ranges:
\[
(z_{n+m-1}^+)^*(x\circ_i y)=
\begin{cases}(z_n^+)^*(x)\circ_{i-1}y,& 2\le i\le n,\\[2pt]
(z_m^+)^*(y)\circ_m(z_n^+)^*(x),
& i=1,\ m\ge 1.
\end{cases}
\]
\end{enumerate}
A \emph{cyclic operad} is an operad equipped with a cyclic structure.
\end{definition}

A map $f:\calP\to\calQ$ of cyclic operads is a map of operads commuting with the $\Sigma_n^+$-actions. We write $\Cyc(\bE)$ for the category of cyclic operads in $\bE$.
As discussed in Appendix~\ref{sec: homotopy theory of mod op and truncated mod op}, the categories $\Op(\bE)$ and $\Cyc(\bE)$ admit model category structures in which a map $f:\calP\to\calQ$ is a weak equivalence
if and only if each $f(n):\calP(n)\to\calQ(n)$ is a weak equivalence in $\bE$ for all $n\ge 0$.

\begin{remark}\label{rmk:reduce-circij}
Cyclic operads are often presented using two-slot partial compositions
\[
\circ_i^j:\calP(n)\times \calP(m)\longrightarrow \calP(n+m-1),
\]
where $i\in\{0,\ldots,n\}$ and $j\in\{0,\ldots,m\}$, rather than using the one-slot compositions $\circ_i$ of an ordinary operad together with an extended
symmetric group action. The two formalisms are equivalent.

Indeed, using the $\Sigma_m^+$-action on $\calP(m)$, one may first relabel the second operation so that its $j$th slot becomes the distinguished slot $0$.
One then composes using the operation $\circ_i^0$, and finally reindexes the output using the induced $\Sigma_{n+m-1}^+$-action. For example, if
$z_m^+\in\Sigma_m^+$ denotes the cycle $z_m^+(r)=r+1\pmod{m+1}$, then, up to the chosen convention for the final labelling, one may write
\[
x\circ_i^j y
:= \rho^*\Bigl(x\circ_i^0 \bigl(((z_m^+)^{-j})^*y\bigr)\Bigr),
\]
where $\rho\in\Sigma_{n+m-1}^+$ is the permutation determined by the chosen labelling convention for the boundary of the output. For an explicit computation
of $\rho$, see \cite{DCH}.
\end{remark}

\subsection{Modular operads}\label{sec: modular operads}
For a fixed monoidal category $\bE$, a \emph{genus-graded
sequence} in $\bE$ is a collection of objects
\[
\calP=\{\calP(g,n+1)\}_{g\geq0,\;n\geq0}
\]
in $\bE$, such that each $\calP(g,n+1)$ carries a right
action of the extended symmetric group
$\Sigma_n^+:=\Aut(\{0,1,\ldots,n\})$.

\begin{definition}\label{def: modular operad}
A \emph{modular operad} in $\bE$ consists of a genus-graded
sequence $\calP=\{\calP(g,n+1)\}$ equipped with the following
structure.
\begin{enumerate}
\item A distinguished unit object $\id\in\calP(0,2)$.

\item A family of partial compositions
\[
\begin{tikzcd}
\calP(g,n+1)\times\calP(h,m+1)
\arrow[r, "\circ_i^j"]
&
\calP(g+h,n+m),
\end{tikzcd}
\]
for $n+m\geq1$, $0\leq i\leq n$, and $0\leq j\leq m$.

\item A family of contractions
\[
\begin{tikzcd}
\calP(g,n+1)
\arrow[r, "\xi_i^j"]
&
\calP(g+1,n-1),
\end{tikzcd}
\]
for $n\geq2$ and $0\leq i<j\leq n$.
\end{enumerate}

The composition maps satisfy the associativity, equivariance,
and unitality axioms of
\cite[Definition~3.3.1]{hv_cyclic} or
\cite[Theorem~3.7]{gk_modular}.
The contractions satisfy the following compatibility axioms.
\begin{enumerate}
\setcounter{enumi}{3}

\item For $\sigma\in\Sigma_n^+$ and
$a\in\calP(g,n+1)$, with $n\geq2$ and $0\leq i<j\leq n$,
\[
\bar\sigma^*(\xi_i^j a)
=
\xi_{\sigma^{-1}(i)}^{\sigma^{-1}(j)}(\sigma^*a).
\]
The contraction on the right is along the unordered pair
$\{\sigma^{-1}(i),\sigma^{-1}(j)\}$.
Here $\bar\sigma$ is induced by the restriction of $\sigma$
from the complement of this pair to the complement of
$\{i,j\}$, after reindexing both complements in increasing
order.

\item For $n\geq4$ and disjoint pairs
$\{i,j\},\{k,\ell\}\subseteq\{0,\ldots,n\}$,
the contractions commute,
\[
\xi_i^j\circ\xi_k^\ell
=
\xi_k^\ell\circ\xi_i^j,
\]
with the evident relabelling of indices after the first
contraction.

\item For $a\in\calP(g,n+1)$ and
$b\in\calP(h,m+1)$, with $n\geq3$,
\[
\xi_1^2\bigl(a\circ_n^0 b\bigr)
=
(\xi_1^2 a)\circ_{n-2}^0 b.
\]

\item For $a\in\calP(g,n+1)$ and
$b\in\calP(h,m+1)$, with $m\geq2$ and $n+m\geq3$,
\[
\xi_n^{n+1}\bigl(a\circ_n^0 b\bigr)
=
a\circ_n^0(\xi_1^2 b).
\]

\item For $a\in\calP(g,n+1)$ and
$b\in\calP(h,m+1)$, with $n,m\geq1$ and $n+m\geq3$,
\[
\xi_{n-1}^{n}\bigl(a\circ_n^0 b\bigr)
=
\xi_{n+m-2}^{\,n+m-1}
\Bigl(a\circ_{n-1}^0(z_m^+)^*b\Bigr),
\]
where $z_m^+\in\Sigma_m^+$ is the cycle $(0\,1\,\dots\,m)$.
\end{enumerate}
\end{definition}

A map $f:\calP\to\calQ$ of modular operads is a map of the underlying genus-graded sequences that is equivariant for the $\Sigma_n^+$-actions and commutes with the unit, partial compositions, and contractions. We write $\Mod(\bE)$ for the category of modular operads in $\bE$. As we discuss in Theorem~\ref{model structure on Mod(E)}, $\Mod(\bE)$ admits a model category structure in which a map $f:\calP\to\calQ$ is a weak equivalence if each component
\[
f(g,n+1):\calP(g,n+1)\longrightarrow \calQ(g,n+1)
\]
is a weak equivalence in $\bE$ for all $g\ge 0$ and $n\ge 0$.

\begin{remark}\label{rmk:indexing-modular-operads}
Our indexing conventions for modular operads are chosen to match the mapping class group conventions of Section~\ref{subsec: mapping class groups}. Thus we write a genus-graded sequence as $\{\calP(g,n+1)\}_{g\geq 0,\;n\geq 0}$, so that the boundary labels are $\{0,1,\ldots,n\}$ and the symmetric group action is by $\Sigma_n^+=\Aut(\{0,1,\ldots,n\})$. This is a slight shift from the convention used in parts of the literature, where one often writes $\{\calP(g,n)\}_{g,n\geq 0}$.

We also exclude operations with no boundary components. Equivalently, in the above structure maps we only include compositions and contractions whose targets still have nonempty boundary; see Appendix~\ref{sec: homotopy theory of mod op and truncated mod op}.
\end{remark}

\begin{definition}\label{def:decorated graph}
Let $G=(G,\lambda,\ell,\epsilon)$ be a connected labelled, genus-graded graph with nonempty boundary, and let $\calP=\{\calP(g,n+1)\}$ be a genus-graded sequence in $\bE$. A \emph{$\calP$-decoration} of $G$ is an ordered tuple 
$\underline{p}=(p_{\lambda(1)},\ldots,p_{\lambda(k)})$, where $V_G=\{v_{\lambda(1)},\ldots,v_{\lambda(k)}\}$ and 
$p_{\lambda(i)}\in \calP\bigl(\epsilon(v_{\lambda(i)}),\,|\nb(v_{\lambda(i)})|\bigr)$ for each $i$.
Equivalently,
\[
\underline{p}\in \prod_{v_{\lambda(i)}\in V_G}\calP\bigl(\epsilon(v_{\lambda(i)}),\,|\nb(v_{\lambda(i)})|\bigr).
\]
\end{definition}

Let $G=(G,\lambda,\ell,\epsilon)$ and $G'=(G',\lambda',\ell',\epsilon')$ be labelled graphs with nonempty boundary. For
\[
i\in\{0,\ldots,|\partial(G)|-1\}
\qquad\text{and}\qquad
j\in\{0,\ldots,|\partial(G')|-1\},
\]
the \emph{grafting} $G\circ_i^j G'$ is the labelled graph obtained by gluing the boundary leg $\ell_G(i)\in\partial(G)$ to the boundary leg $\ell_{G'}(j)\in\partial(G')$, with the remaining structure inherited in the evident way. Similarly, for $0\le i<j\le |\partial(G)|-1$, the \emph{contraction} $\xi_i^j G$ is the labelled graph obtained by gluing together the boundary legs $\ell_G(i)$ and $\ell_G(j)$.

\begin{definition}\label{def: free modular operad}
Let $\calP=\{\calP(g,n+1)\}$ be a genus-graded sequence in $\Set$. The \emph{free modular operad} $\mathcal{F}\calP$ is defined by letting $\mathcal{F}\calP(g,n+1)$ be the set of isomorphism classes of $\calP$-decorated, connected, labelled, genus-graded graphs of total genus $g$ with $n+1$ boundary legs.

The group $\Sigma_n^+$ acts on $\mathcal{F}\calP(g,n+1)$ by permuting the boundary labels. The partial compositions and contractions in $\mathcal{F}\calP$ are given by grafting and contracting decorated graphs.
\end{definition}

To make this more explicit, for each $g\ge 0$ and $n\ge 0$ one may describe $\mathcal{F}\calP(g,n+1)$ as the quotient
\[
\mathcal{F}\calP(g,n+1)=
\Biggl(
\coprod_{[G]} \prod_{v\in V_G}\calP\bigl(\epsilon(v),|\nb(v)|\bigr)
\Biggr)\Big/\sim,
\]
where the coproduct ranges over isomorphism classes $[G]$ of connected labelled genus-graded graphs $G=(G,\lambda,\ell,\epsilon)$
with $|\partial(G)|=n+1$ and total genus
\[
g=\beta_1(G)+\sum_{v\in V_G}\epsilon(v).
\]

The equivalence relation $\sim$ identifies two decorated labelled genus-graded graphs $(G,\lambda,\ell,\epsilon,\underline{p})$ and $(G',\lambda',\ell',\epsilon',\underline{p}')$ whenever there exists an isomorphism of genus-graded graphs $\iota:G\xrightarrow{\cong} G'$ preserving the boundary labels, such that, for each vertex $v_{\lambda(i)}\in V_G$, the induced bijection $\iota:\nb(v_{\lambda(i)})\xrightarrow{\cong}\nb(\iota(v_{\lambda(i)}))$ determines a permutation
\[
\sigma_{v_{\lambda(i)}}:=
(\ell'_{\iota(v_{\lambda(i)})})^{-1}\circ \iota\circ \ell_{v_{\lambda(i)}}
\in \Sigma_{|\nb(v_{\lambda(i)})|-1}^+,
\]
and the decorations satisfy
\[
p'_{\lambda'(j)}=\bigl(\sigma_{v_{\lambda(i)}}\bigr)^*(p_{\lambda(i)})
\]
whenever $\iota(v_{\lambda(i)})=v'_{\lambda'(j)}$.  For further details and variants of this construction, see \cite[2.20]{gk_modular}, \cite[2.1]{hry_modular_nerve}, and \cite[13.2.2]{yau_modular}.

\subsubsection{Truncations of modular operads and adjunctions}\label{subsec: intro to truncations}
In Appendix~\ref{subsec: truncations} we define the \emph{genus-$g$ truncation} of a modular operad
\[
\calP=\{\calP(h,n+1)\}_{h\geq 0,\;n\geq 0},
\]
denoted $\trun{g}\calP$, as the truncated modular operad obtained by retaining only the operations whose output has genus at most $g$. Thus compositions and contractions whose output has genus greater than $g$ are no longer defined. More concretely, we say 
\[
x\circ_i^j y
\qquad\text{and}\qquad
\xi_i^j w
\]
is undefined whenever
\begin{itemize}
\item $x\in\calP(h_1,m_1+1)$ and $y\in\calP(h_2,m_2+1)$ with
$h_1+h_2>g$, or
\item $w\in\calP(h,m+1)$ with $h\geq g$, so that
$\xi_i^j w\in\calP(h+1,m-1)$ has genus greater than $g$.
\end{itemize}
We write $\Mod_{\leq g}(\bE)$ for the category of genus-$g$ truncated modular
operads in $\bE$. The truncation construction defines a functor
\[
\trun{g}^*:\Mod(\bE)\longrightarrow \Mod_{\leq g}(\bE),
\]
which admits both a left adjoint and a right adjoint:
\[
\begin{tikzcd}
\Mod(\bE) \arrow[r, "\trun{g}^*"'] &
\Mod_{\leq g}(\bE)
\arrow[l, bend right=50, "(\trun{g})_!"']
\arrow[l, bend left=50, "(\trun{g})_*"]
.
\end{tikzcd}
\]

\begin{prop}
The category of cyclic operads is naturally equivalent to the category of genus-$0$ truncated modular operads. Moreover, the truncation functor
\[
\trun{0}^*:\Mod(\bE)\longrightarrow \Cyc(\bE)
\]
preserves weak equivalences and is the right adjoint in a Quillen adjunction.
\end{prop}

Similarly, there is a forgetful functor $u^*:\Cyc(\bE)\to \Op(\bE)$ which forgets the cyclic structure. This functor
admits a left adjoint, called the \emph{cyclic envelope}:
\[
\begin{tikzcd}
\Op(\bE)\arrow[r, shift left=1, "u_!"] &
\Cyc(\bE)\arrow[l, shift left=1, "u^*"]
.
\end{tikzcd}
\]
An explicit description of $u_!$ is somewhat involved; see \cite[Section~3.1]{DCH} or \cite[Section~9]{Ward}. The forgetful functor $u^*$ preserves weak equivalences, and the adjunction above is a Quillen adjunction.

\subsection{Mapping Class Groups}\label{subsec: mapping class groups}
A surface of type $(g,n+1)$ is a compact, oriented surface $S$ of genus $g$ with $n+1$ boundary components. 
% We equip each boundary component with a chosen \emph{collar}, i.e.\ a fixed neat embedding $[0,1)\times S^1\hookrightarrow S$ whose image is a neighbourhood of that boundary component.

\begin{definition}\label{def: labelled surface}
A \emph{labelled surface} of type $(g,n+1)$ is a pair $(S,\delta)$, where
\[
\begin{tikzcd}
\delta: \{0,\ldots,n\} \arrow[r, "\cong"] & \pi_0(\partial S)
\end{tikzcd}
\]
is a labelling of the boundary components, i.e. we write $\partial_i S:=\delta(i)$ for the $i$th boundary component. A \emph{parameterized labelled surface} $(S,\delta,\{x_i\})$ consists of a labelled surface together with a choice of point $x_i\in \partial_i S$ for each $0\le i\le n$. 
\end{definition}

\begin{definition}\label{def: mapping class group}
Let $S$ be a surface of type $(g,n+1)$. The associated \emph{pure mapping class group} is
    \[ \Gamma(S)\;:=\;\pi_0\,\Diff^{+}(S;\partial S), \]
where $\Diff^{+}(S;\partial S)$ denotes the group of orientation-preserving diffeomorphisms of $S$ that fix $\partial S$ pointwise. Elements of the mapping class group $\Gamma(S)$ are called \emph{mapping classes}. The group structure on $\Gamma(S)$ is induced by composition, and we use the usual order of composition of functions when writing the group multiplication.

Any two surfaces of type $(g,n+1)$ have isomorphic pure mapping class groups and we denote this group by $\Gamma_{g,n+1}$.
\end{definition}

\begin{remark}
    More generally, a \emph{mapping class} may refer to the isotopy class of a diffeomorphism between different manifolds that is not required to fix the boundary pointwise. Mapping classes that do fix the boundary pointwise are often called \emph{pure mapping classes}. Since all but one of the mapping classes considered in this paper (see Figure~\ref{fig:B-move}) fix the boundary pointwise, we will use the term \emph{mapping class} to mean \emph{pure mapping class}.
\end{remark}

\subsubsection{Presentations of mapping class groups}\label{subsubsec:presentations of mapping class groups} 

Mapping class groups of orientable surfaces are well understood, and several presentations are available \cite{GERVAIS2001703,humphries,Hatcher_Thurston,Wajnryb}. We highlight two cases that will play an important role throughout the paper. For a comprehensive reference, see \cite{farb-marg}.

\begin{example}
    The mapping class group of the disc is trivial, i.e. $\Gamma(D^2)=*$. This is known as Alexander’s lemma \cite[Lemma~2.1]{farb-marg}.
\end{example}

We now recall the mapping classes that are used as standard generators for higher genus surfaces.
Throughout this paper, we use the terms \emph{curve} and \emph{loop} to mean the isotopy class of an embedded closed curve in $S$. The geometric \emph{intersection number} $|a\cap b|$ of two curves $a$ and $b$ is the minimum number of intersection points among all representatives in their isotopy classes. Similarly, the \emph{self-intersection number} of a curve is the minimum self-intersection number among all representatives. A \emph{simple loop} is a loop of self-intersection number $0$. 

A fundamental family of mapping classes is given by \emph{Dehn twists} about simple loops. We denote the Dehn twist about a simple loop $a$ by $D_a$. 
% The mapping class group $\Gamma_{g,n+1}$ is generated by finitely many Dehn twists. In practice we work with a standard Humphries-type generating set, represented in Figure~\ref{fig: humphries generators}. 
See \cite[Section~4.4]{farb-marg} for details.

\begin{prop}\cite{Gervais96,Luo97}\label{prop: generators and relations of mcg}
    For a surface $S$ of type $(g,n+1)$ the mapping class group $\Gamma(S)$ is generated by all Dehn twists along simple closed curves together with the following relations:
    \begin{itemize}
        \item[(C)] If $|a\cap b|=0$,  then $D_aD_b=D_bD_a$.
        \item[(B)] If $|a\cap b|=1$ and $c=D_b(a)$, then $D_c = D_b D_a D_b^{-1}$.

        Relation (B) implies the \emph{braid relation} $D_a D_b D_a=D_b D_a D_b$, if $|a\cap b|=1$.
        \item[(D)] The \emph{doughnut relation}, $(D_{a}D_{b}D_{a})^4=D_\partial$ which takes place on a subsurface of $S$ of type $(1,1)$, $a$ and $b$ are the curves represented in Figure~\ref{fig:curves-for-doughnut-rel}, and $D_\partial$ denotes the twist around the boundary. 
        \item[(L)] The \emph{lantern relation}, $D_{\partial_0}D_{\partial_1}D_{\partial_2}D_{\partial_3}=D_{b_1}D_{b_2}D_{b_3}$, where the $\partial_i$ are loops in $S$ bounding a subsurface of type $(0,4)$ and the $b_j$ are as in Figure~\ref{fig:curves-for-lantern-rel}.
    \end{itemize}
\end{prop}

    % \[\Gamma(S) \cong \left<D_{a_{1}},D_{b_1},D_{a_2}, D_{b_2}\ldots,D_{b_{g-1}},D_{a_g},D_{m_1},D_{m_2},D_{c_0}, D_{c_1}\ldots,D_{c_{n}}\mid (\text{C}), (\text{B}), (\text{D}), \ \text{and} \ (\text{L})\right>\] where $D_{a_i}$, $D_{b_i}$, $D_{m_i}$, and $D_{c_i}$ denote the Dehn twists along the corresponding curves represented in Figure \ref{fig: humphries generators}. The relations are given as follows: 
% 

\begin{figure}[ht]
        \centering
        \hfill\hfill
        \begin{subfigure}[t]{0.45\textwidth}
            \centering
            \includegraphics[]{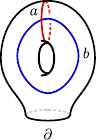}
            \caption{}
            \label{fig:curves-for-doughnut-rel}
            \end{subfigure}
                \hfill
            \begin{subfigure}[t]{0.45\textwidth}
                \centering
                \includegraphics[]{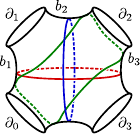}
                \caption{}
                \label{fig:curves-for-lantern-rel}
            \end{subfigure}
            \hfill\hfill
            \caption{}
        \label{fig:curves for presentation of mcg}
\end{figure}

We will also make repeated use of the following:
\begin{lemma}\label{lemma: conjugation by mapping classes}
    Let $S,S'$ be parametrized labelled surfaces and let $f:S\to S'$ be a diffeomorphism taking the marked point on the $i$th boundary component of $S$ to the corresponding marked point of $S'$. Let $a$ be a simple loop in $S$. Then
    \[ D_{f(a)} = f\cdot D_a\cdot f^{-1} \]
    as elements of $\Gamma(S')$. Moreover, the mapping class $f\cdot D_a\cdot f^{-1}$ depends only on the isotopy class of $f$ relative to the boundary.
\end{lemma}

\begin{proof}
    The proof is standard; see \cite[Fact~3.7]{farb-marg}.
\end{proof}

\subsubsection{A diagrammatic representation of mapping classes}\label{subsubsec: diagrammatical representation of mapping classes} 
The goal of this subsection is to give a concrete combinatorial model for mapping classes in genus $0$. More precisely, we describe how isotopy classes of boundary-fixing diffeomorphisms can be encoded by embedded trees. This provides the diagrammatic framework needed later to give a presentation of our modular operad.

Let $(S,\delta,\{x_i\}_{i=0}^{n})$ and $(S',\delta',\{x'_i\}_{i=0}^{n})$ be parametrised labelled surfaces of type $(0,n+1)$, each equipped with a chosen point $x_i\in \partial_i S$ and $x'_i\in \partial_i S'$ on the $i$th boundary component. We write $\Diff^+\bigl(S,S'; \partial \bigr)$ for the space of
orientation-preserving diffeomorphisms $\phi:S\to S'$ that preserve the boundary labels and satisfy $\phi(x_i)=x'_i$ for all $0\le i\le n$. Its set of path components
\[
\pi_0\ \Diff^+\bigl(S,S'; \partial \bigr)
\]
is the set of isotopy classes of such diffeomorphisms. In genus $0$ we will encode these isotopy classes by combinatorial data coming from embedded trees.

\begin{definition}\label{def: marking of genus 0 surface}
Let $T$ be a tree with $|\partial(T)|=n+1$ equipped with a boundary labelling
\[
\begin{tikzcd}
\ell_T: \{0,1,\ldots,n\} \arrow[r, "\cong"] & \partial(T)
\end{tikzcd}
\]
Let $(S,\delta,\{x_i\})$ be a labelled surface of type $(0,n+1)$ with chosen points $x_i\in \partial_i S$ for $0\le i\le n$. A \emph{marking} of $(S,\delta,\{x_i\})$ by $T$ is an equivalence class of embeddings $m:T\hookrightarrow S$ such that
\[
m\bigl(\ell_T(i)\bigr)=x_i\in \partial_i S, \qquad \text{ for all }i\in\{0,\ldots,n\},
\] 
where $m\sim m'$ if $m'$ is obtained from $m$ by an isotopy relative to $\partial S$ or by precomposition with an element of $\Homeo(T)$.

\medskip 

A \emph{marked genus-zero surface} is a pair $\bigl((S,\delta,\{x_i\}), [m]\bigr)$ consisting of a labelled genus-zero surface $(S,\delta,\{x_i\})$ and a marking $m:T\hookrightarrow S$ by a boundary-labelled tree $T$. 
\end{definition}

Any marked surface determines additional combinatorial structure on the underlying tree. Since $T$ is embedded in the oriented surface $S$, the incident half-edges at each vertex inherit a cyclic ordering from the orientation of $S$. A graph equipped with a cyclic ordering of the incident half-edges at each vertex is called a \emph{ribbon graph}, and the corresponding data is called a \emph{ribbon structure}.

In addition, if $U_T$ is a sufficiently small tubular neighbourhood of $T$ in $S$, then $\partial U_T$ is an oriented circle, and the legs of $T$ meet $\partial U_T$ in a cyclically ordered set. Via the boundary labelling $\ell_T$, this induces a cyclic order on the set of boundary labels $\{0,\ldots,n\}$.

\begin{figure}[h!t]
    \includegraphics[width=\textwidth]{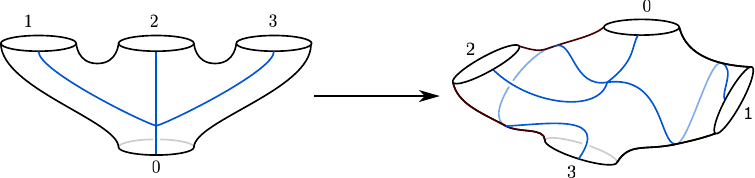}
    \centering
    \caption{Diagrammatic representation of a diffeomorphism $S\rightarrow S'$.}
    \label{fig: diagramatic rep of diffeo}
\end{figure}

\medskip 

Let $S=(S,\delta,\{x_i\}_{i=0}^{n})$ and $S'=(S',\delta',\{x'_i\}_{i=0}^{n})$ be parametrized labelled surfaces of type $(0,n+1)$, and let $[m]$ and $[m']$ be markings on $S$ and $S'$, respectively, that induce the same cyclic order on the boundary labels. Choose representatives
\[
m:T\hookrightarrow S
\qquad\text{and}\qquad
m':T'\hookrightarrow S'
\]
of the markings $[m]$ and $[m']$. Using these representatives, we construct an orientation-preserving diffeomorphism $S\to S'$ up to isotopy relative to the boundary as follows.

Choose small tubular neighbourhoods $U_T\subset S$ and $U_{T'}\subset S'$ of the embedded trees $m(T)$ and $m'(T')$. After smoothing corners, both $U_T$ and $U_{T'}$, as well as their complements
\[
S\setminus \operatorname{int}(U_T)
\qquad\text{and}\qquad
S'\setminus \operatorname{int}(U_{T'}),
\]
are discs. The markings and boundary data determine an orientation-preserving homeomorphism
\[
\partial U_T \longrightarrow \partial U_{T'}
\]
well-defined up to isotopy relative to the marked boundary points. Such a boundary map extends over each of the two discs uniquely up to isotopy relative to the boundary, since the mapping class group of the disc is trivial (Alexander’s lemma). We therefore obtain diffeomorphisms 
\[
\phi_{T,T'}^1:U_T\to U_{T'}
\qquad\text{and}\qquad
\phi_{T,T'}^2:S\setminus \operatorname{int}(U_T)\to S'\setminus \operatorname{int}(U_{T'}).
\]
Gluing these along $\partial U_T$ produces an orientation-preserving diffeomorphism
\[
\phi:S\to S',
\]
well-defined up to isotopy relative to $\partial S$.  The following lemma shows that the isotopy class of $\phi:S\rightarrow S'$ depends only on the isotopy classes of the markings on $S$ and $S'$, respectively.

\begin{lemma}\label{lemma: marking-independence} 
Let $S=(S,\delta,\{x_{i}\}_{i=0}^{n})$ and $S'=(S',\delta',\{x'_{i}\}_{i=0}^{n})$ be two parameterized labelled surfaces of type $(0,n+1)$, and let $[m]$ and $[m']$
be markings on $S$ and $S'$, respectively, inducing the same cyclic order on the boundary labels $\{0,\ldots,n\}$.

Let
\[
m:T\hookrightarrow S\qquad \text{and}\qquad \tilde m:\widetilde T\hookrightarrow S
\]
be two representatives of the marking $[m]$, and let
\[
m':T'\hookrightarrow S'\qquad \text{and}\qquad  \tilde m':\widetilde T'\hookrightarrow S'
\]
be two representatives of the marking $[m']$.
Then the diffeomorphisms obtained from the above construction using $(m,m')$ and $(\tilde m,\tilde m')$ determine the same isotopy class of the diffeomorphism $\phi: S\to S'.$ \qed
\end{lemma}

\subsection{The Complex of Markings}\label{sec: complex of markings}
In the study of presentations of mapping class groups, it is standard to consider a $2$-dimensional CW complex whose vertices encode decompositions of a surface and whose edges correspond to elementary moves between them. To give a presentation of our modular operad of mapping class groups, we use a variant of the complex of maximal markings of \cite[Section~5]{bk_marked_surfaces}.

A \emph{cut system} on a surface $S$ is a finite collection of pairwise disjoint simple loops $C=\{c_1,\ldots,c_k\}$, up to isotopy, such that the complement is a disjoint union of genus-zero punctured surfaces. Each puncture can be canonically replaced by a boundary component, yielding a disjoint union of genus-zero surfaces with boundary. By abuse of notation, we denote the resulting disjoint union of surfaces by
\[
\overline{S\setminus C}=\overline{S\setminus \bigcup_i c_i} \;=\; \coprod_j S_j.
\]
Each surface $S_j$ is canonically associated to a subsurface of $S$, and we refer to the surfaces $S_j$ as the \emph{cut-subsurfaces} determined by $C$. See, for example, \cite{Hatcher_Thurston,hls,moore1989classical}. 
%\lbnote{Slightly re-written: if a reader is very nitpicky, they would interpret the closure $\overline{S\setminus \bigcup_i c_i}$ in $S$. But then this closure is just the whole $S$. Original below}\mrnote{good point. }

We will enhance a cut system by equipping the cut-subsurfaces with compatible markings. These markings assemble into an embedding of a graph $G$ recording how the genus-zero pieces are connected: the vertices of $G$ correspond to cut-subsurfaces, while the internal edges correspond to the cut curves.

\begin{definition}\label{def: marked cut system}\cite[Definition~3.3]{bk_marked_surfaces}
Let $S=(S,\delta,\{x_i\}_{i=0}^n)$ be a labelled parameterized surface of type $(g,n+1)$ with $n\geq 0$. A \emph{marked cut system} on $S$ consists of a pair $(C,[m])$, where:
\begin{itemize}
    \item $C=\{c_1,\ldots,c_k\}$ is a cut system on $S$;

    \item $[m]$ is an equivalence class of embeddings $m:G\hookrightarrow S$, where $m\sim m'$ if $m'$ is obtained from $m$ by an isotopy relative to $\partial S$ or by precomposition with an element of $\Homeo(G)$.
\end{itemize}

These data are required to satisfy the following conditions:
\begin{enumerate}
    \item The embedding $m$ is compatible with the decomposition determined by
    $C$ in the following sense:
    \begin{itemize}
        \item the internal edges $iE_G$ of $G$ are in bijection with the curves in $C$, and the image of each internal edge meets the corresponding curve transversely in a single point and is disjoint from all other curves in $C$;
        
        \item the image of each vertex lies in a distinct cut-subsurface of $S$, and this determines a bijection of the vertices $V_G$ with the connected components of $\overline{S\setminus C}$;
        
        \item for each vertex $v\in V_G$, the half-edges incident to $v$ are in bijection with the boundary components of the corresponding cut-subsurface $S_v$.
    \end{itemize}

    \item The embedding $m$ sends each leg of $G$ to a marked boundary point of $S$. Equivalently, for every leg $e\in \partial(G)$ there exists $i\in\{0,\ldots,n\}$ such that the endpoint of $e$ is mapped to $x_i\in \delta(i)\subset \partial S$.

    \item For each vertex $v\in V_G$, the restriction of $m$ to the corolla subgraph $C_v\subset G$ induces a marking
    \[
    m_v:C_v\hookrightarrow S_v
    \]
    in the sense of Definition~\ref{def: marking of genus 0 surface}, where $S_v$ is the cut-subsurface corresponding to $v$ under (1).
\end{enumerate}

A marked cut system $(C,[m])$ is called a \emph{marked pants decomposition} if every cut-subsurface $S_j$ is a surface of type $(0,3)$.
\end{definition}

\begin{example}\label{ex:standard-markings}
For the pair of pants $P$, markings are represented by embeddings of the corolla $m:C_3\hookrightarrow P$. Figure~\ref{fig:standard-markings} depicts one such marking, which we call the \emph{standard marking} $m_{P}$ on $P$.

The surfaces $S_{\forkone}$ and $S_{\forktwo}$ are obtained by gluing two copies of $P$ in the two ways shown in Figure~\ref{fig:standard-markings}. The same figure shows the corresponding marked pants decompositions, which we call the \emph{standard markings} on $S_{\forkone}$ and $S_{\forktwo}$.

The surface $S_{\contraction}$ is obtained from a single copy of $P$ by gluing the boundary components $\partial_1$ and $\partial_2$. Its standard marking, also shown in Figure~\ref{fig:standard-markings}, is induced by gluing the corresponding half-edges of the embedded corolla in $P$.

\begin{figure}[h!]
    \centering
    \includegraphics[width=\linewidth]{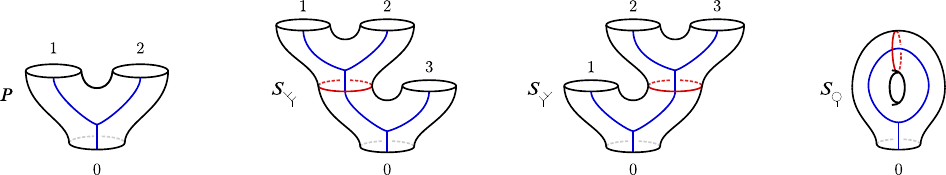}
    \caption{Standard markings on $P$, $S_{\forkone}$, $S_{\forktwo}$, and
    $S_{\contraction}$.}
    \label{fig:standard-markings}
\end{figure}
\end{example}

The complex of markings is a CW complex whose vertices are marked cut systems. Its edges are generated by the following mapping classes, written using the diagrammatic representation introduced in
Section~\ref{subsubsec: diagrammatical representation of mapping classes}.

\begin{definition}\label{def:standard moves}
    Let $P$, $S_{\forkone}$, $S_{\forktwo}$, and $S_{\contraction}$ be the surfaces described in Example~\ref{ex:standard-markings}.
    \begin{itemize}
        \item The standard $T$-move is the mapping class
            \[
            \tau:P\longrightarrow P
            \]
            represented in  Figure~\ref{fig:T-move}. This is precisely the Dehn twist about the boundary component labelled $0$.
        
        \item The standard $B$-move on the pair of pants is the mapping class 
            \[
            \beta:P\longrightarrow P
            \]
            represented in Figure~\ref{fig:B-move}. Note that this mapping class is actually an isotopy class of a diffeomorphism that does \emph{not} fix the boundary of $P$. In fact, it exchanges boundaries labelled $1$ and $2$.
            
            \begin{figure}[h!]
                \centering
                \begin{subfigure}[t]{0.45\textwidth}
                    \centering
                \includegraphics[]{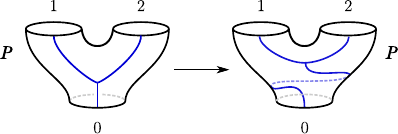}
                \caption{Standard $T$-move, $\tau$}
                \label{fig:T-move}
                \end{subfigure}
                \hfill
                \begin{subfigure}[t]{0.45\textwidth}
                    \centering
                    \includegraphics[]{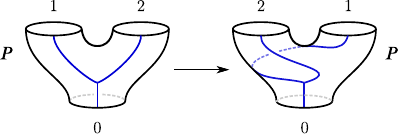}
                    \caption{Standard $B$-move, $\beta$}
                    \label{fig:B-move}
                \end{subfigure}
                \caption{The standard $T$ and $B$-moves.}
                \label{fig:standard-moves-t-b}
            \end{figure}
        
        \item The standard $A$-move is the mapping class
            \[
            \alpha:S_{\forkone}\longrightarrow S_{\forktwo}
            \]
            represented in Figure~\ref{fig:A-move}.    
            \begin{figure}[h]
                \centering
                \includegraphics[]{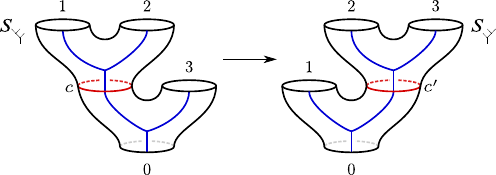}
                \caption{Standard $A$-move, $\alpha$}
                \label{fig:A-move}
            \end{figure}
        
        \item The standard $S$-move is the mapping class
            \[
            \sigma:S_{\contraction}\longrightarrow S_{\contraction}
            \] 
            given by
            \[
            \sigma=D_aD_bD_a\in \Gamma(S_{\contraction}),
            \] where $a$ and $b$ are the curves represented in Figure~\ref{fig:S-move}. It is simple to check that $\sigma$ sends the marked cut system $(\{a\},[m_{\contraction}])$ to $(\{b\},[\sigma\circ m_{\contraction}])$, as in Figure~\ref{fig:S-move}. 
            \begin{figure}[h]
                \centering
                \includegraphics[]{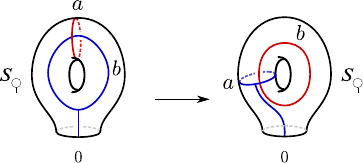}
                \caption{Standard $S$-move, $\sigma$}
                \label{fig:S-move}
            \end{figure}
    \end{itemize}
\end{definition}

We now define the CW complex of markings used throughout the paper. It is a variant of the \emph{complex of maximal markings} of \cite[Section~5]{bk_marked_surfaces}. The main difference is that we restrict to markings over pants decompositions, so the cut-subsurfaces are all pairs of pants. We refer to Appendix~\ref{app: complex of markings} for the full list of $2$-cell relations, and record below those most relevant for this paper.

\begin{definition}\label{def: complex of markings}

Let $S=(S,\delta,\{x_i\}_{i=0}^{n})$ be a parametrized, labelled surface of type $(g,n+1)$, with $n\ge 0$. The \emph{complex of markings over pants decompositions on $S$}, denoted $\markS$, is the $2$-dimensional CW complex defined as follows.

\smallskip
\noindent\textbf{Vertices.}
The vertices of $\markS$ are marked pants decompositions $(C,[m])$ on $S$.

\smallskip
\noindent\textbf{$1$-cells.}
The $1$-cells are generated by the following moves.

\begin{enumerate}
\item \textbf{$T$-moves:}
Given marked pants decompositions $(C,[m])$ and $(C',[m'])$, we define a $1$-cell
\[
(C,[m])\overset{T_a}{\rightsquigarrow}(C',[m'])
\]
whenever $C=C'$ and the markings $m$ and $m'$ differ only on a cut-subsurface $S'\subset \overline{S\setminus C}$ having $a$ as one of its boundary curves, in the following way: under an identification $\phi:S'\xrightarrow{\cong} P$ satisfying $\phi\circ m|_{S'}=m_P$ and $\phi(a)=\partial_0$,
we have
    \[\phi\circ m'|_{S'}=\tau\circ m_P.\]

Thus a $T$-move keeps the underlying pants decomposition fixed and changes the
marking by a Dehn twist along one cut curve.

\item \textbf{$B$-moves:}
Given marked pants decompositions $(C,[m])$ and $(C',[m'])$, we define a $1$-cell
\[
(C,[m])\overset{B_{a,b,c}}{\rightsquigarrow}(C',[m'])
\]
whenever $C=C'$ and the markings $m$ and $m'$ differ only on a cut-subsurface $S'\subset \overline{S\setminus C}$ with boundary curves $a,b,c$, in the following way: under an identification
$\phi:S'\xrightarrow{\cong}P$ satisfying
\[
\phi\circ m|_{S'}=m_P,\qquad
\phi(a)=\partial_1,\qquad
\phi(b)=\partial_2,\qquad \text{and} \qquad
\phi(c)=\partial_0,
\]
we have
    \[ \phi\circ m'|_{S'}=\beta\circ m_P. \]
When the curves $a$ and $b$ uniquely determine the pair of pants $S'$, we omit $c$ from the notation and write $B_{a,b}$.

\item \textbf{$A$-moves:} 
Given marked pants decompositions $(C,[m])$ and $(C',[m'])$, we define a $1$-cell
\[
(C,[m])\overset{A_{a,b}}{\rightsquigarrow}(C',[m'])
\] 
whenever $C\setminus\{a\}=C'\setminus\{b\}$ and the markings $m$ and $m'$ differ only on a cut-subsurface 
\[
S'\subset
\overline{S\setminus (C\setminus\{a\})}
=
\overline{S\setminus (C'\setminus\{b\})}
\]
of type $(0,4)$, in the following way: under an identification $\phi:S'\xrightarrow{\cong} S_{\forkone}$ taking the local marked cut system $(\{a\},m|_{S'})$ to the standard marking on $S_{\forkone}$, the composite
\[
\alpha\circ\phi:S'\xrightarrow{\cong} S_{\forktwo}
\]
takes the local marked cut system $(\{b\},m'|_{S'})$ to the standard marking on $S_{\forktwo}$.

\item \textbf{$S$-moves:} 
Given marked pants decompositions $(C,[m])$ and $(C',[m'])$, we define a $1$-cell
\[
(C,[m])\overset{S_{c',c}}{\rightsquigarrow}(C',[m'])
\] 
whenever $C\setminus\{c\}=C'\setminus\{c'\}$ and the markings $m$ and $m'$ differ only on a cut-subsurface 
\[
S'\subset
\overline{S\setminus (C\setminus\{c\})}
=
\overline{S\setminus (C'\setminus\{c'\})}
\] 
of type $(1,1)$, in the following way: under an identification $S'\xrightarrow{\cong} S_{\contraction}$ taking the local marked cut system $(\{c\},m|_{S'})$ to the standard marking on $S_{\contraction}$, we have
\[
c'=\sigma(c)
\qquad\text{and}\qquad
m'|_{S'}=\sigma\circ m|_{S'}.
\]

\end{enumerate}
\smallskip
\smallskip
\smallskip
\noindent\textbf{$2$-cells.}
The $2$-cells are attached along loops expressing relations among the $1$-cells.
Most of these coincide with the relations in \cite{bk_marked_surfaces}; see also
Appendix~\ref{app: complex of markings}.

In describing these relations, we use the following conventions. For
$f\in\Gamma(S)$ and a marking $m:G\hookrightarrow S$, we write $f\circ m$ for
the marking represented by $G\xrightarrow{m}S\xrightarrow{f}S$. For a subsurface
$S'\subset S$, we write $m|_{S'}$ for the induced marking obtained by restricting
the image of $m$ to the corresponding subgraph. In particular, if $S'$ is a
cut-subsurface, then $m|_{S'}$ is the
corolla marking $m_v:C_v\hookrightarrow S'$ associated to the corresponding
vertex $v\in V_G$.

We include the following families of relations, which are described in detail in Appendix~\ref{app: complex of markings}:
\begin{enumerate}
\item commutativity of disjoint moves;
\item boundary twist relations;
\item pentagon relation;
\item two hexagon relations;
\item self-duality of associativity,
% \[
% A_{c,c'}A_{c',c}=\id;
% \]
\item  twist--braiding relation;
\item  gluing relations for Dehn twists.
\end{enumerate}

\smallskip
\noindent
In addition to the local relations above, we impose the following genus-one relations. For surfaces of type $(1,1)$, we impose
\begin{align}
    % S_{a,b} S_{b,a} B_{\partial_0,a,a} &= \id, \tag{G1.1}\label{eq:g1.1}\\
    SSB &= \id, \tag{G1.1}\label{eq:g1.1}\\
    T^{-1}ST^{-1}ST^{-1} &= S\tag{G1.2}\label{eq:g1.2}
\end{align}
shown in more detail in Figures~\ref{fig:G2-cell-G1-1}
and~\ref{fig:G2-cell-G1-2}.

\begin{figure}[h!]
    \centering
    \begin{subfigure}[t]{0.415\linewidth}
        \centering
        \includegraphics[width=\linewidth]{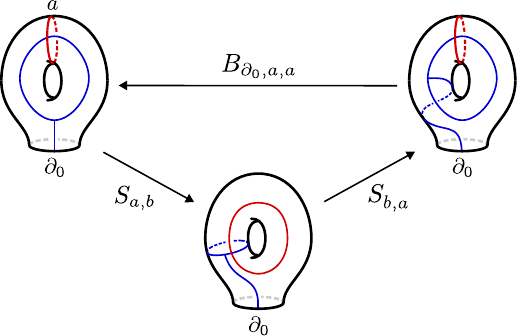}
        \caption{$2$-cell (G1.1)}
        \label{fig:G2-cell-G1-1}
    \end{subfigure}
    \hfill
    \begin{subfigure}[t]{0.545\linewidth}
        \centering
        \includegraphics[width=\linewidth]{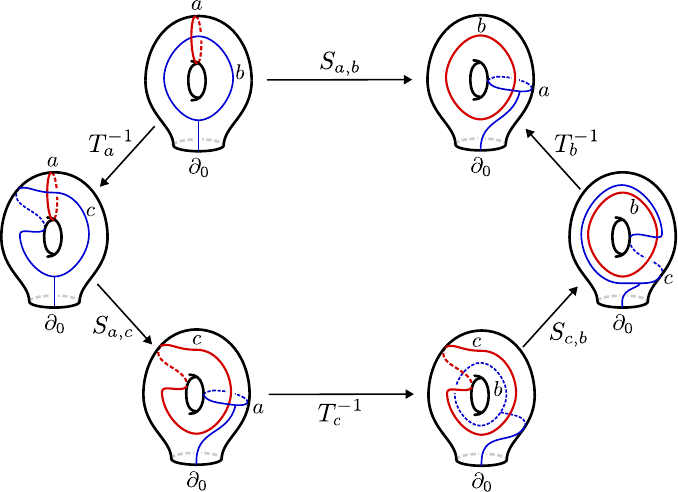}
        \caption{$2$-cell (G1.2)}
        \label{fig:G2-cell-G1-2}
    \end{subfigure}
    \caption{The genus-one $2$-cells (G1.1) and (G1.2).}
    \label{fig:G2-cells-G1}
\end{figure}
For surfaces of type $(1,2)$, we also impose the relation shown in
Figure~\ref{fig:G2-cell-NS-complex-of-markings}. This relation differs slightly
from the corresponding relation in \cite{bk_marked_surfaces}.

\begin{figure}[h!]
    \centering
    \includegraphics[width=\linewidth]{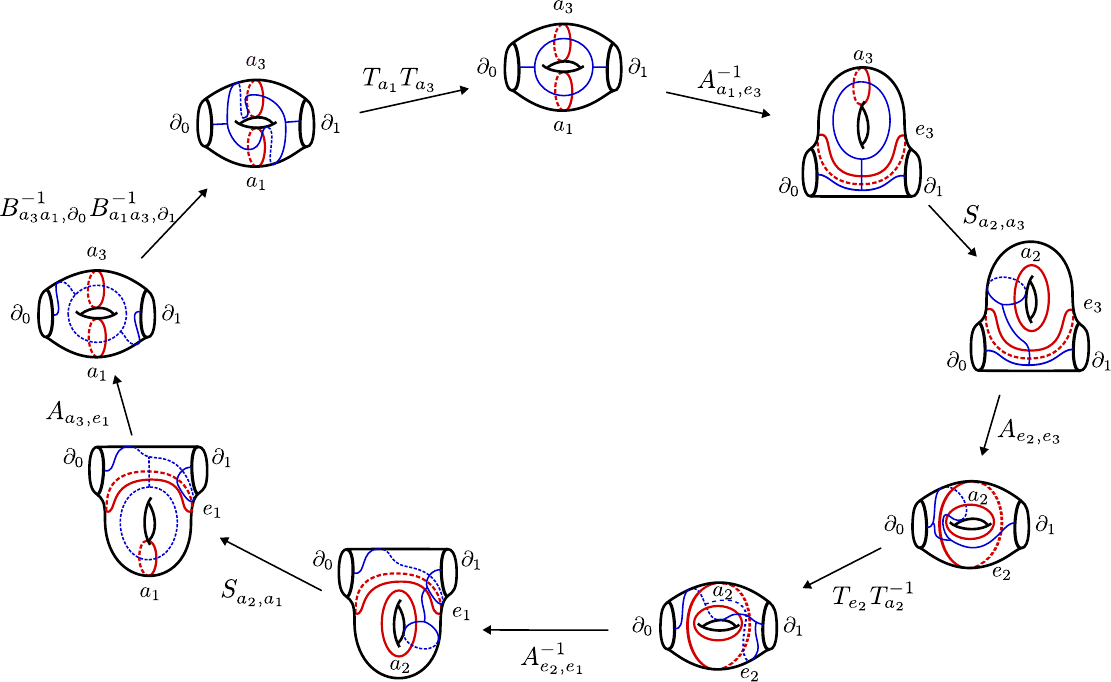}
    \caption{$2$-cell (G2)}
    \label{fig:G2-cell-NS-complex-of-markings}
\end{figure}
\end{definition}

The relations above are chosen so that the resulting complex has the expected connectivity properties. The result we will use later is the following.

\begin{prop}\label{prop:our-complex-of-markings-is-simply-connected}
Let $S=(S,\delta,\{x_i\}_{i=0}^{n})$ be a parameterized, labelled surface of type $(g,n+1)$, with $n\geq 0$. Then the complex of markings over pants decompositions on $S$, $\markS,$ is connected and simply connected.
\end{prop}

A proof is given in Appendix~\ref{app: complex of markings}.

\section{A Modular Operad of Mapping Class Groupoids}\label{sec: modular operad in groupoids}
In this section, we define a modular operad in groupoids $\bS=\{\bS(g,n+1)\}_{g\ge 0,\, n\ge 0},$ which homotopically approximates a modular operad constructed from the mapping class groups in the sense that, \[\cs\bS(g,n+1)\simeq \cs\Gamma_{g,n+1}.\] In Section~\ref{sec: presentation of modular operad} we will compare each groupoid $\bS(g,n+1)$ with the corresponding complex of markings. This comparison will allow us to give a finite presentation of $\bS$ by generators and relations.

\subsection{The Modular Operad $\bS$}\label{sec: modular operad of mapping class groups}
Recall that we write $\Sigma_n^+ = \Aut({0,1,\ldots,n})$ for the extended symmetric group. If $C_{n+1}$ is a corolla with (n+1) boundary legs, then a labelling of its boundary by the set ${0,1,\ldots,n}$ is equivalent to a choice of element of $\Sigma_n^+$. Indeed, once a reference labelling is fixed, every other labelling is obtained by precomposition with a unique permutation in $\Sigma_n^+$.

In the case $n=2$, the alternating subgroup of $\Sigma_2^+ = \Aut(\{0,1,2\})$ is given by $
A_2^+ = \{\id,(012),(021)\}$. The quotient $\Sigma_2^+/A_2^+$ has two elements, represented by
\[
\id\cdot A_2^+
\qquad\text{and}\qquad
(12)\cdot A_2^+.
\]
%Geometrically, these elements correspond to the two planar orientations of the trivalent corolla with legs labelled by $\{0,1,2\}$. 
The permutations in $A_2^+$ preserve the cyclic ordering of the labels, so the three labellings in each coset determine the same cyclic ordering of the boundary legs. 

\begin{definition}\label{def:free modular operad generated by mu}
Let $\Upsilon=\{\Upsilon(g,n+1)\}$ be the free modular operad in $\Set$ generated by the set \[
\Upsilon(0,3)=\Sigma_2^+/A_2^+.
\] 
The set $\Upsilon(0,3)$ consists of the two planar trivalent corollas with boundary labelled by $\{0,1,2\}$, which we denote by
\[
P=\id\cdot A_2^+
\qquad\text{and}\qquad
(12)^*P=(12)\cdot A_2^+.  
\] Figure~\ref{fig:objects-of-PP2} depicts elements of the set $\Upsilon(0,3)$.
\end{definition}

 \begin{figure}[ht]
        \centering
        \includegraphics[width=0.4\textwidth]{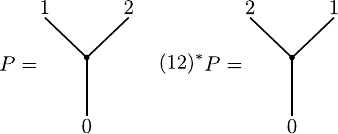}
        \caption{Generators of the modular operad $\Upsilon$.}
        \label{fig:objects-of-PP2}
    \end{figure}
    
\medskip 
Following Definition~\ref{def: free modular operad}, we see that an element of $\Upsilon(g,n+1)$ is an isomorphism class of connected decorated graphs $G=(G,\lambda, \ell_G,\epsilon,\underline{\upsilon})$ where:
\begin{enumerate}
    \item every vertex $v\in V_G$ is trivalent and has genus $\epsilon(v)=0$;
    \item The vertex decoration $\underline{\upsilon} = (\upsilon_{\lambda(1)}, \ldots, \upsilon_{\lambda(k)})$ consists of a sequence of elements of $\Upsilon(0,3)$. In other words, each vertex $v_{\lambda(i)}$ is decorated by one of the two planar trivalent corollas in $\Upsilon(0,3)$, equivalently by a bijection
    \[
    \ell_{v_{\lambda(i)}}:\{0,1,2\}\xrightarrow{\cong}\nb(v_{\lambda(i)})
    \]
    defined up to the action of $A_2^+$;
    \item the boundary of $G$ is equipped with a labelling
    \[
    \ell_G:\{0,\ldots,n\}\xrightarrow{\cong}\partial(G);
    \]
    \item the total genus of $G$ is $g$, equivalently $\beta_1(G)=g$.
\end{enumerate}

Condition~(2) is equivalent to equipping each vertex of $G$ with a cyclic ordering of its incident half-edges. In other words, $G$ is a ribbon graph, and the isomorphism classes above are isomorphism classes of ribbon graphs. Accordingly, we will often regard elements of $\Upsilon(g,n+1)$ as ribbon graphs with first Betti number $g$ and $n+1$ ordered legs. However, the interpretation of $\Upsilon(g,n+1)$ as the free modular operad generated by a symmetric sequence will be essential in later sections.

\medskip

Our next goal is to define, for each $g,n\geq0$ satisfying $2g+n\geq1$, a groupoid $\bS(g,n+1)$ whose objects are the decorated graphs in $\Upsilon(g,n+1)$. The construction is arranged so that, for every $G\in\Upsilon(g,n+1)$, there is a natural isomorphism
\[
\Aut_{\bS(g,n+1)}(G)\cong\Gamma(S_G),
\] where $S_G$ is the surface associated to $G$. 

We now describe this surface.  Let $G=(G,\lambda,\ell_G,\epsilon,\underline{\upsilon})\in \Upsilon(g,n+1)$
be a decorated graph. From $G$ we construct a labelled, parametrized surface
\[
S_G=(S_G,\delta,\{x_i\}_{i=0}^n)
\]
of type $(g,n+1)$, together with a marking over a pants decomposition in the
sense of Definition~\ref{def: marked cut system}, as follows.
\begin{enumerate}
    \item For each vertex $v_{\lambda(i)}\in V_G$, let $P_{v_{\lambda(i)}}=(P_{v_{\lambda(i)}},\delta_{v_{\lambda(i)}},\{x_0,x_1,x_2\})$ denote a copy of the standard pair of pants corresponding to the decoration $\upsilon_{\lambda(i)}\in \Upsilon(0,3).$ Thus the boundary components of $P_{v_{\lambda(i)}}$ are labelled by $\{0,1,2\}$ according to the labelling of the legs of the corolla decorated by $\upsilon_{\lambda(i)}$, i.e.  \[
    \begin{tikzcd}
    \delta := \ell_{P_{v_{\lambda(i)}}}: \{0,1,2\} \arrow[r] & \pi_0(\partial P_{v_{\lambda(i)}}).
    \end{tikzcd}
    \] Moreover, $P_{v_{\lambda(i)}}$ comes equipped with the standard marking (Figure~\ref{fig:standard-markings})
\[
m_{v_{\lambda(i)}}:C_{v_{\lambda(i)}}\hookrightarrow P_{v_{\lambda(i)}},
\]
where the image of the $j$-th leg determines the marked point $x_j$ on the $j$th boundary component of $P_{v_{\lambda(i)}}$.

\item If two vertices of $G$ are joined by an internal edge, we glue the corresponding boundary components of the associated pairs of pants. In this way one obtains a surface of type $(g,n+1)$, which we denote by $S_G$. 

    More precisely, if the graph $G$ is presented as the coequalizer
    \[
    \begin{tikzcd}
        \coprod\limits_{e\in iE_G}\updownarrow 
        \arrow[r, shift left=1.2] 
        \arrow[r, shift right=1.2]
        & 
        \coprod\limits_{v_{\lambda(i)}\in V_G} C_{v_{\lambda(i)}} 
        \arrow[r] 
        & 
        G,
    \end{tikzcd}
    \]
    as in Lemma~\ref{lem:coequalizer-graph}, then the associated surface $S_G$ is the coequalizer
    \[
    \begin{tikzcd}
        \coprod\limits_{e\in iE_G} S^1 
        \arrow[r, shift left=1.2] 
        \arrow[r, shift right=1.2]
        & 
        \coprod\limits_{v_{\lambda(i)}\in V_G} P_{v_{\lambda(i)}} 
        \arrow[r] 
        & 
        S_G.
    \end{tikzcd}
    \]
We note that to ensure that $S_G$ is a smooth surface one should perform the gluing along the collar neighbourhoods of the boundary components. We omit these technical details for brevity. 

\item The boundary components of $S_G$ are labelled via the boundary labelling of $G$, that is,
    \[
    \begin{tikzcd}
    \delta := \ell_G : \{0,1,\ldots,n\} \arrow[r] & \pi_0(\partial S_G).
    \end{tikzcd}
    \]
    For each $i\in\{0,\ldots,n\}$, let $x_i\in \delta(i)$ denote the point where the marking $m_G$ meets the boundary component $\delta(i)$.

\item     By construction, the image of $\coprod_{e\in iE_G} S^1$ in $S_G$ defines a pants decomposition.  The local markings $m_{v_{\lambda(i)}}:C_{v_{\lambda(i)}}\hookrightarrow P_{v_{\lambda(i)}}$ glue together to give a natural embedding $m_G:G\hookrightarrow S_G.$ Together this defines a marked pants decomposition $(iE_G, [m_G])$ of the surface $S_G$. 

\end{enumerate}

\begin{definition}\label{def: surface modular operad}
We define a family of groupoids $\bS(g,n+1)$, for
$g,n\geq0$ satisfying $2g+n\geq1$, as follows.

For $(g,n+1)=(0,2)$, let $\bS(0,2)$ be the groupoid with a
single object $|$ and
\[
\Hom_{\bS(0,2)}(|,|)=\mathbb Z.
\]
We regard this as the groupoid version of the mapping class group of
a thin cylinder, whose morphisms correspond to Dehn twists along the
boundary; e.g.~\cite[Section~3.1.1]{Wahl_infinite_loop_space}.

\medskip

For every other pair $g,n\geq0$ satisfying $2g+n\geq1$, the objects
of $\bS(g,n+1)$ are the decorated graphs in $\Upsilon(g,n+1)$:
\[
\ob(\bS(g,n+1))=\Upsilon(g,n+1).
\]
Given $G,G'\in\Upsilon(g,n+1)$, define
\[
\Hom_{\bS(g,n+1)}(G,G')
:=
\pi_0\Diff(S_G,S_{G'};\partial),
\]
where the diffeomorphisms preserve the boundary labelling and send
the marked points on the boundary of $S_G$ to the corresponding
marked points on the boundary of $S_{G'}$. Composition is induced by
composition of diffeomorphisms.
\end{definition}

The group $\Sigma_1^+$ acts trivially on the exceptional entry
$\bS(0,2)$. On every other entry $\bS(g,n+1)$, the extended symmetric
group $\Sigma_n^+$ acts by permuting the labels on the boundary legs
of the underlying decorated graphs. Thus
\[
\bS=\{\bS(g,n+1)\}_{g,n\geq0,\;2g+n\geq1}
\]
is a genus-graded sequence in groupoids.

We now define the modular operad structure. For $n+m\geq1$ and boundary labels $0\leq i\leq n$ and $0\leq j\leq m$, the composition functor
\[
\circ_i^j:
\bS(g,n+1)\times\bS(h,m+1)
\longrightarrow\bS(g+h,n+m)
\]
is defined on objects by grafting decorated graphs, so that $(G,H)$ is sent to $G\circ_i^j H$, as in Definition~\ref{def: free modular operad}.

The unique object $|\in\bS(0,2)$ is the unit for graph grafting. The generating morphism $1\in \mathbb Z=\Hom_{\bS(0,2)}(|,|)$ acts by Dehn twists: composing it at the $i$th input with a morphism $[\phi]:S_G\to S_H$ produces the mapping class $[D_{\partial_i}\circ \phi]:S_G\to S_H,$ where $\partial_i$ denotes the $i$th boundary component of $S_H$. 

For all other pairs $(g,n)$, the associated surface $S_{G\circ_i^j H}$ is obtained by gluing $S_G$ and $S_H$ along the boundary components labelled $\ell_G(i)$ and $\ell_H(j)$. The same gluing operation defines the functor on morphisms. Namely, if $\phi:S_G\to S_{G'}$ and $\psi:S_H\to S_{H'}$ are boundary-label-preserving diffeomorphisms, then they carry the chosen boundary components to the corresponding boundary components of $S_{G'}$ and $S_{H'}$. Hence they glue to a diffeomorphism $\phi\circ_i^j\psi:
S_{G\circ_i^j H}
\longrightarrow
S_{G'\circ_i^j H'}.$
This construction is compatible with isotopy classes, and therefore defines the desired functor.

% For $(g,n)=(0,2)$, there is a single object $|$ in $\bS(0,2)$ which behaves as the identity in graph grafting. The generating morphism $1\in\Z\in S(0,2)$ when composed with with a morphism $[\phi]:S_G\to S_H$ at an input $i$, gives as a result the mapping class $[D_{\partial_i}\cdot \phi]:S_G\to S_H$, where $\partial_i$ denotes the $i$th boundary of $S_H$. For all other pairs $(g,n)$, we see that the associated surface
% $S_{G\circ_i^j H}$ is obtained by gluing $S_G$ and $S_H$ along the boundary
% components labelled $\ell_G(i)$ and $\ell_H(j)$.

% The same gluing operation defines the functor on morphisms. Namely, if
% $\phi:S_G\to S_{G'}$ and $\psi:S_H\to S_{H'}$ are boundary-label-preserving
% diffeomorphisms, then they carry the chosen boundary components to the
% corresponding boundary components of $S_{G'}$ and $S_{H'}$. Hence they glue to a
% diffeomorphism
% $\phi\circ_i^j\psi:S_{G\circ_i^j H}\to S_{G'\circ_i^j H'}$.
% This construction is compatible with isotopy classes, and therefore defines the
% desired functor. 

Similarly, for $n\geq2$ and distinct labels
$0\leq i,j\leq n$, the contraction functor
\[
\xi_i^j:\bS(g,n+1)\longrightarrow\bS(g+1,n-1)
\]
is defined on objects by grafting together the two boundary legs of $G$
labelled $\ell_G(i)$ and $\ell_G(j)$. On surfaces, this amounts to
gluing the corresponding boundary components of $S_G$, thereby
increasing the genus by one and reducing the number of boundary
components by two.

On morphisms, a boundary-label-preserving diffeomorphism
$\phi:S_G\to S_{G'}$ induces, after gluing, a diffeomorphism
$S_{\xi_i^jG}\to S_{\xi_i^jG'}$. This construction is compatible with
isotopy classes and hence defines the contraction functor.

\begin{definition}\label{def: modular operad of mapping class groups}
The \emph{modular operad of surfaces}, 
\[
\bS=\{\bS(g,n+1)\}_{2g+n\geq 1},
\]
is the modular operad in groupoids whose $(g,n+1)$-component is the groupoid $\bS(g,n+1)$ of genus $g$ surfaces with $n+1$ marked boundary components.  The modular operad structure is induced by gluing boundary components, or equivalently by the grafting and contraction operations on graphs.

\medskip

The \emph{cyclic operad of genus-zero surfaces}, denoted $\trun{0}\bS$, is
obtained by restricting $\bS$ to genus zero. Its component with $n+1$ boundary
components is
\[
\trun{0}\bS(n+1)=\bS(0,n+1).
\]
The cyclic operad structure is induced by the gluing operations in $\bS$ which glue distinct genus-zero surfaces along boundary components. On markings, these operations correspond to grafting the underlying labelled trees.
\end{definition} 

By construction, for every object $G\in \ob(\bS(g,n+1))$, the automorphism group of
$G$ is the mapping class group of the associated surface:
\[
\Aut_{\bS(g,n+1)}(G)=\Gamma(S_G)\cong \Gamma_{g,n+1}.
\]
Since $\bS(g,n+1)$ is a connected groupoid, the following lemma follows
immediately.

\begin{lemma}
For each $g\geq 0$ and $n\geq 0$ satisfying $2g+n\geq 1$, there is a weak equivalence
\[
\cs\bS(g,n+1)\simeq \cs\Gamma_{g,n+1}.
\]
\end{lemma}

%\begin{proof}
%By construction, $\bS(g,n+1)$ is a connected groupoid. Moreover, for any object $G\in \bS(g,n+1)$, its automorphism group is naturally isomorphic to the mapping class group:
%\[
%\Aut_{\bS(g,n+1)}(G)\cong \Gamma_{g,n+1}.
%\]
%Since a connected groupoid is equivalent to the one-object groupoid associated to the automorphism group of any of its objects, it follows that $\bS(g,n+1)$ is equivalent to the one-object groupoid with endomorphism group $\Gamma_{g,n+1}$. Therefore
%\[
%\cs\bS(g,n+1)\simeq \cs\Gamma_{g,n+1}.
%\]
%\end{proof}

\begin{remark}
The modular operad in Definition~\ref{def: modular operad of mapping class groups} is closely related to the surface operads of \cite[Construction 2.2]{Till_surface_operad} and \cite[Section 3.1]{Wahl_infinite_loop_space}. These constructions all build surfaces by gluing basic pieces along boundary components, but they differ in which pieces are taken as ``atomic''. In the constructions of Tillmann and Wahl, the atomic surfaces include the pair of pants, the one-holed torus, and the disc. In our setting, the pair of pants is the only atomic surface.
\end{remark}

\subsubsection{A presentation of the modular operad $\bS$}\label{sec: presentation of modular operad}
Fix an object $G\in\ob(\bS(g,n+1))$, regarded as a decorated graph, and let $S_G$ denote the associated surface equipped with its standard marked pants
decomposition
\[
(iE_G,[m_G]:G\hookrightarrow S_G).
\]
We compare the topological encoding of markings of $S_G$, given by the complex of markings, with the categorical encoding given by the coslice category $G\downarrow \bS(g,n+1)$.

Recall the complex of markings $\calM(S_G)$ from Definition~\ref{def: complex of markings}. We write $\mathbb{M}_G$ for the groupoid presented by this $2$-complex. Its objects are the $0$-cells of $\calM(S_G)$, its generating morphisms are the edges of
$\calM(S_G)$, and its relations are the boundaries of the $2$-cells. Equivalently, $\Hom_{\mathbb{M}_G}(v,w)$ consists of homotopy classes, relative endpoints, of edge paths from $v$ to $w$.

\begin{definition}\label{defn: coslice S under G}
Fix a graph $G\in \ob(\bS(g,n+1))$. The \emph{coslice category} $G\downarrow \bS(g,n+1)$ is the category whose objects are morphisms $\phi:G\to H$ in $\bS(g,n+1)$, as $H$ ranges over the objects of $\bS(g,n+1)$.

A morphism from $\phi_1:G\to H_1$ to $\phi_2:G\to H_2$ is a morphism $\psi:H_1\to H_2$ in $\bS(g,n+1)$ such that $\psi\phi_1=\phi_2$. In other words, a morphism from $\phi_1\rightarrow \phi_2$ is a morphism $\psi:H_1\to H_2$ which makes the diagram
\[
\begin{tikzcd}
& G \arrow[dl, "\phi_1"'] \arrow[dr, "\phi_2"] & \\
H_1 \arrow[rr, "\psi"'] && H_2
\end{tikzcd}
\]
commute.
\end{definition}

\begin{lemma}\label{lemma: the complexes of morphisms and markings}
Fix a graph $G\in \ob(\bS(g,n+1))$ and let $S_G$ be the associated labelled, parametrized surface. Then the coslice category $G\downarrow \bS(g,n+1)$ is isomorphic to the groupoid $\mathbb{M}_G$.
\end{lemma}

%This lemma will be crucial in our future proofs. In particular, the observation that morphisms that compose to $\id_G$ correspond to loops in $\calM(S_G)$.

\begin{proof}
We first construct a functor $I:\mathbb{M}_G\longrightarrow G\downarrow \bS(g,n+1)$ and then show that it is an isomorphism.

To define $I$ on objects,  let $(C,[m]:H\hookrightarrow S_G)$ be a vertex of the complex of markings $\calM(S_G)$. Since $C$ is a pants decomposition, the components of $\overline{S_G\setminus C}$ are pairs of pants. Thus the graph $H$ is trivalent, with internal edges corresponding to the curves of $C$ and vertices corresponding to the pair-of-pants
components. Since every pants decomposition of a surface of type $(g,n+1)$ has $2g+n-1 = |\chi(S_G)|$ pairs of pants, the graph $H$ has $2g+n-1$ trivalent vertices. An Euler characteristic argument then shows that $\beta_1(H)=g$.

The marking $[m]$ determines the additional structure required to regard $H$ as an object of $\bS(g,n+1)$. Indeed, the orientation of $S_G$ induces a cyclic ordering on the half-edges incident to each vertex, while the labelling of the boundary components of $S_G$ induces a labelling of the boundary legs of $H$. Thus the marked pants decomposition determines an object of $\bS(g,n+1)$, which we again denote by $H$.

%We now convert the marked cut system $(C,m)$ into an object of $\ob(\bS(g,n+1))$.  For each trivalent vertex $v\in V_{H}$ the corresponding subsurface $S'_v\subset S_G$ is a pair of pants. The embedding $m: H\hookrightarrow S_G$ determines a unique labelling permutation $\ell_v\in\Sigma_2^+$ (relative to the chosen ordering of the boundary components of the standard pair of pants), and we decorate $v$ by the corresponding element \[\mu_v = \mu(\ell_v(0),\ell_v(1),\ell_v(2)).\]

The object $H$ comes with its standard surface $S_H$ and standard marked pants decomposition $(C_H,[m_H])$. By construction, the marked pants decomposition $(C,[m])$ of $S_G$ and the standard marked pants decomposition of $S_H$ have the same associated graph. Therefore the local identifications of the
corresponding pairs of pants glue to a boundary-label-preserving diffeomorphism $\phi:S_G\longrightarrow S_H,$ which is well defined up to isotopy. We set
\[
I(C,[m]) = [\phi:G\to H],
\]
viewed as an object of the coslice category $G\downarrow \bS(g,n+1)$.

%Thus we obtain a planar, labelled, decorated graph \[H = (H, \lambda_{m}, \ell_{m}, {\mu}) \in \ob(\bS(g,n+1)).\] Recall that to such an object $H$ we associate a surface $S_H$ equipped with its standard marking over a pants decomposition $(C_H,m_H)$. For each vertex $v\in V_H$, let $S_v\subset S_G\setminus C$ and $S'_v\subset S_H\setminus C_H$ denote the corresponding subsurfaces, both of type $(0,3)$. The markings on $S_G$ and $S_H$ restrict to markings on $S_v$ and $S'_v$, which determine a unique mapping class \[\phi_v : S_v \longrightarrow S'_v\] as in Definition~\ref{def: diagramatic description of diffeo}. 

%Since the boundary parameterizations agree along the cut curves, the mapping classes $\phi_v$ are compatible along the glued boundary components and therefore assemble to a unique global mapping class    \[[\phi] : S_G \longrightarrow S_H.\] We define    \[I(C,m) := [\phi] \in \Hom_{\bS(g,n+1)}(G,H).\]

\medskip

We now define $I$ on morphisms. Let $\alpha:(C,[m])\longrightarrow (C',[m'])$ be a morphism in $\mathbb{M}_G$, and suppose
\[
I(C,[m])=[\phi:G\to H]
\qquad \text{and} \qquad
I(C',[m'])=[\phi':G\to H'].
\]
There is a unique morphism in the coslice category from $[\phi]$ to $[\phi']$, namely the mapping class
\[
[\phi'\phi^{-1}]:H\longrightarrow H'.
\]
We set $I(\alpha)=[\phi'\phi^{-1}]$. %Equivalently, $I(\alpha)$ is the unique morphism making the triangle
%\[
%\begin{tikzcd}
%& G \arrow[dl, "{[\phi]}"'] \arrow[dr, "{[\phi']}"] & \\
%H \arrow[rr, "{[\phi'\phi^{-1}]}"'] && H'
%\end{tikzcd}
%\]
%commute.

\medskip 
It remains to show that $I$ is an isomorphism. On objects, $I$ is bijective. Indeed, given an object $[\phi]:G\to H$ of $G\downarrow \bS(g,n+1)$, one can pull back the standard marked pants decomposition $(C_H,[m_H])$ of $S_H$ along $\phi$ to define a marked pants decomposition
\[
(\phi^{-1}(C_H),[\phi^{-1}m_H])
\] of $S_G$. This construction is inverse to the one above.

For morphisms, first observe that the coslice category $G\downarrow \bS(g,n+1)$ has exactly one morphism between any two objects. Indeed, if the objects are $[\phi]:G\to H$ and $[\phi']:G\to H'$, then a morphism between them is a morphism $[\psi]:H\to H'$ satisfying $[\psi]\circ[\phi]=[\phi']$. Since $\bS(g,n+1)$ is a groupoid, $[\phi]$ is invertible, so $[\psi]$ is necessarily $[\phi']\circ[\phi]^{-1}$. %Conversely, this morphism clearly satisfies the required relation.

On the other hand, $\mathbb{M}_G$ is the fundamental groupoid of $\calM(S_G)$ on its vertices. Since $\calM(S_G)$ is connected and simply connected by Proposition~\ref{prop:our-complex-of-markings-is-simply-connected}, there is exactly one homotopy class of paths between any two vertices. Thus $\mathbb{M}_G$ also has exactly one morphism between any two objects. Therefore $I$ is bijective on all Hom-sets. Hence $I$ is an isomorphism of categories. It follows that
\[
\mathbb{M}_G\cong G\downarrow \bS(g,n+1).
\]
\end{proof}

\medskip

Recall the mapping classes that described the standard $T$-, $B$-, $A$-, and $S$-moves in Definition~\ref{def:standard moves}. We write
\[
\tau \in \Hom_{\bS(0,2)}(|,|),\quad
\beta \in \Hom_{\bS(0,3)}\bigl(\triv,(12)^*\triv\bigr),\quad
\alpha \in \Hom_{\bS(0,4)}\bigl(\forkone,\forktwo\bigr),\quad
\sigma \in \Hom_{\bS(1,1)}\bigl(\contraction,\contraction\bigr)
\]
for the corresponding morphisms in $\bS$.

\begin{prop}\label{prop:S-generated-by-standard-moves}
Every morphism in $\bS$ can be obtained from $\tau$, $\beta$, $\alpha$, $\sigma$, and their inverses using categorical composition, operadic composition, symmetric group actions, and contractions.
\end{prop}

\begin{proof} 
Let $[\phi]:G\to H$ be a morphism in $\bS(g,n+1)$, represented by a boundary-label-preserving diffeomorphism
\[
\phi:S_G\longrightarrow S_H.
\]
By Lemma~\ref{lemma: the complexes of morphisms and markings}, this morphism corresponds to the marked pants decomposition of $S_G$ obtained by pulling back the standard marked pants decomposition of $S_H$ along $\phi$.

We first translate a single edge of the complex of markings into a morphism of $\bS$. Suppose that
\[
(C,[m])\rightsquigarrow (C',[m'])
\]
is an edge of $\calM(S_G)$ of type $T$, $B$, $A$, or $S$. Let
\[
[\phi]:G\to H
\qquad\text{and}\qquad
[\psi]:G\to H'
\]
be the objects of the coslice category $G\downarrow\bS(g,n+1)$ corresponding to $(C,[m])$ and $(C',[m'])$, respectively, under Lemma~\ref{lemma: the complexes of morphisms and markings}. The edge determines the unique morphism in the coslice category from $[\phi]$ to $[\psi]$; equivalently, there is a unique morphism $[f]:H\to H'$ such that
\[
[\psi]=[f]\circ[\phi].
\]
We claim that $[f]$ is obtained from identities and one of the standard generators $\tau,\beta,\alpha,\sigma$ by operadic composition, symmetric group actions, and contractions.

We verify the claim for a $B$-move; the other cases are analogous. Suppose that
\[
(C,[m]:H\hookrightarrow S_G)
\overset{B_{a,b,c}}{\rightsquigarrow}
(C',[m']:H'\hookrightarrow S_G)
\]
is a $B$-move. Then $C=C'$, so the two marked pants decompositions have the same underlying cut system. The move is supported on a single cut-subsurface
\[
S'\subset \overline{S_G\setminus C}
\]
with boundary components $a,b,c$, and the markings $m$ and $m'$ agree on the complement of $S'$.

Choose an identification $\xi:S'\xrightarrow{\cong}P$ with the standard pair of pants such that
\[
\xi\circ m|_{S'}=m_P,\qquad
\xi(a)=\partial_1,\qquad
\xi(b)=\partial_2,\qquad
\xi(c)=\partial_0.
\]
By the definition of the $B$-move, under this identification the new marking is given by
\[
\xi\circ m'|_{S'}=\beta\circ \xi\circ m|_{S'}.
\]

Let $v\in V(H)$ and $v'\in V(H')$ be the vertices corresponding to the component $S'$. Since the move is supported on $S'$, the graphs $H$ and $H'$ agree away from these vertices, with all boundary labels preserved. We therefore define a mapping class
\[
f:S_H\longrightarrow S_{H'}
\]
by taking it to be the identity on every pair-of-pants component except the one corresponding to $v$, and by taking it on that component to be the standard $B$-move $\beta$ under the above identifications. If the local identification with the standard pair of pants differs by a cyclic relabelling, then the local map is the corresponding cyclic translate of $\beta$.

Thus $[f]:H\to H'$ is obtained from $\beta$ by operadic composition with identities, together with the relevant symmetric group action. Moreover, the marking corresponding to $[\psi]$ is exactly the marking obtained from the one corresponding to $[\phi]$ by applying $[f]$ after $[\phi]$. Hence
\[
[\psi]=[f]\circ[\phi].
\]

We now complete the proof. Let $[\phi]\in \bS(g,n+1)(G,H)$, and let $(C_\phi,[m_\phi])$ be the corresponding marked pants decomposition of $S_G$. Since $\calM(S_G)$ is connected, there is an edge path
\[
(C_0,[m_0])\rightsquigarrow (C_1,[m_1])
\rightsquigarrow \cdots \rightsquigarrow (C_k,[m_k])
\]
from the standard marking $(C_0,[m_0])=(iE_G,[m_G])$ to $(C_k,[m_k])=(C_\phi,[m_\phi])$.

Under Lemma~\ref{lemma: the complexes of morphisms and markings}, the standard marking corresponds to the identity morphism $\id_G:G\to G$. By the single-edge argument above, each edge in the path corresponds to postcomposition with a morphism
\[
[f_i]:H_{i-1}\to H_i
\]
obtained from identities and one of $\tau,\beta,\alpha,\sigma$ by operadic composition, symmetric group actions, and contractions. Hence
\[
[\phi]=[f_k]\circ\cdots\circ [f_1]\circ \id_G.
\]
Thus $[\phi]$ lies in the subcategory generated by $\tau,\beta,\alpha,\sigma$ under categorical composition, operadic composition, symmetric group actions, and contractions. Since $[\phi]$ was arbitrary, the claim follows.
\end{proof}

%By definition of the complex $\calM(S_{G})$, a loop based at the standard cut system $(iE_G, m_G: G\hookrightarrow S_G)$ in $\calM_{S_G}$ is a sequence of edges \[(iE_G,m_G)\overset{e_1}{\rightarrow} (C_1,m_1) \overset{e_2}{\rightarrow}\dots \overset{e_{k-1}}{\rightarrow}(C_{k-1},m_{k-1})\overset{e_k}{\rightarrow} (iE_G,m_G)\] and, by Lemma~\ref{lemma: the complexes of morphisms and markings}, it corresponds to a sequence of diffeomorphisms \[\phi_{e_{k-1}}\phi_{e_{k-2}}\dots\phi_{e_1}=\id_{G} \in \Hom_{\bS(g,n+1)}(G,G).\]  Theorem 5.1 of \cite{bk_marked_surfaces} says that the complex $\calM(S_G)$ is connected and simply connected. In particular the fundamental group $\pi_1(\calM(S_G); (iE_G, m_G))$ is trivial and the boundaries of the $2$-cells based at the point $(iE_G, m_G)$ in $\calM(S_{G})$ translate to sequences of diffeomorphisms $\phi_{e_{k-1}}\phi_{e_{k-2}}\dots\phi_{e_1}=\id_{G} \in \Hom_{\bS(g,n+1)}(G,G)$ in which each of the $\phi_{e_i}$ are one of the elementary diffeomorphisms $\tau, \beta, \alpha$ and $\sigma$.  Translating the diffeomorphisms $\phi_{e_{k-1}}\phi_{e_{k-2}}\dots\phi_{e_1}=\id_{G} \in \Hom_{\bS(g,n+1)}(G,G)$ into a series of modular operadic and categorical compositions, provides a series of equations which are satisfied by the mapping classes $\tau, \beta, \alpha$ and $\sigma$ in $\bS(g,n+1)$. The following theorem states that the equations imposed by the $2$-cells of $\calM(S_{G})$ are sufficient to determine a map of modular operads out of $\bS$. 

The preceding proposition identifies the generators of the morphisms in $\bS$. The relations among these generators are encoded by the $2$-cells of the complex of markings.  Fix $G\in \ob(\bS(g,n+1))$. A loop in $\calM(S_G)$ based at the standard marking $(iE_G,[m_G])$ is represented by an edge path
\[
(iE_G,[m_G])
\overset{e_1}{\rightsquigarrow}
(C_1,[m_1])
\overset{e_2}{\rightsquigarrow}
\cdots
\overset{e_k}{\rightsquigarrow}
(iE_G,[m_G]).
\]
Under Lemma~\ref{lemma: the complexes of morphisms and markings}, this loop corresponds to a sequence of morphisms in $\bS(g,n+1)$ whose composite is the identity morphism of $G$:
\[
[f_k]\circ \cdots \circ [f_1]=\id_G
\qquad\text{in }\Aut_{\bS(g,n+1)}(G).
\]
Each edge $e_i$ is one of the elementary moves $T$, $B$, $A$, or $S$, so each $[f_i]$ is obtained from $\tau,\beta,\alpha,\sigma$ by operadic composition, symmetric group actions, and contractions.

The boundaries of the $2$-cells in $\calM(S_G)$ therefore give equations among
the generators $\tau,\beta,\alpha,\sigma$. The goal of the next few sections is
to show that these are the only equations satisfied by these generators. This
will give a presentation of $\bS$, and hence a classification of maps out of the
modular operad $\bS$.

Before proving this in full generality, we first study the genus-zero part of $\bS$ and relate it to operads whose presentations are already known.

\subsection{Genus-zero surfaces and cyclic ribbon braids} \label{sec: operad of ribbon braids}
As a first step toward a presentation of the modular operad of surfaces, we study its genus-zero truncation. We use the known presentation of the operad of parenthesized ribbon braids to obtain a presentation of $\trun{0}\bS$, and then show that this identification is compatible with the cyclic structures. %The key input is the classical identification of pure ribbon braid groups with mapping class groups of genus-zero surfaces with boundary.

\subsubsection{Braids and ribbon braids}

For $n\geq 1$, let $\Br_n$ denote the braid group on $n$ strands. It may be defined as the fundamental group of the configuration space of $n$ unordered distinct points in the complex plane. It admits the standard presentation
\[
\Br_n=
\left\langle
\beta_1,\ldots,\beta_{n-1}
\middle|
\beta_i\beta_j=\beta_j\beta_i \ \text{for } |i-j|\ge 2,
\beta_i\beta_{i+1}\beta_i=\beta_{i+1}\beta_i\beta_{i+1}
\right\rangle .
\]
For each $n\geq 1$, there is a short exact sequence
\[
\begin{tikzcd}
1 \arrow[r] & \PB_n \arrow[r] & \Br_n \arrow[r, "\pi"] & \Sigma_n \arrow[r] & 1,
\end{tikzcd}
\]
where $\pi$ sends the braid generator $\beta_i$ to the transposition $(i\ i+1)$. The kernel of $\pi$ is the \emph{pure braid group}, denoted $\PB_n$. It is generated by the elements
\[
x_{ij}=\beta_{j-1}\cdots \beta_{i+1}\beta_i^2\beta_{i+1}^{-1}\cdots \beta_{j-1}^{-1},
\qquad 1\le i<j\le n.
\]

The \emph{ribbon braid group} on $n$ strands, denoted $\RB_n$,
may be defined as the fundamental group of the configuration space
of $n$ unordered distinct points in the plane, each equipped with
a framing (equivalently, each equipped with a point of $S^1$).
It is generated by the braid generators $\beta_1,\ldots,\beta_{n-1}$
together with twists $\twist_1,\ldots,\twist_n$, subject to the braid
relations above and the additional relations
\[
\beta_i\twist_j=\twist_j\beta_i\quad(j\notin\{i,i+1\}),\qquad
\beta_i\twist_i=\twist_{i+1}\beta_i,\qquad
\beta_i\twist_{i+1}=\twist_i\beta_i,\qquad
\twist_i\twist_j=\twist_j\twist_i\quad(i\neq j).
\]
The \emph{pure ribbon braid group}, denoted $\PRB_n$, is the kernel
of the natural projection
\[
\RB_n\longrightarrow \Sigma_n.
\]

The ribbon braid groups are closely related to mapping class groups of genus-zero surfaces. Indeed, there is a natural homomorphism
\[
\PRB_n\longrightarrow \Gamma_{0,n+1}
\]
sending the pure braid generator $x_{ij}$ to the Dehn twist about the simple closed curve separating the boundary components labelled $i$ and $j$ from the remaining boundary components, and sending the twist $\twist_k$ to the full Dehn twist about the $k$th boundary component. The following proposition is standard, cf., for example, \cite[Section 1.5]{Wahl_Thesis}. 

\begin{prop}\label{prop: RBn is Gamma0n+1}
There is an isomorphism of groups
\[
\PRB_n\cong \Gamma_{0,n+1}.
\]
\end{prop}

\subsubsection{The operad of parenthesized ribbon braids} 
The objects of the operad of parenthesized ribbon braids are parenthesized binary products. We encode these products by planar rooted binary trees.% whose input legs are labelled.
\begin{definition}\label{def: magma operad}
The \emph{magma operad} $\Omega=\{\Omega(n)\}_{n\geq 1}$ is the operad in $\Set$ whose elements are isomorphism classes of planar rooted binary trees with $n$ input legs, equipped with a labelling
\[
\ell_T:\{0,1,\dots,n\}\xrightarrow{\cong}\partial(T),
\]
where $\ell_T(0)$ is the root.

Equivalently, $\Omega$ is the free operad generated by the two rooted binary corollas
\[
\mu=\mu(0;1,2)
\qquad\text{and}\qquad
(12)^*\mu=\mu(0;2,1).
\]
These are the two corollas obtained by decorating the rooted binary corolla by the two elements of $\Sigma_2$, or equivalently by choosing the order of the two input legs relative to the fixed root. Operadic composition is given by grafting the root of one tree onto an input leg of another. For further details, see \cite[Section~6.1]{FresseBook1} and \cite[Definition~6.10]{Boavida-Horel-Robertson}.
\end{definition}

\begin{remark}
Elements of $\Omega(n)$ may be identified with fully parenthesized words in the symbols $\{1,\dots,n\}$. Under this identification, the two generators $\mu$ and $(12)^*\mu$ correspond to the binary words $(12)$ and $(21)$, respectively, and every parenthesized word is obtained from these generators by operadic composition. For example,
\[
(((12)3)4), \qquad ((32)(14)) \qquad \text{and} \qquad (1((23)4))
\]
are elements of $\Omega(4)$.
\end{remark}

\begin{definition}
For each $n\geq 1$, let $\PaRB(n)$ be the groupoid whose objects are the elements of $\Omega(n)$. %Thus an object is a parenthesized word, or equivalently a labelled planar rooted binary tree with $n$ input legs.

Given two objects $T_1,T_2\in\Omega(n)$, their boundary labellings determine a unique bijection
\[
\rho_{T_1,T_2}:\partial(T_1)\xrightarrow{\cong}\partial(T_2)
\]
such that $\rho_{T_1,T_2}\circ \ell_{T_1}=\ell_{T_2}$. A morphism $r\in \Hom_{\PaRB(n)}(T_1,T_2)$ is a ribbon braid $r\in \RB_n$ whose underlying permutation is the permutation of the input legs induced by $\rho_{T_1,T_2}$.
\end{definition}

Composition in $\PaRB(n)$ is given by multiplication in the ribbon braid group $\RB_n$. In particular, for any object $T\in\Omega(n)$, the endomorphism group is naturally identified with the pure ribbon braid group:
\[
\Aut_{\PaRB(n)}(T)\cong \PRB_n.
\] The symmetric group $\Sigma_n$ acts on $\PaRB(n)$ by permuting the labels of the input legs. The resulting symmetric sequence
\[
\PaRB=\{\PaRB(n)\}_{n\geq 1}
\] forms an operad in groupoids. On objects, operadic composition is the grafting operation in $\Omega$. On morphisms, it is given by inserting a ribbon braid in $\RB_m$ into the $i$th strand of a ribbon braid in $\RB_n$. See \cite[Definition~6.11]{Boavida-Horel-Robertson} for details.

\medskip
An important feature of $\PaRB$ is that its morphisms are generated by three elementary morphisms:
\[
\tau\in\Hom_{\PaRB(1)}(\mid,\mid),\qquad
\beta\in\Hom_{\PaRB(2)}(\mu,(12)^*\mu),
\qquad
\alpha\in\Hom_{\PaRB(3)}(\mu\circ_1\mu,\mu\circ_2\mu).
\]
Here $\tau$ is the strand twist, $\beta$ is the braiding, and $\alpha$ is the re-parenthesization isomorphism from $\mu\circ_1\mu$ to $\mu\circ_2\mu$. More precisely, every morphism in $\PaRB(n)$ can be written as an iterated categorical and operadic composite of these morphisms and their inverses; see \cite[Section~7]{Boavida-Horel-Robertson}, following the general framework of \cite[Chapter~6]{FresseBook1}.  This presentation allows one to classify maps out of $\PaRB$. 

The following is a restatement of Lemmas~7.1 and~7.4 of \cite{Boavida-Horel-Robertson}, based on \cite[Theorem~6.2.4]{FresseBook1}. We write $\cdot$ for categorical composition, read from right to left, and $\circ_i$ for operadic composition. Permutations are written in cycle notation and act on
the right.

\begin{thm}\label{thm: maps_out_of_P_0}
Let $\calQ$ be an operad in groupoids. To give an operad map $f:\PaRB\longrightarrow\calQ$  is equivalent to giving a quadruple
$(\mathbf m,\mathbf t,\mathbf b,\mathbf a)$ consisting of an object $f(\mu)=\mathbf m\in\ob(\calQ(2))$ and morphisms
\[
f(\tau)=\mathbf t\in \Hom_{\calQ(1)}(\mathbf1,\mathbf1),\qquad f(\beta)=\mathbf b\in
\Hom_{\calQ(2)}(\mathbf m,(12)^*\mathbf m),
\]
and
\[
f(\alpha)=\mathbf a\in \Hom_{\calQ(3)} (\mathbf m\circ_1\mathbf m,\mathbf m\circ_2\mathbf m),
\]
satisfying the following relations.

\begin{equation}\label{twist_rel_morph}\tag{T}
\mathbf t\circ_1 1_{\mathbf m} = \bigl(1_{\mathbf m}\circ(\mathbf t,\mathbf t)\bigr) \cdot (12)^*\mathbf b\cdot \mathbf b \qquad \text{in } \calQ(2).
\end{equation}

\begin{equation}\label{hex_rel_1_morph}\tag{H1}
\bigl((12)^*(1_{\mathbf m}\circ_2\mathbf b)\bigr) \cdot \bigl((12)^*\mathbf a\bigr) \cdot (1_{\mathbf m}\circ_1\mathbf b) = \bigl((132)^*\mathbf a\bigr) \cdot (\mathbf b\circ_2 1_{\mathbf m}) \cdot \mathbf a \qquad \text{in } \calQ(3).
\end{equation}

\begin{equation}\label{hex_rel_2_morph}\tag{H2}
\bigl((23)^*(1_{\mathbf m}\circ_1\mathbf b)\bigr) \cdot \bigl((23)^*\mathbf a^{-1}\bigr) \cdot (1_{\mathbf m}\circ_2\mathbf b) = \bigl((123)^*\mathbf a^{-1}\bigr) \cdot (\mathbf b\circ_1 1_{\mathbf m}) \cdot \mathbf a^{-1} \qquad \text{in } \calQ(3).
\end{equation}

\begin{equation}\label{pent_rel_1_morph}\tag{P}
(\mathbf a\circ_3 1_{\mathbf m}) \cdot (\mathbf a\circ_1 1_{\mathbf m}) = (1_{\mathbf m}\circ_2\mathbf a) \cdot (\mathbf a\circ_2 1_{\mathbf m})
\cdot (1_{\mathbf m}\circ_1\mathbf a)
\qquad \text{in } \calQ(4).
\end{equation}
\end{thm}

\medskip 

Recall from Proposition~\ref{prop:S-generated-by-standard-moves}
that the genus-zero generating morphisms in $\bS$ are
\[
\tau\in\Hom_{\bS(0,2)}(|,|),\quad
\beta\in\Hom_{\bS(0,3)}
\bigl(\triv,(12)^*\triv\bigr)
\quad\text{and}\quad
\alpha\in\Hom_{\bS(0,4)}(\forkone,\forktwo),
\]
where these are the mapping classes associated to the standard
$T$-, $B$-, and $A$-moves in the complex of markings, respectively.
To distinguish these mapping classes from the generating morphisms
of $\PaRB$, throughout this subsection we denote their images in
$u^*\trun{0}\bS$ by $\mathbf t$, $\mathbf b$, and $\mathbf a$,
respectively.

The genus-zero truncation is naturally a cyclic operad, whose
arity-$(n+1)$ component is
$\trun{0}\bS(n+1)=\bS(0,n+1)$.
Its underlying ordinary operad $u^*\trun{0}\bS$ is obtained
by choosing the boundary component labelled $0$ as the root,
so that
$u^*\trun{0}\bS(n)=\trun{0}\bS(n+1)=\bS(0,n+1)$.
These are the same groupoids, with the symmetric-group actions
restricted to permutations fixing the root. In particular,
the two generating objects in arity $2$ are $P$ and $(12)^*P$,
corresponding to the two rooted planar corollas obtained from
the two classes in $\Sigma_2^+/A_2^+$ after choosing the
boundary labelled $0$ as the root.

\begin{lemma}\label{lemma:operad-map-PaRB-to-S0}
There is an isomorphism of operads
\[
I:\PaRB\longrightarrow u^*\trun{0}\bS
\]
defined on generators by $I(\mu)=P$, $I(\tau)=\mathbf t$,
$I(\beta)=\mathbf b$, and $I(\alpha)=\mathbf a$.
\end{lemma}

\begin{proof}
By Theorem~\ref{thm: maps_out_of_P_0}, the assignment extends
to a map of operads if it satisfies
\eqref{twist_rel_morph},
\eqref{hex_rel_1_morph},
\eqref{hex_rel_2_morph}, and
\eqref{pent_rel_1_morph}.

We first check the twist--braiding relation
\eqref{twist_rel_morph}. Under the assignment in the
statement, this becomes
\begin{equation}\label{eq:I-applied-to-twist}
\mathbf t\circ_1\id_P
=
\bigl(\id_P\circ(\mathbf t,\mathbf t)\bigr)
\cdot(12)^*\mathbf b\cdot\mathbf b
\end{equation}
in $\Aut_{u^*\trun{0}\bS(2)}(P)$.
We interpret this equation using the equivalence
$\mathbb M_P\cong P\downarrow\bS(0,3)$ from
Lemma~\ref{lemma: the complexes of morphisms and markings}.
The left-hand side is the full Dehn twist $T_0$ about
the boundary component labelled $0$.
The morphisms $\mathbf b$ and $(12)^*\mathbf b$
correspond to $B_{1,2}$ and $B_{2,1}$, while
$\id_P\circ(\mathbf t,\mathbf t)$ gives the boundary
twists $T_1T_2$. Thus
\eqref{eq:I-applied-to-twist} becomes
$T_0=T_1T_2B_{2,1}B_{1,2}$.
Since the boundary twists commute, this is equivalent to
\[
B_{2,1}B_{1,2}=T_0T_1^{-1}T_2^{-1},
\]
which is the twist--braiding $2$-cell in
$\mathbb M_P$. Hence the assigned morphisms satisfy
\eqref{twist_rel_morph}.

We next check the pentagon relation
\eqref{pent_rel_1_morph}. We rewrite it as
\[
(\alpha\circ_1\id_\mu)^{-1}
\cdot(\alpha\circ_3\id_\mu)^{-1}
\cdot(\id_\mu\circ_2\alpha)
\cdot(\alpha\circ_2\id_\mu)
\cdot(\id_\mu\circ_1\alpha)
=1
\]
in
$\Aut_{\PaRB(4)}((\mu\circ_1\mu)\circ_1\mu)$.

Let $\forkone=(P\circ_1P)\circ_1P$.
After replacing the generators by their assigned images,
the equivalence
$\mathbb M_{S_{\forkone}}
\cong\forkone\downarrow\bS(0,5)$ of
Lemma~\ref{lemma: the complexes of morphisms and markings}
identifies the resulting composite with the loop
\[
(A_{b,e})^{-1}(A_{a,d})^{-1}
A_{c,e}A_{b,d}A_{a,c}.
\]
The curves defining these $A$-moves are shown in
Figure~\ref{fig:pentagon}. This loop is the boundary
of the pentagon $2$-cell, and hence represents the
identity morphism in $\mathbb M_{S_{\forkone}}$.
Its image is therefore the identity in
$\Aut_{u^*\trun{0}\bS(4)}(\forkone)$, so the
assigned morphisms satisfy \eqref{pent_rel_1_morph}.

The hexagon relations \eqref{hex_rel_1_morph} and
\eqref{hex_rel_2_morph} are verified in the same way.
They may be rewritten as loops based at
$\mu\circ_1\mu$ and $\mu\circ_2\mu$ in $\PaRB(3)$,
respectively. Replacing the generators by their
assigned images and using the corresponding
equivalences with the fundamental groupoids of the
marking complexes gives the two hexagon loops shown
in Figure~\ref{fig:hexagons}. Each bounds a hexagon
$2$-cell, so both composites are identities.
It follows from Theorem~\ref{thm: maps_out_of_P_0}
that the assignment extends to a map of operads
$I:\PaRB\to u^*\trun{0}\bS$.

It remains to show that $I$ is an isomorphism.
On objects, $I$ is the identity under the identification
of labelled rooted trees.
For every $T\in\PaRB(n)$, the induced map on
automorphism groups is the identification
$\PRB_n\cong\Gamma_{0,n+1}$ of
Proposition~\ref{prop: RBn is Gamma0n+1}.
Since these groupoids are connected, each $I(n)$
is fully faithful. It is also bijective on objects,
and hence is an isomorphism of groupoids.
Therefore $I$ is an isomorphism of operads.
\end{proof}

\medskip

The operad $\PaRB$ admits a cyclic structure in the
sense of Definition~\ref{def: cyclic structure}.
See \cite{campos2019configuration,MW,robertson2025grothendieck}.
We write $\PaRB^{\mathrm{cyc}}$ for the resulting
cyclic operad.

It is enough to describe the action of the transposition
$(01)$, since this action, together with the ordinary
symmetric-group actions, determines the full extended
symmetric-group action.
On objects, $(01)$ changes the choice of root, making
the leg labelled $1$ the new root. In particular,
$(01)^*\mu=(12)^*\mu$.
Together with the cyclic operad axioms, the action
on the generating morphisms determines the action
on all morphisms. With our conventions, it is given by
\begin{equation}\label{cyclic action on PaRB generators}\tag{$\star$}
(01)^*\tau=\tau,\qquad
(01)^*\beta
=
(\id_\mu\circ_2\tau^{-1})\cdot\beta^{-1}
\quad\text{and}\quad
(01)^*\alpha=(132)^*\alpha^{-1}.
\end{equation}

\begin{remark}
In \cite[Section~3.1]{campos2019configuration},
the subscripts in $A_{231}^{-1}$ record the
relabelling of the associator. We instead write
the permutation acting on it, using cycle notation
and right actions throughout the paper.
Thus their $A_{231}^{-1}$ is our
$(132)^*\alpha^{-1}$.
\end{remark}

We now show that the operad isomorphism of
Lemma~\ref{lemma:operad-map-PaRB-to-S0} respects
this root-changing action and therefore gives
an isomorphism of cyclic operads.

\begin{prop}\label{prop:PaRB-cyclic-is-S0}
The isomorphism of operads
$I:\PaRB\to u^*\trun{0}\bS$
is compatible with the cyclic structure on $\PaRB$.
Hence it induces an isomorphism of cyclic operads
\[
\PaRB^{\mathrm{cyc}}\cong\trun{0}\bS.
\]
\end{prop}

\begin{proof}
Since the $(01)$-action determines the cyclic structure
together with the ordinary symmetric-group actions,
it suffices to verify
$I((01)^*x)=(01)^*I(x)$
for $x=\mu,\tau,\beta,\alpha$.

For the generating object, recall that $I(\mu)=P$,
where $P$ is the rooted trivalent corolla with
boundary labels $\{0,1,2\}$.
The objects of $\trun{0}\bS$ are generated by the
two rooted planar corollas represented by
$P=\id\cdot A_2^+$ and
$(12)^*P=(12)\cdot A_2^+$, where
$A_2^+=\langle(012)\rangle$.
Since $(01)$ and $(12)$ represent the same
nontrivial coset modulo $A_2^+$, we have
$(01)^*P=(12)^*P$.
Thus $I((01)^*\mu)=(01)^*I(\mu)$.

For the twist generator, $I(\tau)=\mathbf t$.
The transposition $(01)$ acts trivially on
$\bS(0,2)$ by definition, so
$I((01)^*\tau)=I(\tau)=\mathbf t=(01)^*I(\tau)$.

For the braiding generator, applying $(01)^*$ to
$\mathbf b:P\to(12)^*P$ gives
$(01)^*\mathbf b:(12)^*P\to P$.
Through the equivalence
$\mathbb M_{(12)^*P}
\cong(12)^*P\downarrow\trun{0}\bS(3)$,
this morphism is identified with the opposite
$B$-move followed by the boundary-twist correction
that puts the relabelled marking in standard form.
As depicted in Figure~\ref{fig: cyclic}, this gives
$(01)^*B_{1,2}=T_0^{-1}T_1\cdot B_{2,1}$.
The twist--braiding relation gives
$B_{2,1}=T_0T_1^{-1}T_2^{-1}\cdot B_{1,2}^{-1}$.
Substituting this expression and cancelling the
commuting boundary twists gives
\[
(01)^*B_{1,2}
=
T_0^{-1}T_1T_0T_1^{-1}T_2^{-1}
\cdot B_{1,2}^{-1}
=
T_2^{-1}\cdot B_{1,2}^{-1}.
\]
In operadic notation, this is
$(01)^*\mathbf b
=(\id_P\circ_2\mathbf t^{-1})\cdot\mathbf b^{-1}$.
This agrees with the image under $I$ of
\eqref{cyclic action on PaRB generators}, so
$I((01)^*\beta)=(01)^*I(\beta)$.

It remains to check the associativity generator.
The labelled rooted graphs in Figure~\ref{fig: cyclic}
give
\[
(01)^*(P\circ_1P)=(132)^*(P\circ_2P),
\qquad
(01)^*(P\circ_2P)=(132)^*(P\circ_1P).
\]
Thus both $(01)^*\mathbf a$ and
$(132)^*\mathbf a^{-1}$ have source
$(132)^*(P\circ_2P)$ and target
$(132)^*(P\circ_1P)$.
Their comparison is the loop
\[
(132)^*\mathbf a\cdot(01)^*\mathbf a
\in
\Aut_{\trun{0}\bS(4)}
\bigl((132)^*(P\circ_2P)\bigr).
\]
Under the equivalence
$\mathbb M_{(132)^*(P\circ_2P)}
\cong(132)^*(P\circ_2P)\downarrow\trun{0}\bS(4)$,
this is the loop obtained by performing the
$(01)$-relabelled $A$-move followed by the
$(132)$-relabelled return move.
The self-duality relation
$A_{c',c}A_{c,c'}=\id$
makes this loop trivial. Consequently,
$(01)^*\mathbf a=(132)^*\mathbf a^{-1}$,
and $I((01)^*\alpha)=(01)^*I(\alpha)$.

Therefore $I$ is compatible with the cyclic
structures and induces an isomorphism
$\PaRB^{\mathrm{cyc}}\cong\trun{0}\bS$.
\end{proof}

\begin{remark}
In the proof of
Lemma~\ref{lemma:operad-map-PaRB-to-S0},
we used diagrammatic depictions of genus-zero
mapping classes and their corresponding paths
in the complexes of markings.
This is justified by
Lemma~\ref{lemma: the complexes of morphisms and markings}
and the discussion in
Section~\ref{subsubsec: diagrammatical representation of mapping classes},
which shows that these representatives are well defined.
Figure~\ref{fig: cyclic} depicts the action of $(01)$
on the $B$- and $A$-moves.

\begin{figure}[h!t]
    \centering
    \includegraphics[width=\textwidth]{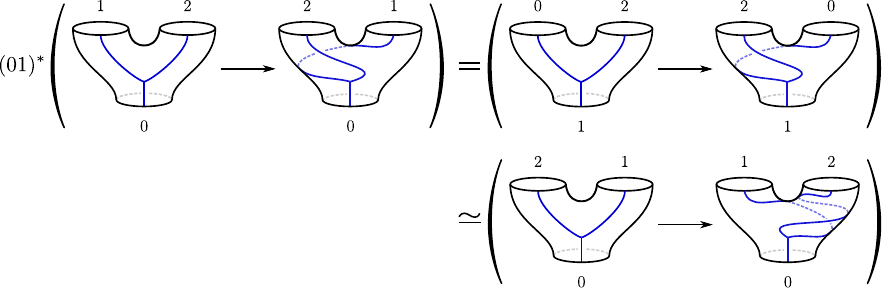}

    \vspace{.5cm}

    \includegraphics[width=\textwidth]{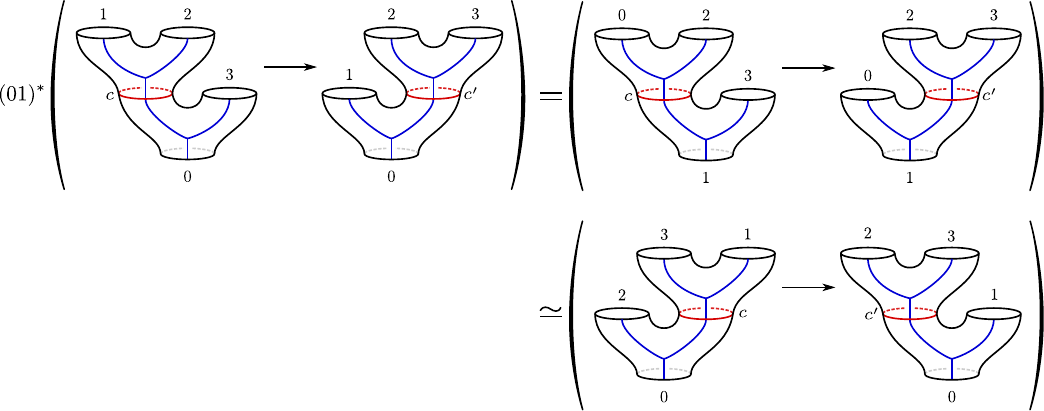}
    \caption{Action of the $(01)$-transposition on the
    $B$- and $A$-moves.}
    \label{fig: cyclic}
\end{figure}
\end{remark}

\subsection{A presentation for maps out of $\bS$}
\label{sec: maps out of bS}

We use $\cdot$ for categorical composition in the groupoids
$\bS(g,n+1)$, write $\circ_i^j$ for cyclic operadic composition,
and write $\xi_i^j$ for contractions. When only the underlying
operad structure is involved, we write $\circ_i^0=\circ_i$.

The goal of this subsection is to prove the main presentation
theorem for maps out of the modular operad of surfaces $\bS$.
As in the genus-zero case, the proof uses the comparison with
complexes of markings over pants decompositions established in
Lemma~\ref{lemma: the complexes of morphisms and markings}.
These complexes are connected and simply connected by
Proposition~\ref{prop:our-complex-of-markings-is-simply-connected},
so their edges describe generators for the morphisms of $\bS$,
while their $2$-cells supply the relations.

%\subsubsection{A presentation of the modular operad $\bS$} Proof of the main theorem. 

\begin{thm}\label{thm: presentation for maps out of bS}
Let $\calQ$ be a modular operad in groupoids. A map of modular
operads $f:\bS\longrightarrow\calQ$ is uniquely determined by a tuple
\[
(\mathbf m,\mathbf t,\mathbf b,\mathbf a,\mathbf s)
\]
consisting of an object
$f(P)=\mathbf m\in\ob(\calQ(0,3))$
and morphisms
\[
f(\tau)=\mathbf t\in\Hom_{\calQ(0,2)}(|,|),
\qquad
f(\beta)=\mathbf b\in
\Hom_{\calQ(0,3)}(\mathbf m,(12)^*\mathbf m),
\]
\[
f(\alpha)=\mathbf a\in
\Hom_{\calQ(0,4)}
(\mathbf m\circ_1\mathbf m,\mathbf m\circ_2\mathbf m),
\qquad
f(\sigma)=\mathbf s\in
\Hom_{\calQ(1,1)}
(\xi_{12}\mathbf m,\xi_{12}\mathbf m),
\]
satisfying the following relations.

\begin{equation}\label{sym_rel_morph}\tag{S}
(012)^*\mathbf m=\mathbf m
\qquad \text{in } \ob(\calQ(0,3)).
\end{equation}

\begin{equation}\label{twist_rel_morphS}\tag{T}
\mathbf t\circ_1\id_{\mathbf m}
=
\bigl(\id_{\mathbf m}\circ(\mathbf t,\mathbf t)\bigr)
\cdot(12)^*\mathbf b\cdot\mathbf b
\qquad \text{in } \calQ(0,3).
\end{equation}

\begin{equation}\label{hex_rel_1_morphS}\tag{H1}
\bigl((12)^*(\id_{\mathbf m}\circ_2^0\mathbf b)\bigr)
\cdot\bigl((12)^*\mathbf a\bigr)
\cdot(\id_{\mathbf m}\circ_1^0\mathbf b)
=
\bigl((132)^*\mathbf a\bigr)
\cdot(\mathbf b\circ_2^0\id_{\mathbf m})
\cdot\mathbf a
\qquad \text{in } \calQ(0,4).
\end{equation}

\begin{equation}\label{hex_rel_2_morphS}\tag{H2}
\bigl((23)^*(\id_{\mathbf m}\circ_1^0\mathbf b)\bigr)
\cdot\bigl((23)^*\mathbf a^{-1}\bigr)
\cdot(\id_{\mathbf m}\circ_2^0\mathbf b)
=
\bigl((123)^*\mathbf a^{-1}\bigr)
\cdot(\mathbf b\circ_1^0\id_{\mathbf m})
\cdot\mathbf a^{-1}
\qquad \text{in } \calQ(0,4).
\end{equation}

\begin{equation}\label{pent_rel_1_morphS}\tag{P}
(\mathbf a\circ_3\id_{\mathbf m})
\cdot(\mathbf a\circ_1\id_{\mathbf m})
=
(\id_{\mathbf m}\circ_2\mathbf a)
\cdot(\mathbf a\circ_2\id_{\mathbf m})
\cdot(\id_{\mathbf m}\circ_1\mathbf a)
\qquad \text{in } \calQ(0,5).
\end{equation}

\begin{equation}\label{g1_rel_morph_1}\tag{G1a}
\mathbf s^2
=
\xi_{0,2}\bigl(\mathbf b^{-1}\bigr)
\qquad \text{in } \calQ(1,1).
\end{equation}

\begin{equation}\label{g1_rel_morph_2}\tag{G1b}
\mathbf d^{-1}
\cdot\mathbf s
\cdot\mathbf d^{-1}
\cdot\mathbf s
\cdot\mathbf d^{-1}
=
\mathbf s
\qquad \text{in } \calQ(1,1),
\end{equation}
where $\mathbf d=\xi_{12}(\id_{\mathbf m}\circ_1\mathbf t)$.

\begin{multline}\label{g2_rel_morph_1}\tag{G2}
\xi_{12}\bigl((012)^*\mathbf b\circ_2
(\mathbf t^{-1},\mathbf t^{-1})\circ\mathbf b\bigr)
\\
=
\xi_{12}(\mathbf a)
\cdot(\id_{\mathbf m}\circ_1\mathbf s)
\cdot\xi_{12}(\mathbf a^{-1})
\cdot
\xi_{12}\bigl(\id_{\mathbf m}\circ_2
(\mathbf t^{-1},\mathbf t)\circ\id_{\mathbf m}\bigr)
\\
\cdot\xi_{12}(\mathbf a)
\cdot(\id_{\mathbf m}\circ_1\mathbf s)
\cdot\xi_{12}(\mathbf a^{-1})
\qquad \text{in } \calQ(1,2).
\end{multline}
\end{thm}

Relations \ref{g1_rel_morph_1}, \ref{g1_rel_morph_2} and \ref{g2_rel_morph_1} are added to include the higher genus structure. For instance, 
equation~\ref{g2_rel_morph_1} can be pictured diagrammatically as in Figure~\ref{fig:2-cell from NS}.

\begin{figure}
    \centering
    \includegraphics[width=\linewidth]{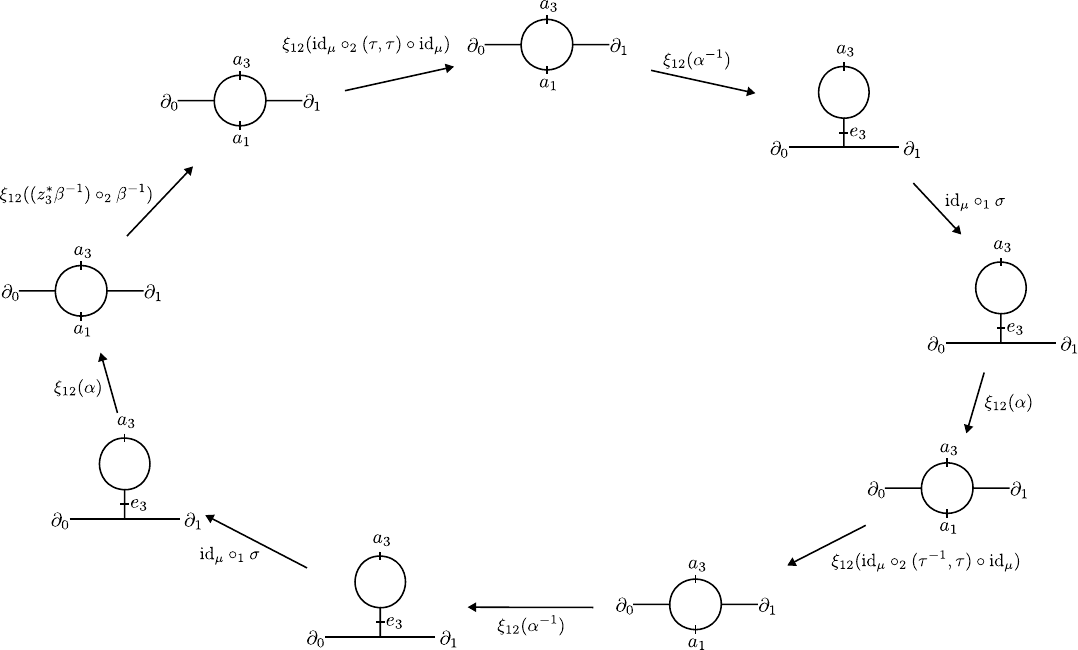}
    \caption{Diagram of relation (G2)}
    \label{fig:2-cell from NS}
\end{figure}

\begin{proof}
A map of modular operads $f:\bS\to \calQ$ necessarily determines the images of the generating morphisms:
\[ f(P)=\mathbf{m},\quad f(\tau)=\mathbf{t},\quad f(\beta)=\mathbf{b},\quad f(\alpha)=\mathbf{a},\quad f(\sigma)=\mathbf{s}.
\]
Moreover, such a map consists of functors $f(g,n+1):\bS(g,n+1)\longrightarrow \calQ(g,n+1)$ compatible with the symmetric group actions, operadic compositions, and contractions. Under the equivalence of Lemma~\ref{lemma: the complexes of morphisms and markings}, the boundaries of the relevant $2$-cells in the marking complexes $\calM(S_G)$ are identified with the relations \eqref{sym_rel_morph}, \eqref{twist_rel_morphS}, \eqref{pent_rel_1_morphS}, \eqref{hex_rel_1_morphS}, \eqref{hex_rel_2_morphS}, \eqref{g1_rel_morph_1}, \eqref{g1_rel_morph_2}, and \eqref{g2_rel_morph_1}. Therefore any map $f:\bS\to\calQ$ gives a tuple $(\mathbf{m},\mathbf{t},\mathbf{b},\mathbf{a},\mathbf{s})$ satisfying the stated relations.

\medskip 

We now prove the converse. Suppose we are given $\mathbf{m},\mathbf{t},\mathbf{b},\mathbf{a},\mathbf{s}$ in $\calQ$ satisfying the stated relations.  Since $\ob(\bS)$ is the free modular operad in sets generated by the genus-zero corolla $P$ and $(12)^*P$, the assignment $\ob(f)(P)=\mathbf{m}\in\ob(\calQ(0,3))$ extends uniquely to a map of modular operads in sets
\[
\ob(f):\ob(\bS)\longrightarrow\ob(\calQ).
\]
In particular, the object assignment is compatible with the extended symmetric group actions, operadic compositions, and contractions.  Note that the object $P\in \bS(0,3)$ is represented by the coset $\id\cdot A_2^+$, where $A_2^+=\langle(012)\rangle$. Thus $(012)^*P=P$, and the condition
\[
(012)^*\mathbf{m}=\mathbf{m}
\]
is precisely the condition needed for the assignment $P\mapsto \mathbf{m}$ to respect the stabiliser of the generating object.

\medskip

Using the cyclic operad isomorphism $\PaRB^{\mathrm{cyc}}\cong \trun{0}\bS$ from Proposition~\ref{prop:PaRB-cyclic-is-S0}, together with Theorem~\ref{thm: maps_out_of_P_0}, the assignments
\[
\mu\mapsto \mathbf{m},\qquad \tau\mapsto \mathbf{t},\qquad \beta\mapsto \mathbf{b},\qquad \text{and} \qquad \alpha\mapsto \mathbf{a}
\]
satisfying \eqref{twist_rel_morphS}, \eqref{pent_rel_1_morphS}, \eqref{hex_rel_1_morphS}, and \eqref{hex_rel_2_morphS} define the genus-zero
part of $f$, namely a map of cyclic operads $\trun{0}f:\trun{0}\bS\longrightarrow \trun{0}\calQ.$

\medskip 

It remains to extend the genus-zero map to all of $\bS$. By Proposition~\ref{prop:S-generated-by-standard-moves}, every morphism in $\bS$ is obtained from the generating morphisms $\alpha,\beta,\tau,\sigma$ by modular operad operations and categorical composition. Thus, for a morphism
\[
\phi\in\Hom_{\bS(g,n+1)}(G,G'),
\]
we choose an expression for $\phi$ in terms of these generators and define its image in $\calQ(g,n+1)$ by replacing $\alpha,\beta,\tau,$ and $\sigma$ by $\mathbf{a},\mathbf{b},\mathbf{t}$ and $\mathbf{s}$. 

We must check that this description of $f(\phi)$ is independent of the chosen expression. So suppose that $\phi$ has two expressions
\[
\phi=\phi^1_k\cdots \phi^1_1=\phi^2_m\cdots \phi^2_1,
\]
in $\bS(g,n+1)$, where each $\phi^r_i$ is one of the generating morphisms of $\bS$. Then
\[
(\phi^2_m\cdots \phi^2_1)^{-1}(\phi^1_k\cdots \phi^1_1)=\id_G \ \in\Aut_{\bS(g,n+1)}(G).
\]

Under the equivalence of Lemma~\ref{lemma: the complexes of morphisms and markings}, this relation is represented by a loop in the marking complex $\calM(S_G)$. Since $\calM(S_G)$ is connected and simply connected \cite[Theorem~5.1]{bk_marked_surfaces}, this loop is generated by the boundaries of the $2$-cells of $\calM(S_G)$.  The boundary relations of these $2$-cells are exactly the relations imposed above, together with the formal relations already forced by the modular operad axioms. 

More precisely, the twist--braiding, pentagon, and two hexagon cells are encoded by \eqref{twist_rel_morphS}, \eqref{pent_rel_1_morphS}, \eqref{hex_rel_1_morphS}, and \eqref{hex_rel_2_morphS}. The self-duality of the associativity move follows from Mac Lane coherence, and therefore from the pentagon and hexagon relations already imposed.  The genus-one cells in types $(1,1)$ and $(1,2)$ are encoded by \eqref{g1_rel_morph_1}, \eqref{g1_rel_morph_2}, and \eqref{g2_rel_morph_1}. The remaining cells express formal compatibility relations for disjointly supported moves and for the modular operad operations. These hold in $\calQ$ because $\calQ$ is a modular operad. 

It follows that
\[
(f(\phi^2_m)\cdots f(\phi^2_1))^{-1} (f(\phi^1_k)\cdots f(\phi^1_1)) = \id_{f(G)}.
\]
Thus the two expressions for $\phi$ have the same image in $\calQ$, so the assignment on morphisms is well defined. By construction it is compatible with categorical composition, symmetric group actions, operadic compositions, and contractions. Hence it defines a map of modular operads $f:\bS\longrightarrow \calQ.$ The theorem follows. 

\end{proof}

\section{The operadic two-level principle}
\label{sec: two level principle}

Grothendieck's \emph{two-level principle} predicts that Galois actions on the Teichmüller tower should be determined by their effect in the first two levels,
namely on
\[
\widehat{\Gamma}_{0,4},\quad
\widehat{\Gamma}_{0,5},\quad
\widehat{\Gamma}_{1,1},\quad
\text{and}\quad
\widehat{\Gamma}_{1,2}.
\]
In this section we formulate and prove an operadic version of this principle. Our main goal is to show that there exists an action of $\galQ$ on the profinite completion of the surface modular operad $\widehat{\bS}$ (Theorem~\ref{main theorem NS action}). Moreover, this action is completely determined by the genus one truncation of $\widehat{\bS}$ (Corollary~\ref{cor: truncation and endomorphisms of the surface mod operad}). 

We begin with a brief discussion of profinite completion for modular operads in groupoids. 

\subsection{Profinite completion of modular operads in groupoids}
\label{sec: profinite completion of modular operads in groupoids}

We aim to study automorphisms of the profinite completion $\widehat{\bS}$. We begin by explaining how profinite completion extends from groupoids to modular operads in groupoids. Proposition~\ref{prop: profinite completion preserves products of groupoids} shows that, under the finite-object hypothesis satisfied by $\bS$, the modular-operad structure can be completed entrywise.

\begin{definition}\label{def: profinite completion of modular operads in groupoids}
Let $\calP=\{\calP(g,n+1)\}$ be a modular operad in groupoids such that each $\calP(g,n+1)$ has finitely many objects. Its \emph{profinite completion} is the modular operad $\widehat{\calP}$ in profinite groupoids defined by
\[
\widehat{\calP}(g,n+1):=\reallywidehat{\calP(g,n+1)}.
\]
Its structure maps are obtained by applying profinite completion to those of $\calP$, using the product comparison of Proposition~\ref{prop: profinite completion preserves products of groupoids}.
\end{definition}

More precisely, whenever $n+m\geq1$, the partial composition
\[
\widehat{\calP}(g,n+1)\times\widehat{\calP}(h,m+1) \longrightarrow \widehat{\calP}(g+h,n+m)
\]
is the composite
\[
\begin{tikzcd} \widehat{\calP}(g,n+1)\times\widehat{\calP}(h,m+1) \arrow[r, "\cong"] & \reallywidehat{\calP(g,n+1)\times\calP(h,m+1)} \arrow[r, "\widehat{(\circ_i^j)}"] & \widehat{\calP}(g+h,n+m).
\end{tikzcd}
\]
The first arrow is the inverse of the product isomorphism from Proposition~\ref{prop: profinite completion preserves products of groupoids}. For $n\geq2$, the contraction maps
\[
\widehat{(\xi_i^j)}: \widehat{\calP}(g,n+1)\longrightarrow \widehat{\calP}(g+1,n-1)
\] are obtained directly by applying profinite completion.

The right adjoint in the profinite-completion adjunction, $|-|:\widehat{\Grpd}\to\Grpd$, preserves finite products and hence induces a functor $|-|:\Mod(\widehat{\Grpd})\to\Mod(\Grpd)$. For any $\calP$ satisfying the finiteness hypothesis above, the componentwise units assemble to a morphism of modular operads
\[
\eta_{\calP}:\calP\longrightarrow|\widehat{\calP}|.
\]

\begin{lemma}\label{lemma: profinite completion adjunction for modular operads}
Let $\calP$ be as above. For every $\calQ\in\Mod(\widehat{\Grpd})$, composition with $\eta_{\calP}$ induces a natural bijection 
\[ \Hom_{\Mod(\widehat{\Grpd})}(\widehat{\calP},\calQ) \cong \Hom_{\Mod(\Grpd)}(\calP,|\calQ|). \]
\end{lemma}

\begin{proof}
For groupoids $\mathcal A$ and $\mathcal B$ with finitely many objects, the product comparison and preservation of products by $|-|$ give natural isomorphisms
\[
|\widehat{\mathcal A\times\mathcal B}| \cong|\widehat{\mathcal A}\times\widehat{\mathcal B}| \cong |\widehat{\mathcal A}|\times|\widehat{\mathcal B}|.
\]
Under these identifications, $\eta_{\mathcal A\times\mathcal B}$ agrees with $\eta_{\mathcal A}\times\eta_{\mathcal B}$. The entrywise profinite-completion adjunction is therefore compatible with partial compositions. Compatibility with contractions, symmetric-group actions, and units follows entrywise. The entrywise adjunction therefore restricts to the stated bijection between morphisms of modular operads.
\end{proof}

\subsubsection{The genus-one truncation}
The presentation theorem (Theorem~\ref{thm: presentation for maps out of bS}) shows that maps out of $\mathbf S$ are determined by the images of the generators and relations appearing in genus at most one. We express this using the truncation formalism introduced in Section~\ref{subsec: intro to truncations} and developed further in Appendix~\ref{subsec: truncations}: $\trun{1}\calP$ denotes the genus-one truncation of a modular operad $\calP$, and the truncation functor admits a left adjoint $(\trun{1})_!$.

\begin{cor}\label{cor: maps out of bS}
The counit of the truncation adjunction gives an isomorphism of modular operads in groupoids
\[
(\trun{1})_!(\trun{1}\bS)\cong \bS.
\]
\end{cor}

\begin{proof}
By Theorem~\ref{thm: presentation for maps out of bS}, a map out of $\bS$ is uniquely determined by the images of the generators appearing in $\trun{1}\bS$, and every choice of such images satisfying the stated relations extends uniquely to a map out of $\bS$. Thus, for every modular operad $\calQ$, restriction induces a natural bijection
\[
\Hom_{\Mod(\Grpd)}(\bS,\calQ) \cong \Hom_{\Mod_{\leq 1}(\Grpd)}(\trun{1}\bS,\trun{1}\calQ).
\]
By the adjunction $(\trun{1})_!\dashv \trun{1}$, the right-hand side is naturally isomorphic to $\Hom_{\Mod(\Grpd)}((\trun{1})_!(\trun{1}\bS),\calQ).$ Under these identifications, the natural bijection is induced by precomposition with the counit
\[
(\trun{1})_!(\trun{1}\bS)\longrightarrow \bS.
\]
Therefore, by Yoneda, the counit is an isomorphism.
\end{proof}

We will use this after profinite completion.  Let $\eta_{\calP}:\calP\to |\widehat{\calP}|$ denote the unit of the profinite-completion adjunction, where $|-|:\widehat{\Grpd}\to\Grpd$ is the forgetful functor. Thus, for any modular operad $\mathcal R$ in profinite groupoids, composition with $\eta_{\calP}$ gives a natural bijection
\[
\Hom_{\Mod(\widehat{\Grpd})}(\widehat{\calP},\mathcal R) \cong \Hom_{\Mod(\Grpd)}(\calP,|\mathcal R|).
\]
Equivalently, every map $f:\calP\to|\mathcal R|$ extends uniquely to a continuous map $\widetilde f:\widehat{\calP}\to\mathcal R$ with $|\widetilde f|\circ\eta_{\calP}=f$.

In particular, taking $\calP=\bS$ and $\mathcal R=\widehat{\bS}$, the presentation theorem describes maps $\bS\to|\widehat{\bS}|$, and the adjunction upgrades each such map uniquely to a continuous endomorphism $\widehat{\bS}\to\widehat{\bS}$.

\begin{cor}\label{cor: truncation and endomorphisms of the surface mod operad}
The truncation functor $\trun{1}$ induces an isomorphism of monoids
\[
\End_0(\widehat{\bS})
\xrightarrow{\cong}
\End_0(\trun{1}\widehat{\bS}).
\]
\end{cor}

\begin{proof}
By the profinite-completion adjunction, endomorphisms of $\widehat{\bS}$ are equivalent to maps $\bS\to|\widehat{\bS}|$. Since the unit $\eta_{\bS}:\bS\to|\widehat{\bS}|$ does not change the object set, this bijection restricts to object-fixing maps:
\[
\End_0(\widehat{\bS})\cong \Hom_0(\bS,|\widehat{\bS}|).
\]
Corollary~\ref{cor: maps out of bS} identifies this set with $\Map_0(\trun{1}\bS,|\trun{1}\widehat{\bS}|)$. Since profinite completion is applied entrywise, it commutes with genus truncation, so $\trun{1}\widehat{\bS}$ is the profinite completion of $\trun{1}\bS$. A second application of the profinite-completion adjunction gives
\[
\Hom_0(\trun{1}\bS,|\trun{1}\widehat{\bS}|) \cong \End_0(\trun{1}\widehat{\bS}).
\]
Combining these bijections gives the desired isomorphism of monoids.
\end{proof}
\subsection{The action of $\galQ$ on $\widehat{\bS}$}\label{sec: NS}
To study Galois actions on higher-genus mapping class groups, we work with explicit profinite groups containing $\galQ$ and admitting well-understood presentations. The most prominent example is the Grothendieck--Teichmüller group $\GT$, introduced by Drinfeld \cite[Section~4]{Drin}. 
In their study of Galois actions on higher-genus mapping class groups, Nakamura and Schneps introduced in \cite{Nakamura-Schneps} a subgroup $\NS$ of the Grothendieck--Teichmüller group, satisfying
\[
\galQ\hookrightarrow \NS\hookrightarrow \GT.
\]
Rather than defining the Galois action directly on the full profinite modular operad $\widehat{\bS}$, we construct it via these explicit profinite groups.

Let $\widehat{\F}_2$ denote the profinite completion of the free
group on two generators $x$ and $y$.
If $f=f(x,y)\in\widehat{\F}_2$ and $a,b$ are elements of a profinite
group $\mathsf G$, we write $f(a,b)$ for the image of $f$ under the
unique continuous homomorphism
$\widehat{\F}_2\to\mathsf G$ sending $x\mapsto a$ and $y\mapsto b$.

For $(\lambda,f)\in\widehat{\mathbb Z}\times\widehat{\F}_2$,
write $\varphi_{(\lambda,f)}:\widehat{\F}_2\to\widehat{\F}_2$
for the continuous endomorphism determined by
\[
x\longmapsto f(x,y)^{-1}x^\lambda f(x,y),
\qquad
y\longmapsto y^\lambda.
\]

\begin{definition}\label{defn:GT}
The \emph{Grothendieck--Teichm\"uller monoid} $\underline{\GT}$
consists of pairs
$(\lambda,f)\in
(1+2\widehat{\mathbb Z})\times\widehat{\F}_2$
satisfying the following relations:
\begin{equation}\label{I}\tag{I}
f(x,y)f(y,x)=1,
\end{equation}
\begin{equation}\label{II}\tag{II}
f(z,x)z^\nu f(y,z)y^\nu f(x,y)x^\nu=1,
\qquad
z=(xy)^{-1},
\qquad
\nu=(\lambda-1)/2,
\end{equation}
and
\begin{equation}\label{III}\tag{III}
f(x_{12},x_{23}x_{24})f(x_{13}x_{23},x_{34})
=
f(x_{23},x_{34})f(x_{12}x_{13},x_{24}x_{34})f(x_{12},x_{23}).
\end{equation}
The first two equations hold in $\widehat{\F}_2$, and the third
holds in $\widehat{\PB}_4$, the profinite completion of the pure
braid group on four strands.

For $F=(\lambda,f)$ and $G=(\mu,g)$, we define the multiplication by
\[
F\cdot G :=\bigl(\lambda\mu,\, g(x,y)f(g(x,y)^{-1}x^\mu g(x,y),y^\mu)\bigr),
\]
so that $\varphi_{F\cdot G}=\varphi_G\circ\varphi_F$. The \emph{Grothendieck--Teichm\"uller group} $\GT$ is the group of invertible elements of the monoid $\underline{\GT}$.
\end{definition}

\begin{remark}
Our multiplication convention is
$\varphi_{F\cdot G}=\varphi_G\circ\varphi_F$:
the endomorphism associated to $F$ is applied first,
followed by the endomorphism associated to $G$.
\end{remark}

Drinfeld described the action of $\GT$ on profinite braid groups in his original definition of the Grothendieck--Teichmüller group~\cite{Drin}. An operadic interpretation of this action was developed in~\cite{Horel_profinite_groupoids}, building on work such as~\cite{BN,FresseBook1}. This was extended to an action on profinite ribbon braid operads by the second author and collaborators in~\cite{Boavida-Horel-Robertson}. The compatibility with the cyclic structure on $\PaRB$ was recently established by the second author and Singh~\cite{robertson2025grothendieck}.

By Proposition~\ref{prop:PaRB-cyclic-is-S0}, there is an isomorphism of cyclic operads
\[
\trun{0}\widehat{\bS}\cong \widehat{\PaRB}^{\mathrm{cyc}}.
\]
We use this identification to regard the genus-zero $\GT$-action on $\widehat{\PaRB}^{\mathrm{cyc}}$ as an action on $\trun{0}\widehat{\bS}$. Explicitly, an element $F=(\lambda,f)\in\underline{\GT}$ determines an object-fixing endomorphism $\iota_F\in\End_0(\trun{0}\widehat{\bS})$ by its action on the standard generators $\tau,\beta,\alpha$:
\[
\iota_F(\tau)=\tau^\lambda,\qquad
\iota_F(\beta)=\beta\cdot\bigl(\beta\cdot(12)^*\beta\bigr)^\nu,\qquad
\iota_F(\alpha)=\alpha\cdot f(D_{c'},D_c),
\]
where $\nu=(\lambda-1)/2$, and where $c$ and $c'$ are the curves shown in Figure~\ref{fig:A-move}. 

In the following proposition, we take our endomorphism monoids equipped with the multiplication $u\cdot v:=v\circ u$, so that $u$ is applied first and then $v$. This agrees with the multiplication convention for $\underline{\GT}$. 

\begin{prop}\label{prop:genus-zero-GT}
There is an isomorphism of monoids
\[
\End_0(\trun{0}\widehat{\bS})
\cong
\underline{\GT}.
\]
\end{prop}

\begin{proof}
This follows from the identification $\trun{0}\widehat{\bS}\cong \widehat{\PaRB}^{\mathrm{cyc}}$ above and the isomorphism $\underline{\GT}\cong \End_0(\widehat{\PaRB}^{\mathrm{cyc}})$ proved in~\cite[Theorem~5.5]{robertson2025grothendieck}.
\end{proof}

\begin{comment}
\begin{cor}
There exists an isomorphism of monoids \[\End_0(\PaRB^{cyc})\cong\End_0(\trun{0}\hat\bS)\cong \underline{\GT}.\]
\end{cor}

\begin{proof}
    This follows from Lemma~\ref{lemma: iso between S0 and cyclic PaRB}, Theorem~\ref{thm: maps_out_of_P_0}, Theorem~\ref{thm: PaRB and GT}, and Theorem~\ref{thm: presentation for maps out of bS} together with the fact that any endomorphism of $\trun{0}\bS$ fixing the objects will immediately satisfy relation \ref{sym_rel_morph}.
\end{proof}

\end{comment}

In \cite{Nakamura-Schneps}, Nakamura and Schneps introduced a subgroup $\NS\subseteq \GT$ by imposing two additional relations, discovered through comparisons of Galois representations in $\widehat{\Gamma}_{0,5}$ and $\widehat{\Gamma}_{1,2}$. These relations encode the genus-one compatibility needed to extend the genus-zero Grothendieck--Teichmüller action to the full tower of profinite mapping class groups.

\begin{definition}\label{def: NS}
The monoid $\underline{\NS}\subseteq\underline{\GT}$ consists of those elements $(\lambda,f)\in\underline{\GT}$ satisfying the following two additional relations:
\begin{equation}\label{eq: III'}\tag{NS I} f(\beta_1\beta_3,\beta_2^2) (\beta_1\beta_2\beta_3)^{4\rho_2} = h((\beta_1\beta_2)^3,(\beta_2\beta_3)^3) f(\beta_1^2,\beta_2^2) f(\beta_3^2,(\beta_1\beta_2)^3) \quad\text{in }\widehat{\Br}_4,
\end{equation}
and
\begin{equation}\label{eq: IV}\tag{NS II}
f(\beta_1,\beta_2^4) = \beta_2^{8\rho_2} f(\beta_1^2,\beta_2^2) \beta_1^{4\rho_2} (\beta_1\beta_2)^{-6\rho_2} \quad\text{in }\widehat{\Br}_3 \cong\widehat{\Gamma}_{1,1}.
\end{equation}

Here $h=h(x,y)$ is the unique element of $\widehat{\mathbb F}_2$ satisfying
\[
f(x,y)=h(y,x)^{-1}h(x,y);
\]
see~\cite{lochak1997cohomological}. We write $\rho_2=\rho_2(\lambda,f)$ for the extension of the Kummer $1$-cocycle associated to the roots of $2$, characterised by
\[
h(x,y)\equiv(xy)^{\rho_2} \quad\text{in }\widehat{\mathbb F}_2^{ab}
\cong\widehat{\mathbb Z}\times\widehat{\mathbb Z};
\]
see~\cite[Definition~5.1]{Nakamura-Schneps}. The group $\NS$ is the group of invertible elements of $\underline{\NS}$.
\end{definition}

In \eqref{eq: III'}, the elements $\beta_1,\beta_2,\beta_3$ are the standard generators of $\Br_4$, and we use the braid-group formulation of the relation given in \cite[Proposition~5.4]{Nakamura-Schneps}. In \eqref{eq: IV}, $\beta_1,\beta_2$ are the standard generators of $\Br_3$, which we identify with $\Gamma_{1,1}$.

Nakamura and Schneps prove that $\NS$ is a subgroup of $\GT$ containing the image of the absolute Galois group \cite[Theorem~1.2]{Nakamura-Schneps}. For $g\in\galQ$, the corresponding element is $(\lambda_g,f_g)$, where $\lambda_g$ is the cyclotomic character and $f_g$ lies in the closed commutator subgroup of the étale fundamental group $\pi_1^{\mathrm{et}}(\calM_{0,4,\overline{\mathbb Q}})\cong\widehat{\mathbb F}_2$.

Nakamura and Schneps construct actions of $\NS$ on profinite
mapping class groups using pants decompositions equipped with
quilts. The markings used above determine the curves and Dehn
twists appearing in their explicit formulae
\cite[Theorem~11.2]{Nakamura-Schneps}.
We express these formulae in the language of marked surfaces
and use them to construct an action on the profinite modular
operad $\widehat{\bS}$.
Its genus-zero restriction agrees with the action of
$\underline{\GT}$ from Proposition~\ref{prop:genus-zero-GT},
restricted along $\underline{\NS}\hookrightarrow\underline{\GT}$.

\begin{theorem}\label{main theorem NS action}
There is an inclusion of monoids $\underline{\NS}\hookrightarrow \End_0(\widehat{\bS})$ whose restriction to genus zero agrees with the isomorphism \[\underline{\GT}\cong \End_0(\trun{0}\widehat{\bS})\] from Proposition~\ref{prop:genus-zero-GT}, after restricting along $\underline{\NS}\hookrightarrow\underline{\GT}$.
\end{theorem}

\begin{proof}
Let $F=(\lambda,f)\in\underline{\NS}$. Recall that an object-fixing map of modular operads $\iota_F:\bS\longrightarrow |\widehat{\bS}|$ extends uniquely to a continuous object-fixing endomorphism $\hat{\iota}_F:\widehat{\bS}\longrightarrow \widehat{\bS}.$ Thus it is enough, using Theorem~\ref{thm: presentation for maps out of bS}, to choose morphisms
\[
\hat{\iota}_F(\tau)=\mathbf{t}_F \in \Hom_{\widehat{\bS}(0,2)}(|,|),
\qquad
\hat{\iota}_F(\beta)= \mathbf{b}_F \in \Hom_{\widehat{\bS}(0,3)}(P,(12)^*P),
\]
and
\[
\hat{\iota}_F(\alpha)=\mathbf{a}_F \in
\Hom_{\widehat{\bS}(0,4)}(P\circ_1 P,P\circ_2 P),
\qquad
\hat{\iota}_F(\sigma)=\mathbf{s}_F \in
\Hom_{\widehat{\bS}(1,1)}(\xi_{12}P,\xi_{12}P)
\]
and check that these choices satisfy the defining relations. Since $\iota_F$ is object-fixing, the object-level symmetry relation \eqref{sym_rel_morph} is automatic. Thus it remains to check the morphism relations \eqref{twist_rel_morphS}, \eqref{pent_rel_1_morphS}, \eqref{hex_rel_1_morphS}, \eqref{hex_rel_2_morphS}, \eqref{g1_rel_morph_1}, \eqref{g1_rel_morph_2}, and \eqref{g2_rel_morph_1}.

We choose the values of our generating morphisms using the action described in \cite{Nakamura-Schneps}. Explicitly, letting $\nu=(\lambda-1)/2$, we set
\begin{align}
    \mathbf{t}_F& = \tau^\lambda
    &
    \mathbf{b}_F &=\beta \cdot (\beta \cdot (12)^*\beta)^{\nu} \nonumber\\
    \mathbf{a}_F &= \alpha \cdot f(D_{d},D_{c})
    &
    \mathbf{s}_F &= \sigma^{\lambda}\cdot D_{b}^{8\rho_2(F)}\cdot f(D_{b}^2,D_a^2) \cdot D_a^{-8\rho_2(F)} .
    \label{NS Action}
\end{align}
Here $\rho_2(F)$ denotes the value, uniquely determined by $F=(\lambda,f)$, of the Nakamura--Schneps extension of the Kummer $1$-cocycle with respect to the roots of $2$. The symbol $D_i$ denotes the Dehn twist about the curve $i$. The curves $a,b,c,d$ appearing in \eqref{NS Action} are shown in Figures~\ref{fig:S-move} and \ref{fig:action-on-A-move}.
    
    \begin{figure}[ht]
        \centering
        \includegraphics[]{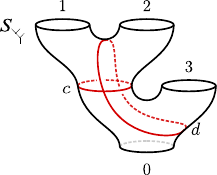}
        \caption{}
        \label{fig:action-on-A-move}
    \end{figure}

We now check the defining relations. The genus-zero relations can be dealt with first. By Proposition~8.1 of \cite{Boavida-Horel-Robertson} and Proposition~\ref{prop:genus-zero-GT}, the object-fixing endomorphisms of $\trun{0}\widehat{\bS}$ are identified with $\underline{\GT}$, and the element $F=(\lambda,f)\in\underline{\NS}\subseteq\underline{\GT}$ acts on the generators by the formulae defining $\mathbf{t}_F$, $\mathbf{b}_F$, and $\mathbf{a}_F$. Therefore these three morphisms define the genus-zero part of $\iota_F$, \[\trun{0}(\iota_F):\trun{0}\bS\longrightarrow |\trun{0}\widehat{\bS}|,\] and the genus-zero relations \eqref{twist_rel_morphS}, \eqref{pent_rel_1_morphS}, \eqref{hex_rel_1_morphS}, and \eqref{hex_rel_2_morphS} are satisfied.

It remains to check the genuinely genus-one relations \eqref{g1_rel_morph_1}, \eqref{g1_rel_morph_2}, and \eqref{g2_rel_morph_1}. In the following calculations we use the mapping class group relations from Proposition~\ref{prop: generators and relations of mcg}, together with the conjugation formulae in Lemma~\ref{lemma: conjugation by mapping classes}.

We first check the two relations in type $(1,1)$, namely \eqref{g1_rel_morph_1} and \eqref{g1_rel_morph_2}. These are relations in
\[
\Aut_{\widehat{\bS}(1,1)}(\xi_{12}P) \cong \widehat{\Gamma}_{1,1}.
\] We will use the curves shown in Figure~\ref{fig: g1-curves}. Since the calculations below are notation-heavy, we write $a$, $b$, and $c$ instead of $D_{a}, D_{b}$ and $D_{c}$ for the Dehn twists about the curves labelled $a$, $b$, and $c$.

\begin{figure}[h!]
    \centering
    \includegraphics[width=0.25\linewidth]{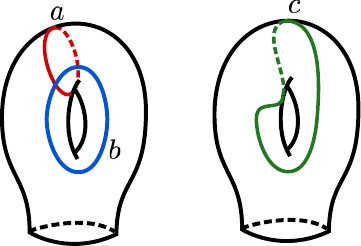}
    \caption{Curves used in the verification of the genus-one relations in $\widehat{\Gamma}_{1,1}$.}
    \label{fig: g1-curves}
\end{figure}

\emph{Checking relation \eqref{g1_rel_morph_1}:}  we first note that $\xi_{02}\beta^{-1}=(aba)^2$ and $\xi_{02}(\beta\cdot (12)^*\beta)=\partial_0^{-1}$. Then 
\begin{align*}
    \hat{\iota}_{F}(G1a)&=(aba)^{\lambda} b^{8\rho_2}f(b^2,a^2)a^{-8\rho_2} \cdot (aba)^{\lambda} b^{8\rho_2}f(b^2,a^2)a^{-8\rho_2} \cdot  (aba)^{-2} \partial_0^{-\nu}.
\end{align*}
Using that $(aba)^{-1}a(aba)=b$ and $(aba)^{-1}b(aba)=a$, we can move the middle $(aba)^\lambda$ term to the left to get:
\begin{align*}
    \hat{\iota}_{F}(G1a)&=(aba)^{2\lambda} a^{8\rho_2}f(a^2,b^2)b^{-8\rho_2} b^{8\rho_2}f(b^2,a^2)a^{-8\rho_2} (aba)^{-2} \partial_0^{-\nu}
\end{align*}
Recalling that $f(a^2,b^2)=f(b^2,a^2)^{-1}$, this gives us
\begin{align*}
    \hat{\iota}_{F}(G1a)&=(aba)^{4\nu} (aba)^2 (aba)^{-2} \partial_0^{-\nu}
\end{align*}
Since $(aba)^4=\partial_0$, relation \eqref{g1_rel_morph_1} follows.

\medskip 

\emph{Checking relation \eqref{g1_rel_morph_2}:} Recall that $d=\xi_{12}(\id_P\circ_1\tau)\in \Aut_{\bS(1,1)}(\xi_{12}P)$. Under the identification $\Aut_{\bS(1,1)}(\xi_{12}P)\cong \Gamma_{1,1}$, the morphism $d$ is the Dehn twist about the curve labelled $a$ in Figure~\ref{fig: g1-curves}. Hence 
    \[a=\xi_{1,2}(\id_P\circ_1 \tau)\]
and therefore its image under $\iota_F$ is represented in $\widehat{\Gamma}_{1,1}$ by $a^\lambda$. We compute the image of the corresponding loop, replacing each occurrence of $d$ by $a^\lambda$:
\begin{align*}
    \iota_F(G1b)
    &=(aba)^{\lambda} b^{8\rho_2}f(b^2,a^2)a^{-8\rho_2}\cdot a^{\lambda}\cdot  a^{8\rho_2}f(a^2,b^2)b^{-8\rho_2}(aba)^{-\lambda}\cdot a^{\lambda}\cdot  a^{8\rho_2}f(a^2,b^2)b^{-8\rho_2}(aba)^{-\lambda}\cdot a^{\lambda}.
\end{align*}

Using that $a^{-1}ba=c$ we get \[\iota_{F}(G1b)=(aba)^{\lambda} b^{8\rho_2}f(b^2,a^2)a^{-8\rho_2}\cdot a^{\lambda}\cdot  a^{8\rho_2}f(a^2,b^2)b^{-8\rho_2}(aba)^{-\lambda}\cdot a^{\lambda}\cdot  a^{8\rho_2}f(a^2,b^2)b^{-8\rho_2}\mathbf{(aba)^{-1}a}(aca)^{-2\nu}\cdot a^{2\nu}\]

Since $(aba)^{-1}a= (ba)^{-1}$ and 
    \begin{align*}
        (ba)a(ba)^{-1}&= c & 
        (ba)b(ba)^{-1}&= a,
            \end{align*}
we get
\[\iota_{F}(G1b)=(aba)^{\lambda} b^{8\rho_2}f(b^2,a^2)a^{-8\rho_2}\cdot a^{\lambda}\cdot  a^{8\rho_2}f(a^2,b^2)b^{-8\rho_2}(aba)^{-\lambda}\cdot \mathbf{a(ba)^{-1}} \cdot c^{2\nu}\cdot  c^{8\rho_2}f(c^2,a^2)a^{-8\rho_2}(aca)^{-2\nu}\cdot a^{2\nu}\]

Since $a(ba)^{-1}=b^{-1}$ and 
    \begin{align*}
        bab^{-1}&=c & 
        bbb^{-1}&=b,
    \end{align*}
we get
\[\iota_{F}(G1b)=(aba)^{\lambda} b^{8\rho_2}f(b^2,a^2)a^{-8\rho_2}\cdot a^{\lambda}\cdot  a^{8\rho_2}f(a^2,b^2)b^{-8\rho_2} \mathbf{(aba)^{-1}b^{-1}} \cdot (cbc)^{-2\nu}\cdot c^{2\nu}\cdot  c^{8\rho_2}f(c^2,a^2)a^{-8\rho_2}(aca)^{-2\nu}\cdot a^{2\nu}\]
Now we use that 
    \begin{align*}
        (b(aba))a((aba)^{-1}b^{-1})&=b &  (b(aba))b((aba)^{-1}b^{-1})&=c 
    \end{align*}
giving us
\[\iota_{F}(G1b)=(aba)^{\lambda} b^{8\rho_2}f(b^2,a^2)a^{-8\rho_2}\cdot \mathbf{a(aba)^{-1}b^{-1}} \cdot b^{2\nu}\cdot  b^{8\rho_2}f(b^2,c^2)c^{-8\rho_2}  (cbc)^{-2\nu}\cdot c^{2\nu}\cdot  c^{8\rho_2}f(c^2,a^2)a^{-8\rho_2}(aca)^{-2\nu}\cdot a^{2\nu}\]

Since $a(aba)^{-1}b^{-1}=(bab)^{-1}$ and 
    \begin{align*}
    (bab)a(bab)^{-1} &=b &  (bab)b(bab)^{-1}&=a
    \end{align*}
we get
\[\iota_{F}(G1b)=\mathbf{(aba)(aba)^{-1}} \cdot(bab)^{2\nu} a^{8\rho_2}f(a^2,b^2)b^{-8\rho_2}\cdot  b^{2\nu}\cdot  b^{8\rho_2}f(b^2,c^2)c^{-8\rho_2}  (cbc)^{-2\nu}\cdot c^{2\nu}\cdot  c^{8\rho_2}f(c^2,a^2)a^{-8\rho_2}(aca)^{-2\nu}\cdot a^{2\nu}\]

We now note that $(aba)^2$, $(aca)^2$, and $(bcb)^2$ are all equal and generate the centre of $\widehat{\Gamma}_{1,1}$, so we have
\[\iota_{F}(G1b)=(aba)^{-2\nu} a^{8\rho_2}f(a^2,b^2)\cdot  b^{2\nu}\cdot  f(b^2,c^2)\cdot c^{2\nu}\cdot  f(c^2,a^2)\cdot a^{2\nu}a^{-8\rho_2}\]
Finally, using that $c=a^{-1}ba$, we can check that $b^2a^2c^2=(aba)^2$ so using \cite[Relation (1.5.1)]{Nakamura-Schneps} (recall they write multiplication in the opposite order), we have
\[\iota_{F}(G1b)=(aba)^{-2\nu} a^{8\rho_2}\cdot(aba)^{2\nu}\cdot a^{-8\rho_2}=1.\]
Hence, relation \eqref{g1_rel_morph_2} holds.

\medskip

\emph{Checking relation \eqref{g2_rel_morph_1}:} We will study relation \eqref{g2_rel_morph_1} through the composition of morphisms depicted in Figure~\ref{fig:2-cell from NS} starting at the top of the diagram. In relation \eqref{g2_rel_morph_1}, the mapping class $\alpha$ appears conjugating certain mapping classes. However, by Lemma~\ref{lemma: conjugation by mapping classes}, these conjugations can be expressed entirely in terms of mapping classes on the surface $S_{\graphgtwo}$.\footnote{Here the graph $\graphgtwo$ is the composition $\xi_{2}^{3}(P\circ_1P)$ in $\Upsilon$.} Hence all our computations can be made on the profinite completion of the mapping class group of $S_{\graphgtwo}$.

For such computations, we will use as a reference the curves in Figure~\ref{fig:curves for g2} :

\begin{figure}[h!]
\centering
\includegraphics[width=0.85\linewidth]{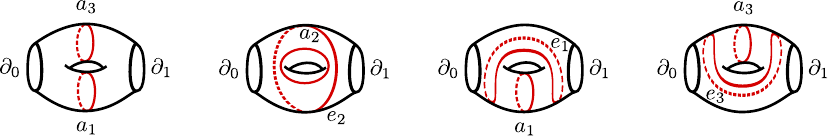}
\caption{}
\label{fig:curves for g2}
\end{figure} 

We start by noting the following equality on mapping class groups
    \[\xi_{12}[(z_3^*\beta)\circ_2(\tau^{-1},\tau^{-1})\circ \beta)]=\partial_0^{\frac{1}{2}}\partial_1^{\frac{1}{2}}\]
where $\partial_0^{\frac{1}{2}}\partial_1^{\frac{1}{2}}$ denotes the simultaneous half-twists along boundaries $\partial_0$ and $\partial_1$. Together with the definition of $\iota_{F}$ and relation \eqref{twist_rel_morphS}, this yields
    \[\hat{\iota}_{F}[\xi_{12}((z_3^*\mathbf{b})\circ_2(\mathbf{t}^{-1},\mathbf{t}^{-1})\circ \mathbf{b})]=\partial_0^{\nu}\partial_1^{\nu}(\partial_0^{\frac{1}{2}}\partial_1^{\frac{1}{2}}).\]
Since $\lambda=2\nu+1$, we abuse notation and denote this element by $\partial_0^{\frac{\lambda}{2}}\partial_1^{\frac{\lambda}{2}}$. Hence, checking relation \eqref{g2_rel_morph_1} amounts to showing that the following expression is trivial in the profinite mapping class group of $S_{\graphgtwo}$:
\begin{multline*}
    \iota_{F}(G2)=\partial_0^{-\frac{\lambda}{2}}\partial_1^{-\frac{\lambda}{2}}f(a_1,e_3)\cdot (a_3 a_2 a_3)^{\lambda} a_2^{8\rho_2} f(a_2^2,a_3^2) a_3^{-8\rho_2} \cdot f(e_3,a_1) \cdot a_3^{-\lambda} a_1^{\lambda} \cdot f(a_1,e_3) \cdot (a_3 a_2 a_3)^{\lambda} a_2^{8\rho_2} f(a_2^2,a_3^2) a_3^{-8\rho_2}\cdot f(e_3,a_1)
\end{multline*}

As in the proof of relation \eqref{g1_rel_morph_2}, we repeatedly use the fact that $\lambda=2\nu+1$ to separate one copy of each profinite power and move it to the left-hand side of the expression. Starting from the rightmost such term:
\begin{multline*}
    \iota_{F}(G2)=\partial_0^{-\frac{\lambda}{2}}\partial_1^{-\frac{\lambda}{2}}f(a_1,e_3)\cdot (a_3 a_2 a_3)^{\lambda} a_2^{8\rho_2} f(a_2^2,a_3^2) a_3^{-8\rho_2} \cdot f(e_3,a_1) \cdot a_3^{-\lambda} a_1^{\lambda} \cdot f(a_1,e_3) \cdot \\\mathbf{(a_3 a_2 a_3)}\cdot (a_3 a_2 a_3)^{2\nu} a_2^{8\rho_2} f(a_2^2,a_3^2) a_3^{-8\rho_2}\cdot f(e_3,a_1)
\end{multline*}

Now using that 
    \begin{align*}
        (a_3 a_2 a_3)^{-1}e_3(a_3 a_2 a_3)&= e_3 & (a_3 a_2 a_3)^{-1}a_1(a_3 a_2 a_3)&=e_2\\
        (a_3 a_2 a_3)^{-1}a_3(a_3 a_2 a_3)&= a_2
    \end{align*}
we get
\begin{multline*}
    \iota_{F}(G2)=\partial_0^{-\frac{\lambda}{2}}\partial_1^{-\frac{\lambda}{2}}f(a_1,e_3)\cdot (a_3 a_2 a_3)^{\lambda} a_2^{8\rho_2} f(a_2^2,a_3^2) a_3^{-8\rho_2} \cdot f(e_3,a_1) \cdot \mathbf{[a_3^{-1}a_1(a_3 a_2 a_3)]}\cdot a_2^{-2\nu} e_2^{2\nu} \cdot  \\ f(e_2,e_3) \cdot(a_3 a_2 a_3)^{2\nu} a_2^{8\rho_2} f(a_2^2,a_3^2) a_3^{-8\rho_2}\cdot f(e_3,a_1)
\end{multline*}

Since $a_1$ and $a_3$ are disjoint, they commute. Hence the expression in bold is equal to $a_1a_2a_3$. We now use the following relations in the mapping class group
    \begin{align*}
        (a_1a_2a_3)^{-1} a_1 (a_1a_2a_3)&=e_2 & (a_1a_2a_3)^{-1}e_3(a_1a_2a_3)&= e_1\\
        (a_1a_2a_3)^{-1} a_3 (a_1a_2a_3)&= a_2 & (a_1a_2a_3)^{-1} a_2 (a_1a_2a_3)&= a_1
    \end{align*}
then pulling the bold terms to the left, we get
\begin{multline*}
    \iota_{F}(G2)=\partial_0^{-\frac{\lambda}{2}}\partial_1^{-\frac{\lambda}{2}}f(a_1,e_3)\cdot \mathbf{[(a_3a_2a_3)(a_1 a_2 a_3)]}\cdot(a_2 a_1 a_2)^{2\nu} a_1^{8\rho_2} f(a_1^2,a_2^2) a_2^{-8\rho_2} \cdot f(e_1,e_2) \cdot  a_2^{-2\nu} e_2^{2\nu} \cdot  \\ f(e_2,e_3) \cdot(a_3 a_2 a_3)^{2\nu} a_2^{8\rho_2} f(a_2^2,a_3^2) a_3^{-8\rho_2}\cdot f(e_3,a_1)
\end{multline*}

Now note that
    \begin{align*}
        [(a_3a_2a_3)(a_1 a_2 a_3)]^{-1}e_3[(a_3a_2a_3)(a_1 a_2 a_3)]&=e_1 \\
        [(a_3a_2a_3)(a_1 a_2 a_3)]^{-1}a_1[(a_3a_2a_3)(a_1 a_2 a_3)]&=a_3
    \end{align*}
And we get
    \begin{multline*}
        \iota_{F}(G2)=\partial_0^{-\frac{\lambda}{2}}\partial_1^{-\frac{\lambda}{2}}\mathbf{[(a_3a_2a_3)(a_1 a_2 a_3)]}\cdot f(a_3,e_1)\cdot (a_2 a_1 a_2)^{2\nu} a_1^{8\rho_2} f(a_1^2,a_2^2) a_2^{-8\rho_2} \cdot f(e_1,e_2) \cdot  a_2^{-2\nu} e_2^{2\nu} \cdot  \\ f(e_2,e_3) \cdot(a_3 a_2 a_3)^{2\nu} a_2^{8\rho_2} f(a_2^2,a_3^2) a_3^{-8\rho_2}\cdot f(e_3,a_1)
    \end{multline*}
Since
    \begin{align*}
        (a_3a_2a_3)(a_1 a_2 a_3)&= \partial_0^{\frac{1}{2}}\partial_2^{\frac{1}{2}},
    \end{align*}
our expression for $\iota_{F}(G2)$ simplifies to 
    \begin{multline*}
        \iota_{F}(G2)=\partial_0^{-\nu}\partial_1^{-\nu}\cdot f(a_3,e_1)\cdot (a_2 a_1 a_2)^{2\nu} a_1^{8\rho_2} f(a_1^2,a_2^2) a_2^{-8\rho_2} \cdot f(e_1,e_2) \cdot  a_2^{-2\nu} e_2^{2\nu} \cdot  \\ f(e_2,e_3) \cdot(a_3 a_2 a_3)^{2\nu} a_2^{8\rho_2} f(a_2^2,a_3^2) a_3^{-8\rho_2}\cdot f(e_3,a_1)
    \end{multline*}
which is shown to be trivial in \cite[Proof of Claim 8.3, relation (6AS), end of page 542]{Nakamura-Schneps} (recall that in that paper they compose mapping classes in the opposite order).

\medskip 

\emph{Homomorphism:} We have now constructed a map $\NS\hookrightarrow\End_0(\hat{S})$ and we finish by showing it preserves the composition in $\NS$. Recall that, for \(F=(\lambda,f)\) and \(G=(\mu,g)\), our multiplication
convention is defined by
\[
F\cdot G =(\lambda\mu,\, g(x,y)\,f(g(x,y)^{-1}x^{\mu} g(x,y),y^{\mu})),
\]
so that \(\varphi_{F\cdot G}=\varphi_G\circ\varphi_F\).  It remains to verify that the composition $\iota_G\circ \iota_F$ agrees with $\iota_{F\cdot G}$ on the generators $\tau$, $\beta$, $\alpha$, and $\sigma$. The verification for $\tau$ and $\beta$ is immediate. The arguments for $\alpha$ and $\sigma$ are similar, so we only discuss the latter, which is slightly more involved.

We now show that
    \[\iota_G\circ \iota_F(\sigma)=\iota_{F\cdot G}(\sigma).\]
The left-hand side is equal to 
    \begin{align}\label{eq:homomorphism eq for sigma}
        \iota_G(\sigma^{\lambda}\cdot {b}^{8\rho_2(F)}\cdot f({b}^2,a^2) \cdot a^{-8\rho_2(F)})=\iota_G(\sigma)^{\lambda}\cdot {\iota_G(b)}^{8\rho_2(F)}\cdot f({\iota_G(b)}^2,\iota_G(a)^2) \cdot \iota_G(a)^{-8\rho_2(F)}
    \end{align}
The term $\iota_G(\sigma)$ is explicitly described in \eqref{NS Action} and, as discussed in the proof of \eqref{g2_rel_morph_1}, we have $\iota_G(a)=a^\mu$. Moreover, since $b=\sigma^{-1}a\sigma$, we get 
    \begin{align*}
        \iota_G(b)&=\iota_G(\sigma)^{-1}  \cdot \iota_G(a)\cdot \iota_G(\sigma)\\
        &=[\sigma^{\mu}\cdot b^{8\rho_2(G)}\cdot g(b^2,a^2) \cdot a^{-8\rho_2(G)}]^{-1} \cdot a^\mu \cdot [\sigma^{\mu}\cdot b^{8\rho_2(G)}\cdot g(b^2,a^2) \cdot a^{-8\rho_2(G)}]\\
        &=a^{8\rho_2(G)}\cdot  g(b^2,a^2)^{-1} \cdot b^{-8\rho_2(G)}\cdot \sigma^{-\mu} \cdot a^\mu \cdot \sigma^{\mu}\cdot b^{8\rho_2(G)}\cdot g(b^2,a^2) \cdot a^{-8\rho_2(G)}.
    \end{align*}
Since $\mu$ is of the form $2\nu'+1$, and that $\sigma^2$ commutes with $a$ and $b$, we get
    \begin{align*}
        \iota_G(b)
        &=a^{8\rho_2(G)}\cdot  g(b^2,a^2)^{-1} \cdot b^{-8\rho_2(G)}\cdot  b^\mu \cdot b^{8\rho_2(G)}\cdot g(b^2,a^2) \cdot a^{-8\rho_2(G)}\\
        &=a^{8\rho_2(G)}  g(b^2,a^2)^{-1}  b^\mu g(b^2,a^2)  a^{-8\rho_2(G)}.
    \end{align*}
Substituting these expressions into \eqref{eq:homomorphism eq for sigma}, we obtain
    \begin{multline}\label{eq: homomorphism on sigma-2}
        \iota_G\circ \iota_F(\sigma)=[\sigma^{\mu}\cdot b^{8\rho_2(G)}\cdot g(b^2,a^2) \cdot a^{-8\rho_2(G)}]^\lambda \cdot a^{8\rho_2(G)}  g(b^2,a^2)^{-1}  b^{8\rho_2(F)\mu} g(b^2,a^2)  a^{-8\rho_2(G)}\cdot\\
         \cdot f(a^{8\rho_2(G)}  g(b^2,a^2)^{-1}  b^{2\mu} g(b^2,a^2)  a^{-8\rho_2(G)},a^{2\mu})\cdot a^{-8\rho_2(F)\mu}
    \end{multline}
We begin by simplifying the first factor on the right-hand side using that:
    \begin{align*}
        [\sigma^{\mu}\cdot b^{8\rho_2(G)}\cdot g(b^2,a^2) \cdot a^{-8\rho_2(G)}]^2&=\sigma^{\mu}\cdot b^{8\rho_2(G)}\cdot g(b^2,a^2) \cdot a^{-8\rho_2(G)}\cdot\sigma^{\mu}\cdot b^{8\rho_2(G)}\cdot g(b^2,a^2) \cdot a^{-8\rho_2(G)}\\
        &=\sigma^{2\mu}\cdot a^{8\rho_2(G)}\cdot g(a^2,b^2) \cdot b^{-8\rho_2(G)}\cdot b^{8\rho_2(G)}\cdot g(b^2,a^2) \cdot a^{-8\rho_2(G)}\\
        &=\sigma^{2\mu}
    \end{align*}
Since $\lambda=2\nu+1$, equation~\eqref{eq: homomorphism on sigma-2} becomes
    \begin{multline*}
        \iota_G\circ \iota_F(\sigma)=[\sigma^{\lambda\mu}\cdot b^{8\rho_2(G)}\cdot g(b^2,a^2) \cdot a^{-8\rho_2(G)}] \cdot a^{8\rho_2(G)}  g(b^2,a^2)^{-1}  b^{8\rho_2(F)\mu} g(b^2,a^2)  a^{-8\rho_2(G)}\cdot\\
         \cdot f(a^{8\rho_2(G)}  g(b^2,a^2)^{-1}  b^{2\mu} g(b^2,a^2)  a^{-8\rho_2(G)},a^{2\mu})\cdot a^{-8\rho_2(F)\mu}.
    \end{multline*}
After cancellation, this simplifies to
    \begin{align*}
        \iota_G\circ \iota_F(\sigma)&=\sigma^{\lambda\mu}\cdot b^{8(\rho_2(G)+\rho_2(F)\mu)} \cdot g(b^2,a^2)\cdot f(g(b^2,a^2)^{-1}  b^{2\mu} g(b^2,a^2),a^{2\mu})\cdot a^{-8(\rho_2(G)+\rho_2(F)\mu)}.
    \end{align*}
Using that $\rho_2(F\cdot G)=\rho_2(G)+\rho_2(F)\mu$, as discussed in \cite[Corollary 5.2]{Nakamura-Schneps}, we see that this is precisely the formula defining $\iota_{F\cdot G}(\sigma)$, as required.  Finally, the homomorphism is injective because its restriction to the
genus-zero truncation agrees with the inclusion
\(\underline{\NS}\subseteq \underline{\GT}\) under the isomorphism
\[
\underline{\GT}\cong
\End_0(\operatorname{tr}_{\leq 0}\widehat{\mathbf S})
\]
of Proposition~\ref{prop:genus-zero-GT}.
\end{proof}

\begin{remark}
The curve $e_2$ in Figure~\ref{fig:curves for g2} differs slightly from the curve labelled $e_2$ in \cite{Nakamura-Schneps}. The relation used in their calculation, however, is
\[
(a_2a_3)^{-1}a_1(a_2a_3)=e_2,
\]
which is the relation satisfied by the curve $e_2$ drawn here. Thus our figure should be read as using the curve determined by this conjugation relation.

We also note that Nakamura--Schneps write $\mu$ for the quantity we denote by $\nu=(\lambda-1)/2$, since $\mu$ is already reserved in this paper for the generating object of $\PaRB$.
\end{remark}

Theorem~\ref{main theorem NS action}, together with the inclusion $\galQ\hookrightarrow \NS$ constructed by Nakamura--Schneps, immediately gives the corresponding Galois action on the profinite modular operad of
surfaces.

\begin{cor}
The absolute Galois group $\galQ$ acts on $\widehat{\bS}$.
\end{cor}

\begin{proof}
This follows from Theorem~\ref{main theorem NS action} and the inclusion
$\galQ\hookrightarrow \NS$ of Nakamura--Schneps.
\end{proof}
\section{Modular $\infty$-operads}\label{sec:modular infinity operads}

To take profinite completions of modular operads in spaces, we need a
model for homotopy-coherent modular operads. The relevant indexing
category is the modular dendroidal category $\bM$, whose objects are
connected genus-graded labelled graphs with non-empty boundary. The
modular dendroidal nerve of a strict modular operad satisfies
\[
(\nerve\mathcal P)_G
\cong
\prod_{v\in V_G}(\nerve\mathcal P)_{\corolla_v}.
\]
See \cite[Theorem~3.6]{hry_modular_nerve}. A modular $\infty$-operad is
obtained by replacing this strict product condition by a Segal
condition.

We first recall the modular dendroidal category. Let $G=(G,\lambda,\ell,\epsilon)$ be a genus-graded labelled graph. For each vertex $v\in V_G$, let $H_v$ be a genus-graded labelled graph whose boundary is identified with the half-edges incident to $v$. We require
$H_v$ to have total genus $\epsilon(v)$, equivalently
\[
\epsilon(v)=\beta_1(H_v)+ \sum_{w\in V_{H_v}}\epsilon_{H_v}(w),
\]
where $\epsilon_{H_v}$ denotes the genus function on $H_v$. Then we can define a new graph $G\{H_v\}_{v\in V_G}$ by the coequaliser
\[
\begin{tikzcd}
\displaystyle\coprod_{e\in iE_G}\updownarrow
\arrow[r, shift left=1, "u"]
\arrow[r, shift right=1, "d"']
&
\displaystyle\coprod_{v\in V_G} H_v
\arrow[r, "\pi"]
&
G\{H_v\}_{v\in V_G}.
\end{tikzcd}
\]
Here, the maps $u$ and $d$ record the two incidences of each internal edge of $G$. The coequaliser glues the corresponding boundary legs of the graphs $H_v$. This is the graph-substitution construction of \cite[Construction~1.18]{hry1}.

\begin{definition}\label{def: modular dendroidal category}
The \emph{modular dendroidal category} $\bM$ is the category whose objects are connected genus-graded labelled graphs with non-empty boundary. Its morphisms, called \emph{graphical maps}, are generated
by isomorphisms, the graph-substitution maps $G\to G\{H_v\}_{v\in V_G}$, and the maps $H_v\to G\{H_w\}_{w\in V_G}$ supplied by the coequaliser. See \cite[Definition~1.31]{hry1} for full details.
\end{definition}

\begin{figure}[h!]
\[
\begin{tikzpicture}[x=0.75pt,y=0.75pt,yscale=-1,xscale=1]
%uncomment if require: \path (0,823); %set diagram left start at 0, and has height of 823

%Straight Lines [id:da6488749677933408] 
\draw    (114.84,123.97) -- (143.88,188.12) ;
%Straight Lines [id:da18732783760567262] 
\draw    (73.89,176.06) -- (107.15,133.2) ;
%Straight Lines [id:da7889407840448126] 
\draw    (51.03,124.07) -- (102.31,123.97) ;
%Straight Lines [id:da4431489597110816] 
\draw    (143.88,188.12) -- (115.59,245.03) ;
%Shape: Circle [id:dp31356981671318795] 
\draw  [color={rgb, 255:red, 4; green, 146; blue, 194 }  ,draw opacity=1 ][fill={rgb, 255:red, 255; green, 255; blue, 255 }  ,fill opacity=1 ][dash pattern={on 4.5pt off 4.5pt}] (102.31,123.97) .. controls (102.31,117.05) and (107.92,111.44) .. (114.84,111.44) .. controls (121.76,111.44) and (127.37,117.05) .. (127.37,123.97) .. controls (127.37,130.89) and (121.76,136.5) .. (114.84,136.5) .. controls (107.92,136.5) and (102.31,130.89) .. (102.31,123.97) -- cycle ;
%Curve Lines [id:da771171194316091] 
\draw    (122.04,114.52) .. controls (166.89,52.51) and (199.98,153.31) .. (126.66,129.14) ;
%Straight Lines [id:da8453143109032858] 
\draw    (170.22,268.12) -- (115.59,245.03) ;
%Shape: Ellipse [id:dp2086114067940088] 
\draw  [color={rgb, 255:red, 237; green, 130; blue, 14 }  ,draw opacity=1 ][fill={rgb, 255:red, 255; green, 255; blue, 255 }  ,fill opacity=1 ][dash pattern={on 4.5pt off 4.5pt}] (156,191.3) .. controls (154.25,197.99) and (147.41,202) .. (140.71,200.25) .. controls (134.01,198.5) and (130.01,191.65) .. (131.76,184.95) .. controls (133.51,178.26) and (140.36,174.25) .. (147.05,176) .. controls (153.75,177.75) and (157.76,184.6) .. (156,191.3) -- cycle ;
%Shape: Boxed Bezier Curve [id:dp583141332441497] 
\draw    (121.93,256.52) .. controls (159.11,323.41) and (53.78,310.69) .. (106.73,254.51) ;
%Shape: Circle [id:dp21641337625529522] 
\draw  [color={rgb, 255:red, 140; green, 219; blue, 54 }  ,draw opacity=1 ][fill={rgb, 255:red, 255; green, 255; blue, 255 }  ,fill opacity=1 ][dash pattern={on 4.5pt off 4.5pt}] (127.71,248.21) .. controls (125.96,254.9) and (119.12,258.91) .. (112.42,257.16) .. controls (105.72,255.41) and (101.72,248.56) .. (103.47,241.86) .. controls (105.22,235.17) and (112.07,231.16) .. (118.76,232.91) .. controls (125.46,234.66) and (129.47,241.51) .. (127.71,248.21) -- cycle ;
%Shape: Circle [id:dp3919998773078771] 
\draw  [color={rgb, 255:red, 4; green, 146; blue, 194 }  ,draw opacity=1 ][dash pattern={on 4.5pt off 4.5pt}] (42.04,70.45) .. controls (42.04,50.91) and (57.88,35.07) .. (77.42,35.07) .. controls (96.96,35.07) and (112.8,50.91) .. (112.8,70.45) .. controls (112.8,89.99) and (96.96,105.83) .. (77.42,105.83) .. controls (57.88,105.83) and (42.04,89.99) .. (42.04,70.45) -- cycle ;
%Shape: Circle [id:dp5469801557415326] 
\draw  [color={rgb, 255:red, 237; green, 130; blue, 14 }  ,draw opacity=1 ][dash pattern={on 4.5pt off 4.5pt}] (172.85,175.51) .. controls (172.85,160.37) and (185.12,148.1) .. (200.26,148.1) .. controls (215.4,148.1) and (227.68,160.37) .. (227.68,175.51) .. controls (227.68,190.65) and (215.4,202.93) .. (200.26,202.93) .. controls (185.12,202.93) and (172.85,190.65) .. (172.85,175.51) -- cycle ;
%Straight Lines [id:da10145939265222326] 
\draw    (453.36,104.93) -- (480.57,165.38) ;
%Straight Lines [id:da6840236886882134] 
\draw    (349.18,152.84) -- (379.77,113.42) ;
%Straight Lines [id:da3915242304438935] 
\draw    (328.16,105.03) -- (375.32,104.93) ;
%Straight Lines [id:da13158542747729218] 
\draw    (398.37,104.93) -- (441.84,104.93) ;
%Straight Lines [id:da9875493906550447] 
\draw    (421.52,51.15) -- (386.85,104.93) ;
%Straight Lines [id:da3512817237369321] 
\draw    (421.52,51.15) -- (453.36,104.93) ;
%Shape: Circle [id:dp11097495085655085] 
\draw  [fill={rgb, 255:red, 255; green, 255; blue, 255 }  ,fill opacity=1 ] (375.32,104.93) .. controls (375.32,98.57) and (380.48,93.41) .. (386.85,93.41) .. controls (393.21,93.41) and (398.37,98.57) .. (398.37,104.93) .. controls (398.37,111.3) and (393.21,116.45) .. (386.85,116.45) .. controls (380.48,116.45) and (375.32,111.3) .. (375.32,104.93) -- cycle ;
%Shape: Circle [id:dp6559922512513051] 
\draw  [fill={rgb, 255:red, 255; green, 255; blue, 255 }  ,fill opacity=1 ] (410,51.15) .. controls (410,44.79) and (415.15,39.63) .. (421.52,39.63) .. controls (427.88,39.63) and (433.04,44.79) .. (433.04,51.15) .. controls (433.04,57.52) and (427.88,62.68) .. (421.52,62.68) .. controls (415.15,62.68) and (410,57.52) .. (410,51.15) -- cycle ;
%Shape: Ellipse [id:dp28458230797732953] 
\draw  [fill={rgb, 255:red, 255; green, 255; blue, 255 }  ,fill opacity=1 ] (441.84,104.93) .. controls (441.84,98.57) and (447,93.41) .. (453.36,93.41) .. controls (459.73,93.41) and (464.89,98.57) .. (464.89,104.93) .. controls (464.89,111.3) and (459.73,116.45) .. (453.36,116.45) .. controls (447,116.45) and (441.84,111.3) .. (441.84,104.93) -- cycle ;
%Straight Lines [id:da7405793901835506] 
\draw    (91.99,70.14) -- (103.35,95.25) ;
%Straight Lines [id:da2704008409978078] 
\draw    (47.65,90.52) -- (60.67,73.75) ;
%Straight Lines [id:da23190885369214975] 
\draw    (42.04,70.45) -- (58.78,70.14) ;
%Straight Lines [id:da3111650015820755] 
\draw    (91.99,70.14) -- (112,78.43) ;
%Straight Lines [id:da6112752101701214] 
\draw    (91.99,70.14) -- (109.19,53.12) ;
%Straight Lines [id:da591579984214321] 
\draw    (68.59,70.14) -- (87.08,70.14) ;
%Straight Lines [id:da42403731742552686] 
\draw    (78.44,47.25) -- (63.68,70.14) ;
%Straight Lines [id:da8604597036244791] 
\draw    (78.44,47.25) -- (91.99,70.14) ;
%Shape: Ellipse [id:dp48890303362623866] 
\draw  [fill={rgb, 255:red, 255; green, 255; blue, 255 }  ,fill opacity=1 ] (58.78,70.14) .. controls (58.78,67.43) and (60.97,65.23) .. (63.68,65.23) .. controls (66.39,65.23) and (68.59,67.43) .. (68.59,70.14) .. controls (68.59,72.85) and (66.39,75.04) .. (63.68,75.04) .. controls (60.97,75.04) and (58.78,72.85) .. (58.78,70.14) -- cycle ;
%Shape: Ellipse [id:dp4835970024442289] 
\draw  [fill={rgb, 255:red, 255; green, 255; blue, 255 }  ,fill opacity=1 ] (73.53,47.25) .. controls (73.53,44.55) and (75.73,42.35) .. (78.44,42.35) .. controls (81.14,42.35) and (83.34,44.55) .. (83.34,47.25) .. controls (83.34,49.96) and (81.14,52.16) .. (78.44,52.16) .. controls (75.73,52.16) and (73.53,49.96) .. (73.53,47.25) -- cycle ;
%Shape: Ellipse [id:dp09294465671430874] 
\draw  [fill={rgb, 255:red, 255; green, 255; blue, 255 }  ,fill opacity=1 ] (87.08,70.14) .. controls (87.08,67.43) and (89.28,65.23) .. (91.99,65.23) .. controls (94.69,65.23) and (96.89,67.43) .. (96.89,70.14) .. controls (96.89,72.85) and (94.69,75.04) .. (91.99,75.04) .. controls (89.28,75.04) and (87.08,72.85) .. (87.08,70.14) -- cycle ;

%Straight Lines [id:da6207155341595856] 
\draw    (187.17,152.19) -- (198.74,177.74) ;
%Straight Lines [id:da49058651435496536] 
\draw    (198.74,177.74) -- (187.47,200.41) ;
%Curve Lines [id:da03965419119301694] 
\draw    (201.73,173.18) .. controls (219.59,148.48) and (232.77,188.63) .. (203.56,179) ;
%Shape: Circle [id:dp18101776392820634] 
\draw  [fill={rgb, 255:red, 255; green, 255; blue, 255 }  ,fill opacity=1 ] (193.75,177.74) .. controls (193.75,174.98) and (195.98,172.75) .. (198.74,172.75) .. controls (201.49,172.75) and (203.73,174.98) .. (203.73,177.74) .. controls (203.73,180.5) and (201.49,182.73) .. (198.74,182.73) .. controls (195.98,182.73) and (193.75,180.5) .. (193.75,177.74) -- cycle ;
%Straight Lines [id:da26185445967315535] 
\draw    (480.57,165.38) -- (441.22,250.03) ;
%Curve Lines [id:da7928497690636905] 
\draw    (487.6,154.65) .. controls (529.62,96.56) and (560.61,190.99) .. (491.93,168.35) ;
%Shape: Ellipse [id:dp5472298127523607] 
\draw  [fill={rgb, 255:red, 255; green, 255; blue, 255 }  ,fill opacity=1 ] (468.83,165.38) .. controls (468.83,158.89) and (474.08,153.64) .. (480.57,153.64) .. controls (487.05,153.64) and (492.31,158.89) .. (492.31,165.38) .. controls (492.31,171.86) and (487.05,177.12) .. (480.57,177.12) .. controls (474.08,177.12) and (468.83,171.86) .. (468.83,165.38) -- cycle ;
%Straight Lines [id:da9275037610543019] 
\draw    (491.23,271.46) -- (440.85,250.17) ;
%Straight Lines [id:da17246183653719316] 
\draw    (441.22,250.03) -- (390.84,228.74) ;
%Shape: Ellipse [id:dp6439493847369149] 
\draw  [fill={rgb, 255:red, 255; green, 255; blue, 255 }  ,fill opacity=1 ] (379.26,228.74) .. controls (379.26,222.34) and (384.45,217.16) .. (390.84,217.16) .. controls (397.24,217.16) and (402.43,222.34) .. (402.43,228.74) .. controls (402.43,235.14) and (397.24,240.33) .. (390.84,240.33) .. controls (384.45,240.33) and (379.26,235.14) .. (379.26,228.74) -- cycle ;
%Shape: Ellipse [id:dp15675917651922633] 
\draw  [fill={rgb, 255:red, 255; green, 255; blue, 255 }  ,fill opacity=1 ] (429.27,250.17) .. controls (429.27,243.78) and (434.45,238.59) .. (440.85,238.59) .. controls (447.25,238.59) and (452.44,243.78) .. (452.44,250.17) .. controls (452.44,256.57) and (447.25,261.76) .. (440.85,261.76) .. controls (434.45,261.76) and (429.27,256.57) .. (429.27,250.17) -- cycle ;
%Shape: Boxed Bezier Curve [id:dp7358717336524252] 
\draw    (447.15,260.8) .. controls (481.33,322.31) and (384.47,310.61) .. (433.17,258.95) ;
%Curve Lines [id:da34900819714117004] 
\draw    (459.98,95.54) .. controls (501.23,38.51) and (531.65,131.21) .. (464.23,108.98) ;
%Shape: Ellipse [id:dp8284371495354373] 
\draw  [color={rgb, 255:red, 4; green, 146; blue, 194 }  ,draw opacity=1 ][dash pattern={on 4.5pt off 4.5pt}] (362.3,89.16) .. controls (362.3,57.92) and (387.62,32.59) .. (418.87,32.59) .. controls (450.11,32.59) and (475.44,57.92) .. (475.44,89.16) .. controls (475.44,120.41) and (450.11,145.73) .. (418.87,145.73) .. controls (387.62,145.73) and (362.3,120.41) .. (362.3,89.16) -- cycle ;
%Shape: Ellipse [id:dp6470131354645342] 
\draw  [color={rgb, 255:red, 237; green, 130; blue, 14 }  ,draw opacity=1 ][dash pattern={on 4.5pt off 4.5pt}] (455.36,165.38) .. controls (455.36,151.45) and (466.64,140.17) .. (480.57,140.17) .. controls (494.49,140.17) and (505.78,151.45) .. (505.78,165.38) .. controls (505.78,179.3) and (494.49,190.59) .. (480.57,190.59) .. controls (466.64,190.59) and (455.36,179.3) .. (455.36,165.38) -- cycle ;
%Shape: Circle [id:dp2882731415599845] 
\draw  [color={rgb, 255:red, 140; green, 219; blue, 54 }  ,draw opacity=1 ][dash pattern={on 4.5pt off 4.5pt}] (368.91,232.75) .. controls (368.91,204.39) and (391.9,181.39) .. (420.27,181.39) .. controls (448.64,181.39) and (471.63,204.39) .. (471.63,232.75) .. controls (471.63,261.12) and (448.64,284.12) .. (420.27,284.12) .. controls (391.9,284.12) and (368.91,261.12) .. (368.91,232.75) -- cycle ;
%Straight Lines [id:da6792600183091927] 
\draw    (79.07,214.04) -- (69.3,235.44) ;
%Straight Lines [id:da7033421088998842] 
\draw    (79.72,240.37) -- (69.13,235.51) ;
%Straight Lines [id:da9885858772107592] 
\draw    (69.3,235.44) -- (46.16,225.67) ;
%Shape: Ellipse [id:dp2330553093227623] 
\draw  [fill={rgb, 255:red, 255; green, 255; blue, 255 }  ,fill opacity=1 ] (40.84,225.67) .. controls (40.84,222.73) and (43.22,220.35) .. (46.16,220.35) .. controls (49.1,220.35) and (51.48,222.73) .. (51.48,225.67) .. controls (51.48,228.61) and (49.1,230.99) .. (46.16,230.99) .. controls (43.22,230.99) and (40.84,228.61) .. (40.84,225.67) -- cycle ;
%Shape: Ellipse [id:dp8591699508470692] 
\draw  [color={rgb, 255:red, 140; green, 219; blue, 54 }  ,draw opacity=1 ][dash pattern={on 4.5pt off 4.5pt}] (36.09,227.51) .. controls (36.09,214.48) and (46.65,203.92) .. (59.67,203.92) .. controls (72.7,203.92) and (83.26,214.48) .. (83.26,227.51) .. controls (83.26,240.54) and (72.7,251.1) .. (59.67,251.1) .. controls (46.65,251.1) and (36.09,240.54) .. (36.09,227.51) -- cycle ;
%Straight Lines [id:da4787390203115717] 
\draw    (69.13,235.51) -- (55.35,250.77) ;
%Straight Lines [id:da5055444695966994] 
\draw    (69.3,235.44) -- (74.85,245.89) ;
%Shape: Ellipse [id:dp14002901248011645] 
\draw  [fill={rgb, 255:red, 255; green, 255; blue, 255 }  ,fill opacity=1 ] (63.81,235.51) .. controls (63.81,232.57) and (66.19,230.19) .. (69.13,230.19) .. controls (72.07,230.19) and (74.45,232.57) .. (74.45,235.51) .. controls (74.45,238.45) and (72.07,240.83) .. (69.13,240.83) .. controls (66.19,240.83) and (63.81,238.45) .. (63.81,235.51) -- cycle ;

%Straight Lines [id:da32349549764022567] 
\draw    (259.5,179.74) -- (318.5,179.74) ;
\draw [shift={(320.5,179.74)}, rotate = 180] [color={rgb, 255:red, 0; green, 0; blue, 0 }  ][line width=0.75]    (10.93,-3.29) .. controls (6.95,-1.4) and (3.31,-0.3) .. (0,0) .. controls (3.31,0.3) and (6.95,1.4) .. (10.93,3.29)   ;

% Text Node
\draw (109.45,118.88) node [anchor=north west][inner sep=0.75pt]   [align=left] {$\displaystyle v$};
% Text Node
\draw (137.57,182.75) node [anchor=north west][inner sep=0.75pt]   [align=left] {$\displaystyle w$};
% Text Node
\draw (110.1,239.69) node [anchor=north west][inner sep=0.75pt]   [align=left] {$\displaystyle u$};
% Text Node
\draw (104.99,24.86) node [anchor=north west][inner sep=0.75pt]   [align=left] {$\displaystyle H_{v}$};
% Text Node
\draw (219.29,134.89) node [anchor=north west][inner sep=0.75pt]   [align=left] {$\displaystyle H_{w}$};
% Text Node
\draw (22,184.83) node [anchor=north west][inner sep=0.75pt]   [align=left] {$\displaystyle H_{u}$};

\end{tikzpicture}
\]
\caption{An example of a graphical map}\label{graph map}
\end{figure}
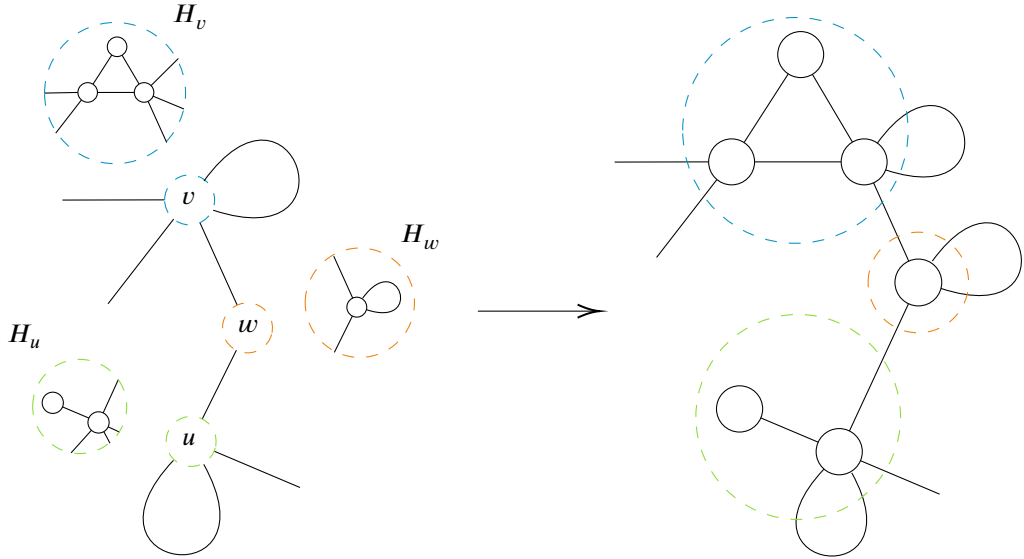

The category of \emph{modular dendroidal objects} in $\bE$, denoted $\dMod(\bE)$, is the category of contravariant functors $X:\bM^{\mathrm{op}}\longrightarrow\bE$ and natural transformations. We write $X_G$ for the value of $X\in\dMod(\bE)$ on a graph $G\in\bM$, and we write
\[
\bM[G]:=\Hom_{\bM}(-,G)
\]
for the representable functor at $G$.

\begin{definition}\label{defn: Segal core}
\cite[Section~3.1]{hry1}
Let $G$ be a graph with at least one vertex. The \emph{Segal core} of $G$ is the coequaliser in $\dMod(\Set)$
\[
\begin{tikzcd}
\displaystyle\coprod_{e\in iE_G}\bM[\updownarrow]
\arrow[r, shift left=1]
\arrow[r, shift right=1]
&
\displaystyle\coprod_{v\in V_G}\bM[\corolla_v]
\arrow[r]
&
\Sc[G].
\end{tikzcd}
\]
The two arrows record the two incidences of each internal edge with the corollas at its endpoints. The Segal core is a subobject of the representable $\bM[G]$, with inclusion $\Sc[G]\longrightarrow\bM[G]$ induced by the inclusions of the local vertex data $\bM[\corolla_v]\longrightarrow\bM[G]$.
\end{definition}

The Segal core detects whether the value of a modular dendroidal object
on a graph is determined, up to homotopy, by its values on the corollas
at the vertices. When $\bE$ is a cofibrantly generated Cartesian
monoidal model category, the category $\dMod(\bE)$ admits the
projective model structure, in which a map $X\to Y$ is a weak
equivalence, respectively a fibration, if and only if
$X_G\to Y_G$ is a weak equivalence, respectively a fibration,
in $\bE$ for every $G\in\bM$.
See \cite[Theorem~11.6.1]{hirsch}.

A modular dendroidal object $X$ is called \emph{Segal} if, for every
$G\in\bM$, the map induced by the Segal core inclusion
\[
\mathbb{R}\Map(\bM[G],X)
\longrightarrow
\mathbb{R}\Map(\Sc[G],X)
\]
is a weak equivalence of spaces.\footnote{Here $\bM[G]$ and $\Sc[G]$
are regarded as discrete objects of $\dMod(\bE)$ via the functor
$\Set\to\bE$.} A modular dendroidal object $X$ is called
\emph{reduced} if $X_{\updownarrow}=*$.

\begin{definition}\label{def: infty modular}
A \emph{modular $\infty$-operad} in $\bE$ is a modular dendroidal object
$X:\bM^{\mathrm{op}}\to\bE$ satisfying:
\begin{itemize}
    \item $X_{\updownarrow}=*$;
    \item for every graph $G\in\bM$, the Segal map
    \[
    \mathbb{R}\Map(\bM[G],X)
    \longrightarrow
    \mathbb{R}\Map(\Sc[G],X)
    \]
    is a weak equivalence of spaces.
\end{itemize}
\end{definition}

We write $\Mod_\infty(\bE)$ for the $\infty$-category of modular $\infty$-operads in $\bE$. The model-categorical presentation of this $\infty$-category is recalled in Appendix~\ref{sec: nerve is homotopically fully faithful}.

The condition $X_{\updownarrow}=*$ means that we are considering
one-coloured modular $\infty$-operads. It also simplifies the Segal
condition. Indeed, for a reduced object $X$, the edge-compatibility
conditions in the Segal core are trivial, since each internal edge
contributes a copy of $X_{\updownarrow}=*$. Thus, after replacing $X$
by an entrywise fibrant object if necessary, the Segal map may be
identified with the comparison
\[
X_G\longrightarrow
\prod_{v\in V_G}(X_{\corolla_v})_f.
\]
Here $(X_{\corolla_v})_f$ denotes a fibrant replacement in $\bE$, so
that the product computes the corresponding homotopy product.
Equivalently, a reduced modular dendroidal object is a modular
$\infty$-operad precisely when, for every graph $G$, the value $X_G$ is
weakly equivalent to the homotopy product of its values on the corollas
at the vertices of $G$.

In the cases used in this paper we do not need to worry much about
fibrant replacements. If $\bE=\Grpd$, every object is fibrant, so the
homotopy product is computed by the ordinary product. If $\bE=\sSet$
and $X$ is entrywise fibrant, for example if $X$ is fibrant in the
localised Reedy model structure, the ordinary product again computes
the homotopy product. In this situation the Segal condition can be
written simply as requiring
\[
X_G\longrightarrow
\prod_{v\in V_G}X_{\corolla_v}
\]
to be a weak equivalence.

\subsection{The modular dendroidal nerve}

Let $\Mod_{\infty}(\bE)$ denote the $\infty$-category of modular
$\infty$-operads, as in Definition~\ref{def: infinity cat of infinity
modular operads}. A strict modular operad assigns operations to
genus-graded corollas. Viewing a general genus-graded graph as a
pattern for composing such operations, one obtains a modular
dendroidal object by decorating each vertex with an operation of the
corresponding genus and arity. This construction defines the modular
dendroidal \emph{nerve} functor
\[
\nerve:\Mod(\bE)\longrightarrow\Mod_{\infty}(\bE).
\]

\begin{definition}\label{def: modular nerve}
The \emph{nerve} of a modular operad $\calP$ is the contravariant functor $\nerve\calP:\bM^{\mathrm{op}}\to\bE$ defined on a graph $G=(G,\lambda,\ell,\epsilon)$ by
\[
(\nerve\calP)_G
:=
\prod_{v\in V_G}
\calP\bigl(\epsilon(v),|\nb(v)|\bigr).
\]
In particular, if $\corolla_v$ is the corolla at a vertex $v$, then
\[
(\nerve\calP)_{\corolla_v} = \calP\bigl(\epsilon(v),|\nb(v)|\bigr).
\]
\end{definition}

For a strict modular operad $\calP$, the Segal map is an isomorphism rather than just a weak equivalence. Indeed, for every graph $G$, the definition of the nerve identifies
\[
(\nerve\calP)_G
\cong
\prod_{v\in V_G}(\nerve\calP)_{\corolla_v}.
\]
Thus the nerve of a strict modular operad satisfies the Segal condition strictly. In particular, when ordinary finite products compute finite homotopy products in $\bE$, the object $\nerve\calP$ is a modular $\infty$-operad. This applies in the cases used in this paper, for example when $\bE=\Grpd$, and for entrywise fibrant objects in $\sSet$.   In Appendix~\ref{sec: nerve is homotopically fully faithful}, we prove that the modular dendroidal nerve is homotopically fully faithful; see Theorem~\ref{thm: nerve is homotopically fully faithful}.

\subsection{Truncations of modular dendroidal objects}

We now introduce genus truncations of modular dendroidal
objects. Let $\bM_{\leq g}\subseteq\bM$ denote the full
subcategory spanned by genus-graded graphs of total genus
at most $g$.

\begin{lemma}
The full subcategory $\bM_{\leq g}$ is a sieve on $\bM$.
That is, if $\phi:G\to H$ is a morphism in $\bM$ and
$H\in\bM_{\leq g}$, then $G\in\bM_{\leq g}$.
\end{lemma}

\begin{proof}
If $G$ is the exceptional edge, the assertion is immediate.
Otherwise, the active--embedding factorisation
of~\cite[Proposition~2.14]{hry1} expresses $\phi$ as
\[
G\longrightarrow G\{K_v\}_{v\in V_G}
\longrightarrow H,
\]
where the first map is given by graph substitution and
the second is a graphical embedding. Each $K_v$ has total
genus $\epsilon_G(v)$, so graph substitution preserves
total genus. Moreover, a graphical embedding preserves
vertex genera and cannot decrease the first Betti number.
Consequently,
\[
g(G)
=
g\bigl(G\{K_v\}_{v\in V_G}\bigr)
\leq g(H).
\]
Thus $H\in\bM_{\leq g}$ implies $G\in\bM_{\leq g}$.
\end{proof}

We write $\dMod_{\leq g}(\bE) =\operatorname{Fun}(\bM_{\leq g}^{\op},\bE)$ for the category of genus-$g$ truncated modular dendroidal objects in $\bE$. The inclusion $i_g:\bM_{\leq g}\hookrightarrow\bM$ induces a restriction functor
\[
i_g^*:\dMod(\bE)\longrightarrow\dMod_{\leq g}(\bE).
\]
We call $i_g^*X$ the genus-$g$ truncation of $X$ and write it as $\trun{g}X$.
\section{From profinite groupoids to profinite spaces}\label{sec: homotopy action} 
Theorem~\ref{main theorem NS action} constructs an action of the monoid $\underline{\NS}$ on the modular operad $\widehat{\bS}$ in profinite groupoids. In this section, we pass from this strict action to homotopy automorphisms and construct a faithful $\NS$-action on the associated modular $\infty$-operad $\widehat{\cs}\nerve\widehat{\bS}$ in profinite spaces.

We also show that $\widehat{\cs\nerve\bS}$ is a modular $\infty$-operad whose corolla values model the étale homotopy types of the tangent-vector moduli stacks. Comparing these two constructions identifies their genus-one truncations and gives the operadic two-level principle.

\subsection{Profinite completion of modular $\infty$-operads}\label{sec: profinite modular infty operads}
We next explain why profinite completion of modular dendroidal objects may be computed entrywise. Let $\mathscr K$ be a small category with finite hom-sets and let $\bC$ be a small category with finite limits. A functor $A:\mathscr K^{\op}\to\bC$ is \emph{finitely coskeletal} if it is the right Kan extension of its restriction to a finite full subcategory of $\mathscr K^{\op}$. We write $\operatorname{Fun}_{\mathrm{fcosk}}(\mathscr K^{\op},\bC)$ for the full subcategory
of such functors. The dual of \cite[Theorem~2.3]{BM20} gives an equivalence
\[
\Pro\bigl(\operatorname{Fun}_{\mathrm{fcosk}} (\mathscr K^{\op},\bC)\bigr) \simeq \operatorname{Fun}(\mathscr K^{\op},\Pro(\bC)).
\]

The modular dendroidal category $\bM$ is small and has finite hom-sets, and hence is the filtered union of its finite full subcategories. For each finite full subcategory $\mathscr D\subseteq\bM$, restriction to $\mathscr D^{\op}$ has a right adjoint $\operatorname{cosk}_{\mathscr D}$. Every modular dendroidal object $X$ is canonically the limit of the finitely coskeletal diagrams $\operatorname{cosk}_{\mathscr D} \bigl(X|_{\mathscr D^{\op}}\bigr),$ as $\mathscr D$ ranges over the finite full subcategories of $\bM$. Indeed, after evaluation at a graph $G$, this system is constant with value $X_G$ once $\mathscr D$ contains $G$. It follows that pro-$\bC$ completion of modular dendroidal objects is computed entrywise.

We apply this with $\bE=\sSet$ or $\Grpd$.
In the groupoid case, $\bC$ is the category of finite
groupoids. In the spatial case, $\bC$ is the category
of lean simplicial sets. These are degreewise finite
simplicial sets that are $n$-coskeletal for some $n\geq0$,
as in \cite[Section~2]{BM20}.
In this case, the resulting model structure on
$\Pro(\bC)$ agrees with Quick's model structure on
profinite spaces by
\cite[Example~5.5 and Corollary~6.6]{BM20}.

\begin{definition}\label{def: profinite modular operad}
Let $X:\bM^{\op}\to\bE$ be a modular dendroidal object.
Its \emph{profinite completion} is the modular dendroidal
object
\[
\widehat X:\bM^{\op}\longrightarrow\Pro(\bC),
\qquad
(\widehat X)_G=\widehat{X_G},
\]
obtained by applying pro-$\bC$ completion entrywise.
\end{definition}

Profinite completion does not in general preserve finite products. Consequently, even if $X$ satisfies the Segal condition, its entrywise profinite completion $\widehat X$ need not do so. For groupoids, however, Proposition~\ref{prop: profinite completion preserves products of groupoids}
shows that completion preserves finite products under the finiteness hypothesis used in Section~\ref{sec: profinite completion of modular operads in groupoids}. The following lemma applies this comparison to the Segal condition.

We equip $\widehat{\Grpd}$ with the model structure whose weak equivalences are those of
\cite[Definition~4.8]{Horel_profinite_groupoids}. We write
$\Mod_\infty(\widehat{\Grpd})$ for the $\infty$-category presented by
the relative category of modular $\infty$-operads
$X:\bM^{\op}\to\widehat{\Grpd}$. A morphism $f:X\to Y$ is a weak
equivalence in this relative category if
$f_G:X_G\to Y_G$ is a weak equivalence in $\widehat{\Grpd}$ for every
$G\in\bM$.

\begin{lemma}\label{lemma: profinite-completion-preserves-segal-groupoids}
Let $X:\bM^{\op}\to\Grpd$ be a modular $\infty$-operad such that $X_C$
has finitely many objects for every corolla $C\in\bM$. Then entrywise
profinite completion defines an object
\[
\widehat X\in\Mod_\infty(\widehat{\Grpd}).
\]
\end{lemma}

\begin{proof}
The completion of the trivial groupoid is the trivial profinite
groupoid, so $\widehat X_{\updownarrow}=*$. For every graph $G$, the
Segal map is an equivalence of groupoids
\[
X_G\longrightarrow\prod_{v\in V_G}X_{C_v}.
\]
Profinite completion preserves equivalences of groupoids, since it is
a left Quillen functor and every groupoid is cofibrant. Applying
completion to the Segal map and then using
Proposition~\ref{prop: profinite completion preserves products of groupoids}
gives
\[
\widehat X_G
\xrightarrow{\sim}
\left(\prod_{v\in V_G}X_{C_v}\right)^{\wedge}
\xrightarrow{\cong}
\prod_{v\in V_G}\widehat X_{C_v}.
\]
The product comparison applies because $G$ has finitely many vertices
and each $X_{C_v}$ has finitely many objects. Thus the Segal map for
$\widehat X$ is a weak equivalence in $\widehat{\Grpd}$.
\end{proof}

If $\mathcal P$ is an entrywise fibrant modular operad in profinite groupoids, its nerve satisfies the Segal condition strictly and therefore determines an object $\nerve\mathcal P\in\Mod_\infty(\widehat{\Grpd})$. This applies to the profinite modular operads considered below.

For a modular operad in groupoids satisfying the finiteness hypothesis of Definition~\ref{def: profinite completion of modular operads in groupoids},
the following lemma identifies the completion of its nerve with the nerve of its completion.

\begin{lemma}\label{lemma: N commutes with completion for groupoids}
Let $\calP$ be a modular operad in groupoids such that
$\calP(g,n+1)$ has finitely many objects for all $g\geq0$ and
$n\geq0$. Then there is a natural isomorphism of modular
$\infty$-operads
\[
\widehat{\nerve\calP}\cong\nerve\widehat{\calP}.
\]
\end{lemma}

\begin{proof}
For every graph $G\in\bM$, the definitions of the nerve and entrywise
profinite completion give
\[
(\widehat{\nerve\calP})_G
=
\left(
\prod_{v\in V_G}
\calP\bigl(\epsilon(v),|\nb(v)|\bigr)
\right)^{\wedge}.
\]
If $V_G\neq\varnothing$, iterating
Proposition~\ref{prop: profinite completion preserves products of groupoids}
identifies this naturally with
\[
\prod_{v\in V_G}
\widehat{\calP}\bigl(\epsilon(v),|\nb(v)|\bigr)
=
(\nerve\widehat{\calP})_G.
\]
The product comparison applies because $G$ has finitely many vertices
and each $\calP(g,n+1)$ has finitely many objects. If
$V_G=\varnothing$, both sides are terminal.

The product comparison isomorphisms are natural and compatible with
finite products. Since the modular-operad structure on
$\widehat{\calP}$ is defined using these isomorphisms, the resulting
componentwise isomorphisms commute with the structure maps associated
to morphisms in $\bM$. They therefore assemble, naturally in $\calP$,
to an isomorphism of modular dendroidal objects
\[
\widehat{\nerve\calP}\cong\nerve\widehat{\calP}.
\]

The nerve $\nerve\calP$ is a modular $\infty$-operad in groupoids, and
its corolla values have finitely many objects. Hence
Lemma~\ref{lemma: profinite-completion-preserves-segal-groupoids}
shows that $\widehat{\nerve\calP}$ is an object of
$\Mod_\infty(\widehat{\Grpd})$. The isomorphism above shows the same
for $\nerve\widehat{\calP}$ and is therefore an isomorphism of modular
$\infty$-operads.
\end{proof}

\subsection{From endomorphisms to homotopy endomorphisms}
\label{sec: homotopy endomorphisms}

To compare strict endomorphisms with derived mapping spaces, we first construct a cofibrant replacement of $\bS$.

\subsubsection{A cofibrant replacement of $\bS$}

Recall that the modular operad of objects of $\bS$ is the free
modular operad $\Upsilon$ generated by the $\Sigma_2^+$-set
$\Sigma_2^+/A_2^+$
(Definition~\ref{def:free modular operad generated by mu}).
We resolve this generating orbit by the regular $\Sigma_2^+$-set.

\begin{definition}
Let $\widetilde{\Upsilon}$ be the genus-graded sequence
concentrated in $(0,3)$, with
$\widetilde{\Upsilon}(0,3)=\Sigma_2^+$ carrying its regular
right action. As for $\Upsilon$, we use the same notation
for the free modular operad generated by this sequence.
\end{definition}

An object of $\widetilde{\Upsilon}$ may be represented by a
connected trivalent graph of the kind representing an object
of $\bS$, together with, at every vertex $v$, a choice of
representative of its cyclic labelling
\[
\ell_v:\{0,1,2\}\xrightarrow{\cong}\nb(v).
\]
Equivalently, one retains the cyclic ordering at each vertex
and additionally marks one of its three incident half-edges;
the marked half-edge is assigned the label $0$, and the cyclic
ordering determines the labels $1$ and $2$.
The quotient
$\Sigma_2^+\twoheadrightarrow\Sigma_2^+/A_2^+$
induces a morphism of modular operads in sets, which we denote by
\[
\ob(q):\widetilde{\Upsilon}\longrightarrow\Upsilon=\ob(\bS).
\]

\begin{definition}\label{def: cofibrant modular operad S}
For each $g,n\geq0$, define $\bS_{\mathrm{cof}}(g,n+1)$
to be the groupoid whose set of objects is
$\widetilde{\Upsilon}(g,n+1)$ and whose morphism sets are
\[
\Hom_{\bS_{\mathrm{cof}}(g,n+1)}(G,H)
:=
\Hom_{\bS(g,n+1)}
\bigl(\ob(q)(G),\ob(q)(H)\bigr).
\]
Identities and categorical composition are inherited from
$\bS(g,n+1)$.

The symmetric-group actions, partial compositions, and
contractions are those of $\widetilde{\Upsilon}$ on objects
and those of $\bS$ on morphisms. These are well defined
because $\ob(q)$ is a morphism of modular operads in sets.
Thus $\bS_{\mathrm{cof}}$ is a modular operad in groupoids.
\end{definition}

There is a morphism of modular operads
$q:\bS_{\mathrm{cof}}\to\bS$ with object map $\ob(q)$,
given on morphisms by the identity under the defining
identification of morphism sets.
Thus $\bS_{\mathrm{cof}}$ has the same mapping-class morphisms
as $\bS$, but separates the different lifts of the cyclic
ordering at each vertex.

\begin{lemma}\label{lemma: S cof is cofibrant}
The modular operad $\bS_{\mathrm{cof}}$ is cofibrant in
$\Mod(\Grpd)$.
\end{lemma}

\begin{proof}
The modular operad
$\ob(\bS_{\mathrm{cof}})=\widetilde{\Upsilon}$
is free on the regular $\Sigma_2^+$-set
$\widetilde{\Upsilon}(0,3)=\Sigma_2^+$.
Hence $\bS_{\mathrm{cof}}$ is cofibrant by
Corollary~\ref{cor: cofibrant modular operad}.
\end{proof}

\begin{prop}\label{prop: S cofibrant replacement}
The morphism $q:\bS_{\mathrm{cof}}\longrightarrow\bS$
is an entrywise trivial fibration. Consequently, it is a
cofibrant replacement of $\bS$ in $\Mod(\Grpd)$.
\end{prop}

\begin{proof}
For every $g,n\geq0$, the functor $q(g,n+1)$ induces the
identity under the defining identification of morphism sets,
and is therefore fully faithful.

The functor $q(g,n+1)$ is also surjective on objects.
Apart from the exceptional edge representing the unit,
an object of $\bS(g,n+1)$ is represented by a connected
trivalent graph with its vertex and boundary labellings
and, at each vertex $v$, an $A_2^+$-orbit of bijections
$\{0,1,2\}\xrightarrow{\cong}\nb(v)$.
Choosing a representative of each orbit gives an object of
$\bS_{\mathrm{cof}}(g,n+1)$ mapping to the given object.
The exceptional edge lifts to itself.

Thus $q(g,n+1)$ is fully faithful and surjective on objects,
and hence is a trivial fibration of groupoids.
Since fibrations and weak equivalences in $\Mod(\Grpd)$
are defined entrywise
(Theorem~\ref{model structure on Mod(E)}),
$q$ is a trivial fibration of modular operads.
Its source is cofibrant by
Lemma~\ref{lemma: S cof is cofibrant}, so $q$ is a
cofibrant replacement of $\bS$.
\end{proof}

\begin{cor}\label{cor: truncated S cofibrant replacements}
The morphism
\[
\trun{0}q:\trun{0}\bS_{\mathrm{cof}}
\longrightarrow\trun{0}\bS
\]
is a cofibrant replacement in $\Mod_{\leq0}(\Grpd)$.
Under the identification
$\trun{0}\bS\cong\PaRB^{\mathrm{cyc}}$,
this gives a cofibrant replacement of $\PaRB^{\mathrm{cyc}}$.
\end{cor}

\begin{proof}
Since partial compositions and contractions do not decrease
genus, $\trun{0}\widetilde{\Upsilon}$ is the free cyclic operad
generated by the regular $\Sigma_2^+$-set
$\widetilde{\Upsilon}(0,3)=\Sigma_2^+$.
The argument of Lemma~\ref{lemma: S cof is cofibrant}
therefore shows that $\trun{0}\bS_{\mathrm{cof}}$ is cofibrant.
Moreover, $\trun{0}q$ is an entrywise trivial fibration by
Proposition~\ref{prop: S cofibrant replacement}.
The final assertion follows from
Proposition~\ref{prop:PaRB-cyclic-is-S0}.
\end{proof}

\begin{comment}
\begin{cor}
The morphism
\[
\trun{1}q:\trun{1}\bS_{\mathrm{cof}}
\longrightarrow\trun{1}\bS
\]
is a cofibrant replacement in $\Mod_{\leq1}(\Grpd)$.
\end{cor}

\begin{proof}
Since partial compositions and contractions do not decrease
genus, $\trun{1}\widetilde{\Upsilon}$ is the free genus-one
truncated modular operad generated by the regular
$\Sigma_2^+$-set
$\widetilde{\Upsilon}(0,3)=\Sigma_2^+$.
The argument of Lemma~\ref{lemma: S cof is cofibrant}
therefore shows that $\trun{1}\bS_{\mathrm{cof}}$ is cofibrant.
The morphism $\trun{1}q$ is an entrywise trivial fibration by
Proposition~\ref{prop: S cofibrant replacement}.
\end{proof}
\end{comment}

\subsubsection{Homotopy automorphisms}
We use an elementary notion of homotopy for modular operads in groupoids and profinite groupoids. Let $\mathscr I$ be the finite groupoid with objects $\{0,1\}$ and a unique morphism between any two objects, regarded as a profinite groupoid when necessary. For a modular operad $\calQ$ in groupoids or profinite groupoids, let $\calQ^{\mathscr I}$ denote the modular operad defined entrywise by
\[
(\calQ^{\mathscr I})(g,n+1)=\calQ(g,n+1)^{\mathscr I}.
\]
The modular-operad operations on $\calQ^{\mathscr I}$ are defined pointwise. Evaluation at the two objects of $\mathscr I$ gives morphisms $d_0,d_1:\calQ^{\mathscr I}\to\calQ$.

Given morphisms $f,g:\calP\to\calQ$, a \emph{homotopy} from $f$ to $g$ is a morphism of modular operads $H:\calP\to\calQ^{\mathscr I}$ satisfying $d_0H=f$ and $d_1H=g$. Equivalently, a homotopy is a natural isomorphism from $f$ to $g$ compatible with the modular-operad structure, continuous in the
profinite case.

This defines an equivalence relation on morphisms and is compatible with composition. We write $[\calP,\calQ]$ for the resulting set of homotopy classes. In particular, $[\calP,\calP]$ is a monoid, and we write $\HoAut(\calP)$ for its group of units. We write $\Aut_0(\calP)$ for the group of strict automorphisms of $\calP$ that fix every object.

\begin{prop}\label{prop:injection}
    The map 
        \[\Aut_0(\widehat{\bS})\to \HoAut(\widehat{\bS})\]
    is injective.
\end{prop}

Before we prove Proposition~\ref{prop:injection} we require the following Lemma about $\widehat\Gamma_{1,1}\cong \widehat{\Gamma(S_{\contraction})}$. We use as a reference Figure~\ref{fig:S-move}. \begin{lemma}\label{lemma:centralisers in the profinite completion of Gamma11} Let $a$ denote the curve in $S_{\contraction}$ depicted in Figure~\ref{fig:S-move}. Then
\[C_{\widehat{\Gamma}_{1,1}}(D_a) = \overline{C_{\Gamma_{1,1}}(D_a)} = \overline{\langle D_\partial,D_a,(D_aD_bD_a)^2\rangle}.
\]
\end{lemma}

\begin{proof}
Capping the boundary by gluing a punctured disc induces a short exact sequence
\[
1\to\langle D_\partial\rangle \to\Gamma_{1,1} \xrightarrow{\pi}\Gamma_1^1 \to1,
\]
where $\Gamma_1^1$ denotes the mapping class group of the
one-punctured torus.

Since profinite completion is right exact, this gives an exact sequence
\[
\widehat{\langle D_\partial\rangle} \to\widehat\Gamma_{1,1} \xrightarrow{\pi}\widehat\Gamma_1^1 \to1.
\]
Since $D_\partial$ is central, restricting to the normaliser $N_{\widehat\Gamma_{1,1}}(\overline{\langle D_a\rangle})$ gives the exact sequence
\begin{equation}\label{eq:exact-sequence-for-injectivity}
\widehat{\langle D_\partial\rangle} \to N_{\widehat\Gamma_{1,1}}(\overline{\langle D_a\rangle}) \xrightarrow{\pi} \pi\bigl( N_{\widehat\Gamma_{1,1}}(\overline{\langle D_a\rangle}) \bigr) \to1.
\end{equation}

We have
\begin{equation}\label{eq:injectivity-proof-1}
\pi\bigl( N_{\widehat\Gamma_{1,1}}(\overline{\langle D_a\rangle}) \bigr) \subseteq N_{\widehat\Gamma_1^1} (\overline{\langle\pi(D_a)\rangle}),
\end{equation} and we will show that this inclusion is an equality. Since $\Gamma_1^1\cong SL(2,\Z)$ is finitely generated and virtually free, \cite[Corollary~2.9]{Ribes-Pavel} gives
\[
N_{\widehat\Gamma_1^1} (\overline{\langle\pi(D_a)\rangle}) = \overline{N_{\Gamma_1^1}(\langle\pi(D_a)\rangle)}.
\]
We now describe the normaliser of $\langle\pi(D_a)\rangle$ in $\Gamma_1^1$. Note that $\pi(D_a)$ is the Dehn twist in the once-punctured torus along the curve corresponding to $a$, so we will abuse notation and denote it also by $D_a$.

If $f$ normalises $\langle D_a\rangle$, then $fD_af^{-1}=D_a^{\pm1}$. By Lemma~\ref{lemma: conjugation by mapping classes}, $fD_af^{-1}=D_{f(a)}$. Since orientation-preserving mapping classes conjugate positive Dehn twists to positive Dehn twists, the negative exponent is excluded. Thus $f$ fixes
the isotopy class of $a$ and centralises $D_a$. Consequently,
\[
N_{\Gamma_1^1}(\langle\pi(D_a)\rangle) = C_{\Gamma_1^1}(D_a).
\]
This centraliser consists precisely of the mapping classes that fix the isotopy class of $a$. It is the subgroup generated by $D_a$ and $(D_aD_bD_a)^2$, as can be seen by cutting along $a$ and allowing mapping classes to permute the two resulting boundary components corresponding to $a$.

Substituting this description into \eqref{eq:injectivity-proof-1} gives
\[
\pi\bigl( N_{\widehat\Gamma_{1,1}}(\overline{\langle D_a\rangle}) \bigr) \subseteq \overline{\langle D_a,(D_aD_bD_a)^2\rangle}.
\]
Moreover, both $D_a$ and $(D_aD_bD_a)^2$ in $\Gamma_1^1$ are images of the corresponding elements of $\Gamma_{1,1}$. These lifts commute with $D_a$ and hence lie in $N_{\widehat\Gamma_{1,1}}(\overline{\langle D_a\rangle})$.
Therefore
\[
\langle D_a,(D_aD_bD_a)^2\rangle \subseteq \pi\bigl( N_{\widehat\Gamma_{1,1}}(\overline{\langle D_a\rangle}) \bigr).
\]
Since normalisers of closed subgroups are closed in profinite groups, $N_{\widehat\Gamma_{1,1}}(\overline{\langle D_a\rangle})$ is compact. Its image under $\pi$ is therefore closed,
and we obtain
\[
\overline{\langle D_a,(D_aD_bD_a)^2\rangle} \subseteq \pi\bigl(N_{\widehat\Gamma_{1,1}}(\overline{\langle D_a\rangle})\bigr).
\]
Together with the preceding inclusion, this proves equality.

Combining this with
\eqref{eq:exact-sequence-for-injectivity} gives the exact
sequence
\[\widehat{\langle D_\partial\rangle} \to N_{\widehat\Gamma_{1,1}}(\overline{\langle D_a\rangle}) \xrightarrow{\pi} \overline{\langle D_a,(D_aD_bD_a)^2\rangle} \to1.
\]
It follows that
$N_{\widehat\Gamma_{1,1}}(\overline{\langle D_a\rangle})$
is topologically generated by
\[
D_\partial,\qquad D_a,\qquad (D_aD_bD_a)^2.
\]
Indeed, the closed subgroup generated by these elements contains the kernel of $\pi$ and maps onto the quotient
in this exact sequence, so it equals the normaliser.

All three generators commute with $D_a$ in $\Gamma_{1,1}$. Since the profinite centraliser is closed, the normaliser is therefore contained in $C_{\widehat{\Gamma}_{1,1}}(D_a)$. The reverse inclusion holds because every element centralising $D_a$ normalises $\overline{\langle D_a\rangle}$. Moreover, the three generators belong to $C_{\Gamma_{1,1}}(D_a)$, whose closure is contained in the profinite centraliser. Hence
\[
C_{\widehat{\Gamma}_{1,1}}(D_a) = \overline{C_{\Gamma_{1,1}}(D_a)} = \overline{\langle D_\partial,D_a,(D_aD_bD_a)^2\rangle}.
\]
\end{proof}

We are now ready to return to Proposition~\ref{prop:injection}:

\begin{proof}[Proof of Proposition~\ref{prop:injection}]
It suffices to show that an object-fixing automorphism homotopic to the identity is the identity.  Suppose, therefore, that $F:\widehat{\bS}\to\widehat{\bS}$ is an object-fixing automorphism and that $H:F\Rightarrow\id_{\widehat{\bS}}$ is a homotopy. We first pass to genus zero. Truncation induces a commutative diagram
\[
\begin{tikzcd}
\Aut_0(\widehat{\bS}) \arrow[r] \arrow[d, "\trun{0}"'] & \HoAut(\widehat{\bS}) \arrow[d, "\trun{0}"] \\ \Aut_0(\trun{0}\widehat{\bS}) \arrow[r] & \HoAut(\trun{0}\widehat{\bS}).
\end{tikzcd}
\]
Since profinite completion is defined entrywise, Proposition~\ref{prop:PaRB-cyclic-is-S0} identifies $\trun{0}\widehat{\bS}$ with $\widehat{\PaRB^{\mathrm{cyc}}}$. After forgetting the cyclic structure, an object-fixing automorphism and a homotopy between such automorphisms determine the corresponding data for $\widehat{\PaRB}$. It therefore follows from \cite[Proposition~8.1]{Boavida-Horel-Robertson} that the bottom horizontal map is injective.

The homotopy $H$ restricts to a homotopy from $\trun{0}F$ to $\id_{\trun{0}\widehat{\bS}}$. Consequently, $\trun{0}F=\id_{\trun{0}\widehat{\bS}}.$ 

By Lemma~\ref{lemma: profinite completion adjunction for modular operads} and Theorem~\ref{thm: presentation for maps out of bS}, the automorphism $F$ is determined by its restriction to genus zero and its value on the standard $S$-move $\sigma$. Since $\trun{0}F=\id$, it remains to prove that $F(\sigma)=\sigma$.

Recall that $\widehat\bS(1,1)$ has a single object $\contraction$ and morphism group $\widehat\Gamma_{1,1}$. Let $a$ denote the curve in $S_{\contraction}$ depicted in Figure~\ref{fig:S-move}, and let $\sigma=D_aD_bD_a$ be the standard $S$-move defined in Definition~\ref{def:standard moves}.
    
The natural transformation  $H:F\Rightarrow\id_{\widehat{\bS}}$ implies there is a morphism $H_{\contraction}$ in $\widehat{\bS}(1,1)$ making the following squares commute
        \[\begin{tikzcd}
        	{{\contraction}} & {{\contraction}} \\
        	{{\contraction}} & {{\contraction}}
        	\arrow["{F(D_a)}", from=1-1, to=1-2]
        	\arrow["{H_{\contraction}}"', from=1-1, to=2-1]
        	\arrow["{H_{\contraction}}", from=1-2, to=2-2]
        	\arrow["{D_a}"', from=2-1, to=2-2]
        \end{tikzcd}
        \hspace{2cm}
        \begin{tikzcd}
        	{{\contraction}} & {{\contraction}} \\
        	{{\contraction}} & {{\contraction}}
        	\arrow["{F(\sigma)}", from=1-1, to=1-2]
        	\arrow["{H_{\contraction}}"', from=1-1, to=2-1]
        	\arrow["{H_{\contraction}}", from=1-2, to=2-2]
        	\arrow["{\sigma}"', from=2-1, to=2-2]
        \end{tikzcd}\]
Since $D_a=\xi_{12}(\id_P\circ_1\tau)$ is obtained from a genus-zero morphism by composition and contraction, and since $F$ is the identity in genus zero, we have $F(D_a)=D_a$. The left-hand square therefore shows that $H_{\contraction}$ centralises $D_a$, while the second gives $F(\sigma)=H_{\contraction}^{-1}\sigma H_{\contraction}.$

%Since $D_a=\xi_{12}(\id_P\circ_1\tau)$, the mapping class $D_a$ is obtained from genus-zero morphisms. Since $F$ restricts to the identity in genus zero, we conclude that $F(D_a)=D_a$. In particular, then the left-hand commutative square implies that $H_{\contraction}$ is in the centraliser of $D_a$ in $\widehat\Gamma_{1,1}$. 

%By Lemma~\ref{lemma:centralisers in the profinite completion of Gamma11}, we get that $C_{\widehat{\Gamma}_{1,1}}(D_a)$ is generated by $D_\partial$, $(D_aD_bD_a)^2$, and $ D_a,$  which gives that   \[H_{\contraction}=(D_aD_bD_a)^{2\chi}\cdot D_\partial^\mu \cdot D_a^\nu. \qquad \text{ for some } \chi,\mu,\nu\in\widehat{\Z}\] and
        \[F(\sigma)=((D_aD_bD_a)^{2\chi}\cdot D_\partial^\mu \cdot D_a^\nu)^{-1}\sigma ((D_aD_bD_a)^{2\chi}\cdot D_\partial^\mu \cdot D_a^\nu)\]

By Lemma~\ref{lemma:centralisers in the profinite completion of Gamma11}, the centraliser of $D_a$ is generated by $D_\partial$, $\sigma^2$, and $D_a$. Since these elements commute, there exist $\chi,\mu,\nu\in\widehat{\mathbb Z}$ such that
\[
H_{\contraction}=(D_aD_bD_a)^{2\chi}\cdot D_\partial^\mu \cdot D_a^\nu. 
\] As $\sigma=D_aD_bD_a$ commutes with $(D_aD_bD_a)^2$ and $D_\partial$ is in the centre of $\Gamma_{1,1}$, we get \[F(\sigma)= D_a^{-\nu}\cdot\sigma \cdot D_a^\nu= D_a^{-\nu}\cdot(D_aD_bD_a)\cdot D_a^\nu.\]
    
    As $F$ must satisfy relation \eqref{g2_rel_morph_1}, using that $F$ is the identity in genus zero and the above formula for $F(\sigma)$, we get (using the curves and notation in Figure~\ref{fig:curves for g2})
        \begin{align}
            D_{\partial_0}^{\frac{1}{2}}D_{\partial_1}^{\frac{1}{2}} &= (D_{a_3}^{-\nu}\cdot(D_{a_3}D_{a_2}D_{a_3}) \cdot D_{a_3}^\nu)\cdot D_{a_3}^{-1} D_{a_1} \cdot (D_{a_3}^{-\nu}\cdot(D_{a_3}D_{a_2}D_{a_3}) \cdot D_{a_3}^\nu)\\
            &=  D_{a_3}^{-\nu}\cdot((D_{a_3}D_{a_2}D_{a_3}) \cdot D_{a_3}^{-1} D_{a_1} \cdot (D_{a_3}D_{a_2}D_{a_3})) \cdot D_{a_3}^\nu\\
            &=D_{a_3}^{-\nu}\cdot(D_{a_2}D_{a_3}D_{a_1}D_{a_2}D_{a_3}D_{a_2}) \cdot D_{a_3}^\nu.\label{eq:injection proof G2}
        \end{align}
    Now recall that $D_{\partial_0}^{\frac{1}{2}}D_{\partial_1}^{\frac{1}{2}}=(D_{a_2}D_{a_3}D_{a_1}D_{a_2}D_{a_3}D_{a_2})$ in $\Gamma_{1,2}$ and that, by the Lemma~\ref{lemma: conjugation by mapping classes}, we have
        \[(D_{\partial_0}^{\frac{1}{2}}D_{\partial_1}^{\frac{1}{2}})D_{a_3}(D_{\partial_0}^{\frac{1}{2}}D_{\partial_1}^{\frac{1}{2}})^{-1}=D_{a_1}.\]
    Putting these in equation \eqref{eq:injection proof G2}, we get
        \begin{align*}
        D_{a_2}D_{a_3}D_{a_1}D_{a_2}D_{a_3}D_{a_2} &=D_{a_3}^{-\nu}D_{a_1}^\nu\cdot(D_{a_2}D_{a_3}D_{a_1}D_{a_2}D_{a_3}D_{a_2}),
        \end{align*}    
    giving that
        $1= D_{a_3}^{-\nu}D_{a_1}^\nu.$
    Since $a_1$ and $a_3$ are distinct non-intersecting curves, this forces $\nu=0$. Therefore $F(\sigma)=\sigma$, as required.
\end{proof}

\begin{lemma}\label{lemma: object-fixing representative}
Every automorphism of $\widehat{\bS}$ is homotopic to an object-fixing automorphism.
\end{lemma}

\begin{proof}
Let $F:\widehat{\bS}\to\widehat{\bS}$ be an automorphism and write $F_0=\trun{0}F$. Under the identification $\trun{0}\widehat{\bS}\cong\widehat{\PaRB^{\mathrm{cyc}}}$ of Proposition~\ref{prop:PaRB-cyclic-is-S0}, the corresponding cyclic result gives an object-fixing automorphism $F'_0$ and a homotopy $H^{(0)}:F'_0\Rightarrow F_0$ (\cite[Proof of Proposition 5.6]{robertson2025grothendieck}).

Recall that $\ob\widehat{\bS}\cong\Upsilon$ is freely
generated by the $\Sigma_2^+$-set $\Upsilon(0,3)=\Sigma_2^+/A_2^+$ (Definition~\ref{def:free modular operad generated by mu}).
The components $H^{(0)}_P:P\to F(P)$ on these generators form a $\Sigma_2^+$-equivariant family, since $H^{(0)}$
is a cyclic-operad homotopy. Freeness therefore extends this family uniquely to a morphism of modular operads in sets $\widetilde H:\ob\widehat{\bS}\to \ob(\widehat{\bS}^{\mathscr I})$. Its endpoints agree on generators with
$\id_{\ob\widehat{\bS}}$ and $\ob F$, respectively, so
they agree everywhere.

The resulting morphisms $H_G:G\to F(G)$ consequently satisfy
\[
H_{G\circ_i^jK}=H_G\circ_i^jH_K,\qquad H_{\xi_i^jG}=\xi_i^jH_G, \qquad\text{and} \qquad
H_{\rho^*G}=\rho^*H_G,
\]
and the component at the unit is the identity. Define $F'$ to be the identity on objects and, for
$\phi:G\to G'$, set
\[
F'(\phi)=H_{G'}^{-1}\cdot F(\phi)\cdot H_G.
\]
The compatibility identities for $H_G$ imply that $F'$ respects partial compositions, contractions, symmetric-group actions, and units. Each induced map on morphism sets is a homeomorphism, being the composite of the corresponding map for $F$ with composition by fixed isomorphisms. Thus $F'$ is a continuous object-fixing automorphism
of modular operads.

Finally, $H_{G'}\cdot F'(\phi)=F(\phi)\cdot H_G$, so the family $H_G$ defines a natural transformation $H:F'\Rightarrow F$. Its compatibility with the
modular-operad structure makes it the required homotopy.
\end{proof}

We next compare object-fixing automorphisms of $\widehat{\bS}$
with homotopy automorphisms of the modular $\infty$-operad
$\widehat{\nerve\bS}$. Following \cite[1.2]{dk}, we write
$\mathbb R\Aut(\calP)\subseteq\mathbb R\Map(\calP,\calP)$
for the union of the connected components representing
invertible elements of the monoid
$\pi_0\mathbb R\Map(\calP,\calP)$.

For $\calQ\in\Mod(\Grpd)$, constant functors and endpoint
evaluations give a factorisation of the diagonal
\[
\calQ\longrightarrow\calQ^{\mathscr I}
\xrightarrow{(d_0,d_1)}
\calQ\times\calQ.
\]
This is a path object in $\Mod(\Grpd)$. Since every object
of $\Mod(\Grpd)$ is fibrant, for cofibrant $\calP$ there
is a natural identification
\[
[\calP,\calQ]_{\Mod(\Grpd)}
\cong
\pi_0\mathbb R\Map_{\Mod(\Grpd)}(\calP,\calQ).
\]

\begin{prop}\label{prop: strict-to-homotopy-endomorphisms}
There is an isomorphism of groups
\[
\Aut_0(\widehat{\bS})
\cong
\pi_0\mathbb R\Aut_{\Mod_\infty(\widehat{\Grpd})}
(\widehat{\nerve\bS}).
\]
\end{prop}

\begin{proof}
Let $q:\bS_{\mathrm{cof}}\xrightarrow{\sim}\bS$ be the
cofibrant replacement of
Proposition~\ref{prop: S cofibrant replacement}.
We first show that precomposition induces a bijection
\[
q^*:
[\bS,|\widehat{\bS}|]_{\Mod(\Grpd)}
\longrightarrow
[\bS_{\mathrm{cof}},|\widehat{\bS}|]_{\Mod(\Grpd)}.
\]

Since $q$ is fully faithful in every component, for
$G,H\in\ob(\bS_{\mathrm{cof}}(g,n+1))$ with
$\ob(q)(G)=\ob(q)(H)$ there is a unique morphism
$\iota_{G,H}:G\to H$ satisfying
$q(\iota_{G,H})=\id_{\ob(q)(G)}$.
As $q$ is also surjective on objects, a morphism
$u:\bS_{\mathrm{cof}}\to|\widehat{\bS}|$ factors uniquely
through $q$ if and only if it sends these comparisons
to identities.

This condition can be checked in genus zero. Indeed,
choose representatives of $G$ and $H$ on the same
underlying graph. Their local labellings determine pairs
of generators in $\widetilde{\Upsilon}(0,3)=\Sigma_2^+$
with the same image in
$\Upsilon(0,3)=\Sigma_2^+/A_2^+$.
Applying the partial compositions, contractions, and
relabellings defining the graph to their comparison
morphisms gives $\iota_{G,H}$, by full faithfulness of $q$.
Thus $u$ factors through $q$ if and only if $\trun{0}u$
factors through $\trun{0}q$.

For surjectivity, let
$u:\bS_{\mathrm{cof}}\to|\widehat{\bS}|$.
By Corollary~\ref{cor: truncated S cofibrant replacements}
and the cyclic result of \cite{robertson2025grothendieck},
there is a morphism
$u'_0:\trun{0}\bS\to|\trun{0}\widehat{\bS}|$
with the canonical map on objects and a homotopy
$\trun{0}u\simeq u'_0\circ\trun{0}q$.
Since $\ob(\bS_{\mathrm{cof}})$ is freely generated in
genus zero, the construction in the proof of
Lemma~\ref{lemma: object-fixing representative} extends
the components of this homotopy and conjugates $u$ to
a homotopic morphism satisfying
$\trun{0}u=u'_0\circ\trun{0}q$.
It therefore factors uniquely as $u=u'\circ q$.
Moreover, $\ob(u')=\ob(\eta_{\bS})$, since this holds
on the generators, where
$\eta_{\bS}:\bS\to|\widehat{\bS}|$ is the unit of the
adjunction of
Lemma~\ref{lemma: profinite completion adjunction for modular operads}.
This proves surjectivity of $q^*$.

For injectivity, suppose
$K:u'\circ q\Rightarrow v'\circ q$ is a homotopy.
Naturality at $\iota_{G,H}$ gives $K_G=K_H$ whenever
$\ob(q)(G)=\ob(q)(H)$.
Thus $K$ descends uniquely to a homotopy $u'\Rightarrow v'$:
fullness of $q$ gives naturality, and surjectivity on
objects gives compatibility with the modular-operad
structure. Hence $q^*$ is injective.

Since $q$ is a cofibrant replacement, the path-object
identification above now gives
\[
[\bS,|\widehat{\bS}|]_{\Mod(\Grpd)}
\cong
\pi_0\mathbb R\Map_{\Mod(\Grpd)}
(\bS,|\widehat{\bS}|).
\]
By Theorem~\ref{thm: nerve is homotopically fully faithful},
\begin{equation}\label{eq: modular nerve mapping comparison}
\mathbb R\Map_{\Mod(\Grpd)}(\bS,|\widehat{\bS}|)
\simeq
\mathbb R\Map_{\Mod_\infty(\Grpd)}
(\nerve\bS,\nerve|\widehat{\bS}|).
\end{equation}
Lemma~\ref{lemma: N commutes with completion for groupoids}
and preservation of finite products by $|-|$ give
\begin{equation}\label{eq: underlying completed nerve}
|\widehat{\nerve\bS}|
\cong|\nerve\widehat{\bS}|
\cong\nerve|\widehat{\bS}|.
\end{equation}
Thus the right-hand side of
\eqref{eq: modular nerve mapping comparison} is identified
with
\begin{equation}\label{eq: completed nerve target comparison}
\mathbb R\Map_{\Mod_\infty(\Grpd)}
(\nerve\bS,\nerve|\widehat{\bS}|)
\simeq
\mathbb R\Map_{\Mod_\infty(\Grpd)}
(\nerve\bS,|\widehat{\nerve\bS}|).
\end{equation}

The Quillen adjunction
$(-)^{\wedge}:\Grpd\rightleftarrows\widehat{\Grpd}:|-|$
of \cite[Proposition~4.22]{Horel_profinite_groupoids}
induces a derived adjunction between the $\infty$-categories
associated to
$\operatorname{Fun}(\bM^{\op},\Grpd)$ and
$\operatorname{Fun}(\bM^{\op},\widehat{\Grpd})$.
Here the derived functors are computed by entrywise
completion and $|-|$: every groupoid is cofibrant, and
$\nerve\widehat{\bS}$ is entrywise fibrant by
\cite[Corollary~4.24 and Proposition~4.40]{Horel_profinite_groupoids}.
Since modular $\infty$-operads form full
$\infty$-subcategories of these functor $\infty$-categories,
the adjunction gives
\begin{equation}\label{eq: completed nerve mapping adjunction}
\mathbb R\Map_{\Mod_\infty(\Grpd)}
(\nerve\bS,|\widehat{\nerve\bS}|)
\simeq
\mathbb R\Map_{\Mod_\infty(\widehat{\Grpd})}
(\widehat{\nerve\bS},\widehat{\nerve\bS}).
\end{equation}

The strict completion adjunction of
Lemma~\ref{lemma: profinite completion adjunction for modular operads}
also identifies homotopies, since
$|\widehat{\bS}^{\mathscr I}|
\cong|\widehat{\bS}|^{\mathscr I}$.
Combining this with the bijection $q^*$ and taking
connected components in
\eqref{eq: modular nerve mapping comparison},
\eqref{eq: completed nerve target comparison}, and
\eqref{eq: completed nerve mapping adjunction} gives
\begin{equation}\label{eq: strict and derived endomorphism classes}
[\widehat{\bS},\widehat{\bS}]_{\Mod(\widehat{\Grpd})}
\cong
\pi_0\mathbb R\Map_{\Mod_\infty(\widehat{\Grpd})}
(\widehat{\nerve\bS},\widehat{\nerve\bS}).
\end{equation}
Under \eqref{eq: underlying completed nerve}, the adjoint
of $\nerve F$ is $\nerve F'$, where
$F'=|F|\circ\eta_{\bS}$.
Thus \eqref{eq: strict and derived endomorphism classes}
sends the class of $F$ to the component of $\nerve F$
and is an isomorphism of monoids.

The construction above gives an object-fixing
representative $F$ for every class on the left-hand side
of \eqref{eq: strict and derived endomorphism classes}.
If its class is invertible, the inverse is represented
by a strict endomorphism, giving a homotopy inverse for $F$.
Its underlying functor at each $(g,n+1)$ is therefore an
equivalence of groupoids. Since $F$ fixes objects, it
induces continuous bijections on morphism sets, which
are homeomorphisms of profinite sets.
Hence $F$ is a strict automorphism, proving surjectivity.

Finally, if two object-fixing automorphisms determine
the same component,
\eqref{eq: strict and derived endomorphism classes}
implies that they are homotopic. They are equal by
Proposition~\ref{prop:injection}, proving injectivity.
\end{proof}

Theorem~\ref{main theorem NS action} gives an object-fixing
action of the monoid $\underline{\NS}$ on $\widehat{\bS}$.
Restricting to invertible elements gives an action of
$\NS$ by object-fixing automorphisms. We now pass to
homotopy automorphisms and record the compatibility
with genus-zero truncation.

\begin{cor}\label{cor: homotopy NS action}
The action of $\NS$ on $\widehat{\bS}$ induces a faithful
action on $\widehat{\nerve\bS}$ by homotopy automorphisms.
Equivalently, there is an injective homomorphism
\[
\NS
\longrightarrow
\pi_0\mathbb R\Aut_{\Mod_\infty(\widehat{\Grpd})}
(\widehat{\nerve\bS}).
\]
Its genus-zero restriction agrees with the restriction
of the $\GT$-action along $\NS\hookrightarrow\GT$.
\end{cor}

\begin{proof}
Theorem~\ref{main theorem NS action} gives an injective
homomorphism $\NS\hookrightarrow\Aut_0(\widehat{\bS})$.
Composing it with the isomorphism of
Proposition~\ref{prop: strict-to-homotopy-endomorphisms}
gives the required injective homomorphism.

By Theorem~\ref{main theorem NS action}, genus-zero
truncation of the strict action agrees with restriction
of the $\GT$-action along $\NS\hookrightarrow\GT$.
Since the nerve and entrywise completion commute with
restriction, we obtain a commutative diagram
\[
\begin{tikzcd}
\NS
\arrow[r, hook]
\arrow[d, hook]
&
\pi_0\mathbb R\Aut_{\Mod_\infty(\widehat{\Grpd})}
(\widehat{\nerve\bS})
\arrow[d, "\trun{0}"]
\\
\GT
\arrow[r]
&
\pi_0\mathbb R\Aut_{\Cyc_\infty(\widehat{\Grpd})}
(\widehat{\nerve\trun{0}\bS}).
\end{tikzcd}
\]
After forgetting the cyclic structure and applying $\widehat{\cs}$, the bottom map gives the faithful action on the profinite genus-zero surface operad of \cite[Theorem~8.4]{Boavida-Horel-Robertson}.
\end{proof}

\subsection{The Teichmüller tower of profinite homotopy $1$-types}

We now construct a tower of profinite homotopy $1$-types
from $\widehat{\bS}$ and equip it with a faithful
$\NS$-action. Consider the composite
\[
\bM^{\op}
\xrightarrow{\nerve\bS}\Grpd
\xrightarrow{\widehat{(-)}}\widehat{\Grpd}
\xrightarrow{\widehat{\cs}}\widehat{\isSet}.
\]
By Lemma~\ref{lemma: N commutes with completion for groupoids},
this modular dendroidal object is identified with
$\widehat{\cs}\nerve\widehat{\bS}$.

The nerve $\nerve\widehat{\bS}$ is reduced, satisfies
the Segal condition strictly, and is entrywise fibrant by
\cite[Corollary~4.24 and Proposition~4.40]{Horel_profinite_groupoids}.
Since the right Quillen functor $\widehat{\cs}$ preserves
finite products and the terminal object,
$\widehat{\cs}\nerve\widehat{\bS}$ defines a modular
$\infty$-operad in profinite spaces. At the corolla
$C_{(g,n+1)}$, its value is equivalent to
$\widehat{\cs}\widehat{\Gamma}_{g,n+1}$. It therefore
has profinite fundamental group $\widehat{\Gamma}_{g,n+1}$
and trivial higher profinite homotopy groups.

\begin{prop}\label{NS action on topological operad}
The action of $\NS$ on $\widehat{\bS}$ induces a faithful
action by homotopy automorphisms on
$\widehat{\cs}\nerve\widehat{\bS}$.
Equivalently, there is an injective homomorphism
\[
\NS\longrightarrow
\pi_0\mathbb R\Aut_{\Mod_\infty(\widehat{\isSet})}
\bigl(\widehat{\cs}\nerve\widehat{\bS}\bigr).
\]
Restriction along $\galQ\hookrightarrow\NS$ gives a
faithful action of $\galQ$ on this modular $\infty$-operad.
\end{prop}

\begin{proof}
The Quillen adjunction
\[
\begin{tikzcd}
\widehat{\sSet}
\arrow[r, shift left=1, "\widehat{\Pi}_1"] &
\widehat{\Grpd}
\arrow[l, shift left=1, "\widehat{\cs}"]
\end{tikzcd}
\]
of \cite[Equation~(5.1)]{Horel_profinite_groupoids}
has fully faithful right derived functor by
\cite[Proposition~5.7]{Horel_profinite_groupoids}.
Applying this functor entrywise gives a fully faithful
functor between the $\infty$-categories associated to
$\operatorname{Fun}(\bM^{\op},\widehat{\Grpd})$ and
$\operatorname{Fun}(\bM^{\op},\widehat{\sSet})$.
Since it preserves finite products, it restricts to
their full subcategories of modular $\infty$-operads.

The entrywise fibrancy of $\nerve\widehat{\bS}$
therefore gives an isomorphism
\[
\pi_0\mathbb R\Aut_{\Mod_\infty(\widehat{\Grpd})}
\bigl(\nerve\widehat{\bS}\bigr)
\cong
\pi_0\mathbb R\Aut_{\Mod_\infty(\widehat{\isSet})}
\bigl(\widehat{\cs}\nerve\widehat{\bS}\bigr).
\]
Corollary~\ref{cor: homotopy NS action} and the
identification
$\nerve\widehat{\bS}\cong\widehat{\nerve\bS}$
of Lemma~\ref{lemma: N commutes with completion for groupoids}
give an injective homomorphism from $\NS$ to the
left-hand side. Composing with the displayed isomorphism
gives the faithful $\NS$-action. Restriction along
$\galQ\hookrightarrow\NS$ gives the faithful Galois action.
\end{proof}

\subsection{Étale homotopy types and the profinite-space model}\label{sec: tower of etale homotopy types}

We now take classifying spaces before profinite completion and show that $\widehat{\cs\nerve\bS}$ defines a modular $\infty$-operad in profinite spaces. Its corolla values will be identified with the étale homotopy types of the tangent-vector moduli stacks. To verify the Segal condition, we must show that profinite completion preserves the finite products occurring in its Segal maps. 

We write $\isSet$ for the $\infty$-category of spaces and $\isSet_\pi\subset\isSet$ for the full subcategory of $\pi$-finite spaces. The $\infty$-category of profinite spaces is $\widehat{\isSet}:=\Pro(\isSet_\pi)$. Profinite completion $\widehat{(-)}:\Pro(\isSet)\to\widehat{\isSet}$ is left adjoint to the inclusion. We regard spaces as constant prospaces when applying this functor.

Let $\isSet_{<\infty}\subset\isSet$ denote the full subcategory of truncated spaces, and write $\tau_{<\infty}:\Pro(\isSet)\to\Pro(\isSet_{<\infty})$ for protruncation. Following \cite[Notation~2.11]{haine2024profinite}, let $\mathcal H\subset\Pro(\isSet_{<\infty})$ be the smallest full subcategory containing $\widehat{\isSet}$ and closed under geometric realisations, retracts, and cofiltered limits. We call a prospace $X$ \emph{admissible} if $\tau_{<\infty}X$ lies in $\mathcal H$.

\begin{prop}\label{prop:haine-product}
Let $X_1,\ldots,X_r$ be admissible prospaces. Then the natural map
\[
\left(\prod_{i=1}^r X_i\right)^{\wedge}\longrightarrow\prod_{i=1}^r\widehat{X_i}
\]
is an equivalence in profinite spaces.
\end{prop}

\begin{proof}
Profinite completion factors through protruncation, which preserves finite products. The assertion therefore follows
by iterating \cite[Theorem~2.13]{haine2024profinite}, with the set of primes taken to be all primes.
\end{proof}

For a quasi-compact and quasi-separated (qcqs) scheme $X$, the protruncated étale homotopy type $\Pi^{\acute et}_{<\infty}(X):=\tau_{<\infty}\Pi^{\acute et}_{\infty}(X)$ admits a $\Pro(\isSet_\pi)$-resolution by \cite[Example~A.9]{haine2024profinite}. Consequently, it belongs to $\mathcal H$, so $\Pi^{\acute et}_{\infty}(X)$ is admissible. The following finite-product comparison follows from Proposition~\ref{prop:haine-product}. The case of two qcqs schemes appears in \cite[Example~2.17]{haine2024profinite}.

\begin{cor}\label{cor:haine-qcqs-schemes}
Let $X_1,\ldots,X_r$ be quasi-compact and quasi-separated
schemes. Then the natural map
\[
\left(\prod_{i=1}^r\Pi^{\acute et}_{\infty}(X_i)\right)^{\wedge}\longrightarrow\prod_{i=1}^r \widehat{\Pi^{\acute et}_{\infty}(X_i)}
\]
is an equivalence in profinite spaces.
\end{cor}

We now consider the geometric moduli stacks over $\overline{\mathbb Q}$, with a chosen embedding $\overline{\mathbb Q}\hookrightarrow\mathbb C$. For $2g-2+n>0$, Example~\ref{example: Oda} gives the identification
\[
\widehat{\Pi^{\acute et}_{\infty}(\mathcal M_g^n)}\simeq \widehat{\cs\Gamma_g^n}
\]
between the completed étale homotopy type of $\mathcal M_g^n$ and the profinite completion of the classifying space of the marked-point mapping class group. To obtain the boundary mapping class groups $\Gamma_{g,n}$ occurring in our modular operad, we equip each marked point with a non-zero tangent vector.

Let $\mathcal M_{g,n}^{\tan}$ denote the stack of smooth genus $g$ curves with $n$ ordered marked points equipped with non-zero tangent vectors. It is the
$(\mathbb G_m)^n$-torsor over $\mathcal M_g^n$ obtained by taking the fibre product of the tangent-line bundles with their zero sections removed. For $n\geq1$ and $2g-2+n>0$, the corresponding capping map on mapping class groups gives the central extension
\[
1\longrightarrow\mathbb Z^n\longrightarrow\Gamma_{g,n}\longrightarrow\Gamma_g^n \longrightarrow1,
\]
whose kernel is generated by the boundary Dehn twists.

We next compare the étale homotopy types of the
tangent-vector moduli stacks with the classifying spaces
of boundary mapping class groups after profinite completion.

\begin{lemma}\label{lemma:tangent-moduli-comparison}
For every $g\geq0$ and $n\geq0$ satisfying $2g+n\geq1$,
there is an equivalence in profinite spaces
\[
\widehat{\Pi_\infty^{\acute et}
(\mathcal M_{g,n+1}^{\tan})}
\simeq
\widehat{\cs\Gamma_{g,n+1}}.
\]
\end{lemma}

\begin{proof}
We first justify passage from $\overline{\mathbb Q}$ to
$\mathbb C$ using a Čech presentation of
$\mathcal M_{g,n+1}^{\tan}$.
Choose a smooth surjective atlas
$U\to\mathcal M_{g,n+1}^{\tan}$ with $U$ affine and of
finite type over $\overline{\mathbb Q}$
\cite[Tag~04YC]{stacks-project}.
Let $U_\bullet$ be the Čech nerve of this atlas, so that
$U_k$ is the fibre product of $k+1$ copies of $U$ over
$\mathcal M_{g,n+1}^{\tan}$.

We verify that every $U_k$ is a qcqs scheme.
Let $(C_1,\mathbf p_1,\mathbf t_1)$ and
$(C_2,\mathbf p_2,\mathbf t_2)$ be the two families
obtained by pulling back along the projections
$U\times U\to U$. There is a natural identification
\[
U_1\cong
\operatorname{Isom}_{U\times U}
\bigl((C_1,\mathbf p_1,\mathbf t_1),
      (C_2,\mathbf p_2,\mathbf t_2)\bigr)
\hookrightarrow
\operatorname{Isom}_{U\times U}(C_1,C_2).
\]
The algebraic space on the right is separated and of
finite presentation over $U\times U$ by
\cite[Tag~0DSQ]{stacks-project}.
The inclusion is closed because preservation of the
marked sections and tangent vectors imposes the equations
$f(p_{1,j})=p_{2,j}$ and $df(t_{1,j})=t_{2,j}$ for
$1\leq j\leq n+1$.
Thus $U_1$ is a separated algebraic space of finite type
over $U\times U$.

When $2g+n>1$, the pointed curves are stable and have
finite automorphism groups, so the geometric fibres
of $U_1\to U\times U$ are finite.
For $(g,n)=(0,1)$, an automorphism of $\mathbb P^1$
fixing the two marked points and their tangent vectors
is the identity, so the same conclusion holds.
Hence $U_1\to U\times U$ is quasi-finite, and $U_1$
is a scheme by \cite[Tag~03XX]{stacks-project}.
For $k\geq1$, the term $U_k$ is the fibre product of
$k$ copies of $U_1$ over $U$.
Consequently, every $U_k$ is a separated scheme of
finite type over $\overline{\mathbb Q}$, and hence qcqs.

The atlas admits étale local sections, giving
$|U_\bullet|\simeq\mathcal M_{g,n+1}^{\tan}$ in étale
stacks. By \cite[Lemma~2.3.10]{Carchedi}, the functor
$X\mapsto\operatorname{Sh}_\infty(X_{\acute et})$
preserves colimits. Since shape and profinite completion
also preserve colimits, we obtain
\[
\widehat{\Pi_\infty^{\acute et}
(\mathcal M_{g,n+1}^{\tan})}
\simeq
\left|\widehat{\Pi_\infty^{\acute et}(U_\bullet)}\right|
\]
in profinite spaces.

Write $\mathcal M_{g,n+1,\mathbb C}^{\tan}$ for the
base change along the chosen embedding
$\overline{\mathbb Q}\hookrightarrow\mathbb C$.
Its base-changed atlas gives the analogous presentation.
Invariance under extension of algebraically closed fields
of characteristic zero applies to each $U_k$, so these
presentations identify the étale homotopy types of
$\mathcal M_{g,n+1}^{\tan}$ and
$\mathcal M_{g,n+1,\mathbb C}^{\tan}$ after profinite
completion.
The comparison theorem
\cite[Corollary~5.14]{Carchedi} therefore identifies
$\widehat{\Pi_\infty^{\acute et}
(\mathcal M_{g,n+1}^{\tan})}$
with the profinite completion of the homotopy type of
$(\mathcal M_{g,n+1,\mathbb C}^{\tan})^{\mathrm{top}}$.

Suppose first that $2g+n>1$.
Forgetting the tangent vectors gives a principal
$(\mathbb C^\times)^{n+1}$-bundle and hence a homotopy
fibre sequence
\[
(\mathbb C^\times)^{n+1}
\longrightarrow
(\mathcal M_{g,n+1,\mathbb C}^{\tan})^{\mathrm{top}}
\longrightarrow
(\mathcal M_{g,\mathbb C}^{n+1})^{\mathrm{top}}.
\]
The base has homotopy type $\cs\Gamma_g^{n+1}$ by
contractibility of pointed Teichmüller space, while
$(\mathbb C^\times)^{n+1}\simeq\cs\mathbb Z^{n+1}$.
The long exact homotopy sequence shows that the middle
term is aspherical. Its fundamental-group sequence is
the central extension
\[
1\longrightarrow\mathbb Z^{n+1}
\xrightarrow{\,e_j\mapsto D_{\partial_j}\,}
\Gamma_{g,n+1}
\longrightarrow\Gamma_g^{n+1}
\longrightarrow1
\]
associated with replacing the marked points by parametrised
boundary components. Consequently,
$(\mathcal M_{g,n+1,\mathbb C}^{\tan})^{\mathrm{top}}
\simeq\cs\Gamma_{g,n+1}$.

It remains to consider $(g,n)=(0,1)$.
Placing the marked points at $0,\infty\in\mathbb P^1$
gives
\[
\mathcal M_{0,2}^{\tan}
\cong[(\mathbb G_m)^2/\mathbb G_m],
\qquad
t\cdot(u,v)=(tu,t^{-1}v).
\]
The invariant map $(u,v)\mapsto uv$ identifies this quotient
with $\mathbb G_m$. Hence
$(\mathcal M_{0,2,\mathbb C}^{\tan})^{\mathrm{top}}
\simeq S^1\simeq\cs\Gamma_{0,2}$,
since $\Gamma_{0,2}\cong\mathbb Z$.

In both cases, the comparison theorem gives the required
equivalence
\[
\widehat{\Pi_\infty^{\acute et}
(\mathcal M_{g,n+1}^{\tan})}
\simeq
\widehat{\cs\Gamma_{g,n+1}}.
\]
\end{proof}
The comparison above identifies the corolla values with
étale homotopy types. To verify the Segal condition, we
establish the required product comparison directly for
the classifying spaces of boundary mapping class groups,
using their finite CW models.

\begin{prop}\label{prop:haine-boundary-product}
Let $g_1,\ldots,g_r\geq0$ and $n_1,\ldots,n_r\geq0$.
Then the natural map
\[
\left(\prod_{i=1}^r\cs\Gamma_{g_i,n_i+1}\right)^{\wedge}
\longrightarrow
\prod_{i=1}^r(\cs\Gamma_{g_i,n_i+1})^{\wedge}
\]
is an equivalence in profinite spaces.
\end{prop}

\begin{proof}
Since $n_i+1>0$, each group $\Gamma_{g_i,n_i+1}$ is
torsion-free and admits a cocompact classifying space
for proper actions
\cite[Section~2.2 and Theorem~4.2]{mislin}.
The action is therefore free, and the quotient is a
finite CW model for $\cs\Gamma_{g_i,n_i+1}$.

Consequently, each $\cs\Gamma_{g_i,n_i+1}$ is a compact
object of $\isSet$. By
\cite[Observation~2.12]{haine2024profinite}, its
protruncation belongs to $\mathcal H$, so it is admissible.
Proposition~\ref{prop:haine-product} now gives the required
equivalence.
\end{proof}
We use this product comparison to show that $\widehat{\cs\nerve\bS}$ satisfies the Segal condition.

\begin{lemma}\label{lemma:teichmuller-tower-is-modular-infty-operad}
The composite
\[
\bM^{\op}
\xrightarrow{\nerve\bS}\Grpd
\xrightarrow{\cs}\sSet
\xrightarrow{\widehat{(-)}}\widehat{\isSet}
\]
defines a modular $\infty$-operad in profinite spaces.
\end{lemma}

\begin{proof}
Let $G$ be a graph with vertices $v_1,\ldots,v_r$, and write
$\epsilon(v_i)=g_i$ and $|\nb(v_i)|=n_i+1$.
Since the modular dendroidal nerve satisfies the Segal
condition strictly and $\cs$ preserves finite products,
there is a natural isomorphism
\[
(\cs\nerve\bS)_G
\cong
\prod_{i=1}^r\cs\bS(g_i,n_i+1).
\]
The groupoid $\bS(g_i,n_i+1)$ is connected with automorphism
group $\Gamma_{g_i,n_i+1}$, so
$\cs\bS(g_i,n_i+1)\simeq\cs\Gamma_{g_i,n_i+1}$.
Proposition~\ref{prop:haine-boundary-product} therefore
identifies the profinite completion of the product above
with the product of the completed corolla values.
Thus the Segal map
\[
(\widehat{\cs\nerve\bS})_G
\longrightarrow
\prod_{i=1}^r(\widehat{\cs\nerve\bS})_{C_{v_i}}
\]
is a weak equivalence for every graph $G$.
The exceptional edge is still sent to the point.
Hence $\widehat{\cs\nerve\bS}$ is a modular
$\infty$-operad in profinite spaces.
\end{proof}

Together, Lemmas~\ref{lemma:teichmuller-tower-is-modular-infty-operad}
and~\ref{lemma:tangent-moduli-comparison} identify the
corolla values of $\widehat{\cs\nerve\bS}$ with the étale
homotopy types of $\mathcal M_{g,n+1}^{\tan}$.
We now compare this modular $\infty$-operad with the tower
of profinite homotopy $1$-types constructed in the preceding
subsection.

The two orders of applying profinite completion and
classifying spaces are displayed in the square
\[
\begin{tikzcd}
\Grpd
\arrow[r, "\cs"] \arrow[d, "\widehat{(-)}"'] &
\sSet \arrow[d, "\widehat{(-)}"] \\
\widehat{\Grpd}
\arrow[r, "\widehat{\cs}"'] &
\widehat{\sSet}.
\end{tikzcd}
\]
The profinite-completion adjunctions give a natural
comparison between its two composites. At a groupoid
$\mathscr E$, this is the map
$\widehat{\cs\mathscr E}\to
\widehat{\cs}\widehat{\mathscr E}$.
Applying the comparison entrywise and using
Lemma~\ref{lemma: N commutes with completion for groupoids}
gives a morphism
\begin{equation}\label{eq: profinite etale tower comparison}
\widehat{\cs\nerve\bS}
\longrightarrow
\widehat{\cs}\nerve\widehat{\bS}.
\end{equation}
On corollas, this map passes from the étale homotopy types
to their profinite homotopy $1$-types.
The comparison need not be an equivalence for an arbitrary
groupoid. A sufficient condition is goodness in the sense
of Serre.

\begin{definition}\label{def: good groups}
A discrete group $\mathsf G$ is \emph{good} if, for every
finite discrete $\mathsf G$-module $\mathsf M$, the natural
map
$H^*_{\mathrm{cts}}(\widehat{\mathsf G},\mathsf M)
\to H^*(\mathsf G,\mathsf M)$
is an isomorphism. A groupoid with finitely many objects
is called good if the automorphism group of each object
is good.
\end{definition}

By~\cite[Proposition~5.9]{Horel_profinite_groupoids},
if $\mathscr E$ is a good groupoid with finitely many
objects, then the comparison map
$\widehat{\cs\mathscr E}\to
\widehat{\cs}\widehat{\mathscr E}$
is a weak equivalence of profinite spaces.

\begin{lemma}\label{lemma: B of good groupoids commutes with hat}
For $0\leq g\leq1$ and $n\geq0$ satisfying $2g+n\geq1$,
the comparison map
\[
\widesthat{\cs\bS(g,n+1)}
\longrightarrow
\widehat{\cs}\bigl(\widehat{\bS}(g,n+1)\bigr)
\]
is a weak equivalence of profinite spaces.
\end{lemma}

\begin{proof}
The groupoid $\bS(g,n+1)$ has finitely many objects.
For every object $A\in\ob(\bS(g,n+1))$, the automorphism
group $\Aut_{\bS(g,n+1)}(A)$ is isomorphic to the mapping
class group $\Gamma_{g,n+1}$. These groups are good for
$g\leq1$; see~\cite[p.~94]{Oda}.
Hence $\bS(g,n+1)$ is a good groupoid, and the result
follows from~\cite[Proposition~5.9]{Horel_profinite_groupoids}.
\end{proof}

\begin{prop}\label{prop: genus-one B N and completion commute}
The comparison
\eqref{eq: profinite etale tower comparison}
induces an equivalence of genus-one truncated modular
$\infty$-operads in profinite spaces
\[
\trun{1}\widehat{\cs\nerve\bS}
\xrightarrow{\sim}
\trun{1}\bigl(\widehat{\cs}\nerve\widehat{\bS}\bigr).
\]
\end{prop}

\begin{proof}
Both modular dendroidal objects satisfy the Segal condition
by Lemma~\ref{lemma:teichmuller-tower-is-modular-infty-operad}
and the construction preceding
Proposition~\ref{NS action on topological operad}.
For every graph $G\in\bM_{\leq1}$, the comparison fits
into the commutative square
\[
\begin{tikzcd}[column sep=large]
(\widehat{\cs\nerve\bS})_G
\arrow[r] \arrow[d, "\simeq"'] &
(\widehat{\cs}\nerve\widehat{\bS})_G
\arrow[d, "\simeq"] \\
\displaystyle\prod_{v\in V_G}
(\widehat{\cs\nerve\bS})_{C_v}
\arrow[r] &
\displaystyle\prod_{v\in V_G}
(\widehat{\cs}\nerve\widehat{\bS})_{C_v}.
\end{tikzcd}
\]
Every vertex of $G$ has genus at most one, so the bottom
map is an equivalence by
Lemma~\ref{lemma: B of good groupoids commutes with hat}.
The vertical maps are the Segal equivalences, and hence
the top map is an equivalence as well.
\end{proof}

Combining Proposition~\ref{prop: genus-one B N and completion commute}
with Theorem~\ref{thm: presentation for maps out of bS}
gives the following operadic two-level principle.
It identifies the homotopy endomorphisms of the full
tower of profinite homotopy $1$-types with those of the
genus-one truncation of the tower of étale homotopy types.

\begin{cor}\label{prop: two-level profinite space model}
Restriction and the comparison
\eqref{eq: profinite etale tower comparison}
induce an isomorphism of monoids
\[
\pi_0\mathbb R\End_{\Mod_\infty(\widehat{\isSet})}
\bigl(\widehat{\cs}\nerve\widehat{\bS}\bigr)
\cong
\pi_0\mathbb R\End_{\Mod_{\infty,\leq1}(\widehat{\isSet})}
\bigl(\trun{1}\widehat{\cs\nerve\bS}\bigr).
\]
Consequently, $\NS$ and $\galQ$ act faithfully by homotopy
automorphisms on $\trun{1}\widehat{\cs\nerve\bS}$.
The genus-zero restriction of the $\NS$-action agrees
with the $\GT$-action along $\NS\hookrightarrow\GT$.
\end{cor}

\begin{proof}
Theorem~\ref{thm: presentation for maps out of bS}
and Lemma~\ref{lemma: profinite completion adjunction for modular operads}
show that every endomorphism of $\trun{1}\widehat{\bS}$
extends uniquely to $\widehat{\bS}$.
Applying the same presentation to maps into
$|\widehat{\bS}^{\mathscr I}|$ gives the corresponding
extension of homotopies. Restriction therefore induces
an isomorphism
\[
[\widehat{\bS},\widehat{\bS}]_{\Mod(\widehat{\Grpd})}
\cong
[\trun{1}\widehat{\bS},\trun{1}\widehat{\bS}]
_{\Mod_{\leq1}(\widehat{\Grpd})}.
\]

By \eqref{eq: strict and derived endomorphism classes},
these homotopy classes compute the components of the
derived endomorphism spaces of the nerves.
The same identification holds after genus-one truncation.
Indeed, the cofibrant-replacement argument in
Proposition~\ref{prop: strict-to-homotopy-endomorphisms}
applies after truncation, since the object operad of
the replacement is freely generated in genus zero.
These identifications commute with restriction and
preserve composition.

The fully faithful derived profinite classifying-space
functor, used in
Proposition~\ref{NS action on topological operad},
then identifies these monoids with the homotopy
endomorphism monoids of
$\widehat{\cs}\nerve\widehat{\bS}$ and its genus-one
truncation. Proposition~\ref{prop: genus-one B N and completion commute}
identifies the latter truncation with
$\trun{1}\widehat{\cs\nerve\bS}$, proving the asserted
isomorphism.

The faithful actions follow from
Proposition~\ref{NS action on topological operad}
under this isomorphism. Compatibility with the
genus-zero $\GT$-action follows from
Corollary~\ref{cor: homotopy NS action}.
\end{proof}
\appendix
\section{Homotopy Theory of Modular Operads}
\label{sec: homotopy theory of mod op and truncated mod op}

Modular operads can be viewed as algebras over a certain coloured operad
(cf.~\cite[Appendix~A]{dch_coextension}). This allows us to obtain model
structures on categories of modular operads from general results
concerning algebras over coloured operads. Since our arguments rely on
several specific features of the coloured operad governing modular
operads, we begin with a brief review of coloured operads and their
algebras.

\begin{definition}
Let $\mathfrak{C}$ be a nonempty set of \emph{colours}. A
$\mathfrak{C}$-coloured \emph{sequence} in $\bE$ is a collection of objects
\[
\mathcal{P}(c_1,\ldots,c_k;c_0)
\]
indexed by lists $(c_1,\ldots,c_k;c_0)$ of colours in $\mathfrak{C}$,
together with right actions of the symmetric groups. Thus, for each
$\sigma\in\Sigma_k$, there is an isomorphism
\[
\mathcal{P}(c_1,\ldots,c_k;c_0)
\xrightarrow{\ \sigma^*\ }
\mathcal{P}(c_{\sigma(1)},\ldots,c_{\sigma(k)};c_0).
\]

A $\mathfrak{C}$-coloured \emph{operad} is a $\mathfrak{C}$-coloured
sequence $\mathcal{P}$ equipped with:
\begin{enumerate}
    \item unit maps
    $\iota_c:1_{\bE}\to\mathcal{P}(c;c)$
    for each $c\in\mathfrak{C}$;

    \item composition maps
    \[
    \mathcal{P}(c_1,\ldots,c_k;c_0)\otimes
    \mathcal{P}(d_1,\ldots,d_j;d_0)
    \xrightarrow{\ \circ_i\ }
    \mathcal{P}(c_1,\ldots,c_{i-1},d_1,\ldots,d_j,
    c_{i+1},\ldots,c_k;c_0),
    \]
    defined whenever $c_i=d_0$ and $1\leq i\leq k$.
\end{enumerate}
These data are required to satisfy the standard associativity,
unitality, and equivariance axioms.
\end{definition}

A map $f:\mathcal{P}\to\mathcal{Q}$ of coloured operads is a map of the
underlying coloured sequences that respects the units, compositions,
and symmetric group actions. The category of $\mathfrak{C}$-coloured
operads in $\bE$ is denoted $\Op_{\mathfrak{C}}(\bE)$. For further
details, see \cite[Definition~1.1]{bm_resolutions} or
\cite[§2.2]{DCH}.

\begin{definition}
Let $\mathcal{P}$ be a $\mathfrak{C}$-coloured operad. A
\emph{$\mathcal{P}$-algebra} in $\bE$ is a collection of objects
$\{X(c)\}_{c\in\mathfrak{C}}$ together with structure maps
\[
\mathcal{P}(c_1,\ldots,c_k;c_0)\otimes
X(c_1)\otimes\cdots\otimes X(c_k)
\longrightarrow X(c_0)
\]
satisfying associativity, unitality, and equivariance with respect to
the operad structure. The category of such algebras is denoted
$\Alg_{\mathcal{P}}(\bE)$.
\end{definition}

\begin{example}\label{def: operad for modular operads}
The \emph{operad of modular operads} is an
$\mathbb{N}\times\mathbb{N}_{\geq 0}$-coloured operad
$\mathcal{MOp}$. We use the convention that the colour $(g,r)$
corresponds to the $(g,r+1)$-component of a modular operad. The set
\[
\mathcal{MOp}\bigl((g_1,r_1),\ldots,(g_n,r_n);(g,r)\bigr)
\]
consists of strict isomorphism classes of connected genus-graded,
labelled graphs $G=(G,\lambda,\ell,\epsilon)$ with $n$ vertices
and $r+1$ legs, satisfying:
\begin{itemize}
    \item the vertex labelled by $i$ has valence
    $|\nb(v_{\lambda(i)})|=r_i+1$
    and genus $\epsilon(v_{\lambda(i)})=g_i$;

    \item the total genus is
    $g=\beta_1(G)+\sum_{i=1}^n g_i.$
\end{itemize}

For $\sigma\in\Sigma_n$, permutation of the vertex labelling induces
\[
\sigma^*:
\mathcal{MOp}\bigl((g_1,r_1),\ldots,(g_n,r_n);(g,r)\bigr)
\longrightarrow
\mathcal{MOp}\bigl(
(g_{\sigma(1)},r_{\sigma(1)}),\ldots,
(g_{\sigma(n)},r_{\sigma(n)});(g,r)\bigr),
\]
with the indexing determined by our convention for right
symmetric-group actions.

When the ambient category $\bE$ is understood, we regard this
set-valued coloured operad as an operad in $\bE$ by taking coproducts
of copies of the monoidal unit. For $\sSet$ and $\Grpd$, this is
simply the discrete-object functor.

\smallskip
Operadic composition is defined by graph substitution. Given
\[
G\in\mathcal{MOp}\bigl((g_1,r_1),\ldots,(g_n,r_n);(g,r)\bigr),
\qquad
H_{v_{\lambda(i)}}\in
\mathcal{MOp}\bigl((h_1,s_1),\ldots,(h_m,s_m);(g_i,r_i)\bigr),
\]
the composite $G\circ_iH_{v_{\lambda(i)}}$ is obtained by substituting
$H_{v_{\lambda(i)}}$ into the vertex $v_{\lambda(i)}$ of $G$:
\[
G\circ_iH_{v_{\lambda(i)}}:=G\{H_{v_{\lambda(i)}}\}.
\]
As in~\cite[Construction~1.18]{hry1}, this defines an element of
\[
\mathcal{MOp}\bigl(
(g_1,r_1),\ldots,(g_{i-1},r_{i-1}),
(h_1,s_1),\ldots,(h_m,s_m),
(g_{i+1},r_{i+1}),\ldots,(g_n,r_n);(g,r)\bigr).
\]

Associativity, equivariance, and unitality follow from the
corresponding properties of graph substitution. For a detailed
construction and verification, see~\cite[Appendix~A]{dch_coextension}
or~\cite[§13.4]{yau_modular}.
\end{example}

\begin{prop}\label{prop: modular operads are algebras over operads}
The category of algebras over the coloured operad $\mathcal{MOp}$
is naturally isomorphic to the category of modular operads in $\bE$:
\[
\Alg_{\mathcal{MOp}}(\bE)\cong \Mod(\bE).
\]
\end{prop}

This identification allows us to import general results about
algebras over coloured operads. In particular, under suitable
hypotheses on the base category $\bE$, it gives a transferred model
structure on $\Mod(\bE)$.

We will use this when $\bE$ is one of the following Cartesian
monoidal model categories:

\begin{itemize}
  \item the category of simplicial sets $\sSet$, equipped with the
  Kan--Quillen model structure;
  \item the category of groupoids $\Grpd$, equipped with the
  canonical model structure in which weak equivalences are
  equivalences of categories~\cite{And};
%  \item the category of profinite groupoids $\widehat{\Grpd}$,
%  equipped with the fibrantly generated model structure described
%  by Horel~\cite[Theorem~4.12]{Horel_profinite_groupoids}.
\end{itemize}

In each case, $\bE$ is a cofibrantly generated Cartesian monoidal
model category with cofibrant unit, a symmetric monoidal fibrant
replacement functor, and a cocommutative coalgebra interval.
Therefore Theorem~2.1 of \cite{bm_resolutions} applies and we have
the following.

\begin{thm}\label{model structure on Mod(E)}
For each of the categories $\bE$ listed above, the category
$\Mod(\bE)$ admits a cofibrantly generated model category structure
in which a map of modular operads
$f:\mathcal{P}\longrightarrow\mathcal{Q}$ is a weak equivalence,
respectively a fibration, if and only if each component
\[
f(g,n+1):\mathcal{P}(g,n+1)\longrightarrow\mathcal{Q}(g,n+1)
\]
is a weak equivalence, respectively a fibration, in $\bE$, for all
$g\geq0$ and $n\geq0$.
\end{thm}

In this model structure, the fibrant objects are precisely the
modular operads that are entrywise fibrant in $\bE$. In particular,
since every groupoid is fibrant in the canonical model structure,
every modular operad in $\Mod(\Grpd)$ is fibrant.

The cofibrant objects are less immediate to characterize. For
modular operads in groupoids, we use an adaptation of a criterion
of Horel~\cite{Horel_profinite_groupoids}. Let
$\ob:\Grpd\longrightarrow\Set$ denote the functor sending a
groupoid to its set of objects. It has a left adjoint
$\operatorname{Disc}$, sending a set to the corresponding discrete
groupoid, and a right adjoint $\operatorname{Codisc}$, sending a
set to the corresponding codiscrete groupoid. Since these functors
preserve finite products, they induce functors between modular
operads in sets and modular operads in groupoids, denoted by the
same symbols. Thus there are adjunctions
\[
\operatorname{Disc}:\Mod(\Set)\rightleftarrows\Mod(\Grpd):\ob,
\qquad
\ob:\Mod(\Grpd)\rightleftarrows\Mod(\Set):\operatorname{Codisc}.
\]
We call a morphism in $\Mod(\Set)$ \emph{entrywise surjective} if
it is surjective in every component $(g,n+1)$.

\begin{prop}\label{prop: cofibrant modular operad}
A morphism $u:\mathcal{P}\to\mathcal{Q}$ in $\Mod(\Grpd)$ is a
cofibration if and only if
\[
\ob(u):\ob(\mathcal{P})\longrightarrow\ob(\mathcal{Q})
\]
has the left lifting property with respect to every entrywise
surjective morphism in $\Mod(\Set)$.
\end{prop}

\begin{proof}
This is the modular-operad analogue of
\cite[Proposition~6.8]{Horel_profinite_groupoids}.

Suppose first that $u$ is a cofibration, and let
$v:\mathcal{R}\to\mathcal{S}$ be entrywise surjective. For every
$(g,n+1)$, the functor
\[
\operatorname{Codisc}(v(g,n+1)):
\operatorname{Codisc}(\mathcal{R}(g,n+1))
\longrightarrow
\operatorname{Codisc}(\mathcal{S}(g,n+1))
\]
is a trivial fibration of groupoids. Hence
$\operatorname{Codisc}(v)$ is a trivial fibration in $\Mod(\Grpd)$.
By the adjunction $\ob\dashv\operatorname{Codisc}$, any lifting
problem for $\ob(u)$ against $v$ corresponds to a lifting problem
for $u$ against $\operatorname{Codisc}(v)$, which admits a solution.

Conversely, suppose that $\ob(u)$ has the stated lifting property,
and let $p:\mathcal{X}\to\mathcal{Y}$ be a trivial fibration in
$\Mod(\Grpd)$. For every $(g,n+1)$, the functor
\[
p(g,n+1):\mathcal{X}(g,n+1)\longrightarrow\mathcal{Y}(g,n+1)
\]
is fully faithful and surjective on objects. Since $\ob(p)$ is
entrywise surjective, the induced lifting problem in $\Mod(\Set)$
admits a solution. Thus we obtain a lift on objects, compatible
with the modular-operad structure.

At each $(g,n+1)$, full faithfulness of $p$ extends this lift
uniquely to a functor on morphisms. For each modular-operad
operation, the two morphisms obtained by applying the operation
before and after the lift have the same source and target and the
same image under $p$. They are therefore equal. Thus the lift
commutes with the $\Sigma_n^+$-actions, the partial compositions
$\circ_i^j$, and the contractions $\xi_i^j$. It is consequently
a morphism of modular operads, so $u$ is a cofibration.
\end{proof}

\begin{cor}\label{cor: cofibrant modular operad}
A modular operad $\mathcal{P}\in\Mod(\Grpd)$ is cofibrant whenever
$\ob(\mathcal{P})$ has the left lifting property with respect to
every entrywise surjective morphism in $\Mod(\Set)$.
\end{cor}

\begin{proof}
Apply Proposition~\ref{prop: cofibrant modular operad} to the unique
morphism from the initial modular operad to $\mathcal P$.
\end{proof}

\medskip

The functors $\Pi_1:\sSet\to\Grpd$ and
$\cs:\Grpd\to\sSet$ preserve finite products and therefore induce an
adjunction
\begin{equation}\label{adj: modular operads}
\begin{tikzcd}
\Mod(\sSet) \arrow[r, shift left=1, "\Pi_1"] &
\Mod(\Grpd) \arrow[l, shift left=1, "\cs"].
\end{tikzcd}
\end{equation}
Since the model structures on modular operads are defined entrywise,
and $\cs$ sends isofibrations and trivial isofibrations of groupoids to
Kan fibrations and trivial fibrations of simplicial sets, the
adjunction~\eqref{adj: modular operads} is a Quillen adjunction.

\begin{lemma}\label{lemma: B is homotopically fully faithful}
The adjunction~\eqref{adj: modular operads} is a Quillen adjunction,
and its right derived functor is fully faithful. In particular, for all
$\mathcal{P},\mathcal{Q}\in\Mod(\Grpd)$, the induced map
\[
\mathbb{R}\Map_{\Mod(\Grpd)}(\mathcal{P},\mathcal{Q})
\longrightarrow
\mathbb{R}\Map_{\Mod(\sSet)}
\bigl(\cs\mathcal{P},\cs\mathcal{Q}\bigr)
\]
is a weak equivalence of spaces.
\end{lemma}

\begin{proof}
It remains to prove that the right derived functor of $\cs$ is fully
faithful. Since every groupoid is fibrant, every object of
$\Mod(\Grpd)$ is fibrant. Let $\mathcal{P}\in\Mod(\Grpd)$, and choose a
cofibrant replacement
\[
q:Q\cs\mathcal{P}\xrightarrow{\sim}\cs\mathcal{P}
\]
in $\Mod(\sSet)$. The derived counit at $\mathcal{P}$ is represented by
\[
\Pi_1Q\cs\mathcal{P}
\xrightarrow{\Pi_1q}
\Pi_1\cs\mathcal{P}
\xrightarrow{\epsilon_{\mathcal{P}}}
\mathcal{P}.
\]
Since every simplicial set is cofibrant, applying $\Pi_1$ entrywise
to $q$ gives an entrywise equivalence of groupoids. Moreover,
$\epsilon_{\mathcal{P}}:\Pi_1\cs\mathcal{P}\to\mathcal{P}$ is an
entrywise equivalence, since $\Pi_1\cs G\to G$ is an equivalence for
every groupoid $G$. Thus the derived counit is a weak equivalence, and
the right derived functor of $\cs$ is fully faithful.
\end{proof}

\subsection{Truncation of Modular Operads}
\label{subsec: truncations}

We now define genus truncations of modular operads by restricting the
coloured operad $\mathcal{MOp}$ governing modular operads. Recall that
our colours are pairs $(h,p)$, where $h$ is the genus and $p+1$ is the
number of boundary components.

For $g\geq0$, let $\mathcal{MOp}_{\leq g}$ be the coloured suboperad of
$\mathcal{MOp}$ whose colours are the pairs $(h,p)$ with
$0\leq h\leq g$ and $p\geq0$. Its operations are those operations of
$\mathcal{MOp}$ whose output colour has genus at most $g$. Thus
\[
\mathcal{MOp}_{\leq g}\bigl(
(h_1,p_1),\ldots,(h_n,p_n);(h,p)
\bigr)
\]
is the set of strict isomorphism classes of connected genus-graded,
labelled graphs $G=(G,\lambda,\ell,\epsilon)$ with $n$ vertices and
$p+1$ legs, such that the vertex labelled by $i$ has valence
$|\nb(v_{\lambda(i)})|=p_i+1$ and genus
$\epsilon(v_{\lambda(i)})=h_i$, and such that
$h=\beta_1(G)+\sum_i h_i\leq g$.

Graph substitution preserves the total output genus, and hence
$\mathcal{MOp}_{\leq g}$ is a coloured suboperad of
$\mathcal{MOp}$.

A \emph{genus-$g$ truncated modular operad} in $\bE$ is an algebra over
$\mathcal{MOp}_{\leq g}$. We write
\[
\Mod_{\leq g}(\bE):=\Alg_{\mathcal{MOp}_{\leq g}}(\bE)
\]
for the category of genus-$g$ truncated modular operads. Concretely,
an object of $\Mod_{\leq g}(\bE)$ consists of the entries
$\mathcal P(h,p+1)$ for $0\leq h\leq g$ and $p\geq0$, together with
all composition and contraction operations whose output has genus at
most $g$. Thus compositions
$\mathcal P(h_1,p_1+1)\otimes\mathcal P(h_2,p_2+1)
\to\mathcal P(h_1+h_2,p_1+p_2)$ are retained when
$h_1+h_2\leq g$ and $p_1+p_2\geq1$, while contractions
$\xi_i^j:\mathcal P(h,p+1)\to\mathcal P(h+1,p-1)$ are retained for
$0\leq i<j\leq p$, $p\geq2$, and $h+1\leq g$.

The inclusion of coloured operads
\[
\iota_g:\mathcal{MOp}_{\leq g}\hookrightarrow\mathcal{MOp}
\]
induces a restriction functor
\[
\trun{g}^*:=\iota_g^*:
\Mod(\bE)\longrightarrow\Mod_{\leq g}(\bE),
\]
which we call the \emph{genus-$g$ truncation functor}. It admits a left
adjoint $(\trun{g})_!$, given by extension of scalars, and a right
adjoint $(\trun{g})_*$ by the categorical-extension criterion of
Drummond--Cole and Hackney~\cite[Theorem~3.8]{dch_coextension}:
\begin{equation}\label{eq: truncation adjunction}
\begin{tikzcd}[column sep=large]
\Mod(\bE) \arrow[r, "\trun{g}^*"] &
\Mod_{\leq g}(\bE)
\arrow[l, bend right=35, "(\trun{g})_!"']
\arrow[l, bend left=35, "(\trun{g})_*"].
\end{tikzcd}
\end{equation}
The left adjoint generalizes the modular envelope; see~\cite{Ward},
\cite[Appendix~A]{dch_coextension}, and~\cite[§4]{GNPR} for related
constructions.

The inclusion $\iota_g$ is a categorical extension in the sense of
\cite[Definition~3.2]{dch_coextension}. Operations whose input and
output colours all have genus at most $g$ are unchanged by the
inclusion. If the output genus is at most $g$ and one input has genus
greater than $g$, then the operation space is empty, since
$h=\beta_1(G)+\sum_i h_i$ with all terms nonnegative. Thus the
full-faithfulness and ideal conditions of
\cite[Remark~4.16]{dch_coextension} are satisfied.

\begin{prop}\label{prop: truncation is Quillen functor}
Let $\bE$ be a cofibrantly generated Cartesian monoidal model category
with cofibrant unit, a symmetric monoidal fibrant replacement functor,
and a cocommutative coalgebra interval. Then both $\Mod(\bE)$ and
$\Mod_{\leq g}(\bE)$ admit cofibrantly generated model structures in
which weak equivalences and fibrations are defined entrywise.

Moreover, $\trun{g}^*$ preserves weak equivalences and fibrations, both
adjoints in~\eqref{eq: truncation adjunction} are fully faithful, and
there are canonical isomorphisms
\[
\trun{g}^*(\trun{g})_!\cong\id_{\Mod_{\leq g}(\bE)}
\qquad\text{and}\qquad
\trun{g}^*(\trun{g})_*\cong\id_{\Mod_{\leq g}(\bE)}.
\]
\end{prop}

\begin{proof}
The model structures follow from Theorem~2.1 of~\cite{bm_resolutions},
applied to $\mathcal{MOp}$ and $\mathcal{MOp}_{\leq g}$. Since weak
equivalences and fibrations are defined entrywise, restriction along
$\iota_g$ preserves both.

By the preceding discussion and
\cite[Theorem~3.8]{dch_coextension}, $\trun{g}^*$ has both adjoints.
The genus-additivity formula shows that no operation involving a colour
of genus greater than $g$ can contribute to an output of genus at most
$g$. Consequently, restricting either algebraic extension back to
$\mathcal{MOp}_{\leq g}$ recovers the original truncated algebra. The
displayed isomorphisms follow, and hence both adjoints are fully
faithful.
\end{proof}

As for modular operads, cyclic operads are algebras over a coloured
operad. The coloured operad governing cyclic operads is identified
with the genus-zero truncation
\[
\mathcal C\cong\mathcal{MOp}_{\leq0}.
\]
Its operations are represented by labelled unrooted trees, equivalently
by genus-zero graphs
$T\in\mathcal{MOp}_{\leq0}((0,k_1),\ldots,(0,k_n);(0,p))$.
Here $(0,p)$ corresponds to $p+1$ boundary components. See
\cite[Example~A3]{dch_coextension} or
\cite[Section~1.6.4]{lukacs2010cyclic}.

\begin{prop}\label{prop: cyc are genus 0 modular operads}
There is a natural isomorphism of categories
\[
\Alg_{\mathcal C}(\bE)\cong\Cyc(\bE).
\]
\end{prop}

As a consequence of Proposition~\ref{prop: cyc are genus 0 modular operads},
Theorem~2.1 of~\cite{bm_resolutions} applies to cyclic operads as well.
Thus $\Cyc(\bE)$ admits a cofibrantly generated model category
structure in which a map of cyclic operads
$f:\mathcal P\to\mathcal Q$ is a weak equivalence, respectively a
fibration, if and only if each component
$f(n):\mathcal P(n)\longrightarrow\mathcal Q(n)$ is a weak equivalence,
respectively a fibration, in $\bE$, for all $n\geq 0$. Moreover, under
the identification $\Mod_{\leq0}(\bE)\cong\Cyc(\bE)$, the genus-zero
truncation functor becomes
\[
\trun{0}^*:\Mod(\bE)\longrightarrow\Cyc(\bE).
\]
It preserves weak equivalences and fibrations, and hence is the right
adjoint in the Quillen adjunction
\[
(\trun{0})_!:\Cyc(\bE)\rightleftarrows\Mod(\bE):\trun{0}^*.
\]

\section{Homotopical properties of the nerve}
\label{sec: nerve is homotopically fully faithful}

Throughout this section, we take $\bE=\sSet$ or $\Grpd$ and write $\dMod(\bE)=\operatorname{Fun}(\bM^{\op},\bE)$ for the category of modular dendroidal objects. To study the homotopical properties of the nerve, we will use both the projective and Reedy model structures on this category. The projective model structure, denoted $\dMod(\bE)_p$, has weak equivalences and fibrations defined entrywise; see \cite[Theorem~11.6.1]{hirsch}.

Since $\bM$ is a generalised Reedy category \cite[Theorem~2.22]{hry1}, $\dMod(\bE)$ also admits the Reedy model structure, denoted $\dMod(\bE)_R$; see \cite[Theorem~1.6]{bm_reedy}. These two model structures present the same homotopy theory: by \cite[Lemma~3.9]{hry1}, the identity functor induces a Quillen equivalence
\begin{equation}\label{identity qe}
\begin{tikzcd}
\dMod(\bE)_p \arrow[r, shift left=1] & \dMod(\bE)_R \arrow[l, shift left=1].
\end{tikzcd}
\end{equation}

A modular dendroidal object $X$ is \emph{reduced} if
$X_{\updownarrow}=*$. Its reduction is defined by the
pushout
\[
\begin{tikzcd}[column sep=large]
X_{\updownarrow}\times\bM[\updownarrow]
\arrow[d] \arrow[r] &
\bM[\updownarrow]/\Aut_{\bM}(\updownarrow)
\arrow[d] \\
X \arrow[r] & X_*.
\end{tikzcd}
\]
Here the automorphism group acts on the representable by
postcomposition. The upper horizontal map is the projection
followed by the quotient map, and the left-hand map is
evaluation.

Evaluating at the exceptional edge gives
$(X_*)_{\updownarrow}=*$. Every map from $X$ to a reduced
object factors uniquely through $X_*$, so the reduction
functor $(-)_*:\dMod(\bE)\to\dMod(\bE)_*$ is left adjoint
to the inclusion of reduced objects. This category carries
the corresponding projective and Reedy model structures,
denoted $\dMod(\bE)_{*,p}$ and $\dMod(\bE)_{*,R}$,
respectively.

Recall the set of reduced Segal-core inclusions
\[
S=\{\Sc[G]_*\longrightarrow\bM[G]_* \mid G\in\bM\}.
\]
We make these maps weak equivalences by taking the left
Bousfield localisation of each model structure at $S$.
We denote the resulting model structures by
$\mathcal L_S\dMod(\bE)_{*,p}$ and
$\mathcal L_S\dMod(\bE)_{*,R}$. In the projective case,
the fibrant objects are precisely the projectively fibrant
reduced Segal objects; in the Reedy case, they are the
Reedy fibrant reduced Segal objects.

The identity adjunction~\eqref{identity qe} remains a
Quillen equivalence after localisation. Moreover, the two
localisations have the same weak equivalences, allowing us
to use either model structure when establishing local
equivalences.

\begin{definition}\label{def: infinity cat of infinity modular operads}
We write $\Mod_\infty(\bE)$ for the $\infty$-category presented by $\mathcal L_S\dMod(\bE)_{*,p}$ and call it the \emph{$\infty$-category of modular $\infty$-operads}. Equivalently, $\Mod_\infty(\bE)$ is presented by the relative category whose objects are reduced Segal objects in $\dMod(\bE)$ (Definition~\ref{def: infty modular}), whose morphisms are natural transformations, and whose weak equivalences are defined entrywise.
\end{definition}

\begin{thm}\label{thm: nerve is homotopically fully faithful}
For any $\mathcal P,\mathcal Q\in\Mod(\bE)$, the nerve
induces a weak equivalence
\[
\mathbb R\Map_{\Mod(\bE)}(\mathcal P,\mathcal Q) \longrightarrow \mathbb R\Map_{\Mod_\infty(\bE)}(\nerve\mathcal P,\nerve\mathcal Q).
\]
\end{thm}

We delay the proof of Theorem~\ref{thm: nerve is homotopically fully faithful} while we establish some intermediate results about the nerve of a free modular operad.

Let $G=(G,\lambda,\ell,\epsilon)$ be a labelled graph with vertices $v_{\lambda(1)},\ldots,v_{\lambda(k)}$. Write $\epsilon(v_{\lambda(i)})=g_i$ and $|\nb(v_{\lambda(i)})|=n_i+1$. Let $X_G$ be the genus-graded set with one generator $x_i\in X_G(g_i,n_i+1)$ for each labelled vertex and no other generators. We write $\mathcal F(X_G)$ for the free modular operad (Definition~\ref{def: free modular operad}) on $X_G$, where each $x_i$ generates a free $\Sigma_{n_i}^{+}$-orbit under relabelling of the incident half-edges at the corresponding vertex.

To compare representables with nerves of free modular
operads, we keep track of which generator labels each
vertex of $G$. We therefore define the decorated
representable
\[
Z_G:=\bM[G]\times\{(x_1,\ldots,x_k)\},
\]
where the second factor is the constant singleton containing the tuple of vertex generators. This records the choice of $x_i$ at $v_{\lambda(i)}$, which determines the comparison with $\nerve\mathcal F(X_G)$. We write $Z_{G,*}$ for its reduction.

The relation between $Z_{G,*}$ and $\mathcal F(X_G)$ is expressed by the adjunction for the modular dendroidal nerve. Namely, $\nerve:\Mod(\bE)\to\dMod(\bE)_*$ admits a left adjoint $\tau:\dMod(\bE)_*\to\Mod(\bE)$ satisfying $\tau(Z_{G,*})\cong\mathcal F(X_G)$. In particular, $\tau$ sends a reduced corolla representable to the free modular operad on one generator in the corresponding genus and valence. The counit $\tau\nerve\mathcal P\to\mathcal P$ is an isomorphism; see \cite[Theorem~3.6]{hry_modular_nerve}. We first check that this adjunction is compatible with the projective model structure.

\begin{prop}\label{quillen pair projective}
The adjunction
\[
\begin{tikzcd}
\dMod(\bE)_{*,p} \arrow[r, shift left=1, "\tau"] & \Mod(\bE) \arrow[l, shift left=1, "\nerve"]
\end{tikzcd}
\]
is a Quillen pair.
\end{prop}

\begin{proof}
Let $f:\mathcal P\to\mathcal Q$ be a fibration,
respectively a trivial fibration, in $\Mod(\bE)$.
For every graph $G\in\bM$, the strict Segal condition gives
\[
(\nerve f)_G\cong \prod_{v\in V_G}f\bigl(\epsilon(v),|\nb(v)|\bigr).
\]
Each factor is a fibration, respectively a trivial fibration, and finite products preserve both classes of maps. For the exceptional edge, the empty product gives the identity of the terminal object. Thus $\nerve$ preserves projective fibrations and trivial fibrations, and is right Quillen.
\end{proof}

The decorated representable $Z_G$ records the graph $G$ with its specified vertex generators. To describe the whole nerve $\nerve\mathcal F(X_G)$, we also consider graphs whose vertices are decorated by these generators,
allowing relabelling of the incident half-edges and repeated occurrences of a generator. These decorated graphs form a category indexing a homotopy-colimit resolution of the nerve.

Recall from Definition~2.24 that a decoration of a graph is a tuple assigning to each vertex an element of a genus-graded sequence in the appropriate genus and valence. Let $\mathscr A_G$ be the category whose objects are graphs $R\in\bM$ equipped with an $\mathcal F(X_G)$-decoration $p$ such that each vertex decoration belongs to the symmetric-group orbit of a generator of $X_G$. We regard $p$ as an element of $(\nerve\mathcal F(X_G))_R$. A morphism $(R,p)\to(R',p')$ is a graphical embedding $j:R\to R'$ in the sense of \cite[Definitions~1.13 and~1.34]{hry1} that preserves the decorations, meaning $j^*p'=p$. The exceptional edge carries the empty tuple.

These decorated graphs determine a diagram
\[
\mathscr A_G\longrightarrow\dMod(\bE),
\qquad
(R,p)\longmapsto\bM[R]\times\{p\},
\]
whose value at $(G,(x_1,\ldots,x_k))$ is $Z_G$.
By Yoneda, the decorations give compatible maps
$\bM[R]\times\{p\}\to\nerve\mathcal F(X_G)$.
The following lemma shows that the induced map from
the homotopy colimit of this diagram is an entrywise
weak equivalence.

\begin{lemma}\label{lemma: free modular nerve resolution}
The canonical map
\[
\operatorname*{hocolim}_{(R,p)\in\mathscr A_G}
\bigl(\bM[R]\times\{p\}\bigr)
\longrightarrow\nerve\mathcal F(X_G)
\]
is an entrywise weak equivalence in
$\operatorname{Fun}(\bM^{\op},\bE)$.
\end{lemma}

\begin{proof}
We first take $\bE=\sSet$ and consider the diagram
$(R,p)\mapsto\bM[R]\times\{p\}$ indexed by $\mathscr A_G$.
By Yoneda, each decoration determines a map
$\bM[R]\times\{p\}\to\nerve\mathcal F(X_G)$, given at
a graph $H$ by $(f,p)\mapsto f^*p$.
These maps are compatible with morphisms in $\mathscr A_G$:
if $j:(R,p)\to(R',p')$, then $j^*p'=p$, and hence
$(jf)^*p'=f^*p$.

Evaluating the simplicial replacement of this diagram
at $H$ gives the simplicial set with $n$-simplices
\[
\coprod_{(R_0,p_0)\longrightarrow\cdots\longrightarrow(R_n,p_n)}
\Hom_{\bM}(H,R_0)\times\{p_0\},
\]
which represents the homotopy colimit evaluated at $H$.
It is the nerve of the category of elements of
$(R,p)\mapsto\Hom_{\bM}(H,R)\times\{p\}$.
Its objects are triples $(R,p,f)$, and a morphism
$(R,p,f)\to(R',p',f')$ is a morphism
$j:(R,p)\to(R',p')$ making the triangle
\[
\begin{tikzcd}
& H \arrow[dl, "f"'] \arrow[dr, "f'"] & \\
R \arrow[rr, "j"'] && R'
\end{tikzcd}
\]
commute. The canonical map to
$(\nerve\mathcal F(X_G))_H$ sends $(R,p,f)$ to $f^*p$.

Fix $y\in(\nerve\mathcal F(X_G))_H$.
We show that the full subcategory defined by $f^*p=y$
has an initial object. If $H=\updownarrow$, this is
$(\updownarrow,\varnothing,\id)$, since every graphical
map from the exceptional edge is an embedding.

Suppose that $H$ has at least one vertex. For each
$v\in V_H$, choose a decorated graph $R_v$ representing $y_v$, with each vertex decoration belonging to the symmetric-group orbit of a generator of $X_G$ and with
its boundary identified with $\nb_H(v)$. If $y_v$ is the unit, take $R_v$ to be the exceptional edge. Graph substitution gives $R_y=H\{R_v\}_{v\in V_H}$, its decoration $p_y$, and a canonical active map $f_y:H\to R_y$ satisfying $f_y^*p_y=y$.
 
For another triple $(R,p,f)$ with $f^*p=y$, choose an embedding $R'_v\to R$ representing each vertex image $f_1(v)$. The induced decoration represents $y_v$, so Definition~\ref{def: free modular operad} gives a boundary-preserving isomorphism $R_v\cong R'_v$ respecting the decorations. Graph substitution then gives an embedding $j:R_y\to R$ with $jf_y=f$ and $j^*p=p_y$; see \cite[Proposition~1.38 and Lemma~1.41]{hry1}.

To show uniqueness, let $j_1,j_2:R_y\to R$ be two such
embeddings. Uniqueness of active--embedding factorisation
\cite[Proposition~2.14]{hry1} gives an isomorphism
$\theta:R_y\to R_y$ with $j_1=j_2\theta$ and
$\theta f_y=f_y$. Thus $\theta$ preserves $p_y$ and,
since $f_y$ is active, fixes the boundary pointwise.

If $R_y$ is the exceptional edge, this gives $\theta=\id$. Otherwise, a fixed boundary leg determines a fixed incident vertex. Its decoration lies in a free symmetric-group orbit, so preserving this decoration forces $\theta$ to fix every incident half-edge. Connectedness propagates this argument through the graph, giving $\theta=\id$. Hence $j_1=j_2$, and $(R_y,p_y,f_y)$ is initial.

Since morphisms of triples preserve $f^*p$, the category of elements is a disjoint union of contractible categories indexed by $(\nerve\mathcal F(X_G))_H$. The canonical map in the statement is therefore an entrywise weak equivalence.

For groupoids, apply the entrywise left derived functor of $\Pi_1$: it preserves homotopy colimits and sends the discrete simplicial objects in the statement to the corresponding discrete groupoids.
\end{proof}

\begin{lemma}\label{lemma: nerve of free operad}
For $G\in\bM$, the objects $Z_{G,*}$ and
$\nerve\mathcal F(X_G)$ are weakly equivalent in
$\mathcal L_S\dMod(\bE)_{*,R}$.
\end{lemma}

\begin{proof}
Decorating each vertex $v_{\lambda(i)}$ of $G$ by $x_i$
gives an element of $(\nerve\mathcal F(X_G))_G$.
By Yoneda, this element determines a map
$Z_G\to\nerve\mathcal F(X_G)$. Since the nerve is reduced,
the map factors uniquely through the reduction, giving
$Z_{G,*}\to\nerve\mathcal F(X_G)$.

We first take $\bE=\sSet$ and let $Y$ be a reduced,
Reedy fibrant Segal object. In what follows, derived
mapping spaces are computed in
$\operatorname{Fun}(\bM^{\op},\sSet)$.
Mapping the homotopy-colimit resolution of
Lemma~\ref{lemma: free modular nerve resolution} into $Y$
gives a homotopy limit of mapping spaces from decorated
representables. Yoneda identifies these mapping spaces
with the values $Y_R$, so
\begin{equation}\label{eq: free nerve mapping resolution}
\mathbb R\Map(\nerve\mathcal F(X_G),Y)
\simeq
\operatorname*{holim}_{(R,p)\in\mathscr A_G^{\op}}Y_R.
\end{equation}
The Segal maps
$Y_R\xrightarrow{\sim}\prod_{v\in V_R}Y_{C_v}$
are natural in graphical embeddings; see
\cite[Lemma~3.4 and Proposition~3.5]{hry1}.

Let $\mathscr A_{G,i}$ be the category of triples
$(R,p,v)$ with $(R,p)\in\mathscr A_G$ and $v\in V_R$,
such that $p_v$ lies in the symmetric-group orbit of
$x_i$. Its morphisms are the embeddings in $\mathscr A_G$
carrying the chosen vertex to the chosen vertex.
Each vertex decoration belongs to exactly one of these
orbits. Reindexing by the chosen vertices and separating
them according to $i$ therefore gives
\begin{equation}\label{eq: free nerve vertex reindexing}
\operatorname*{holim}_{(R,p)\in\mathscr A_G^{\op}}
\prod_{v\in V_R}Y_{C_v}
\simeq
\prod_{i=1}^{k}
\operatorname*{holim}_{(R,p,v)\in\mathscr A_{G,i}^{\op}}
Y_{C_v}.
\end{equation}

The corolla $C_{(g_i,n_i+1)}$ decorated by $x_i$
is initial in $\mathscr A_{G,i}$.
Indeed, freeness of the symmetric-group action gives
a unique embedding $j:C_{(g_i,n_i+1)}\to R$ onto
the corolla at $v$ with $j^*p=x_i$.
The induced corolla isomorphisms identify the diagram
$(R,p,v)\mapsto Y_{C_v}$ with the constant diagram at
$Y_{C_{(g_i,n_i+1)}}$.
Since $\mathscr A_{G,i}$ is contractible,
\eqref{eq: free nerve mapping resolution}
and~\eqref{eq: free nerve vertex reindexing} give
\begin{equation}\label{eq: free nerve generator mapping spaces}
\mathbb R\Map(\nerve\mathcal F(X_G),Y)
\simeq
\prod_{i=1}^{k}Y_{C_{(g_i,n_i+1)}}.
\end{equation}

Since $Y$ is reduced, the reduction adjunction identifies
$\mathbb R\Map(Z_{G,*},Y)$ with $\mathbb R\Map(Z_G,Y)$.
Yoneda identifies the latter with $Y_G$, and the Segal
condition identifies $Y_G$ with the product of its values
on the vertex corollas. Together with
\eqref{eq: free nerve generator mapping spaces}, these
restrictions give a commutative triangle
\[
\begin{tikzcd}[column sep=large]
\mathbb R\Map(\nerve\mathcal F(X_G),Y)
  \arrow[r]
  \arrow[dr, "\sim"']
&
\mathbb R\Map(Z_{G,*},Y)
  \arrow[d, "\sim"]
\\
&
\displaystyle\prod_{i=1}^{k}Y_{C_{(g_i,n_i+1)}}.
\end{tikzcd}
\]
The top map is induced by
$Z_{G,*}\to\nerve\mathcal F(X_G)$, so it is a weak
equivalence for every local fibrant $Y$.
This proves the simplicial-set case.

Entrywise $\Pi_1$ sends the Segal-core localisation maps
to their groupoid counterparts, so its left derived
functor preserves local equivalences. Since $Z_{G,*}$
and $\nerve\mathcal F(X_G)$ are discrete, their derived
images are represented by the corresponding
groupoid-valued objects. The simplicial-set comparison
therefore gives the asserted local equivalence for
$\bE=\Grpd$.
\end{proof}
More generally, let $\{G_i\}_{i\in I}$ be a family of graphs.
For each $i\in I$, the inclusion
$X_{G_i}\to\coprod_{j\in I}X_{G_j}$ induces a composite
\[
Z_{G_i,*}\longrightarrow\nerve\mathcal F(X_{G_i})
\longrightarrow
\nerve\mathcal F\left(\coprod_{j\in I}X_{G_j}\right).
\]
Together, these maps define the canonical map
\begin{equation}\label{eq: free nerve coproduct comparison}
\coprod_{i\in I}Z_{G_i,*}
\longrightarrow
\nerve\mathcal F\left(\coprod_{i\in I}X_{G_i}\right),
\end{equation}
where the source coproduct is taken in reduced modular
dendroidal objects. This map is a local equivalence:
the preceding proofs apply to the combined set of
generators $\coprod_iX_{G_i}$, allowing decorations
from different summands on the same graph.
For every local fibrant $Y$, restriction to the
decorated vertex corollas identifies the derived
mapping spaces from both sides with
$\prod_i\prod_{v\in V_{G_i}}Y_{C_v}$.
For the source, we use that the $Z_{G_i,*}$ are
projectively cofibrant, so their coproduct computes
a homotopy coproduct.

\begin{lemma}\label{derived unit}
If $X\in\dMod(\bE)_{*,p}$ is cofibrant, then the
derived unit
\[
X\longrightarrow\nerve\bigl(\tau(X)_f\bigr)
\]
is a local equivalence in
$\mathcal L_S\dMod(\bE)_{*,p}$.
\end{lemma}

\begin{proof}
For $\bE=\sSet$, choose a simplicial resolution
$X_\bullet$ of $X$ by coproducts of reduced
representables. This is obtained from the
free--forgetful resolution of reduced presheaves
in each simplicial degree by taking the diagonal;
compare \cite[Section~9.2 and Proposition~9.11]{dugger_hocolim}.
The degreewise augmentations split after evaluation
at each graph, and the latching maps are inclusions
of coproduct summands. Thus $X_\bullet$ is Reedy
cofibrant in the projective model structure.

Here $r$ denotes the resolution degree, with
$X_r=\coprod_{i\in I_r}Z_{G_{r,i},*}$.
The augmentation induces an entrywise weak equivalence
$|X_\bullet|\to X$, where realisation takes the diagonal
of the two simplicial directions. In particular,
\begin{equation}\label{eq: simplicial resolution of X}
X\simeq|X_\bullet|
\simeq\operatorname*{hocolim}_{[r]\in\Delta^{\op}}X_r.
\end{equation}

Since $\tau$ is left Quillen, $\tau X_\bullet$ is
Reedy cofibrant, with
$\tau X_r\cong\mathcal F(\coprod_{i\in I_r}X_{G_{r,i}})$.
Moreover, $\tau$ acts degreewise and therefore commutes
with realisation. Applying it to the augmentation
between cofibrant objects gives a weak equivalence
$|\tau X_\bullet|\cong\tau|X_\bullet|\to\tau X$.

Since the entrywise diagonal preserves finite products,
$\nerve|\tau X_\bullet|\cong|\nerve\tau X_\bullet|$.
The latter realisation computes the homotopy colimit
in the reduced projective model structure: a Reedy
cofibrant replacement of $\nerve\tau X_\bullet$ is an
entrywise weak equivalence in each resolution degree,
and taking diagonals preserves these weak equivalences.
As $\nerve$ also preserves entrywise weak equivalences,
we obtain
\begin{equation}\label{eq: nerve and homotopy colimit}
\nerve\bigl(\tau(X)_f\bigr)
\simeq
\operatorname*{hocolim}_{[r]\in\Delta^{\op}}
\nerve(\tau X_r).
\end{equation}

By \eqref{eq: free nerve coproduct comparison}, the unit
$X_r\to\nerve\tau X_r$ is a local equivalence for every $r$.
Taking homotopy colimits and using naturality of the unit
gives the commutative square
\[
\begin{tikzcd}[column sep=large]
\mid X_\bullet\mid \arrow[r, "\sim"] \arrow[d] & X \arrow[d] \\
\mid \nerve\tau X_\bullet\mid \arrow[r, "\sim"'] & \nerve\bigl(\tau(X)_f\bigr).
\end{tikzcd}
\]
The top map is the augmentation, the bottom map is induced
by $|\tau X_\bullet|\to\tau X\to\tau(X)_f$, and the
vertical maps are induced by the unit. The left-hand
map is a local equivalence, while the horizontal maps
are entrywise weak equivalences. Hence the right-hand
map, which is the derived unit at $X$, is a local
equivalence.
\end{proof}

\begin{remark}
For $\bE=\Grpd$, let $X$ be projectively cofibrant
and choose a cofibrant replacement
$A\xrightarrow{\sim}\cs X$ in $\dMod(\sSet)_{*,p}$.
Then $\Pi_1A\to X$ is an entrywise equivalence
between cofibrant objects.

As in the proof of
Lemma~\ref{lemma: nerve of free operad},
entrywise $\Pi_1$ is left Quillen after localisation.
Preservation of finite products and the corresponding
adjunctions give
$\Pi_1\nerve\cong\nerve\Pi_1$ and
$\Pi_1\tau\cong\tau\Pi_1$.
Applying $\Pi_1$ to the simplicial-set case therefore
shows that $\Pi_1A\to\nerve\tau(\Pi_1A)$ is a local
equivalence; the simplicial fibrant replacement
does not affect this conclusion because $\Pi_1$
preserves weak equivalences.

Naturality of the unit gives a commutative square
\[
\begin{tikzcd}[column sep=large]
\Pi_1A \arrow[r, "\sim"] \arrow[d]
&
X \arrow[d]
\\
\nerve\tau(\Pi_1A) \arrow[r, "\sim"']
&
\nerve\tau X.
\end{tikzcd}
\]
The left-hand map is a local equivalence by the
simplicial-set case. The top map is an entrywise
equivalence, and so is the bottom map, since $\tau$
preserves weak equivalences between cofibrant objects
and $\nerve$ preserves entrywise weak equivalences.
Hence the right-hand map is a local equivalence.
Every modular operad in groupoids is fibrant, so this
is the required derived unit.
\end{remark}

\begin{proof}[Proof of Theorem~\ref{thm: nerve is homotopically fully faithful}]
By Proposition~\ref{quillen pair projective},
$\tau\dashv\nerve$ is a Quillen adjunction for the
projective model structure.
The nerve of a fibrant modular operad is
projectively fibrant and satisfies the Segal
condition strictly. Hence the adjunction descends
to a Quillen adjunction
\[
\begin{tikzcd}
\mathcal L_S\dMod(\bE)_{*,p}
\arrow[r, shift left=1, "\tau"] &
\Mod(\bE)
\arrow[l, shift left=1, "\nerve"] .
\end{tikzcd}
\]

Let $\mathcal P$ be fibrant and choose a cofibrant
replacement
$q:Q\nerve(\mathcal P)\xrightarrow{\sim}
\nerve(\mathcal P)$ in $\dMod(\bE)_{*,p}$.
The derived counit is represented by
$e=\epsilon_{\mathcal P}\circ\tau q$.

Choose a fibrant replacement
$j:\tau Q\nerve(\mathcal P)\xrightarrow{\sim}
(\tau Q\nerve(\mathcal P))_f$
that is a trivial cofibration.
Since $\mathcal P$ is fibrant, $e$ extends to
$e_f:(\tau Q\nerve(\mathcal P))_f\to\mathcal P$.
The triangle identity gives
\[
\nerve(e_f)\circ\nerve(j)\circ
\eta_{Q\nerve(\mathcal P)}=q.
\]
Lemma~\ref{derived unit} implies that
$\nerve(j)\circ\eta_{Q\nerve(\mathcal P)}$
is a local equivalence. Since $q$ is also a
local equivalence, so is $\nerve(e_f)$.

Both its source and target are local fibrant,
so $\nerve(e_f)$ is an entrywise weak equivalence;
see \cite[3.2.13]{hirsch}.
Evaluating at $C_{(g,n+1)}$ shows that $e_f$ is
a weak equivalence in $\Mod(\bE)$, and hence
so is $e=e_f\circ j$.
Thus the derived counit is a weak equivalence,
and the right derived functor of $\nerve$ is
fully faithful.
\end{proof}
\section{The complex of markings over pants decompositions}\label{app: complex of markings}
In this appendix, we prove that the complex of markings $\calM(\surf)$ introduced in Section~\ref{sec: complex of markings} is connected and simply connected. This is the complex used in the proof of Theorem~\ref{thm: presentation for maps out of bS}.

The $2$-dimensional CW complex $\calM(\surf)$ is closely related to the complex $\Mmax(\surf)$ of \cite{bk_marked_surfaces}. The main difference is that $\calM(\surf)$ only retains markings over pants decompositions: disc and cylinder components are not allowed. Following the strategy of \cite{bk_marked_surfaces,complex_of_pants_decomp}, our proof proceeds by comparing $\calM(\surf)$ with a simply connected complex via a forgetful map and applying two general comparison results.

The first lemma allows us to deduce simple connectivity of the target from that of the source.

\begin{lemma}\label{lemma: simply-connect target from domain}
    Let $A$ be a simply connected CW complex, and let $f:A\to B$ be a continuous map such that every loop in $B$ lifts to a loop in $A$. Then $B$ is simply connected. \qed
\end{lemma}

The following proposition gives a criterion for deducing simple connectivity of the source from that of the target.

\begin{prop}\label{prop: simply-connect domain from target}
    Let $A$ and $B$ be $2$-dimensional CW complexes, and let $\pi:A^{[1]}\to B^{[1]}$ be a surjective map of their $1$-skeleta. Suppose that
    \begin{enumerate}
        \item $B$ is connected and simply connected.
        \item For every vertex $b\in B$, the fiber $\pi^{-1}(b)$ is connected, and every loop contained in $\pi^{-1}(b)$ is contractible in $A$.
        \item Let $b_1\xra{e} b_2$ be an edge in $B$, and let $a_1'\xra{e'}a_2'$ and $a_1''\xra{e''}a_2''$ be two lifts to $A$. Then there exist paths $a_1'\xra{e_1}a_1''$ in $\pi^{-1}(b_1)$ and $a_2'\xra{e_2}a_2''$ in $\pi^{-1}(b_2)$ such that the square
            \[
            \begin{tikzcd}
            	{a_1'} & {a_2'} \\
            	{a_1''} & {a_2''}
            	\arrow["{e'}", from=1-1, to=1-2]
            	\arrow["{e_1}"', from=1-1, to=2-1]
            	\arrow["{e_2}", from=1-2, to=2-2]
            	\arrow["{e''}", from=2-1, to=2-2]
            \end{tikzcd}
            \]
        is contractible in $A$.
        \item For every $2$-cell $X$ in $B$, its boundary $\partial X$ can be lifted to a contractible loop in $A$.
    \end{enumerate}
    Then $A$ is connected and simply connected. \qed
\end{prop}

The proofs of the preceding two results are standard. Proposition~\ref{prop: simply-connect domain from target} is established by a cellular lifting argument: one projects a loop in $A$ to $B$ and lifts a nullhomotopy back to $A$ cell by cell. Conditions~\emph{(2)} and~\emph{(3)} are used to compare different choices of lifts, and condition~\emph{(4)} is applied to lift the $2$-cells of $B$. We omit the details.

We now recall the pants complex of Hatcher--Lochak--Schneps \cite{hls} in a form suited to our setting. Let $S=(S,\delta)$ be a labelled surface of type $(g,n+1)$, with $n\ge 0$. Recall that a \emph{pants decomposition} of $S$ is a cut system $C=\{c_1,\ldots,c_k\}$ of pairwise disjoint simple loops on $S$ such that
    \[\overline{S\setminus \bigcup_{r=1}^k c_r}\;=\;\coprod_j S_j\]
is a disjoint union of subsurfaces $S_j$, each of which is of type $(0,3)$. The \emph{pants complex} of $S$, denoted $\pantscpx(S)$, is the $2$-dimensional CW complex whose vertices are isotopy classes of pants decompositions on $S$, whose edges are generated by $A$-moves and $S$-moves between pants decompositions, and whose $2$-cells are attached along loops of type $(3A)$, $(5A)$, $(3S)$, $(6AS)$, and $(C)$, in the sense of \cite[Section~2]{hls}.

By \cite[Theorem~2]{hls}, the complex $\pantscpx(S)$ is connected and simply connected.

Recall from Definition~\ref{def: complex of markings} that the complex of marked pants decompositions on $S$, denoted $\markS$, has vertices given by marked pants decompositions $(C,[m])$ on $S$. The $1$-cells are generated by $T$-, $A$-, $B$-, and $S$-moves. The $2$-cells are slight modifications of those appearing in \cite{bk_marked_surfaces}. Concretely, we include the following families of $2$-cells:
\begin{enumerate}
    \item \textbf{Commutativity of disjoint moves:} if $E_1$ and $E_2$ are moves supported in disjoint subsurfaces of $S$, then
                \[
                E_1E_2 = E_2E_1
                \]
                \cite[Section~4.7]{bk_marked_surfaces}.
          
    \item \textbf{Boundary twist relations:} if $(C,[m])\overset{E}{\rightsquigarrow}(C',[m'])$ is a move supported on a subsurface $S'$ and $a\in C$ is a curve on the boundary of $S'$, then if $E=B_{a,b}$
                \[E \; T_a= T_b\; E\]
            otherwise, we have
                \[E\; T_a=T_a \; E.\]
    
    \item \textbf{Pentagon relation:} whenever there exists a subsurface $S'\subset S$ of type $(0,5)$ together with a diffeomorphism from $S'$ to the standard surface of Figure~\ref{fig:pentagon} carrying the relevant curves and markings to those represented in the figure, we attach a $2$-cell along the boundary loop of the pentagon diagram.
        \begin{figure}[ht]
            \centering
            \includegraphics[width=0.5\linewidth]{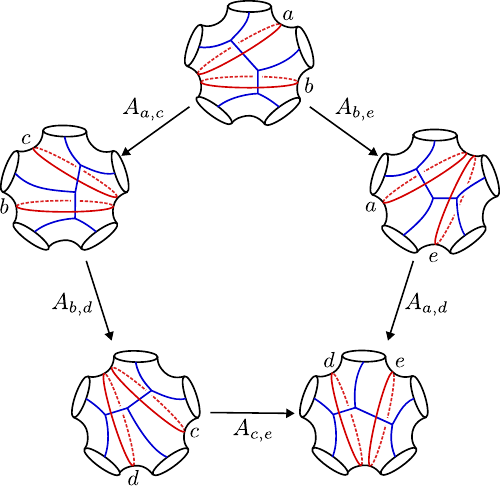}
            \caption{Pentagon relation}
            \label{fig:pentagon}
        \end{figure}
    % \cite[Figure 30, Appendix~C]{bk_marked_surfaces}.

    \item \textbf{Hexagon relations:} whenever there exists a subsurface $S'\subset S$ of type $(0,4)$ together with a diffeomorphism from $S'$ to the standard surface $S_{0,4}$ of Figure~\ref{fig:hexagons} carrying the relevant curves and markings to those represented in the figure, we attach $2$-cells along the two hexagonal loops of Figure~\ref{fig:hexagons}.

    % \begin{figure}[ht]
    %     \includegraphics[width=0.55\linewidth]{Figures/hexagonI.pdf}\hfill\hfill \\
    %     \hfill\includegraphics[width=0.55\linewidth]{Figures/hexagonII.pdf}
    %     \caption{Hexagon II}
    %     \label{fig:hexagonI}
    % \end{figure}
    \begin{figure}[ht]
        % Top-left
        \makebox[\textwidth][l]{%
            \includegraphics[width=0.735\linewidth]{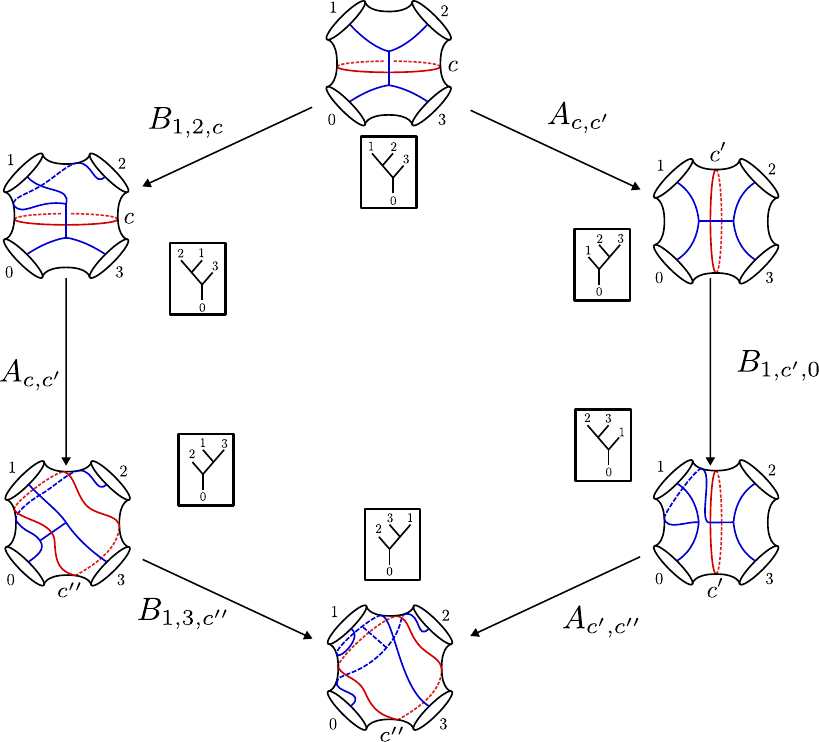}}
        % \vspace{-2.5cm}
        % Bottom-right
        \makebox[\textwidth][r]{%
            \includegraphics[width=0.735\linewidth]{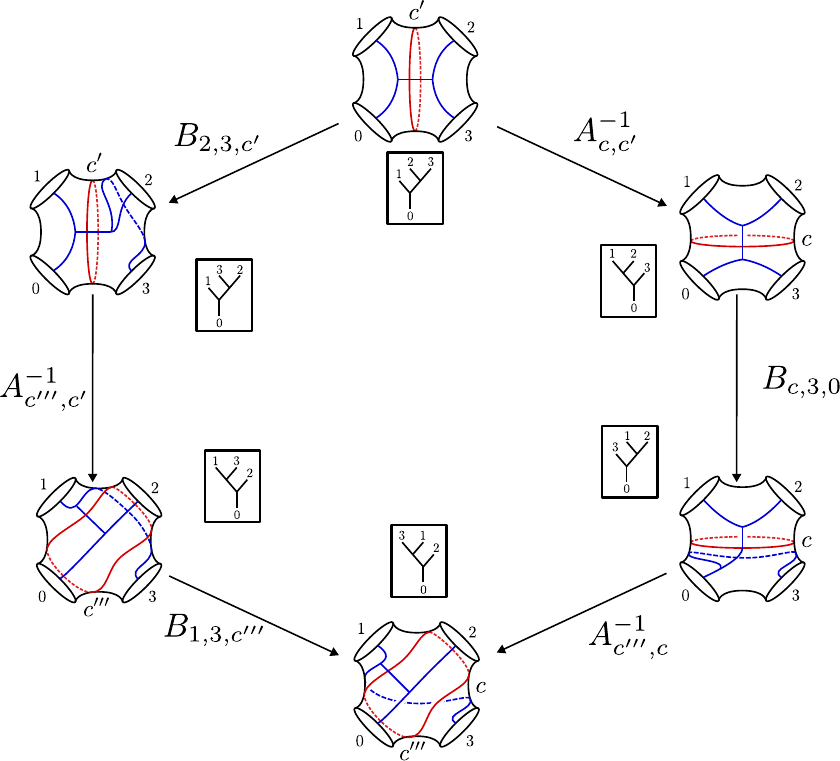}}
        \caption{Hexagon relations: Hexagon I (top left) and Hexagon II (bottom right)}
        \label{fig:hexagons}
    \end{figure}
    
    \item \textbf{Self-duality of associativity:}
        \[A_{c,c'}A_{c',c}=\id.\]

    \item \textbf{Twist--braiding relation:}
        \[B_{1,2}B_{2,1}=T_0T_1^{-1}T_2^{-1}.\]

    \item \textbf{Gluing relations for Dehn twists:} 
        \begin{itemize}
            \item if two pairs of pants $S_1$ and $S_2$ are glued along a curve $a$, then
                \[\id_{S_1}\cup_a T_a = T_a\cup_a \id_{S_2}\]
            where the $T$-move is supported on $S_2$ on the left-hand side and on $S_1$ on the right-hand side.
            \item if a pair of pants $S'$ has two boundary components $a$ and $b$ that get identified in $S$, then 
                \[T_a = T_b.\]
        \end{itemize}

    \item \textbf{Genus-one relations for $(g,n)=(1,1)$:} as in Figure~\ref{fig:G2-cells-G1} in Definition~\ref{def: complex of markings}.

    \item \textbf{Genus-one relation for $(g,n)=(1,2)$:} as in Figure~\ref{fig:G2-cell-NS-complex-of-markings} in Definition~\ref{def: complex of markings}.
\end{enumerate}

\medskip

The complex of markings over the pair of pants $P$, will be of particular relevance for us.

\begin{lemma}\label{lemma: marked complex for pants is simply connected}
Let $P$ be a pair of pants. Then $\mathcal{M}(P)$ is connected and simply connected.
\end{lemma}

\begin{proof}
For a pair of pants $P$, a vertex of $\mathcal M(P)$ is simply a marked pants decomposition of the form $(\emptyset,[m]),$ since the only pants decomposition of $P$ is the empty cut system.

Let $\mathcal M^{\max}(P)$ denote the complex of maximal markings on $P$ in the sense of \cite[Section~5]{bk_marked_surfaces}. A vertex of $\mathcal M^{\max}(P)$ is given by an isotopy class of a cut system
\[
C=\{c_1,\ldots,c_k\}
\]
together with a marking, such that cutting $P$ along $C$ decomposes it into pairs of pants, possibly together with cylinder and disc components. In addition, one records a distinguished half-edge at each vertex of the marking graph.

We compare $\mathcal M^{\max}(P)$ with $\mathcal M(P)$ by means of a map
\[
q:\mathcal M^{\max}(P)\longrightarrow \mathcal M(P).
\]
On vertices, the map $q$ forgets the cut system $C$ and the distinguished half-edges, and contracts those edges of the marking graph that are not incident to a boundary component, thereby producing a vertex $(\emptyset,[m])\in \mathcal{M}(P)$.

On edges, the map $q$ behaves as follows. Moves of type $T$ and $B$ are either contracted to points or taken to the corresponding edges of $\mathcal M(P)$. By contrast, moves of type $F$ forget cuts bounding cylindrical components and therefore become trivial after passing to the empty cut system. Likewise, $A$-moves are supported on a subsurface of type $(0,4)$, and also are contracted to points via $q$. Finally, a $Z$-move changes only the choice of distinguished half-edge, which is not retained in $\mathcal M(P)$. Thus $A$-, $F$-, and $Z$-edges are all sent to vertices.

This defines a cellular map on the $1$-skeleton of $\mathcal M^{\max}(P)$. Moreover, each defining $2$-cell of $\mathcal M^{\max}(P)$ maps either to the corresponding $2$-cell of $\mathcal M(P)$ or to a degenerate loop. Hence $q$ extends to a cellular map
\[
q:\mathcal M^{\max}(P)\to \mathcal M(P).
\]

We next show that every loop in $\mathcal M(P)$ lifts to a loop in $\mathcal M^{\max}(P)$. Given a cellular loop in $\mathcal{M}(P)$, each vertex in the loop is a marking of the form $(\emptyset,[m])$. Any such vertex can be lifted to one in $\mathcal{M}^{\max}(P)$ by choosing a distinguished half-edge for the marking $[m]$. For instance, we may choose the half-edge incident to the boundary component labelled $0$. In this way, all vertices in the loop in $\mathcal{M}(P)$ lift to vertices in $\mathcal{M}^{\max}(P)$.

Since every edge of $\mathcal M(P)$ is a $T$- or $B$-move, each $e_i$ lifts to an edge path in $\mathcal M^{\max}(P)$ consisting of the corresponding $T$- or $B$-move, together with $Z$-moves, if necessary, to adjust the distinguished half-edge before or after performing that move. Concatenating these lifted edge paths yields a loop in $\mathcal M^{\max}(P)$ projecting to the original loop in $\mathcal M(P)$.

By \cite[Theorem~4.9]{bk_marked_surfaces}, the complex $\mathcal M^{\max}(P)$ is connected and simply connected. Lemma~\ref{lemma: simply-connect target from domain} therefore implies that $\mathcal M(P)$ is hence connected and simply connected as well.
\end{proof}

To study the complex of markings $\markS$ for more general surfaces, we compare $\markS$ with $\pantscpx(S)$ via the natural forgetful map $F:\markS\longrightarrow \pantscpx(S)$ defined on vertices by
\[
F(C,[m])=C,
\]
that is, by sending a marked pants decomposition to its underlying pants decomposition.

We use this to prove the main result of this section:

\begin{prop*}[Proposition~\ref{prop:our-complex-of-markings-is-simply-connected}]
    Let $S=(S, \delta, \{x_{i}\}_{i=0}^{n})$ be a parametrised, labelled surface of type $(g,n+1)$, $n\geq 0$. Then $\markS$, the complex of markings over pants decompositions on $S$, is connected and simply connected.
\end{prop*}

\begin{proof}
Let
\[
F:\markS\longrightarrow \pantscpx(S)
\]
be the forgetful map sending a marked pants decomposition $(C,[m])$ to its underlying pants decomposition $C$. By \cite[Theorem~2]{hls}, the pants complex $\pantscpx(S)$ is connected and simply connected. We now verify the hypotheses of Proposition~\ref{prop: simply-connect domain from target} for the map $F$. By construction, $F$ is surjective on vertices and edges.

\emph{Checking condition~$(2)$:} Let $C$ be a vertex of $\pantscpx(S)$, that is, a pants decomposition of $S$. The fibre $F^{-1}(C)$ is the subcomplex of $\markS$ whose vertices are markings with underlying pants decomposition $C$, and whose edges and cells are precisely those built from moves that leave $C$ unchanged. In particular, its $1$-cells are the $T$- and $B$-moves, and its $2$-cells are the relations among these moves that do not alter the underlying pants decomposition, namely commutativity of disjoint moves, the boundary twist relations, the twist--braiding relation, and the Gluing relations for Dehn twists.

We next show that $F^{-1}(C)$ is connected and simply connected. Let $S_1,\dots,S_k$ be the pair-of-pants components obtained as the closures of the connected components of $S\setminus C$. For each curve $c\in C$, choose a basepoint on $c$. These basepoints induce marked points on the corresponding boundary components of the $S_i$, so that each $S_i$ becomes a parametrised, labelled surface of type $(0,3)$ (Definition~\ref{def: labelled surface}). In particular, a marking on each $S_i$ may be viewed as an embedded corolla meeting each boundary component at the chosen marked point.

Using these basepoints, we may glue markings on the $S_i$ along the cut curves: if two boundary components of $S_i$ and $S_j$ arise from the same curve $c\in C$, then we identify them so that the marked points determined by the chosen basepoint of $c$ agree. In this way we obtain a map
\[
\mathcal M(S_1)\times\cdots\times \mathcal M(S_k)\longrightarrow F^{-1}(C),
\]
sending a tuple of markings on the pair-of-pants components to the induced marking on $S$ with underlying pants decomposition $C$. Since each $\mathcal M(S_i)$ is connected and simply connected by Lemma~\ref{lemma: marked complex for pants is simply connected}, the product
    \begin{align}\label{eq:gluing-map-for-markings}
        \mathcal M(S_1)\times\cdots\times \mathcal M(S_k)
    \end{align}
is also connected and simply connected.

The gluing map is surjective on vertices: any marking on $S$ with underlying pants decomposition $C$ can be isotoped so that it intersects each curve $c\in C$ at its chosen marked point. Such a marking then restricts, after cutting along $C$, to markings on the pair-of-pants components $S_1,\dots,S_k$, and gluing these restrictions recovers the original marking. It is also surjective on $1$-cells, since every edge of $F^{-1}(C)$ is a $T$- or $B$-move, hence is supported on a single pair-of-pants component, and therefore is obtained by applying the corresponding edge in one factor $\mathcal M(S_i)$ while leaving the other factors fixed.

The gluing map is not injective on vertices, because isotopies of the glued surface $S$ are not required to preserve the chosen basepoints on the cut curves. In fact, the map \eqref{eq:gluing-map-for-markings} is the quotient of $$\mathcal{M}(S_1)\times \dots\times \mathcal{M}(S_k)$$ by an action of $\Z[C]$, the free abelian group generated by the curves of the cut system $C$. For each $c\in C$, the corresponding generator acts on a tuple
\[
(m_1,\dots,m_k)\in \mathcal M(S_1)\times\cdots\times \mathcal M(S_k)
\]
as follows. If $c$ bounds two distinct components $S_i$ and $S_j$, then
\[
c\cdot(m_1,\dots,m_k)
=
(m_1,\dots,D_c m_i,\dots,D_c^{-1}m_j,\dots,m_k).
\]
If $c$ corresponds to two boundary components $c_1,c_2$ of a single component $S_i$, then
\[
c\cdot(m_1,\dots,m_k)
=
(m_1,\dots,D_{c_1}D_{c_2}^{-1}m_i,\dots,m_k).
\]

It is simple to check that this is a well-defined action on the vertices of $\mathcal M(S_1)\times\cdots\times \mathcal M(S_k)$ which extends to a free and properly discontinuous action on the entire space. Then the quotient has fundamental group $\Z^C$ generated by loops of the type 
\[
(\id_{S_i}\cup_c T_c)\,(T_c^{-1}\cup_c \id_{S_j}),
\]
and similarly in the nonseparating case. Since these loops are precisely the boundaries of the 2-cells (Gluing relations for Dehn twists) in $F^{-1}(C)$, we conclude that $F^{-1}(C)$ is connected and simply connected.

It remains to verify conditions~$(3)$ and~$(4)$ of Proposition~\ref{prop: simply-connect domain from target}.

\emph{Checking condition~$(3)$:} this concerns comparing different lifts of edges in $\pantscpx(S)$. Recall that the edges of $\pantscpx(S)$ are given by $A$- and $S$-moves. We begin with the former.

Let $C\xra{e}C'$ be an $A$-move in $\pantscpx(S)$ exchanging the curve $a\in C$ with $a'\in C'$, and let $e_1$ and $e_2$ be two lifts of $e$. Then $e_1$ and $e_2$ must be edge-paths of the form $\gamma'\circ A_{a,a'}^{\pm 1} \circ \gamma$, where $\gamma$ and $\gamma'$ are edge-paths in $F^{-1}(C)$ and $F^{-1}(C')$, respectively. Since Condition~$(2)$ has already been verified, we may assume that 
    \[e_1:(C,m_1)\xra{}(C',m_1') \quad\text{and}\quad e_2:(C,m_2)\xra{}(C',m_2')\]
are single edges corresponding to $A$-moves. By the $2$-cell (Self-duality of associativity), we may assume without loss of generality that these are $A$-moves, rather than their inverses.  

For a fixed marking over the cut system $C$, there is a unique $A$-move exchanging $a\in C$. We now analyse how the markings $m_1$ and $m_2$ may differ.

By the definition of the $A$-move (Definition~\ref{def: complex of markings}), there exist diffeomorphisms $\phi_1,\phi_2:S'\to S_{\forkone}$ such that $\phi_1(a)=\phi_2(a)=c$, and both $\phi_1\circ m_1$ and $\phi_2\circ m_2$ define the standard marking on $S_{\forkone}$. Moreover, we have $\alpha\phi_1(a')=\alpha\phi_2(a')=c'$, and both $\alpha\circ\phi_1\circ m_1'$ and $\alpha\circ\phi_2\circ m_2'$ define the standard marking on $S_{\forktwo}$. By Lemma~\ref{lemma:aux-for-A-move}, we may further assume that $\phi_1$ and $\phi_2$ send the same boundary components of $S'$ to the components $\{1,2\}$ of $S_{\forkone}$.

It follows that $\phi_1\circ \phi_2^{-1}$ is a diffeomorphism of $S_{\forkone}$ preserving each component of the standard pants decomposition. In particular, it is isotopic to a diffeomorphism fixing the curve $c$ pointwise. Thus, we may regard $\phi_1\circ \phi_2^{-1}$ as a pair of diffeomorphisms of the two pairs of pants in $S'$, that fix the boundary components of the pants corresponding to $c$. From the description of the mapping class group of a pair of pants, it follows that $\phi_1\circ \phi_2^{-1}$ is a composition of Dehn twists along $c$ and the other boundary components, together with possible braiding morphisms that fix $c$.

On the other hand,
    \[ \alpha\phi_1(\alpha\phi_2)^{-1}=\alpha(\phi_1\phi_2^{-1})\alpha^{-1}\]
is a diffeomorphism of $S_{\forktwo}$ fixing $c'$. This implies that $\phi_1\circ \phi_2^{-1}$ also fixes $\alpha^{-1}(c')$.  In particular, by considering intersection numbers, we see that $\phi_1\circ \phi_2^{-1}$ cannot involve Dehn twists about $c$, since such twists would not preserve $\alpha^{-1}(c')$.

It follows that $\phi_1\circ \phi_2^{-1}$ is either a composition of Dehn twists along boundary components, or such a composition followed by the pair of braidings (one on each pair-of-pants component) represented on the right-hand side of Figure~\ref{fig:standard-markings-A-move}. Hence we may assume that $m_1$ and $m_2$ differ from one of the two markings in Figure~\ref{fig:standard-markings-A-move} only by Dehn twists along the boundary components.

% By analysing how braidings and Dehn twists affect the intersection number of $\alpha^{-1}(c')$ with its image, we conclude that applying $\phi_1\circ \phi_2^{-1}$ to the standard marking on $S_{\forkone}$ produces one of the two markings shown in Figure~\ref{fig:standard-markings-A-move}, possibly composed with Dehn twists along the boundary components. Hence we may assume that $m_1$ and $m_2$ differ from one of the two markings in Figure~\ref{fig:standard-markings-A-move} only by Dehn twists along the boundary components.

    \begin{figure}[ht]
        \centering
        \includegraphics[width=0.5\linewidth]{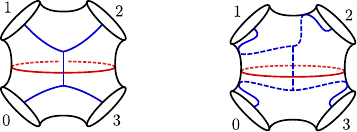}
        \caption{Two standard markings on $S_{0,4}$}
        \label{fig:standard-markings-A-move}
    \end{figure}

If $m_1$ and $m_2$ are obtained from the same marking in Figure~\ref{fig:standard-markings-A-move} by Dehn twists along the boundary components, then the required commutativity of the square follows from the $2$-cell (Boundary twist relations).

It remains to consider the case where $m_1$ and $m_2$ arise from different markings in Figure~\ref{fig:standard-markings-A-move}. Without loss of generality, assume that $m_1$ and $m_2$ correspond to the left and right markings, respectively. Then Figure~\ref{fig:condition-3} gives the required commutative square.

    \begin{figure}
        \centering
        \includegraphics[width=0.65\linewidth]{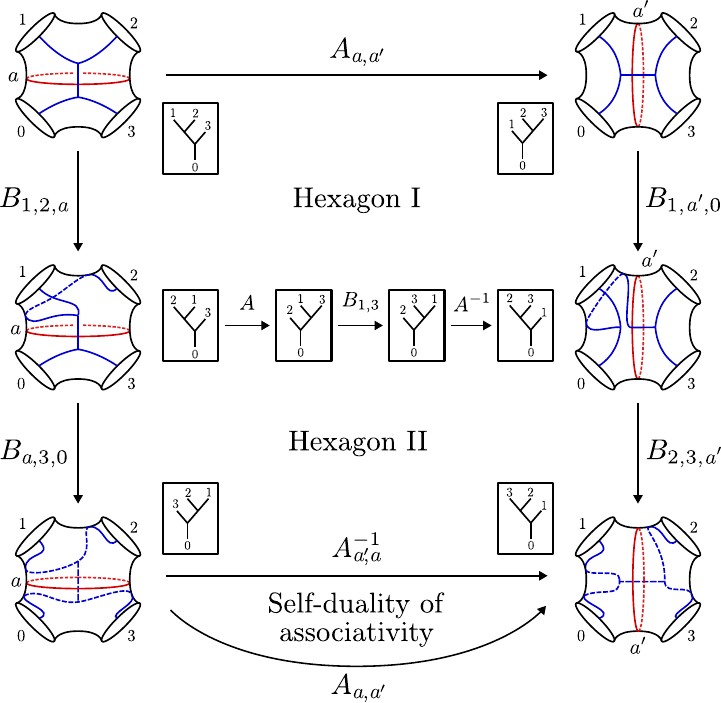}
        \caption{Commutative diagram checking condition (3)}
        \label{fig:condition-3}
    \end{figure}

A very similar argument works for the lifts of the S-move, and can be found in \cite[Section~7.7]{bk_marked_surfaces}. Here, the contractible square is obtained using the (Commutativity of disjoint moves) and the genus-one relation for $(g,n)=(1,1)$ given by \eqref{eq:g1.1}.

\emph{Checking condition~$(4)$:} now we must show that the boundary of each standard $2$-cell of $\pantscpx(S)$ lifts to a contractible loop in $\markS$. In fact, even more, we have that each such boundary can be lifted to the boundary of a standard $2$-cell in $\markS$.

The correspondence is as follows:
\begin{itemize}
    \item The boundary of the $2$-cell $(3A)$ can be lifted to the boundary of the $2$-cell corresponding to the hexagon relations; see Figure~\ref{fig:hexagons};
    \item The boundary of the $2$-cell $(5A)$ can be lifted to the boundary of the $2$-cell corresponding to the pentagon relation; see Figure~\ref{fig:pentagon};
    \item The boundary of the $2$-cell $(3S)$ can be lifted to the boundary of the $2$-cell corresponding to the Genus-one relation \eqref{eq:g1.2}; see Figure~\ref{fig:G2-cell-G1-2};
    \item The boundary of the $2$-cell $(6AS)$ can be lifted to the boundary of the $2$-cell corresponding to the Genus-one relation for $(g,n)=(1,2)$; see Figure~\ref{fig:G2-cell-NS-complex-of-markings};
    \item The boundary of the $2$-cell $(C)$ can be lifted to the boundary of the $2$-cell corresponding to the relation (Commutativity of disjoint moves).
\end{itemize}

Thus the boundary of each $2$-cell of $\pantscpx(S)$ lifts to a loop bounding a $2$-cell of $\markS$. Hence condition~$(4)$ holds. Proposition~\ref{prop: simply-connect domain from target} now implies that $\markS$ is connected and simply connected.
\end{proof}

\begin{lemma}\label{lemma:aux-for-A-move}
    Let $\phi:S'\to S_{\forkone}$ be a diffeomorphism witnessing an $A$-move, as in Definition~\ref{def: complex of markings}, and let $\{c,c'\}$ denote the boundary components of $S'$ that are mapped to the boundary components $\{1,2\}$ of $S_{\forkone}$. Then there exists a diffeomorphism $\phi':S'\to S_{\forkone}$ witnessing the same $A$-move such that the boundary components $\{c,c'\}$ are mapped to $\{0,3\}$.
\end{lemma}

\begin{proof}
    Let $(C,m)\xra{A_{a,b}}(C',m')$ be an $A$-move on $\mathcal{M}(S)$ and let $m|_{S'}$ and $m'_{S'}$ denote representatives of the markings $m$ and $m'$ restricted to $S'$ that, together with $\phi$, witnesses the $A$-move. 

    Let $z:S_{\forkone}\to S_{\forkone}$ denote a difffeomorphism that fixes the standard marking and on the boundary components gives the permutation $(13)(20)$ (this can be constructed using the standard marking and the diagrammatic representation of mapping classes).
    
    Define the diffeomorphism
        \[\phi'\coloneqq z_S\circ \phi.\]
    By construction,  $\phi'\circ m|_{S'}$ is still the standard marking on $S_{\forkone}$, and $\alpha\circ\phi'\circ m'_{S'}$ is still the standard marking on $S_{\forktwo}$. Hence $\phi'$ also witnesses this $A$-move and, by construction, takes the boundary components $\{c,c'\}$ to the boundaries $\{0,3\}$ on $S_{\forkone}$.
    % 
    % Let $z:S_{\forkone}\to S_{\forkone}$ be the standard diffeomorphism which rotates the surface and swaps the boundary components $\{1,2\}$ with $\{0,3\}$ preserving the standard marking (this can be constructed explicitly using the diagrammatic representation of mapping classes). Define the new diffeomorphism by
    %     \[\phi'=z\circ \phi.\]
    % Since $z$ preserves the standard marking, we still have that
    %     \[\phi'(a)=c \quad \text{and} \quad \phi'\circ m= m_{\forkone}.\]
    % All that remains is to check that
    %     \[\alpha\circ\phi'(b)=c' \quad \text{and} \quad \alpha\circ\phi'\circ m'= m_{\forktwo}.\]
    % Substituting, this becomes
    %     \[\alpha\circ z\circ \phi(b)=c' \quad \text{and}\quad \alpha\circ z\circ\phi\circ m'= m_{\forktwo}.\]
\end{proof}

\bibliographystyle{plain}
\bibliography{bibliography.bib}
\end{document}